\documentclass[a4paper,10pt]{article}
\usepackage{amssymb,amsmath,amsthm}
\usepackage{graphicx}

\usepackage{amsfonts}
\usepackage{latexsym}
\usepackage[english]{babel}
\usepackage{hyperref}
\usepackage{xcolor}
\usepackage{todonotes}

\numberwithin{equation}{section}

\newtheorem{theorem}{Theorem}[section]
\newtheorem{proposition}{Proposition}[section]
\newtheorem{lemma}{Lemma}[section]
\newtheorem{corollary}{Corollary}[section]

\newtheorem{remark}{Remark}[section]
\newtheorem{remarks}{Remark}[section]
\newtheorem{definition}{Definition}[section]
\newcommand{\be}{\begin{equation}}
\newcommand{\ee}{\end{equation}}

\newcommand{\e}{\varepsilon}

\newcommand{\R}{\mathbb R}
\newcommand{\C}{\mathbb C}

\newcommand{\Z}{\mathbb Z}

\newcommand{\N}{\mathbb N}
\newcommand{\T}{\mathbb T}

\renewcommand{\b }{\beta }

\newcommand{\ii }{{\rm i} }

\renewcommand{\l }{\lambda }

\begin{document}

\title{{\bf Normal form coordinates for the Benjamin-Ono equation 
having expansions in terms of pseudo-differential operators
}}



\author{
Thomas Kappeler\footnote{Supported in part by the Swiss National Science Foundation.}  ,
Riccardo Montalto\footnote{Supported in part by the Swiss National Science Foundation and INDAM-GNFM.} 
}

\maketitle

\noindent
{\bf Abstract.} 
Near an arbitrary finite gap potential we construct real analytic, canonical coordinates
for the Benjamin-Ono equation on the torus having the following two main properties:  (1) up to a remainder term, which is smoothing to any given order, 
the coordinate transformation is a pseudo-differential operator of order 0 with principal part given by a modified Fourier transform (modification by a phase factor)
and  (2) the pullback of the Hamiltonian of the Benjamin-Ono is in normal form up to order three and
the corresponding Hamiltonian vector field admits an expansion in terms of para-differential operators.
Such coordinates are a key ingredient for studying the stability of finite gap solutions
of the Benjamin-Ono equation under small, quasi-linear perturbations.

\medskip


\medskip



\tableofcontents


\section{Introduction}\label{introduzione paper}
\label{1. Introduction}
The goal of this paper is to construct canonical coordinates for the Benjamin-Ono (BO) equation on the torus $\T := \R/ 2\pi \Z$,
   \begin{equation}\label{1.1}
    \partial_t u =  \partial_x ( |\partial_x| u - u^2) \, ,    \qquad u = u(t, x) \in \R\, ,
   \end{equation}
   which are well suited for studying the stability of finite gap solutions of \eqref{1.1} 
   under semilinear and quasilinear perturbations. 
   Here $|\partial_x|$ denotes the Fourier multiplier, defined by 
   $$
   |\partial_x| : \sum_{n\in\Z} q_n e^{\\i n x} \mapsto  \sum_{n\in\Z}|n| q_n e^{\\i n x} \, . 
   $$
   Equation \eqref{1.1} was introduced  by Benjamin \cite{Ben} and Davis $\&$ Acrivos \cite{DA} as 
a model for internal gravity waves at the interface of two fluids in a special regime. 
We refer to \cite{S} for a derivation of  \eqref{1.1} and a comprehensive survey about results of its solutions.
Optimal results on the well-posedness of 
 \eqref{1.1} on the Sobolev spaces $H^s = \{ q \in H^s_\C: \, q \mbox{ real valued}\}$  
have been recently obtained in \cite{GKT1}, saying that 
   equation \eqref{1.1} is well-posed on $H^s$ for $s >  - 1/2$ and ill-posed for $s \le  -1/2$.
 Here,  for any $s \in \R,$
   \begin{equation}\label{def Hs intro}
   H^s_\C \equiv H^s(\T, \C) : = \big\{ q = \sum_{n \in \Z} q_n e^{ \ii n x} :  \,
   \| q \|_s < \infty\big\} \,,
  \quad \| q \|_s = \big( \sum_{n \in \Z} \langle n \rangle^{2 s} |q_n|^2 \big)^{\frac12}\,,
  \,\,\, \langle n \rangle := {\rm max}\{1, |n| \}\,.
   \end{equation}
   To state our main results, we first need to make preliminary considerations and introduce some notations. 
   It is well known that equation \eqref{1.1} can be written in Hamiltonian form, $\partial_t u = \partial_x \nabla { H}^{bo}$ 
   where $\partial_x $ is the Gardner Poisson structure 
   and where by \eqref{def gradient} below, $\nabla { H}^{bo}$ denotes the $L^2$-gradient of the BO Hamiltonian
   \begin{equation}\label{BO Hamiltonian}
   H^{bo}(q) := \frac{1}{2 \pi}\int_0^{2\pi} \big(\,  \frac12 (|\partial_x|^{\frac 12} q)^2 -  \frac13q^3 \big) d x \, ,
   \end{equation}
   whose domain of definition is the energy space $H^{1/2}$.
   One verifies that $\int_0^{2\pi} u(t, x)\, d x$ is a prime integral for equation \eqref{1.1}. Without loss of generality, we restrict
   our attention to the case where $u$ has zero mean value (cf. Appendix \ref{Birkhoff map for BO}), i.e., we consider solutions $u(t, x)$ of \eqref{1.1} in 
   $H_0^s$ with $s >  -1/2$ where for any $s \in \R,$
   \begin{equation}
   H_0^s = \big\{ q \in H^s : \,\,  \int_0^{2\pi} q(x)\, d x = 0 \big\}\,.
   \end{equation} 
   An important property of equation \eqref{1.1} is that it admits Birkhoff coordinates -- see Appendix \ref{Birkhoff map for BO} for a review. 
   It means that there are (complex) coordinates $\zeta_n = \zeta_n(q)$, $n \ge 1$, defined on $L^2_0 \equiv H_0^0$, 
   so that when expressed in these coordinates, equation \eqref{1.1} takes the form 
   \begin{equation}\label{BO for zeta_n}
   \dot \zeta_n = \ii \, \omega_n^{bo} \zeta_n \,, \qquad \forall n \ge 1\,,
   \end{equation}
   where $\omega_n^{bo}$, $n \ge 1$, denote the BO frequencies (cf. \eqref{BO frequencies} below).
   To be more precise, introduce for any $s \in \R$ the sequence space,
   $$
   h^s_{0,\C} \equiv h^s(\Z \setminus \{ 0 \}, \C) = \big\{ w = (w_n)_{n \neq 0} \subset \C : \| w \|_s < \infty \big\}\,, \qquad
     \| w \|_s := \big( \sum_{n \neq 0} |n|^{2 s} |w_n|^2 \big)^{\frac12}
   $$
   and its real subspace  $ h_0^s := \big\{ (w_n)_{n \neq 0} \in h^s_{0,\C} : w_{- n} = \overline{w}_n \,\, \forall n \geq 1 \big\}$
   and define the weighted complex coordinates 
   \begin{equation}\label{definition z_n}
   z_n(q)  := \sqrt{n} \,  \zeta_n(q) \, ,  \quad  z_{-n}(q) := \sqrt{n} \, \overline{\zeta_n}(q) \, , \qquad \forall n \ge 1\,, 
   \end{equation}
   where $\sqrt{\cdot} \equiv \sqrt[+]{\cdot}\,$ denotes the principal branch of the square root. 
   It then follows from \eqref{BO for zeta_n} that equation \eqref{1.1}, 
   when expressed in the coordinates $z_n$, $n \ne 0$, takes the form
   \begin{equation}\label{BO in Birkhoff}
   \dot z_n = \ii \,  \omega_n^{bo} z_n, \quad \forall n \ne 0, \qquad \qquad
    \omega_{-n}^{bo} :=  - \omega_n^{bo} , \quad \forall \, n \ge 1.
   \end{equation}
   It folllows from Theorem \ref{Theorem Birkhoff coordinates BO} in Appendix \ref{Birkhoff map for BO} that the transformation
   \begin{equation}\label{def Phi^bo}
   \Phi^{bo} : L^2_0 \equiv H^0_0 \to h^0_0, \quad q \mapsto (z_n(q))_{n \neq 0}\,,
   \end{equation}
   is canonical in the sense that 
   $$
   \{ z_n, z_{- n} \}(q) = \frac{1}{2\pi}\int_0^{2 \pi}  \partial_x (\nabla z_n)  \, \nabla z_{- n}\, d x =  - \ii n, \qquad \forall n \ne 0\,,
   $$
   whereas the brackets between all other coordinate functions vanish, and has the property that for any $s \geq 0$, 
   its restriction to $H^s_0$ is a real analytic diffeomorphism with range $h^s_0$, $\Phi^{bo} : H^s_0 \to h^s_0$. 
   Here $\nabla z_n$ denotes the $L^2_0$-gradient of $z_n: L^2_0 \to \C$.\footnote{For notational convenience, whenever the context allows, we  denote both the $L^2_0$-gradient and the $L^2$-gradient by $\nabla$.}
   For notational convenience, we refer to $\Phi^{bo}$ as well as to the map $\Phi$ of Theorem \ref{Theorem Birkhoff coordinates BO} as Birkhoff maps.
   In terms of the coordinates $z_n(q), n \ne 0,$ also referred to as complex Birkhoff coordinates, the action variables
   $ I_n(q)$ are defined by
   \begin{equation}\label{definition actions}
  I_n(q)  := | \zeta_n(q)|^2 = \frac{1}{n} z_n(q) z_{- n}(q) \ge 0\,, \quad \forall n \ge 1\,.
   \end{equation}
   The sequences $I(q)= (I_n(q))_{n \ge1}$ fill out the whole positive quadrant $\mathcal Q_+(\ell^{1,1})$ of  $\ell^{1,1}$,
   where for any $r \ge 0$,
   $$
   \mathcal Q_+(\ell^{1,r}) :=  \{ (x_n)_{n \geq 1} \in \ell^{1,r} : \, x_n \ge 0 \ \forall n \ge 1 \} 
   $$ 
   and $\ell^{1,r}$ denotes the weighted  $\ell^1$-space,
   \begin{equation}\label{weighted ell_1}
   \ell^{1,r} \equiv \ell^{1,r}(\N, \R) := \{ (x_n)_{n \geq 1} \subset \R : \, \sum_{n =1}^\infty n^r |x_n| < \infty \}\,, \qquad \N \equiv \Z_{\geq 1}.
   \end{equation}
    A key feature of the Birkhoff map $\Phi^{bo}$ is that the BO Hamiltonian, expressed in the coordinates $z_n,$ $n \ne 0,$
   $$
    {H}^{bo} \circ \Psi^{bo} : h^1_0 \to \R\,, \qquad \Psi^{bo} := (\Phi^{bo})^{- 1}\,,
   $$
   is a real analytic function ${\cal H}^{bo}$ of the actions $I$ alone. More precisely (cf. Theorem \ref{Theorem Birkhoff coordinates BO},  \cite[Proposition 8.1]{GK1}), 
   $$
   {\cal H}^{bo } : \mathcal Q_+(\ell^{1, 3}) \to \R , \  I= (I_n)_{n \ge1} \mapsto \sum_{n=1}^\infty n^2 I_n-\sum_{n=1}^\infty\big(\sum_{k=n}^\infty I_k\big)^2 \, .
   $$
   The BO frequencies are then defined by (cf. Theorem \ref{Theorem Birkhoff coordinates BO})
   \begin{equation}\label{BO frequencies}
   \omega_n^{bo}(I) := \partial_{I_n} {\cal H}^{bo}(I) = n^2 -2 \sum_{k=1}^\infty \min \{n, k \} I_k , 
   \qquad   \omega_{-n}^{bo}(I) := -  \omega_n^{bo}(I) ,  \qquad \qquad \forall \, n \ge 1.
   \end{equation}
   Furthermore, it is well known that the Hamiltonian $H^{mo}(q):= \frac{1}{2\pi} \int_0^{2\pi} \frac 12 q(x)^2 d x$, referred to as the moment, 
   is a prime integral of \eqref{1.1}.
   When expressed in the coordinates $z_n$, $n \ne 0$, it is a real analytic function of the actions alone, 
   given by (cf. Theorem \ref{spec Lax operator}(i))
   \begin{equation}\label{formula moment}
  H^{mo} \circ \Psi^{bo} =    \sum_{n \ge 1} n \, I_n =    \sum_{n \ge 1}  z_n z_{-n}\, .
   \end{equation}
   Finally,  the differential  $d_0 \Phi^{bo} : L^2_0 \to h^0_0$ of  $\Phi^{bo}$ at $q = 0$ and its inverse are given by 
   (cf. Lemma \ref{results finite gap potentials}(iii), Theorem \ref{Theorem Birkhoff coordinates BO} (NF4)),
   \begin{equation}\label{differential at 0}
   d_0 \Phi^{bo} =   - \mathcal F  \, , \qquad
  d_0 \Psi^{bo} = (d_0 \Phi^{bo})^{-1} =- {\cal F}^{- 1} \, ,
  \end{equation}
  where $\mathcal F$ denotes the Fourier transform, defined for any $s  \in \R$ by
   \begin{equation}\label{Fourier transform}
   {\cal F} : H^s_{0} \to h^s_0,  \quad q \mapsto (q_n)_{n \neq 0}, \qquad q_n :=   \frac{1}{2\pi} \int_0^{2\pi} q(x) e^{-  \ii n x}\, d x\,,
   \quad \forall \, n \ne 0.
   \end{equation} 
For any nonempty, finite subset $S_+ \subseteq \N$, let $ S_+^\bot := \N \setminus S_+$ and define
   \begin{equation}\label{definition S}
    N_S:= \max S , \qquad  S := S_+ \cup (- S_+) ,  \qquad S^\bot := S_+^\bot \cup (-S^\bot_+)\,.
   \end{equation}
    We denote by $M_S \subset L^2_0$ the real analytic manifold of $S$-gap potentials (cf. Appendix \ref{spectral theory}),
   $$
   M_S := \big\{ q \in L^2_0 \, :  \, z_n(q) = 0 \,\, \forall \, n \in S^\bot  \big\}, \quad
   $$
   and by $M_S^o$ the open subset of $M_S$, consisting of proper $S$-gap potentials,
   $$
   M_S^o := \{ q \in M_S \, : \, z_n(q) \ne 0 \ \, \forall \, n \in S \}\,.
   $$
   The manifold $M_S$ is contained in $\cap_{s \geq 0} H^s_0$ and hence consists of $C^\infty$-smooth potentials  
   and $M_S^o$ can be parametrized by the action-angle coordinates 
   $ \theta_S = (\theta_k)_{k \in S_+} \in \T^{S_+}$,  $I_S = (I_k)_{k \in S_+} \in \R^{S_+}$, 
   \begin{equation}\label{definition Psi_S}
    \Psi_S: {\cal M}_S^o  \to M_S^o, \,\, (\theta_S, I_S) \mapsto \Psi^{bo}(z(\theta_S, I_S))\,,
   \qquad {\cal M}_S^o := \T^{S_+} \times \R^{S_+}_{> 0},
   \end{equation}
   where $z(\theta_S, I_S) = (z_n(\theta_S, I_S))_{n \ne 0}$ is given by 
   $$
    z_{\pm n}(\theta_S, I_S) := \sqrt{n I_n} e^{\pm \ii \theta_n}, \quad \forall n \in S_+, \qquad 
   z_n(\theta_S, I_S) = 0, \quad \forall n \in S^\bot .
   $$
 Introduce for any $s \in \R$,
   $$
  h^s_{\bot, \C} := h^s(S^\bot, \C)\,, \qquad  
  h^s_\bot := \big\{ z_\bot = (z_n)_{n \in S^\bot} \in h^s_{\bot, \C} : \, z_{- n} = \overline{z}_n, \quad \forall n \in S^\bot \big\}, 
   $$
and  the partial Fourier transforms $\mathcal F^\pm_{N_S} : L^2_\C \mapsto h^0_{\bot, \C}$, defined for $q \in L^2_\C$ and $n \in S^\bot$ by
\begin{equation}\label{def partial Fourier}
(\mathcal F^+_{N_S}[q])_n := 
\begin{cases}  q_n  \quad \text{if } n \ge  N_S + 1\\ 0 \ \,  \quad \text{if }  n \in S^\bot\setminus [N_S + 1, \infty) 
\end{cases},
\ \ 
(\mathcal F^-_{N_S}[q])_n := 
\begin{cases}  q_{n}  \quad \text{if } n \le  - N_S - 1  \\ 0 \ \,  \quad \text{if }  n \in S^\bot\setminus (- \infty, - N_S - 1] 
\end{cases}.
\end{equation}
For notational convenience, we denote by $( \mathcal F^\pm_{N_S})^{-1}: h^0_{\bot, \C} \to L^2_\C$ the maps
\begin{equation}\label{def inverse partial Fourier}
( \mathcal F^+_{N_S})^{-1} :  z_\bot  \mapsto \sum_{n \ge N_S + 1}  z_n e^{\ii nx}\,, \qquad
( \mathcal F^-_{N_S})^{-1} :  z_\bot  \mapsto \sum_{n \ge N_S + 1} z_{-n} e^{- \ii nx} 
\end{equation} 
   and view $ {\cal M}_S^o \times h^s_\bot$, $s \in \R$, as a subset of $h^s_0$.
   The elements of $ {\cal M}_S^o \times h^s_\bot$ are denoted by
   $$
   {\frak x} := (\theta_S, I_S, z_\bot) , \qquad \quad  \theta_S := (\theta_n)_{n \in S_+}, \quad I_S := (I_n)_{n \in S_+}, \quad z_\bot := (z_n)_{n \in S^\bot}\,. 
   $$
   It is endowed by the standard Poisson bracket (cf. \cite[(6.9)]{GK1}), given by 
   $$
   \{ I_n, \theta_n \} = 1, \,\,\, \forall n \in S_+, \qquad \{ z_n, z_{- n} \} = -  \ii n, \,\,\, \forall n \in S^\bot,
   $$
   whereas the brackets between all other coordinate functions vanish.
   For any $s \in \R$ and any given point $\frak x \in {\cal M}_S^o \times h^s_\bot$, denote by $E_s$  the tangent space of ${\cal M}_S^o \times h^s_\bot$. 
   Note that $E_s$ is independent of $\frak x$ and given by
   $E_s = \R^{S_+} \times \R^{S_+} \times h^s_\bot$. Its elements are denoted by  $\widehat{\frak x} = (\widehat \theta_S, \widehat I_S, \widehat z_\bot)$.
    Furthermore, for any $k \ge 1, $ $\partial_x^{-k} : H^s_{ \C} \to H^{s+k}_{0,\C}$ is  the bounded linear operator, defined by
$$
\partial_x^{-k}[e^{ \ii nx}] = \frac{1}{( \ii n)^k} e^{\ii nx}\, , \quad \forall n \ne 0\,,  
\qquad \quad   \partial_x^{-k}[1] = 0\,.
$$   
With $D:= -\ii \partial_x $, one then obtains for any $n \ge 1$,
$$
D^{-k} [e^{ \ii nx}] = |\partial_x|^{-k}[e^{ \ii nx}] \, , \qquad (-D)^{-k} [e^{ -\ii nx}] = |\partial_x|^{-k}[e^{ - \ii nx}] \, .
$$
Finally, the standard inner products on $L^2_\C$ and on $h^0_{0, \C}$ are defined for any $f, g \in L^2_\C$ and $z, w \in h^0_{0, \C}$ by
\begin{equation}\label{standard inner products}
\langle f | g \rangle \equiv \langle f | g \rangle_{L^2_\C} = \frac{1}{2\pi} \int_0^{2\pi} f(x) \overline{g(x)} d x \, , \qquad 
\langle z | w \rangle \equiv \langle z | w \rangle_{h^0_{0, \C}}  = \sum_{n \ne 0} z_n \overline w_n  \, .
\end{equation}
In addition, we introduce bilinear forms on $L^2_\C$ and on $h^0_{0, C}$, defined for any $f, g \in L^2_\C$
and $z, w \in h^0_{0, \C}$ by
\begin{equation}\label{complex bilinear forms}
\langle f , g \rangle \equiv \langle f , g \rangle_{L^2_\C} = \frac{1}{2\pi} \int_0^{2\pi} f(x) g(x) d x\,,  \qquad 
\langle z , w \rangle \equiv \langle z , w \rangle_{h^0_0}  = \sum_{n \ne 0} z_n w_{-n} \,.
\end{equation}
Note that $\langle \cdot \, , \, \cdot \rangle $ and $\langle \cdot \,  | \, \cdot \rangle$ coincide on the real Hilbert spaces $L^2$ and $h^0_{0}$.
In the sequel, restrictions of $\langle \cdot \, , \, \cdot \rangle $ and $\langle \cdot \, | \, \cdot \rangle$ to subspaces  and extensions
as dual pairings will be denoted in the same way. The gradient of $C^1$-functionals $F: L^2_\C \to \C$ and $G: h^0_{0,\C} \to \C$, 
corresponding to $\langle \cdot \, , \, \cdot \rangle $, 
are denoted by $\nabla F$, respectively $\nabla G$. They are defined by
\begin{equation}\label{def gradient}
dF[\widehat f] = \langle \nabla F , \widehat f \rangle \, , \quad \forall \, \widehat f \in L^2_\C\, , \qquad
dG[\widehat z] = \sum_{n \ne 0} \widehat z_n \partial_{z_n} G = \langle  \nabla G, \widehat z \rangle \, , \quad \forall \, \widehat z \in h^0_{0, \C} \, ,
\end{equation}
where the nth component of $\nabla G $ is given by $(\nabla G)_n = \partial_{z_{-n}}  G$.
Finally, for given Banach spaces $Y_1,$ $Y_2$, we denote by ${\cal B}( Y_1 , Y_2)$ the Banach space of bounded linear operators from $Y_1$ to $Y_2$, endowed with the operator norm. 
\begin{theorem}\label{modified Birkhoff map}
Let $S_+$ be a finite, nonempty subset of $\N$
and let ${\cal K}$ be a subset of ${\cal M}_S^o$ of the form $\T^{S_+} \times {\cal K}_1$ where  ${\cal K}_1$ is a compact subset of $\R^{S_+}_{> 0}$. 
Then there exists an open bounded neighbourhood ${\cal V}$ of ${\cal K} \times \{ 0 \}$ in ${\cal M}_S^o \times h^0_\bot$ 
and a canonical real analytic diffeomorphism 
$\Psi : {\cal V} \to \Psi({\cal V}) \subseteq L^2_0$, satisfying
$$
\Psi(\theta_S, I_S, 0) = \, \Psi_S(\theta_S, I_S), \qquad \forall (\theta_S, I_S, 0) \in {\cal V} \, ,
$$
(with $ \Psi_S$ given by \eqref{definition Psi_S})
and having the property that for any $s \in \Z_{\geq 0}$, $\Psi : {\cal V} \cap ({\cal M}_S^o \times h^s_\bot) \to H^s_0$ 
is a real analytic diffeomorphism onto its image in $H^s_0$, so that the following holds:
\begin{description}
\item[({\bf AE1})] For any $\frak x = (\theta_S, I_S, z_\bot) \in \mathcal V$ and any integer $N \ge 1$, $\Psi(\frak x)$ 
admits an asymptotic expansion  of the form  
$ \Psi_S(\theta_S, I_S) + \mathcal {OP}_N( \frak x; \Psi)+ {\cal R}_{N}(\frak x; \Psi) $ 
where $\mathcal {OP}_N( \frak x; \Psi)$ is given by 
\begin{equation}\label{expansion Psi bo}
\begin{aligned}
 \big( - g_\infty   +  \sum_{k = 1}^N  a^+_{k}(\frak x; \Psi) D^{- k} \, \big) [ (\mathcal F^+_{N_S})^{-1} z_\bot]
 \, + \,  \big( - {\overline{g_\infty}}  +   \sum_{k = 1}^N  a^-_{k}(\frak x; \Psi) (-D)^{- k} \, \big)  [ (\mathcal F^-_{N_S})^{-1} z_\bot] \, ,
\end{aligned}
\end{equation}
with $g_\infty (x) \equiv g_\infty (x ; q)$ defined by
$$
g_\infty (x) := e^{D^{-1}q(x)} =  e^{\ii \partial_x^{-1}q(x)}  \, , \qquad q(x) :=  \Psi_S(\theta_S, I_S)\, .
$$ 
The coefficients $ a^\pm_{k}(\frak x; \Psi) $ and the remainder ${\cal R}_{N}(\frak x; \Psi) $ satisfy
$$
 a^-_{k}(\frak x; \Psi) = \overline{a^+_{k}(\frak x; \Psi)} \, , \quad  \forall \, k \ge 1\, , \qquad  {\cal R}_{N}(\theta_S, I_S, 0; \Psi) = 0\, ,
$$
and for any $s \in \Z_{\geq 0}$  and $ k \ge 1$, 
$$
a_k^{\pm}(\cdot \, ; \, \Psi) : {\cal V} \to H^s_\C \, ,
\qquad {\cal R}_N(\cdot \, ; \, \Psi) : {\cal V} \cap \big( {\cal M}_S^o \times h^s_\bot \big) \to H^{s + N +1}\, , 
$$
are real analytic maps\footnote{In this context, $ H^s_\C$ is considered as a $\R$-Hilbert space. This convention for real analytic maps is used throughout the paper.}
satisfying the tame estimates of Theorem \ref{modified Birkhoff map 2} below. 
\item[({\bf AE2})]
For any $\frak x  \in {\cal V}$,
the transpose $d \Psi(\frak x )^\top$ 
(with respect to the standard inner products) of the differential
$d \Psi(\frak x) : E_1 \to H^1_0$ is a bounded operator $d \Psi(\frak x )^\top : H^1_0 \to E_1$. 
For any integer $N \ge 1$, $d \Psi(\frak x )^\top$ admits an expansion of the form 
$\mathcal {OP}_N(\frak x; d\Psi^\top) +  {\cal R}_{N}(z; d\Psi^\top) $
with 
\begin{equation}\label{expansion d Psi top bo}
\mathcal {OP}_N(\frak x; d\Psi^\top) = \big( 0, \, 0, \,  \mathcal {OP}^+_N(\frak x; d\Psi^\top) + \mathcal {OP}^-_N(\frak x; d\Psi^\top) \big) \, ,
\end{equation}
where for any $\widehat q \in H^1_0$, $\mathcal {OP}^\pm_N(\frak x; d\Psi^\top)[\widehat q]$ are defined  by
$$
\begin{aligned}
\mathcal {OP}^+_N(\frak x; d\Psi^\top) [\widehat q]  
&=    {\cal F}^+_{N_S} \circ \big( - \overline{g_\infty}   +  \, \sum_{k = 1}^N  a^+_{k}(\frak x; d\Psi^\top) D^{- k} \, \big) [\widehat q] \\
&\, + \, {\cal F}^+_{N_S} \circ \sum_{k = 0}^N  \mathcal A^+_k(\frak x; d \Psi^\top)[\widehat q] \, D^{- k} [( \mathcal F^+_{N_S})^{-1} z_\bot] \, , 
\end{aligned}
$$
$$
\begin{aligned}
\mathcal {OP}^-_N(\frak x; d\Psi^\top) [\widehat q]  &= 
   {\cal F}^-_{N_S} \circ \big( - g_\infty   + \,  \sum_{k = 1}^N  a^-_{k}(\frak x; d\Psi^\top) (-D)^{- k} \, \big) [\widehat q]  \\
& \, + \, {\cal F}^-_{N_S} \circ \sum_{k = 0}^N  \mathcal A^-_k(\frak x; d \Psi^\top)[\widehat q] \, (-D)^{- k} [( \mathcal F^-_{N_S})^{-1} z_\bot]  \, .\\
\end{aligned}
$$
Furthermore, for any $\widehat q \in H^1_0$,
$$
a^-_{k}(\frak x; d\Psi^\top) = \overline{a^+_{k}(\frak x; d\Psi^\top)}\, , \quad \forall \, k \ge 1, \qquad
 \mathcal A^-_k(\frak x; d \Psi^\top)[\widehat q]  = \overline{ \mathcal A^+_k(\frak x; d \Psi^\top)[\widehat q] } \, , \quad \forall \, k \ge 0 ,
$$
and for any $s \in \N$,
$$
a_k^\pm (\cdot \, ; \,  {d \Psi^\top}) : {\cal V} 
\to H^s_\C\,, \qquad
\mathcal A_k^\pm (\cdot \, ; \,  {d \Psi^\top}) : {\cal V} 
\to  {\cal B}(H^1_0, H^s_\C)\,, 
$$
$$
 {\cal R}_N( \cdot\, ; \, d \Psi^\top) : {\cal V} \cap ({\cal M}_S^o \times h^s_\bot) \to {\cal B}(H^s_0, E_{s + N +1}), 
$$
are real analytic maps, satisfying the tame estimates of Theorem \ref{modified Birkhoff map 2} below. 
\item[({\bf AE3})] The Hamiltonians ${\cal H} := H^{bo} \circ \Psi$ (cf. \eqref{BO Hamiltonian}) and ${\cal H}^{mo} := H^{mo} \circ \Psi$ (cf. \eqref{formula moment}), 
$$
{\cal H} : {\cal V} \cap (\mathcal M_S^o \times h^1_\bot) \to \R \, , \qquad 
{\cal H}^{mo} : {\cal V} \cap (\mathcal M_S^o \times h^0_\bot) \to \R \, , 
$$  
are in normal form up to order three. More precisely,  
$$
\begin{aligned}
{\cal H}(\theta_S, I_S, z_\bot) & = {\cal H}^{bo}(I_S, 0) + \sum_{n \in S_+^\bot} \Omega_n(I_S) z_n z_{- n} + {\cal P}(\theta_S, I_S, z_\bot) \, , \\
{\cal H}^{mo}(\theta_S, I_S, z_\bot) & =  \sum_{n \in S_+} n I_n+   \sum_{n \in S_+^\bot} z_n z_{- n} + {\cal P}^{mo}(\theta_S, I_S, z_\bot),
\end{aligned}
$$
where for any $n \in S^\bot$, $\Omega_n(I_S) : = \frac{1}{ n} \omega_n^{bo}(I_S, 0)$,  and where
 $ {\cal P}$, $\mathcal P^{mo} : {\cal V} \cap (\mathcal M_S^o \times h^1_\bot) \to \R$ are real analytic 
and satisfy 
$$
{\cal P}(\theta_S, I_S, z_\bot) \, ,  \   {\cal P}^{mo}(\theta_S, I_S, z_\bot)  \ = \, O( \| z_\bot \|_1 \| z_\bot \|_0^2) , 
$$ 
and where $ \omega_n^{bo}$, $n \ne 0$, denote the BO frequencies, introduced in \eqref{BO frequencies}. 

\smallskip
\noindent
{\em Expansion  of $\nabla {\cal P}(\frak x)$:} for any integer $N \ge 1$, 
there exists an integer $\sigma_N \ge N$ (loss of regularity) so that the $L^2$-gradient $\nabla {\cal P}(\frak x)$ of ${\cal P}$ with components 
$\nabla_{\theta_S} {\cal P}$, $\nabla_{I_S} {\cal P}$, and $\nabla_{z_\bot} {\cal P}$ admits an expansion of the form 
$\nabla {\cal P}(\frak x) 
= \big(\, 0, \, 0, \, \mathcal {OP}_N(\frak x; \nabla \mathcal P) \big)  \, + \,  {\cal R}_N(\frak x; \nabla {\cal P}) $, 
$$
\mathcal {OP}_N(\frak x; \nabla \mathcal P) = 
 {\cal F}^+_{N_S} \circ \big(\, \sum_{k = 0}^N  T_{a_k^+( \frak x; {\nabla {\cal P})}} \, D^{- k} \, \big)  [({\cal F}^+_{N_S})^{- 1}z_\bot]   
 \, + \,  {\cal F}^-_{N_S} \circ \big( \, \sum_{k = 0}^N  T_{a_k^-( \frak x; {\nabla {\cal P})}} \, (-D)^{- k}\, \big)  [({\cal F}^-_{N_S})^{- 1}z_\bot]   
$$
where for any $  k \ge 0$, $a_k^-( \frak x \, ; \, {\nabla {\cal P}}) = \overline{a_k^+( \frak x \, ; \, {\nabla {\cal P}})}$
and for any integer $s \ge 0$,
$$
\begin{aligned}
& a_k^\pm( \cdot \, ; \, {\nabla {\cal P}}) : {\cal V} \cap ({\cal M}_S^o \times h^{s + \sigma_N}_\bot) \to H^s_\C \,,  \qquad
 {\cal R}_N( \cdot\, ; \, {\nabla {\cal P}}) : {\cal V} \cap \big( {\cal M}_S^o \times h^{s \lor \sigma_N}_\bot \big) \to E_{s + N + 1} \, ,
\end{aligned}
$$
are real analytic and satisfy the tame estimates of Theorem \ref{modified Birkhoff map 2} below.
Here $T_{a_k^\pm(\cdot \, ; \, \nabla {\cal P})}$ 
denotes the operator of para-multiplication with the function $a_k^\pm( \cdot \, ; \, \nabla {\cal P})$ (cf. Appendix \ref{appendice B}). \\
{\em Expansion  for $\nabla {\cal P}^{mo}(\frak x)$:}  the $L^2$-gradient $\nabla {\cal P}^{mo}(\frak x)$ admits an expansion
similar to the one of $\nabla {\cal P}(\frak x)$ with corresponding coefficients  $a_k^\pm ( \frak x \, ; \, \nabla {\cal P}^{mo})$ and remainder 
 ${\cal R}_N( \frak x \, ; \, \nabla {\cal P}^{mo})$.
\end{description}
\end{theorem}
\begin{remark}
It is a remarkable feature
of the expansions of $\Psi$ and $d\Psi^\top$, which turns out to be relevant for applications, that  for any $k \ge 1$,
the coefficients $a^\pm_k(\frak x; \Psi)$ and $a_k^\pm (\frak x ; d\Psi^\top)$ are $C^\infty$-smooth functions on $\T$ and that the linear operators
$\mathcal A_k^\pm( \frak x; d\Psi^\top)$ are $C^\infty$-smoothing. 
These regularity properties of the expansions of $\Psi$ and $d\Psi^\top$ are a consequence 
of the fact that the map 
$\Psi_L$, introduced in \eqref{definition Psi_L bo} below as the basic ingredient for the construction of the map $\Psi$, is given by the linearization of the Birkhoff map
in normal direction at finite gap potentials and that such potentials are $C^\infty$-smooth (actually, even real analytic).
\end{remark}

In applications, it is of interest to know whether the coordinate transformation $\Psi$ preserves the reversible structure, defined by the maps
$S_{rev} : L^2_0 \to L^2_0$,  $(S_{rev} q)(x) := q(- x)$, and
${\cal S}_{rev} : {\cal M}_S^o \times h^0_\bot \to {\cal M}_S^o \times h^0_\bot$ where 
\begin{equation}\label{definition reversible structure for actions angles}
 {\cal S}_{rev}(\theta_S, I_S, z_\bot) := (\theta^{rev}_S, I_S^{rev}, z_\bot^{rev})\,, \quad
\theta_n^{rev} = - \theta_n , \,\,\, I_n^{rev} = I_n, \,\,\, \forall n \in S_+\,, \quad z_n^{rev} = z_{- n}, \,\,\, \forall n \in S^\bot\,. 
\end{equation}
Note that for any $s \in \R_{\ge 0}$, $S_{rev} : H^s_0 \to H^s_0$ and ${\cal S}_{rev} : {\cal M}_S^o \times h^s_\bot \to {\cal M}_S^o \times h^s_\bot$ 
are linear involutions and that without loss of generality, the neighbourhood ${\cal V}$ of Theorem \ref{modified Birkhoff map} can 
be chosen to be invariant under the map ${\cal S}_{rev}$, i.e., ${\cal S}_{rev} ({\cal V}) = {\cal V}$. 

\medskip

\noindent 
 {\bf Addendum to Theorem \ref{modified Birkhoff map}.} \label{teo reversibilita}
 {\em  The maps $\Psi : {\cal V} \to L^2_0$,  $\Psi^{bo} : h^0_0 \to L^2_0$,  and  $({\cal F}^\pm_{N_S})^{- 1} : h^0_{\bot, \C} \to L^2_{0, \C}$ preserve
 the  reversible structure, i.e.,
  $$
  \Psi \circ {\cal S}_{rev} = S_{rev} \circ \Psi, \qquad \Psi^{bo} \circ {\cal S}_{rev} = S_{rev} \circ \Psi^{bo}\,, 
  \qquad ({\cal F}^-_{N_S})^{- 1} \circ {\cal S}_{rev} = S_{rev} \circ ({\cal F}^+_{N_S})^{- 1}\, ,
  $$
  and so do the maps in the asymptotic expansions {\bf (AE1)} ($\frak x \in \mathcal V$),
  $$
  a^+_k ( {\cal S}_{rev} {\frak x}; \Psi) =   S_{rev}  (a^-_k({\frak x}; \Psi) ) \,, \qquad 
  {\cal R}_N ( {\cal S}_{rev} {\frak x}; \Psi) =  S_{rev} ({\cal R}_N ({\frak x}; \Psi))\,,
  $$
  and the ones in the asymptotic expansions  {\bf (AE2)} ($\frak x \in \mathcal V \cap (\mathcal M_S^o \times h^1_\bot)$, $\widehat q \in H^1_0$), 
  $$
  a^+_k({\cal S}_{rev} \frak x; \, d \Psi^\top) =   { S}_{rev}  (a^-_k(\frak x; \, d \Psi^\top))\,, \qquad
  \mathcal  A^+_k({\cal S}_{rev} \frak x; \, d \Psi^\top) [ S_{rev} \, \widehat q ] =   { S}_{rev}  (\mathcal A^-_k (\frak x; \, d \Psi^\top) [ \widehat q ] ) \, , 
  $$
  $$
  {\cal R}_N({\cal S}_{rev} \frak x ; \, d \Psi^\top) [ S_{rev} \, \widehat q ] =  {\cal S}_{rev} ( {\cal R}_N( \frak x ; \, d \Psi^\top) [ \, \widehat q ] ) \,.
  $$
  Furthermore, the Hamiltonians $H^{bo}$, ${\cal H} = H^{bo} \circ \Psi$, and ${\cal P}$ are reversible, meaning that 
  $$
  H^{bo} \circ S_{rev} = H^{bo}\,, \qquad {\cal H} \circ {\cal S}_{rev} = {\cal H}\,, \qquad {\cal P} \circ {\cal S}_{rev} = {\cal P},
  $$
  and the maps in the asymptotic expansion in {\bf (AE3)} preserve the reversible structure,
   $$
 a_k^+ ( {\cal S}_{rev} {\frak x}; {\nabla \cal P}) =  S_{rev} (a^-_k(\frak x; {\nabla \cal P}) )\,, \qquad \forall \, \frak x \in \mathcal V \cap ({\cal M}_S^o \times h^{1 + \sigma_N}_\bot)\,,
 $$
 $$
  {\cal R}_N ( {\cal S}_{rev} \frak x; {\nabla \cal P}) =  {\cal S}_{rev} ( {\cal R}_{N} (\frak x; {\nabla \cal P}) )\, , \   \quad \forall \, \frak x \in \mathcal V \cap ({\cal M}_S^o \times h^{1 \lor \sigma_N}_\bot)\, . 
  $$
  Corresponding results hold for the Hamiltonians $ H^{mo}$, ${\cal H}^{mo} = H^{mo} \circ \Psi$, and ${\cal P}^{mo}$.
  }
  
  \medskip
  
 \noindent
 Theorem \ref{modified Birkhoff map 2} below states tame estimates for the map $\Psi $ as well as the gradient $\nabla {\cal P}$ of the remainder term ${\cal P}$ 
 in the expansion of ${\cal H}$ and the one of the remainder ${\cal P}^{mo}$ in the expansion of ${\cal H}^{mo}$. 
 Throughout the paper, the stated estimates for maps hold locally uniformly with respect to their arguments.
  \begin{theorem}\label{modified Birkhoff map 2}
  Let $N, l \in \N$. Then
  under the same assumptions as in Theorem \ref{modified Birkhoff map}, the following estimates hold:
  \begin{description}
  \item[(Est1)] For any $\frak x = (\theta_S, I_S, z_\bot) \in {\cal V}$, $k \ge 1$, \, $\widehat{\frak x}_1, \ldots, \widehat{\frak x}_l \in E_0$, $s \in \Z_{\geq 0}$,
  $$
  \begin{aligned}
  & \| a_k^\pm(\frak x; \Psi) \|_s \lesssim_{s, k} 1 
  \,,  \qquad
    \| d^l a_k^\pm(\frak x; \Psi)[\widehat{\frak x}_1, \ldots, \widehat{\frak x}_l]\|_s \lesssim_{s, k, l} \prod_{j = 1}^l \| \widehat{\frak x}_j\|_0\,. 
  \end{aligned}
  $$
  Similarly, for any $\frak x  \in {\cal V} \cap \big( {\cal M}_S^o \times h^s_\bot \big)$, \, $\widehat{\frak x}_1, \ldots, \widehat{\frak x}_l \in E_s$, $s \in \Z_{\geq 0}$,
  $$
  \| {\cal R}_N(\frak x; \Psi)\|_{s + N + 1} \lesssim_{s, N}  \| z_\bot \|_s\, , \qquad \qquad \qquad
  $$
  $$
   \| d^l {\cal R}_N(\frak x; \Psi)[\widehat{\frak x}_1, \ldots, \widehat{\frak x}_l]\|_{s + N + 1} 
   \lesssim_{s, N, l} \sum_{j = 1}^l \| \widehat{\frak x}_j\|_s \prod_{i \neq j} \| \widehat{\frak x}_i\|_0 + \| z_\bot \|_s \prod_{j = 1}^l \| \widehat{\frak x}_j\|_0\,.
  $$
  \item[(Est2)] For any $\frak x = (\theta_S, I_S, z_\bot) \in {\cal V} \cap (\mathcal M_S^o \times h^1_\bot)$, $ k \ge 1$, \, 
  $\widehat{\frak x}_1, \ldots, \widehat{\frak x}_l \in E_1$, $s \in \N$,
  $$
  \begin{aligned}
  &  \| a_k^\pm(\frak x; {d \Psi}^\top)\|_s \lesssim_{s, k} 1 +  \| z_\bot \|_1 \,, \qquad
   \| d^l a_k^\pm(\frak x; {d \Psi}^\top)[\widehat{\frak x}_1, \ldots, \widehat{\frak x}_l]\|_s 
  \lesssim_{s, k, l}  \prod_{j = 1}^l \| \widehat{\frak x}_j\|_1\, ,
  \end{aligned}
  $$
  $$
  \begin{aligned}
  &  \| \mathcal A_k^\pm(\frak x; {d \Psi}^\top)\|_s \lesssim_{s, k} 1 +  \| z_\bot \|_1 \,, \qquad
   \| d^l \mathcal A_k^\pm(\frak x; {d \Psi}^\top)[\widehat{\frak x}_1, \ldots, \widehat{\frak x}_l]\|_s 
  \lesssim_{s, k, l}  \prod_{j = 1}^l \| \widehat{\frak x}_j\|_1\,. 
  \end{aligned}
  $$
  Similarly, for any $\frak x \in {\cal V} \cap \big( {\cal M}_S^o \times h^s_\bot \big)$,  \, $\widehat{\frak x}_1, \ldots, \widehat{\frak x}_l \in E_s$, \, $\widehat q \in H^s_0$, $s \in \N$,
  $$
  \begin{aligned}
 &  \| {\cal R}_N (\frak x; {d \Psi}^\top) [\widehat q]\|_{s + N + 1}  \lesssim_{s, N} \| \widehat q\|_s + \| z_\bot \|_s \| \widehat q\|_1\, , \\
 & \| d^l \big( {\cal R}_N  (\frak x; {d \Psi}^\top) [\widehat q] \big)  [\widehat{\frak x}_1, \ldots, \widehat{\frak x}_l]\|_{s + N + 1} \\
 & \qquad \lesssim_{s, N, l} \| \widehat q\|_s \prod_{j = 1}^l \| \widehat{\frak x}_j\|_1 + \| \widehat q\|_1 \sum_{j = 1}^l \| \widehat{\frak x}_j\|_s \prod_{i \neq j} \| \widehat{\frak x}_i\|_1 + 
  \| \widehat q\|_1 \| z_\bot\|_s \prod_{j = 1}^l \| \widehat{\frak x}_j \|_1\,.
    \end{aligned}
  $$
\item[(Est3)] For any $s \in \Z_{\ge 0}$, $\frak x = (\theta_S, I_S, z_\bot) \in {\cal V} \cap \big( {\cal M}_S^o \times h^{s + \sigma_N}_\bot\big)$, 
$\| z_\bot \|_{\sigma_N} \leq 1$, $1 \le k \le N$,
\, $\widehat{\frak x}_1, \ldots, \widehat{\frak x}_l \in E_{s + \sigma_N}$,
$$
  \| a_k^\pm (\frak x; {\nabla \cal P})\|_{s} \lesssim_{s, k} \| z_\bot \|_{s + \sigma_N}\,, 
  $$
  $$
  \| d^l a_k^\pm(\frak x; {\nabla \cal P})[\widehat{\frak x}_1, \ldots, \widehat{\frak x}_l]\|_s 
\lesssim_{s, k, l} \sum_{j = 1}^l \| \widehat{\frak x}_j\|_{s + \sigma_N} \prod_{n \neq j} \| \widehat{\frak x}_n\|_{\sigma_N} +
 \| z_\bot \|_{s + \sigma_N} \prod_{j = 1}^l \| \widehat{\frak x}_j\|_{\sigma_N}\,. 
$$
For any $s \in \Z_{\ge 0}$, $\frak x \in {\cal V} \cap \big( {\cal M}_S^o \times h_\bot^{s \lor\sigma_N}\big)$ with $\| z_\bot \|_{\sigma_N} \leq 1$,
  \, $ \widehat{\frak x} \in E_{s \lor\sigma_N}$, 
$$
\begin{aligned}
 \| {\cal R}_N(\frak x; {\nabla \cal P})\|_{s + N + 1} & \lesssim_{s, N} \| z_\bot \|_{\sigma_N} \| z_\bot \|_{s \lor\sigma_N}\,,  \\
 \| d {\cal R}_N(\frak x; {\nabla \cal P})[\widehat{\frak x}] \|_{s + N + 1} 
& \lesssim_{s, N} \| z_\bot \|_{\sigma_N} \| \widehat{\frak x}\|_{s \lor \sigma_N} + \| z_\bot \|_{s \lor \sigma_N} \| \widehat{\frak x}\|_{\sigma_N}\, .
\end{aligned}
$$
If in addition $ \widehat{\frak x}_1, \ldots, \widehat{\frak x}_l \in E_{s \lor \sigma_N}$, $l \ge 2,$ then
$$
 \| d^l  {\cal R}_N(\frak x; {\nabla \cal P})[\widehat{\frak x}_1, \ldots, \widehat{\frak x}_l]\|_{s + N + 1} 
\lesssim_{s, N, l} \sum_{j = 1}^l \| \widehat{\frak x}_j\|_{s \lor \sigma_N} \prod_{n \neq j} \| \widehat{\frak x}_n\|_{\sigma_N} + 
\| z_\bot \|_{s \lor \sigma_N} \prod_{j = 1}^l \|\widehat{\frak x}_j \|_{\sigma_N}\,. 
$$
Corresponding estimates hold for the maps $ a_k^\pm (\cdot \, ; \, {\nabla \mathcal P^{mo}})$ and ${\cal R}_N(\cdot \, ; \, {\nabla \mathcal P^{mo}})$.
  \end{description}   
Here and throughout the paper, the meaning of the symbol $\, \lesssim \, $
  with various subindices is the standard one. So for example, the estimate
 $ \| a_k^\pm (\frak x; {\nabla \cal P})\|_{s} 
 \lesssim_{s, k} \| z_\bot \|_{s + \sigma_N}$ 
 in $\bf (Est3)$ means that for any $1 \le k \le N$, there exists  a constant $C(s,k) > 0$, depending on $s$ and $k$, so that 
 $\| a_k^\pm(\frak x; {\nabla \cal P})\|_{s} \le C(s,k) \| z_\bot \|_{s + \sigma_N}$ 
for any $ \frak x $ as indicated in the statement of $\bf (Est3)$, i.e., for any
$ \frak x = (\theta_S, I_S, z_\bot) \in {\cal V} \cap \big( {\cal M}_S^o \times h^{s + \sigma_N}_\bot\big)$ with $\| z_\bot \|_{\sigma_N} \leq 1$.  
  \end{theorem}
\begin{remark}
In the case $S_+ = \emptyset$, one has $M_{\emptyset} = \{ 0 \}$, $S^\bot = \Z\setminus \{0\}$, and $h^0_\bot = h^0_0$.
Since the differential $d \Phi^{bo}$ at $q=0$ is given by $- \mathcal F : L^2_0 \to h^0_0$ (cf. \eqref{differential at 0}), 
one easily verifies that in this case, the map $\Psi$ is given by  $- \mathcal F^{-1}$, hence linear and defined on all of $h^0_0$.
(According to Section \ref{sezione mappa Psi L BO}, $\Psi = \Psi_L \circ \Psi_C$ and in the case at hand, $\Psi_L = - \mathcal F^{-1}$, $\Psi_C = {\rm Id}$.)
As a consequence, $d\Psi^\top = - \mathcal F$. Furthermore, $H^{bo} \circ \Psi : h^1_0 \to \R$ is given by
$$
H^{bo} \circ \Psi (z) = \sum_{n \ge 1} \Omega_n^{bo}(0) z_n z_{- n} + {\cal P}(z)\, , \qquad
\Omega_n^{bo}(0) = n \, ,  \quad \forall \, n \ge 1 \, ,
$$
and ${\cal P}(z) =  \frac{1}{2\pi} \int_0^{2\pi} - \frac 13 (\mathcal F^{-1} z)^3 d x$. Hence by the Bony decomposition (cf. Lemma \ref{primo lemma paraprodotti}(i)), 
the gradient $\nabla \mathcal P$ admits for any $N \ge 1$, $z \in h^{s + \sigma_N}_0$ with $s \ge 0$ and $\sigma_N :=  N+2$, an expansion of the form
$$
\nabla {\cal P}(z)  = 2 T_{ \mathcal F^{-1} z} [ \mathcal F^{-1} z] + \mathcal R^{(B)}(z)\, , \qquad
$$
where for any $s \ge 0$, the Bony remainder $\mathcal R^{(B)} : h^{s \lor \sigma_N}_0 \mapsto h^{s+N + 1}_0$ is real analytic and satisfies the estimate
$\| \mathcal R^{(B)} (z)\|_{s+ N +1} \lesssim_{s, N} \|z\|_{s \lor \sigma_N}  \|z\|_{\sigma_N}$.
Furthermore,
$$
{\cal H}^{mo}(z) = H^{mo} \circ \Psi(z) =   \sum_{n \ge 1} z_n z_{- n} \, .
$$
\end{remark}

 \noindent 
{\em Applications.} The Birkhoff coordinates are well suited to study the initial value problem of \eqref{1.1} (cf. \cite{GKT1} and references therein).
Using the arguments, developed in the case of the Korteweg-de Vries (KdV) equation (cf. e.g. \cite{KP}, \cite{K}), 
one can obtain KAM type results for  semilinear perturbations of \eqref{1.1}. 
However, when equation \eqref{1.1} is expressed in Birkhoff coordinates, 
various features of the BO equation and its perturbations such as being pseudo-differential equations, get lost. 
On the other hand, due to the expansions $({\bf AE1}) - ({\bf AE3})$, the coordinates of Theorem \ref{modified Birkhoff map} allow 
to preserve the essence of such features and in the form stated turn out to be well suited to study quasilinear perturbations 
of the BO equation as well as questions of stability of finite gap solutions.    
We plan to investigate the stability of finite gap solutions of the BO equation
under quasi-linear Hamiltonian perturbations
\begin{equation}\label{main equation}
\partial_t u  =  \partial_x  \big( |\partial_x| u -  u^2 + \e  \nabla { F}(u) \big)   , \qquad x \in \T \, , 
\end{equation}
where $0 < \e <  1$ is a small parameter  
and  $\nabla F $ the $L^2$-gradient  of Hamiltonian $F$ such as
\begin{equation}\label{hamiltoniana perturbazione}
F (u) := \frac{1}{2\pi}\int_{0}^{2\pi} 
f(u(x), |\partial_x|^{1/2} u(x))\, d x \, , 
\qquad f \in {\cal C}^{\infty}( \R \times \R, \R) \, .
\end{equation}
Our goal is to use Theorem \ref{modified Birkhoff map} and Theorem \ref{modified Birkhoff map 2} to prove that for $\e$ small enough, 
there is a large set of finite gap solutions of \eqref{1.1}
with the property that for any initial data, $\e$-close to one of them, the corresponding solution of 
the perturbed equation exists for large  time intervals and stays $\e-$close
to the finite gap solution considered.
  
   \smallskip
   
   \noindent 
   {\em Outline of the construction of the coordinates of Theorem \ref{modified Birkhoff map}.} 
   In \cite{K}, Kuksin presents a general scheme for proving KAM-type theorems 
 for integrable PDEs in one space dimension such as the KdV or the sine-Gordon (sG) equations, which possess a Lax pair formulation 
 and admit finite dimensional integrable subsystems foliated by invariant tori. 
 The starting point is to construct local canonical coordinates, suitable to apply KAM methods.
 Expanding on work of Krichever \cite{Krichever}, Kuksin considers bounded, finite dimensional integrable subsystems of such a PDE 
 which admit action-angle coordinates. The latter are complemented by infinitely many coordinates 
 whose construction is based on a set of time periodic solutions, referred to as Floquet solutions of the PDE, obtained by linearizing 
 the PDE under consideration along a solution evolving in the integrable subsystem. It turns out that the resulting coordinate transformation 
 is typically not symplectic. Extending arguments of Moser and Weinstein to the given infinite dimensional setup (see \cite{K}, Lemma 1.4 and Section 1.7), 
 he constructs a second coordinate transformation so that the composition of the two transformations become symplectic. 
 In our previous work \cite{KM1}, we follow the general scheme of the construction in \cite{K} to construct canonical coordinates for the KdV equation
 near finite dimensional invariant tori,
 which have the property that they admit an expansion  as described in Theorem \ref{modified Birkhoff map} and Theorem \ref{modified Birkhoff map 2}.
Following the arguments developed of  \cite{KM1}, we construct the transformation  $\Psi$ 
 as the composition of $\Psi_L \circ \Psi_C$ of two transformations. 
 The transformation $\Psi_L$ is given by the Taylor expansion of the Birkhoff map $\Psi^{bo}$ of order one in the normal direction $z_\bot$ around $(z_S, 0)$,
$$
   \Psi^{bo}(z_S, 0) + d \Psi^{bo}(z_S, 0)[(0, z_\bot)]\, ,
 $$
 where we write $z$ as $(z_S,  z_\bot)$ with $z_S= (z_n)_{n \in S}$ and $z_\bot = (z_n)_{n \in S^\bot}$.
   The neighbourhood ${\cal V}$ of ${\cal K} \times \{ 0 \}$ (cf. Theorem \ref{modified Birkhoff map}) is chosen sufficiently small so that by the inverse function theorem,
    $\Psi_L$ is a real analytic diffeomorphism onto its image. 
    Using that $\Psi_L$ is given in terms of the Birkhoff map $\Psi^{bo}$ and taking advantage of features of the spectral theory of the Lax operator, 
    we prove that $\Psi_L$ admits a high frequency expansion and tame estimates corresponding to the ones of Theorems \ref{modified Birkhoff map}, \ref{modified Birkhoff map 2}. 
    In a second step we establish the corresponding results for the symplectic corrector $\Psi_C$. 
    
   In comparison with the treatment of the KdV equation, the main differences in the case of the Benjamin-Ono equation consist in the facts
   that the Lax operator of the Benjamin-Ono equation is a pseudo-differential operator (of first order) and not a differential one (of second order)
   and that the Birkhoff map of the Benjamin-Ono equation is a 'quasi-linear' perturbation of the Fourier transform and not a semilinear one
    (cf. Theorem \ref{pseudodiff expansion Psi_L}), making our constructions more challenging.
    

   \smallskip
   
   \noindent
   {\em Comments.} In view of the definition of $\Psi_L$, the map $\Psi$ can be considered as a symplectic version of the Taylor expansion of $\Psi^{bo}$ of order $1$ 
   in normal directions at points in ${\cal M}_S^o \times \{ 0 \}$ and hence as a locally defined symplectic approximation of $\Psi^{bo}$. 
   In the special case $N=1$, Theorem \ref{modified Birkhoff map} implies that 
   $$
    -  g_\infty \cdot ({\cal F}^+_{N_S})^{- 1}[z_\bot] - \overline{g_\infty } \cdot ({\cal F}^-_{N_S})^{- 1}[z_\bot]
    $$ 
    is a high frequency approximation of $\Psi$. More precisely, 
   $$
   \Psi(\theta_S, I_S, z_\bot) + g_\infty \cdot ({\cal F}^+_{N_S})^{- 1}[z_\bot] + \overline{g_\infty } \cdot ({\cal F}^-_{N_S})^{- 1}[z_\bot]
   $$ 
   maps ${\cal V} \cap ({\cal M}_S^o \times h^s_\bot)$ into $H^{s + 1}$ for any $s \geq 0$, i.e., it is one-smoothing. 
   In \cite{GKT5}, such a property has been established  for the differential of $\Psi^{bo}$ 
   and the one of $\Phi^{bo}$.{\footnote{ In the case of the KdV equation (cf. \cite{KST2}, \cite{Kuksin-Perelman}) and the defocusing NLS equation (cf. \cite{KST3}),
   such type of results have been proved for the corresponding differentials of the Birkhoff maps (and their inverses), as well as for the Birkhoff maps themselvees.}}
   In contrast, Theorem \ref{modified Birkhoff map} says that for the map $\Psi$, a much stronger property holds: 
   up to a remainder term, which is $(N+1)$-smoothing, $\Psi$ is a (nonlinear) pseudo-differential operator acting on ${\cal F}_\bot^{- 1}(h^0_\bot)$. 
   
   \smallskip
   
   \noindent
   {\em Notations.}
   The standard inner product on $L^2_\C \equiv H^0_\C$ and the corresponding norm are defined by
$$
\langle f | g \rangle \equiv \langle f | g \rangle_{L^2_\C} = \frac{1}{2\pi} \int_0^{2\pi} f(x) \overline{g(x)} d x\,, 
\quad \| f\| = \langle f | f \rangle^{1/2} \, ,
\qquad \forall \, f, g \in L^2_\C\, .
$$
Restrictions of this inner product to various subspaces and extensions as dual pairings will be denoted in the same way.

Let $h$ and $g$ be real valued functions, depending on various variables. In addition, $h$ might depend on parameters $\alpha, \, \ldots$ .
The notation $ h \lesssim_{\alpha, \, \ldots} g$ means that $h$
satisfies an estimate of the form $h \le C g$ where the constant $C>0$ only depends on the parameters $\alpha, \, \ldots$ .
   
\smallskip
   
   \noindent
   {\em Organization.} 
   The maps $\Psi_L$ and $\Psi_C$ are studied in Section \ref{sezione mappa Psi L BO} and Section \ref{sezione Psi C}, respectively.
   The expansion of the BO Hamiltonian in the new coordinates is treated in Section \ref{Hamiltoniana trasformata} 
   and a summary of the proofs of Theorem \ref{modified Birkhoff map} and Theorem \ref{modified Birkhoff map 2} is given in Section \ref{synopsis of proof}. 
   In Appendix \ref{spectral theory}  - 
   Appendix \ref{AppendixReversability}, 
   we present results needed in the main body of the paper and 
   in Appendix \ref{appendice B} we review material from the pseudo-differential and para-differential calculus.
   
   \bigskip
   
   \noindent
   {\em Acknowledgements.} Thomas Kappeler is supported by the Swiss National Science Foundation. Riccardo Montalto is supported by INDAM-GNFM.


 \section{The map $\Psi_L$}\label{sezione mappa Psi L BO}
In this section we define and study the map $\Psi_L$ described in Section \ref{introduzione paper}. 
First let us introduce some more notation. For $S \subset \Z$ finite as in \eqref{definition S}, let
   \begin{equation}\label{definition h_S bo}
   h_S = \big\{ z_S = (z_n)_{n \in S} \in (\C^\ast)^S :  \,z_{- n} = \overline{z}_n \,\, \forall n \in S_+ \big\}\, , \qquad   \C^\ast :=\C \setminus \{0 \} \, ,
   \end{equation}
  endowed with the norm $\|z_S\| = (\sum_{n \in S} |z_n|^2)^{1/2}$.
   Recall that $N_S = \max\{n \, : \, n \in S \}$.
   For notational convenience, for any $s \in \Z$, we identify $h_S \times h^s_\bot$ with the subset $\{z= (z_n)_{n \ne 0} \in h^s_0 : \, z_n \ne 0 \ \forall n \in S \}$ of $h^s_0$ 
   and write $(z_S, z_\bot)$ for $z \in h^s_0.$
   
   The restriction of the Birkhoff map $\Phi^{bo}$ to the space $M_S^0$ of proper $S$-gap potentials yields a real analytic diffeomorphism,
   ${\Phi^{bo}}|_{ M^0_S} : M^0_S \to h_S$ (cf. Corollary \ref{restriction to finite gap} in Appendix \ref{Birkhoff map for BO}). 
   We endow 
   $h_S$ with the pull back of the standard Poisson 
   structure on $h_0^0$ by the natural embedding $h_S \hookrightarrow h_0^0$, where the standard Poisson structure 
   is the one for which $\{ z_n, z_{- n} \} = - \ii n$ 
   for any $n \ne 0$ and the Poisson brackets among all the other coordinates vanish. 
   By Corollary \ref{restriction to finite gap}, $ M^0_S$ is a real analytic submanifold of $\mathcal U_{N_S}$ 
   where $\mathcal U_{N_S}$ is the set of finite gap potentials $q$
   with the property that $\gamma_{N_S}(q) > 0$ and $\gamma_n(q) = 0$ for any $n \ge N_S +1$. 
%
By Appendix \ref{spectral theory},  any $u \in \mathcal U_{N_S}$ is of the form 
\begin{equation}\label{formula finite gap}
u(x) = \sum_{j=1}^{N_S} \big(  \frac{1 - r_j^2}{1 - 2r_j \cos(x +\alpha_j)+ r_j^2} -1  \big)\, , 
\qquad 0 < r_j < 1, \  0 \le \alpha_j < 2\pi \, , \quad \forall \, 1 \le j \le N_S \, .
\end{equation}

   \smallskip
   
   \noindent

   Consider the partial linearization of the inverse Birkhoff map  $\Psi^{bo} := (\Phi^{bo})^{- 1}$, defined as
   \begin{equation}\label{definition Psi_L bo}
   \Psi_L : h_S \times h_\bot^0 \to L^2_0\,, \,\, (z_S, z_\bot) \mapsto \Psi^{bo}(z_S, 0) +  \Psi_1(z_S)[z_\bot],
    \qquad   \Psi_1(z_S)[z_\bot] :=   d_\bot \Psi^{bo}(z_S, 0)[z_\bot],
   \end{equation}
   where 
   $d_\bot \Psi^{bo}(z_S, 0)$ denotes the Fr\'echet derivative of the map $z_\bot \mapsto \Psi^{bo}(z_S, z_\bot)$, evaluated at the point $(z_S, 0)$. 
  By Theorem \ref{Theorem Birkhoff coordinates BO},
  $\Psi_L$ is a real analytic map. 
   \begin{proposition}\label{prop Psi L Psi bo} The map $\Psi_L$ has the following properties:\\
    (i) For any $z_S \in h_S,$
   $$
   \Psi_L (z_S, 0) = \Psi^{bo}(z_S, 0)\, ,  \qquad \qquad d \Psi_L(z_S, 0)= d \Psi^{bo}(z_S, 0) \, .
   $$
   (ii) For any compact subset ${\mathcal K} \subseteq h_S$ there exists an open neighborhood ${\mathcal V}$ of ${\mathcal K} \times \{ 0 \}$ in $ h_S \times h_\bot^0$ 
   so that for any integer $s \ge 0$, the restriction $\Psi_L|_{\mathcal V \cap h_0^s}$ is a map $\mathcal V \cap h_0^s  \to H^s_0$
   which is a real analytic diffeomorphism onto its image. The neighborhood $\mathcal V$ is chosen of the form $\mathcal V_S \times \mathcal V_\bot$
    where $\mathcal V_S$ is an open, bounded neighborhood of $\mathcal K$ in $h_S$ and $\mathcal V_\bot$ is an open ball
    in $h^0_\bot$ of sufficiently small radius, centered at zero.\\
(iii) For any $z = (z_S, z_\bot) \in {\cal V}$ and  $\widehat z = (\widehat z_S,  \widehat z_\bot ) \in h_S \times h_\bot^0$, 
\begin{equation}\label{formula d Psi L}
d \Psi_L(z) [\widehat z]  = d \Psi_L(z_S, 0)[\widehat z] + d_S \big( d_\bot \Psi^{bo}(z_S, 0)[ z_\bot]  \big)[\widehat z_S] .
\end{equation}
For any $z_S \in \mathcal V_S$, the linear map $d \Psi_L(z_S, 0) =  d \Psi^{bo}( z_S, 0)$
is canonical and for any $z_\bot \in \mathcal V_\bot$, we denote by
 $d_S \big( d_\bot \Psi^{bo}(z_S, 0)[z_\bot]  \big)$ the Fr\'echet derivative
of the map 
$
\mathcal V_S \to L^2_0 ,\, w_S \mapsto d_\bot \Psi^{bo}(w_S, 0)[z_\bot]
$
at $w_S= z_S$.
\end{proposition}
 \begin{proof} (i) The stated formulas follow from the definition of $\Psi_L$. 
 (ii) 
The claimed statements follow from Theorem \ref{Theorem Birkhoff coordinates BO} and the inverse function theorem by arguing
 as in the proof of the corresponding results for the defocusing NLS equation in \cite[Proposition 3.1]{Kappeler-Montalto}. 
 Item (iii) is proved in a straightforward way.
 \end{proof}
 For any $z = (z_S, z_\bot) \in {\cal V} = \mathcal V_S \times \mathcal V_\bot$ and $n \ne 0$, let $q := \Psi^{bo}(z_S, 0) $ 
 and  
 \begin{equation}\label{def W_n bo}
    W_{n}(x) \equiv   W_{n}(x, q) :=  (d_q \Phi^{bo})^{- 1}[e^{(n)}]  , \qquad
     W_{-n}(x) \equiv   W_{-n}(x, q) :=  (d_q \Phi^{bo})^{- 1}[e^{(- n)}] ,
  \end{equation}
  where $e^{(k)}$, $k \ne 0$,  are the standard basis elements in the sequence space $h^0_{0, \C}$, 
   $e^{(k)} = (\delta_{k, j})_{j \neq 0}$.
    (Here we extended $(d_u \Phi^{bo})^{-1}: h^0_{0} \to  L^2_{0} $ as a $\C$-linear map $ h^0_{0, \C} \to L^2_{0, \C}$.) 
 Then $  \Psi_1(z_S)[z_\bot] =  \Psi_L(z_S, z_\bot) - q $ can be written as
   \begin{equation}\label{forma finale Psi L def}
     \Psi_1(z_S)[z_\bot](x)  = \sum_{n \in S^\bot} z_n W_n(x, q) \,.
   \end{equation}
In a next step, we want to analyze $\Psi_1(z_S)[z_\bot]$ further. 
Consider the Hamiltonian vector fields $ \partial_x \nabla z_{ n}$, $n \ne 0$, 
corresponding to the Hamiltonians given by the Birkhoff coordinates $z_{ n}: L^2_0 \to \C, u \mapsto z_n(u)$. 
Since $\Phi^{bo}$ is canonical in the sense that $\{ z_n, z_{- n} \} = - \ii n$ for any $n \ne 0$ 
and  the brackets among all the other coordinates vanish, it follows that for any $u \in L^2_{0}$ and $n \ne 0$, 
   $$
   d_u \Phi^{bo}[\partial_x \nabla z_{ n}] = X_{z_{ n}} \, ,
   $$
   where $X_{z_{ n}}$ is the constant Hamiltonian vector field on $h^0_{0, \C}$ given by (cf. \eqref{def gradient})
   $$
   X_{z_n} = -  \ii n e^{(- n)}\,.
   $$
  Hence for any $n \ne 0$, 
   \begin{equation}\label{formula vector field X z+- n bo}
   (d_u \Phi^{bo})^{- 1}[e^{(n)}] = \frac{1}{ \ii n} \partial_x \nabla z_{- n}  \, .
   \end{equation}   
We need to explicitly compute $\partial_x \nabla z_{n}$ at $q := \Psi^{bo}(z_S, 0)$, $z_S \in \mathcal V_S$, for $|n| \ge N_S+1$ (cf. \eqref{definition h_S bo}).
It requires some elements of the spectral theory of the Lax operator $L_u$, recorded in Appendix \ref{spectral theory}.
The operator $L_u$, $u \in L^2$, is given by $ D - T_u$ and acts on the Hardy space $L^2_+$. Here $D= -\ii \partial_x$ 
and $T_u$ is the Toeplitz operator $f \mapsto \Pi(uf)$ with symbol $u$.
The eigenvalues are listed in increasing order and denoted by $\lambda_n(u)$, $n \ge 0$. 
The corresponding $L^2$-normalized eigenfunctions $f_n(\cdot, u) \in H^1_+$ with the normalization conditions in \eqref{normalized ef}
form an orthonormal basis of $L^2_+$. First let us consider the case $n \ge N_S+1$.
For any $n \ge N_S+1$,
the eigenvalue $\lambda_n : L^2 \to \R$ (cf. Theorem \ref{spec Lax operator}(iv)) 
and the corresponding eigenfunction $f_n : L^2 \to H^1_+$ (cf. Corollary \ref{def kappa}(iii)) of the Lax operator are real analytic functions.
 Denote by $\delta \lambda_n$ and $\delta f_n$ the variation of $\lambda_n$ and respectively, $f_n$ at $q$ in direction $u\in L^2$,
 $$
 \delta \lambda_n = \frac{d}{d\e}|_{\e = 0} \lambda_n(q + \e u)\, , \qquad
  \delta f_n = \frac{d}{d\e}|_{\e = 0} f_n(\cdot, q + \e u)\, .
 $$
    Since $f_n$ is $L^2$-normalized, $\langle \delta f_n | f_n \rangle = \ii \xi_n $ for some $\xi_n \in \R$ and hence (cf. \cite[Section 5]{GK1})
   $$
    \delta f_n =  \ii \xi_n f_n +  \sum_{\ell \ne n}  \langle \delta f_n | f_\ell \rangle f_\ell .
   $$
   The real numbers $\xi_n$, $n \ge 1$, are determined by the normalization conditions \eqref{normalized ef} of $f_n$, $n \ge 0$.
   The derivative of $(D - T_{q+\e u} - \lambda_n(q + \e u))f_n(\cdot, q + \e  u) = 0$ with respect to $\e$ at $\e = 0$ can then be computed as 
   $$
    (- T_u - \delta \lambda_n)f_n + (D - T_q - \lambda_n) \delta f_n = 0.
   $$
   Since $\delta \lambda_n = - \langle u f_n | f_n \rangle$ and $T_uf_n = \sum_{\ell \ge 0} \langle uf_n | f_\ell \rangle f_\ell$ one infers that
   $$
   \delta f_n = \ii \xi_nf_n +  (D - T_q - \lambda_n)^{-1} \sum_{\ell \ne n} \langle u f_n | f_\ell \rangle f_\ell
   =  \ii \xi_nf_n + \sum_{\ell \ne n} \frac{ \langle u f_n | f_\ell \rangle}{\lambda_\ell - \lambda_n} f_\ell \, .
   $$
  By Lemma \ref{results finite gap potentials}(ii), one has
  for any $n \ge N_S+1$,
  $$
  \lambda_n \equiv \lambda(q) = n \, , \qquad f_n(x) \equiv f_n(x, q) = e^{\ii nx} g_\infty(x, q)\, , \qquad 
  g_\infty(x) \equiv g_\infty(x, q) = e^{\ii \partial_x^{-1}q(x)}  \, .
  $$ 
  Since $q$ is a $S$-gap potential and hence $\langle 1 | f_\ell \rangle = 0$ for any $\ell \in S^\bot_+$, it then follows that
   \begin{equation}\label{def S_0}
   1 =  \sum_{\ell \in S_0} \langle 1 | f_\ell \rangle f_\ell \, , \qquad  S_0 := S_+ \cup \{ 0 \} \, ,
   \end{equation}
   and 
   $$
   \langle 1 | \delta f_n \rangle = \sum_{\ell \in S_{0}} \overline{\langle u f_n | f_\ell \rangle} \frac{\langle 1 | f_\ell \rangle}{\lambda_\ell - n}
   =  \langle u , e^{-\ii n x} \overline{g_\infty}  \sum_{\ell \in S_0} \frac{\langle 1 | f_\ell \rangle}{\lambda_\ell - n} f_\ell  \rangle  \, .
   $$
  In view of the definition $z_n = n \frac{\langle 1 | f_n \rangle }{\sqrt{n \kappa_n}}$ and the one of the gradient $\nabla z_n$, 
  one then concludes that
   $$
   \partial_x \nabla z_n (\cdot, q)
   = - \frac{1}{\sqrt{n \kappa_n}} \sum_{\ell \in S_0} \frac{\langle 1 | f_\ell \rangle}{1 - \frac{\lambda_\ell}{n} }
   \partial_x\big(  e^{-\ii n x} \overline{g_\infty}f_\ell \big) \, , \qquad \forall n \ge N_S+1.
   $$
   Taking into account that $z_{-n}= \overline{z_n}$ and $\partial_x = i D$  we are led to the following
   \begin{lemma}\label{lemma zn nabla q bo}
   For any $q \in M^0_S$ and $n \ge N_S + 1$, 
   $$
   \partial_x \nabla z_n (\cdot, q) =  -   e^{-\ii n x} \frac{1}{\sqrt{n \kappa_n}} \sum_{\ell \in S_0} \frac{\langle 1 | f_\ell \rangle}{1 - \frac{\lambda_\ell}{n} }
   \big( - \ii n \, \overline{g_\infty}f_\ell  + \ii D ( \overline{g_\infty}f_\ell ) \big),
   $$
   $$
   \partial_x \nabla z_{- n}(\cdot, q) =  -   e^{ \ii n x} \frac{1}{\sqrt{n \kappa_n}} \sum_{\ell \in S_0} \frac{\overline{ \langle 1 | f_\ell \rangle} }{1 - \frac{\lambda_\ell}{n} }
   \big(  \ii n \, g_\infty \overline{f_\ell}  + \ii D ( g_\infty \overline{f_\ell} ) \big). \qquad
   $$
   Hence by \eqref{def W_n bo} and \eqref{formula vector field X z+- n bo},
  $$
  W_{-n}(x, q) =  \frac{1}{-\ii n} \partial_x \nabla z_n (\cdot, q)= 
    -   e^{- \ii n x} \frac{1}{\sqrt{n \kappa_n}} \sum_{\ell \in S_0} \frac{ \langle 1 | f_\ell \rangle }{1 - \frac{\lambda_\ell}{n} }
   \big( \overline{g_\infty} f_\ell  + \frac{1}{  n} q \overline{g_\infty} \, f_\ell -  \frac{1}{  n}  \overline{ g_\infty} \, D f_\ell   \big) ,
 $$
  $$
  W_n(x, q) = \frac{1}{\ii n} \partial_x \nabla z_{-n} (\cdot, q) = 
    -   e^{ \ii n x} \frac{1}{\sqrt{n \kappa_n}} \sum_{\ell \in S_0} \frac{\overline{ \langle 1 | f_\ell \rangle } }{1 - \frac{\lambda_\ell }{n} }
   \big(  g_\infty \overline{f_\ell}  + \frac{1}{ n } q g_\infty \overline{f_\ell}  - \frac{1}{ n }  g_\infty \overline{ D f_\ell} \big).  \qquad
   $$
   \end{lemma}
   \noindent
 Next we want to show that $\Psi_L(z)$ admits  an expansion of the type stated in Theorem \ref{modified Birkhoff map}. 
 It is convenient to introduce
 $$
 V_S := \Psi^{bo} (\mathcal V_S \times \{ 0 \}) .
 $$
 First note that being a finite gap potential, $q \in V_S$ is $C^\infty$-smooth and so is $W_n(\cdot, q)$ for any $n \in S^\bot$.

 \bigskip
 \noindent  
 \bigskip
 
\begin{theorem}\label{lemma asintotica bo}
For any $q \in V_S$, $N \in \N$, 
$W_{\pm n}(x, q)$, $n \ge N_S+1$, have the expansion as $n \to \infty$ 
\begin{equation}\label{final asymptotic Wn bo}
W_{\pm n}(x, q) = - e^{ \pm \ii n x} \Big( \sum_{k = 0}^N \frac{W_{k}^{ae, \pm }(x, q)}{n^k} \, + \,  \frac{{\cal R}_{N}^{W_{\pm n}}(x, q)}{n^{N + 1}} \Big)
\end{equation}
where for any $1 \le k \le N$, $W_k^{ae, -}(x, q) = \overline{W_k^{ae, +}(x, q)}$ and $W_{k}^{ae, +}(x) \equiv W_{k}^{ae, +}(x, q)$  is given by
$$
W_{k}^{ae, +}(x)  := g_\infty(x) \sum_{\ell \in S_0}  c_{\ell, k} \overline{ \langle 1 | f_\ell \rangle f_\ell(x) } 
+ q(x) g_\infty(x) \sum_{\ell \in S_0}  c_{\ell, k - 1} \overline{ \langle 1 | f_\ell \rangle f_\ell (x)} 
- g_\infty(x)   \sum_{\ell \in S_0} c_{\ell, k -1}  \overline{  \langle 1 | f_\ell \rangle D f_\ell (x)} .
$$
The coefficients $c_{\ell, k}$ are real valued constants and for any $\ell \in S_0$ (cf. \eqref{def S_0}) have the following properties: \\
$(i)$ for $k = -1,$ $0$, one has $c_{\ell, -1}= 0$ and  $c_{\ell, 0}=1$, implying that
\begin{equation}\label{leading term}
W_{0}^{ae, +} =  g_\infty , \qquad
W_{0}^{ ae, -} =  \overline{g_\infty} \,;
\end{equation}
$(ii)$ for any $k \ge 1$, $c_{\ell, k}$ is a polynomial of degree $k$ in the gap lengths $\gamma_j(q)$, $j \in S_+$. 

\smallskip
\noindent
Finally, the remainders ${\cal R}_{N}^{W_{\pm n}}(x) \equiv {\cal R}_{N}^{W_{\pm n}}(x, q)$ satisfy
${\cal R}_{N}^{W_{- n}}(x) = \overline{{\cal R}_{N}^{W_{n}}(x)}$  and ${\cal R}_{N}^{W_{n}}(x)$ is given by 
$$
{\cal R}_{N}^{W_{ n}}(x) := g_\infty(x) \sum_{\ell \in S_0}  r^N_{\ell, n} \overline{ \langle 1 | f_\ell \rangle f_\ell(x) } 
+ q(x) g_\infty(x) \sum_{\ell \in S_0}  r^{N-1}_{\ell, n} \overline{ \langle 1 | f_\ell \rangle f_\ell (x)} 
- g_\infty(x)  \sum_{\ell \in S_0} r^{N-1}_{\ell, n}  \overline{  \langle 1 | f_\ell \rangle D f_\ell (x)} ,
$$
where for any $\ell \in S_0$, the coefficients $r^N_{\ell, n}$ are real, bounded functions of 
$\gamma_j(q)$, $j \in S_+$, independent of $x$.
\end{theorem}
\begin{proof}  The claimed results follow from Lemma \ref{lemma zn nabla q bo}. 
Indeed, for any $n \ge N_S+1,$ one has
$$
n \kappa_n = \frac{n}{\lambda_n - \lambda_0} \prod_{j \ne n}(1 - \frac{\gamma_j}{\lambda_j - \lambda_n})
=  \frac{n}{n - \lambda_0} \prod_{j \in S_+}(1 + \frac{\gamma_j}{ n - \lambda_j })
=  \frac{1}{1 - \frac{\lambda_0}{n}} \prod_{j \in S_+}(1 + \frac{\gamma_j}{ n} \frac{1}{ 1 - \frac{\lambda_j}{ n }}).
$$
For any $\ell \in S_0$, one then obtains an expansion of the form
$$
(n \kappa_n)^{-1/2} \, \frac{1}{1 - \frac{\lambda_\ell}{n}} = 1 + \sum_{k=1}^N \frac{ c_{\ell, k} }{n^k} + \frac{r^{N}_{\ell, n}}{n^{N+1}}
$$
where the coefficients $c_{\ell, k}$, $1 \le k \le N$, and the remainder $r^{N}_{\ell, n}\frac{1}{n^{N+1}}$ can be explicitly computed, using
that for any $a \in \R$ with $|a| < 1$ and any $N \ge 1$
$$
\frac{1}{1-a} = 1 + \sum_{k = 1}^N a^k + a^{N+1}\frac{1}{1 - a}, 
$$
$$
(1-a)^{-1/2} = 1 + \sum_{k=1}^N (-1)^k \begin{pmatrix} -1/2 \\ k  \end{pmatrix} a^k + a^{N+1} C_N \int_0^1 (1-t)^N (1 - t a)^{-1/2 - N - 1} \, d t ,
$$
with $C_N$ being a combinatorial constant.
Furthermore, we set $c_{\ell, -1}:=0$ and $c_{\ell, 0}:=1$ for any $\ell \in S_0$.
Since for any $ j \in S_0,$ one has $\lambda_j = j - \sum_{m= j+1}^{N_S} \gamma_m$, the claimed expansions
for $W_{\pm n}$ with the stated formulas for $W^{ae, \pm}_{k}$ and the ones for ${\cal R}_{N}^{W_{ \pm n}}$
then follow from the formulas for $W_{\pm n}$ of Lemma \ref{lemma zn nabla q bo}. The stated properties of 
$c_{\ell, k}$, $-1 \le k \le N$, and  $r^{N}_{\ell, n}$ can be verified in a straightforward way.
The identities \eqref{leading term} follow from the fact that 
$1 = \sum_{\ell \in S_0} \langle 1 | f_\ell \rangle f_\ell$ (cf. \eqref{def S_0}).
\end{proof}
For notational convenience, in the sequel, we will view $W_k^{ae, \pm}(\cdot, q)$ and ${\cal R}_{N}^{W_{\pm n}}(\cdot, q)$  as functions of $z_S$, 
$$
W_k^{ae, \pm }(\cdot, z_S) \equiv W_k^{ae, \pm }(\cdot, \, \Psi^{bo}(z_S, 0))\,, \qquad  
{\cal R}_{N}^{W_{\pm n}}(\cdot, z_S) \equiv {\cal R}_{N}^{W_{\pm n}}(\cdot, \, \Psi^{bo}(z_S, 0))\,.
$$
Theorem \ref{lemma asintotica bo}, combined with Lemma \ref{even symmetry}(ii), has then the following immediate two applications.
\begin{corollary}
For any $s \in \Z_{\geq 0}$ the following holds:\\
$(i)$ for any $k \ge 0$, the maps
$W_k^{ae, \pm } : \mathcal V_{S} \to H^s_{ \C}$, $z_S \mapsto W_{k}^{ae, \pm }(\cdot, z_S)$ are real analytic;\\
$(ii)$ for any $N \ge 1$ and $n \ge N_S + 1$, the maps
$ {\cal R}_{N}^{W_{\pm n}} :  \mathcal V_{S}  \to H^s_\C, \,z_S  \mapsto {\cal R}_{N}^{W_{\pm n}}(\cdot, z_S)$
are real analytic and  satisfy for any $m \geq 0$, 
$$
{\rm sup}_{\begin{subarray}{c}
0 \leq x \leq 1 \\
n \in S^\bot
\end{subarray}} |\partial_x^m {\cal R}_N^{W_{\pm n}}(x, z_S)| \leq C_{N, m}\,.
$$
The constants $C_{N, m}$ can be chosen locally uniformly for $z_S \in \mathcal V_S$. 
\end{corollary}
The second application concerns the linear operator $ \Psi_1(z_S) : h^0_\bot \to L^2_0$, given for $z_S \in \mathcal V_S$ by
$$
 \widehat z_\bot \mapsto \Psi_1(z_S)[\widehat z_\bot]= \sum_{n \in S^\bot} \widehat z_n W_n( \cdot, z_S) 
= \sum_{|n| \in [1, N_S] \setminus S_+} \widehat z_n W_n( \cdot, z_S)  + \sum_{|n| \ge N_S +1} \widehat z_n W_n( \cdot, z_S).
$$
Note that for any $s \in \Z_{\ge 0}$, the restriction of $\Psi_1(z_S)$ to $h^s_\bot$ is a bounded linear operator $h^s_\bot \to H^s_0$.
Recall that in \eqref{def partial Fourier}, we introduced the partial Fourier transforms
$\mathcal F^+_{N_S} : L^2_\C \mapsto h^0_{\bot, \C}$ and  $\mathcal F^-_{ N_S}: L^2_\C \mapsto h^0_{\bot, \C}$, 
and in \eqref{def inverse partial Fourier} the partial inverses  
$( \mathcal F^\pm_{N_S})^{-1}: h^0_{\bot, \C} \to L^2_\C$ of the Fourier transform.
Using that
 $$
 \sum_{n > N_S} \widehat z_n \frac{1}{ n^k} e^{\ii nx} = D^{-k} (\mathcal F^+_{ N_S})^{-1}[\widehat z_\bot] \, , \qquad
 \sum_{n > N_S} \widehat z_{-n} \frac{1}{ n^k} e^{-\ii nx}  = (-D)^{-k} (\mathcal F^{-}_{N_S})^{-1}[\widehat z_\bot] \, ,
 $$
Theorem \ref{lemma asintotica bo} implies the following two corollaries.
\begin{corollary}\label{pseudodiff expansion Psi_L} 
 For any $z_S \in \mathcal V_S,$
 up to a remainder,  the operator $\Psi_1(z_S): h^0_\bot \to L^2_0$ is a pseudo-differential operator of order $0$. 
More precisely, $\Psi_1(z_S)$ has an expansion to any order $N \ge 1$ of the form $\mathcal {OP}_N(z_S; \Psi_{1}) +  {\cal R}_{N}(z_S; \Psi_1) $ where
$$
\mathcal {OP}_N(z_S; \Psi_{1}) :=  \big( -  g_\infty   +  \sum_{k = 1}^N  a^+_{k}(z_S; \Psi_1) D^{- k} \, \big) \circ ( \mathcal F^+_{N_S})^{-1} 
+  \big( -  \overline{g_\infty}  +   \sum_{k = 1}^N  a^-_{k}(z_S; \Psi_1) (-D)^{- k} \,  \big)  \circ ( \mathcal F^-_{N_S})^{-1}
$$
and
\begin{equation}\label{coefficients for Psi_1}
a^+_{k}(z_S; \Psi_1) := - W_{k}^{ae, +}(\cdot, z_S), \qquad 
a^-_{k}(z_S; \Psi_1) := - W_{k}^{ae, -}(\cdot, z_S) = \overline{ a^+_{k}(z_S; \Psi_1)} , 
\qquad \forall \, k \ge 1, 
\end{equation}
$$
{\cal R}_{N}(z_S; \Psi_1)[\widehat z_\bot] (x) :=  - \sum_{|n| \in [1, N_S] \setminus S_+} \widehat z_n W_n( x, z_S)  
-  \sum_{|n| > N_S} \widehat z_n \frac{{\cal R}_{N}^{W_n}(x, z_S)}{| n |^{N + 1}} e^{ \ii nx}\,.
$$
For any $s \ge 0$, the restriction of ${\cal R}_{N}(z_S; \Psi_1)$ to $h^s_\bot$ defines a bounded linear operator $h^s_\bot \to H^{s+ N+1}$ and the map
$$
\mathcal V_S \to {\cal B}(h^s_\bot, \, H^{s + N+1}), \ z_S \mapsto  {\cal R}_{N}(z_S; \Psi_1)\,,
$$
is real analytic. Correspondingly,  the map $\Psi_L: {\cal V} \to L^2_0,$ defined
in Proposition~\ref{prop Psi L Psi bo}, admits an expansion of the form 
$\Psi^{bo}(z_S, 0) +  \mathcal {OP}_N( z; \Psi_{L})+ {\cal R}_{N}(z; \Psi_L)$ 
where $\mathcal {OP}_N( z; \Psi_{L})$ is given by
\begin{equation}\label{definizione Psi L Phi 1}
\begin{aligned}
 \big( - g_\infty   +  \sum_{k = 1}^N  a^+_{k}(z_S; \Psi_L) D^{- k} \, \big) [( \mathcal F^+_{N_S})^{-1} z_\bot] \,
 + \,  \big( - {\overline{g_\infty}}  +   \sum_{k = 1}^N  a^-_{k}(z_S; \Psi_L) (-D)^{- k} \, \big) [ ( \mathcal F^-_{N_S})^{-1} z_\bot]
\end{aligned}
\end{equation}
with
\begin{equation}\label{coefficients for Psi_L}
a^\pm_{k}(z_S; \Psi_L):= a^\pm_{k}(z_S; \Psi_1), \quad \forall \, k \ge 1, \qquad  {\cal R}_{N}(z; \Psi_L):=  {\cal R}_{N}(z_S; \Psi_1)[z_\bot] \, .
\end{equation}
In particular, one has $a^-_{k}(z_S; \Psi_L) = \overline{a^+_{k}(z_S; \Psi_L)}$.
\end{corollary}

\medskip

\begin{corollary}\label{pseudodiff expansion dPsi_L}
For any $z = (z_S, z_\bot) \in \mathcal V$, $d\Psi_1(z): h^0_0 \to L^2$, given for
any $\widehat z = (\widehat z_S, \widehat z_\bot) \in h^0_0$ by
$$
d\Psi_1(z)[\widehat z] = d_S\big( \Psi_1(z_S)[z_\bot]\big)[\widehat z_S] + \Psi_1(z_S)[\widehat z_\bot]
$$
admits an expansion of the form 
$
\mathcal {OP}_N( z; d\Psi_{1})+ {\cal R}_{N}(z; d\Psi_1) \, ,
$
where the pseudo-differential operator $\mathcal {OP}_N( z; d\Psi_{1})$, when written as a $1 \times 2$ matrix, 
$$
\big( \mathcal {OP}_N( z; d\Psi_{1})^S \ \ \mathcal {OP}_N( z; d\Psi_{1})^\bot \big)\, , \qquad
\mathcal {OP}_N( z; d\Psi_{1})^S : h_S \to L^2 \, , \quad  \mathcal {OP}_N( z; d\Psi_{1})^\bot: h^0_\bot \to L^2 \, ,
$$ 
is given by
$$
\begin{aligned}
&  \mathcal {OP}_N( z; d\Psi_{1})^\bot =    \big( - g_\infty   \, + \,  \sum_{k = 1}^N  a^+_{k}(z_S; d\Psi_1) D^{- k} \, \big) \circ (\mathcal F_{N_S}^+ )^{-1} \
 + \  \big( -  {\overline{g_\infty}}  \, +  \,  \sum_{k = 1}^N  a^-_{k}(z_S; d\Psi_1) (-D)^{- k}  \, \big) \circ (\mathcal F_{N_S}^- )^{-1}  \, , \\
 & \mathcal {OP}_N( z; d\Psi_{1})^S  =  - \,  d_Sg_\infty [\, \cdot \,] \cdot   [( \mathcal F^+_{N_S})^{-1} z_\bot]  \ \ + \ \
 \sum_{k = 1}^N  \mathcal A^+_{k}(z_S; d\Psi_1)[\, \cdot \,] \cdot    D^{- k}  [(\mathcal F^+_{N_S})^{-1} z_\bot]  \\
 & \qquad \qquad  \qquad \quad   - \,    d_S \overline{g_\infty}[\, \cdot \, ] \cdot  [( \mathcal F^-_{N_S})^{-1} z_\bot] \  \ +  \ \
 \sum_{k = 1}^N    \mathcal A^-_{k}(z_S; d\Psi_1) [\, \cdot \, ] \cdot (-D)^{- k} [( \mathcal F^-_{N_S})^{-1} z_\bot] \, ,
 \end{aligned}
$$
where for any $k \ge 1$, $a^\pm_{k}(z_S; d\Psi_1) =  a^\pm_{k}(z_S; \Psi_1)$, and
 $\mathcal A^\pm_{k}(z_S; d\Psi_1)$ are the following linear operators,
$$
\mathcal A^\pm_{k}(z_S; d\Psi_1) : h_S \to L^2_\C , \, \widehat z_S \mapsto d_S a^\pm_k(z_S; \Psi_1)[\widehat z_S] \, .
$$
Furthermore, for any $s \geq 0$, $k \ge 1,$ the maps 
$$ \mathcal V_S \mapsto {\cal B}(h_S, H^s_\C), \, z_S \mapsto \mathcal A^\pm_k(z_S; d \Psi_1) \, , \qquad 
\mathcal V \cap h^s_0 \to {\cal B}\big(h^s_0, H^{s + N +1} \big), \, z \mapsto {\cal R}_{N}(z; d \Psi_1) \, ,
$$ 
are real analytic.
As a consequence, $a^-_{k}(z_S; d\Psi_1) = \overline{a^+_{k}(z_S; d\Psi_1)}$
and $\mathcal A^-_k(z_S; d \Psi_1)[z_S] = \overline{ \mathcal A^+_k(z_S; d \Psi_1)[z_S]}$.
\end{corollary}
\begin{remark}\label{extension of remainder 1} 
$(i)$ Note that for any $N \ge 1$ and $s \in \R$, $\mathcal {OP}_N(z_S; \Psi_{1}) : h^s_\bot \to H^s$ and ${\cal R}_{N}(z_S; \Psi_1): h^s_\bot \to H^{s +N + 1}$ are bounded linear operators.\\
$(ii)$ The fact that up to a remainder term, $ \Psi_L(z_S, \cdot)$ is given by a pseudo-differential operator of order 0,  
 acting on the scale of Hilbert spaces $h^s_\bot$, $s \in \Z_{\ge 0}$, 
 shows that the differential of the Birkhoff map $z \mapsto \Psi^{bo}(z)$ at a finite gap potential, has distinctive features.
 These features are the starting point for the construction of the coordinates of Theorem \ref{modified Birkhoff map}.
\end{remark}
\begin{remark}\label{notation coefficients expansion}
Whenever possible, we will use a similar notation for the coefficients of the expansion of various quantities 
to the one introduced for the coefficients of the expansion of $\Psi_1(z_S)$ and $\Psi_L(z_S, \cdot)$.
If the coefficients are operators, we use the upper case letter $\mathcal A$
and write $\mathcal A_k$ for the kth coefficient, whereas when they are functions (or operators, defined as multiplication by a function),
we use the lower case letter $a$ and write $a_k$ for the kth coefficient. The quantity, which is expanded, is indicated
as an argument of $\mathcal A_k$ and $a_k$.
 \end{remark}
 
\medskip

A straightforward application of Corollary \ref{pseudodiff expansion Psi_L}
 yields an expansion of the transpose $\Psi_1(z_S)^\top$ of the operator $\Psi_1(z_S)$.  
First note that the transpose  $({\cal F}^+_{N_S})^{- \top}$ of $({\cal F}^+_{N_S})^{- 1}$ 
with respect to the standard inner products in $L^2_0$ and $h^0_\bot$ 
is given by ${\cal F}^+_{N_S}$, i.e., for any $\widehat q \in L^2,$
$$
\begin{aligned}
\langle ({\cal F}^{+}_{N_S})^{-1}[z_\bot] \, |  \, \widehat q \rangle  & = 
\frac{1}{2\pi}\int_0^{2\pi} \sum_{n > N_S} z_n e^{ \ii n x} \widehat q (x) d x 
= \sum_{n > N_S} z_n \frac{1}{2\pi}\int_0^{2\pi} \widehat q(x) e^{ \ii n x} d x \\
&= \sum_{n > N_S} z_n \widehat q_{- n}
=  \sum_{n > N_S} z_n \overline{\widehat q_{n}}
= \langle z_\bot | {\cal F}^+_{N_S} \widehat q \rangle \,.
\end{aligned}
$$
Similarly, one computes the transpose  $({\cal F}^{-}_{N_S})^{- \top}$ of $( {\cal F}^{-}_{N_S})^{-1}$. For later reference we record 
\begin{equation}\label{transpose of Fourier restrictions}
({\cal F}^+_{N_S})^{- \top} =  {\cal F}^+_{N_S}\, , \qquad \quad
({\cal F}^-_{N_S})^{- \top} =  {\cal F}^-_{N_S}\, .
\end{equation}
\begin{corollary}\label{lemma asintotiche Phi 1 q}
For any $z_S \in \mathcal V_S$
and $N \in \N$,  
$\Psi_1(z_S)^\top: L^2_0 \to h^0_\bot, \,  \widehat q \mapsto (\langle W_{-n}(\cdot, z_S) |  \, \widehat q \rangle )_{n \in S^\bot}$ 
 has an expansion of the form $\mathcal {OP}_N(z_S; \Psi^\top_{1}) +  {\cal R}_{N}(z_S; \Psi_1^\top) $ with
\begin{align}
\mathcal {OP}_N(z_S; \Psi^\top_{1})  :=
{\cal F}^+_{N_S} \circ \big( -  \overline{g_\infty}   +  \sum_{k = 1}^N  a^+_{k}(z_S; \Psi_1^\top) D^{- k} \, \big)
\, + \,  {\cal F}^-_{N_S} \circ  \big(  -  g_\infty  +   \sum_{k = 1}^N  a^-_{k}(z_S; \Psi_1^\top) (-D)^{- k} \, \big)\, .  \label{espansione finale Phi 1 t bo}
\end{align}
For any $k \ge 1$, $a_k^-( z_S; \Psi_1^\top) = \overline{a_k^+( z_S; \Psi_1^\top)}$
and for any $s \geq 0$, the coefficients $\mathcal V_S \to H^s_\C$, $z_S \mapsto  a_k^\pm( z_S; \Psi_1^\top)$, $k \ge 1$, and the remainder
$\mathcal V_S \to {\cal B}(H^s, \, h^{s + N + 1}_\bot)$, $z_S \mapsto {\cal R}_N( z_S;  \Psi_1^\top),$ are real analytic. 

Corresponding properties hold for the map 
$$
\mathcal V \to \mathcal B(L^2_0, \, h^0_0), \, z \mapsto d\Psi_L(z)^\top \, . 
$$
For any $z \in \mathcal V$ and $N \in \N$, $d\Psi_L(z)^\top$ has an expansion of the form 
$\mathcal {OP}_N(z_S; d\Psi^\top_{L}) +  {\cal R}_{N}(z; d\Psi_L^\top) $
where $\mathcal {OP}_N(z_S; d\Psi^\top_{L})$ is given by
\begin{align}
 \Big( \,  0,  \, {\cal F}^+_{N_S} \circ \big( - \overline{g_\infty}   +  \sum_{k = 1}^N  a^+_{k}(z_S; d\Psi_L^\top) D^{- k} \, \big)
\, +  \, {\cal F}^-_{ N_S} \circ  \big( -  g_\infty  +   \sum_{k = 1}^N  a^-_{k}(z_S; d\Psi_L^\top) (-D)^{- k} \, \big)  \Big) 
 \label{expansion Phi L t}
\end{align}
with $a_k^{\pm}( z_S; d\Psi_L^\top) = a_k^{\pm}( z_S; \Psi_1^\top)$ 
and where for any integer $s \ge 0,$  
$$
\mathcal V \cap h^s_0 \to B(L^2, \, h^{s+N+1}_0), \, z \mapsto {\cal R}_N( z;  d\Psi_L^\top)
$$
is real analytic.
In particular, one has $a_k^-( z_S; d\Psi_L^\top) = \overline{a_k^+( z_S; d\Psi_L^\top)}$.
\end{corollary}
\begin{remark}\label{extension of remainder 2} 
 We record that for any $N \ge 1$ and $s \in \R$, the operators $\mathcal {OP}_N(z_S; \Psi^\top_{1}) : H^s \to h^s_\bot$
and ${\cal R}_N( z_S;  \Psi_1^\top) : H^s \to h_\bot^{s +N + 1}$ are bounded.
\end{remark}
\begin{proof}
Let $N \ge 1$ be given. By Corollary \ref{pseudodiff expansion Psi_L}, 
$\Psi_1(z_S) = \mathcal {OP}_N(z_S; \Psi_{1}) + {\cal R}_{N}(z_S; \Psi_1)$ where for any $z_S \in \mathcal V_S,$
$$
 \mathcal {OP}_N(z_S; \Psi_{1}) =   \big( -  g_\infty   +  \sum_{k = 1}^N  a^+_{k}(z_S; \Psi_1) D^{- k} \, \big) \circ ({\cal F}^+_{N_S})^{-1}  
\, + \,  \big( -  \overline{ g_\infty}  +   \sum_{k = 1}^N  a^-_{k}(z_S; \Psi_1) (-D)^{- k} \, \big)  \circ ({\cal F}^-_{N_S})^{-1}  \, . 
$$
Taking into account \eqref{coefficients for Psi_1} and  \eqref{transpose of Fourier restrictions}, $ \mathcal {OP}_N(z_S; \Psi_{1})^\top $ is given by
$$
\begin{aligned}
&  ({\cal F}^+_{N_S})^{-\top}  \circ \big( -  {\overline{ g_\infty}}  +   \sum_{k = 1}^N  D^{- k} \circ a^-_{k}(z_S; \Psi_1) \, \big)  
+  ({\cal F}^-_{N_S})^{-\top}  \circ \big( - g_\infty  +   \sum_{k = 1}^N   (-D)^{- k} \circ a^+_{k}(z_S; \Psi_1) \, \big)  \\
& =  {\cal F}^+_{N_S}  \circ \big( -   {\overline{g_\infty}}  +   \sum_{k = 1}^N   D^{- k} \circ a^-_{k}(z_S; \Psi_1) \, \big)  
\, + \, {\cal F}^-_{N_S}  \circ \big( -  g_\infty  +   \sum_{k = 1}^N   (-D)^{- k} \circ a^+_{k}(z_S; \Psi_1) \, \big) \, .
\end{aligned}
$$
By Lemma \ref{lemma composizione pseudo},  ${\cal F}^+_{N_S} \circ  D^{- k} \circ a_k^-(z_S; \Psi)$, $k \ge 1$, has an expansion of the form
$$
 {\cal F}^+_{N_S}  \circ \Big(   a^-_{k}(z_S; \Psi_1) \,  D^{ -k} + \sum_{j = 1}^{N - k} C^+_j(k) \, \partial_x^j a^{-}_{k}(z_S; \Psi_1)  \,  D^{ -k - j} \Big)
+  {\cal R}_{N, k, 0}^{\psi do}( a^-_{k} ) 
$$
where  $C^+_j(k)$, $j \ge 1$, are combinatorial constants and ${\cal R}_{N, k, 0}^{\psi do}( a^-_{k} )$ is a remainder term,
which can be obtained from Lemma \ref{lemma composizione pseudo}.
Similarly, one gets for $ {\cal F}^-_{N_S}  \circ  (-D)^{- k} \circ a^+_{k}(z_S; \Psi_1) $ the expansion
$$
 {\cal F}^-_{N_S}  \circ 
  \Big( a^+_{k}(z_S; \Psi_1) \,  (-D)^{ -k} + \sum_{j = 1}^{N - k} C^-_j(k) \, \partial_x^j a^{+}_{k}(z_S; \Psi_1)  \,  (-D)^{ -k - j} \Big)
+  {\cal R}_{N, k, 0}^{\psi do}( a^+_{k} ) .
$$
For any $k \ge 1$, $j \ge 1$, one has  $C^-_j(k) = \overline{ C^+_j(k)}$ and $ {\cal R}_{N, k, 0}^{\psi do}( a^-_{k} ) = \overline{ {\cal R}_{N, k, 0}^{\psi do}( a^+_{k} )}$.
In conclusion, $\mathcal {OP}_N(z_S; \Psi_{1})^\top$ has an expansion of the form
$$
 {\cal F}^+_{N_S} \circ \big( - \overline{g_\infty}   +  \sum_{k = 1}^N  a^+_{k}(z_S; \Psi_1^\top) D^{- k} \, \big)
\, + \,  {\cal F}^-_{N_S} \circ  \big(  -  g_\infty  +   \sum_{k = 1}^N  a^-_{k}(z_S; \Psi_1^\top) (-D)^{- k} \, \big) 
\, + \, \mathcal R^{(1)}_N(z_S) \, ,
$$
where the coefficients $a^\pm_{k}(z_S; \Psi_1^\top)$ have the claimed properties and where
 for any $s \ge 0$, the remainder $\mathcal V_S \to {\cal B}(H^s, \, h^{s + N + 1}_\bot)$, $z_S \mapsto {\cal R}^{(1)}_N( z_S),$ is real analytic
 By Corollary \ref{pseudodiff expansion Psi_L},  ${\cal R}_{N}(z_S; \Psi_1)^\top [\widehat q]$, $ \widehat q \in H^s$, equals
$$
\Big(   \big(  \frac{1}{2\pi}\int_0^{2 \pi} \widehat q(x) W_{-n}( x, z_S)  dx \big)_{|n| \in [1, N_S] \setminus S_+} , \
 \big( \frac{1}{| n|^{N+1}} \frac{1}{2\pi}\int_0^{2 \pi} \widehat q(x) {\cal R}_{N}^{W_{-n}}(x, z_S) e^{\ii nx} \, dx \big)_{|n| > N_S} \Big)
  \in h^{s + N +1}_\bot\,,
$$
implying that for any $s \ge 0$,
${\cal R}_{N}(z_S; \Psi_1)^\top: H^s \to h^{s + N + 1}_\bot$ is bounded, and 
the map $\mathcal V_S \to {\cal B}(H^s, h^{s + N + 1}_\bot)$, $z_S \mapsto {\cal R}_N( z_S;  \Psi_1)^\top$, is real analytic.
Setting $ {\cal R}_{N}(z_S; \Psi_1^\top) := {\cal R}^{(1)}_{N}(z_S) + {\cal R}_{N}(z_S; \Psi_1)^\top$, 
one infers that $\Psi_1(z_S)^\top$ admits an expansion of the form 
$\mathcal {OP}_N(z_S; \Psi^\top_{1}) +  {\cal R}_{N}(z_S; \Psi_1^\top) $ with the claimed properties. 

Since $d \Psi_L(z) [\widehat z] $ is given by the formula
$$
d \Psi_L(z) [\widehat z] = \big( d_S\Psi^{bo}(z_S, 0) +  d_S ( \Psi_1(z_S)[z_\bot]) \big) [\widehat z_S] + \Psi_1(z_S)[\widehat z_\bot],
$$
one verifies in a straightforward way the claimed expansion of $d \Psi_L(z)^\top$ from the one of $\Psi_1(z_S)^\top$.
\end{proof}

\medskip

Next we discuss the properties of the functions $W_n(\cdot, q)$, $n \ne 0$, and the map $\Psi_L(z)$
with regard to the reversible structures, defined by the maps $S_{rev}$ and $\mathcal S_{rev}$, introduced in Section \ref{introduzione paper}. 
For  notational convenience, we also denote 
 by $\mathcal S_{rev}$ the involution $h_0^0 \to h_0^0, (z_n)_{n \ne 0} \mapsto (z_{-n})_{n \ne 0}$, (cf. \eqref{definition reversible structure for actions angles}),
and sometimes write $q_*$ for $S_{rev}q(x)= q(-x) $, $q \in L^2$. 
By Lemma \ref{even symmetry}, for any $q \in L^2$,
\begin{equation}\label{symmetries 1}
\lambda_n(q_*) = \lambda_n(q), \quad f_n(x, q_*) = \overline{f_n(-x, q)}, \qquad \forall \, n \ge 0 \, ,
\qquad  g_\infty(x, q_*) = \overline{ g_\infty(- x, q)} \, ,
\end{equation}
and 
\begin{equation}\label{symmetries 2}
\gamma_n(q_*) =  \gamma_n(q)\, , \quad \kappa_n(q_*) = \kappa_n(q)\, , \qquad \forall n \ge 1\, .
\end{equation}
As a consequence, the manifolds $M_S$ and $M_S^o$ are left invariant by $S_{rev}$.
Without loss of generality, we will assume in the sequel that the neighborhood $\mathcal V$ is invariant under $\mathcal S_{rev}$.
\smallskip

\noindent
{\bf Addendum to Theorem \ref{lemma asintotica bo}}\label{Addendum Theorem 2.1} {\em (i) For any $z_S \in \mathcal V_S$, one has
$$
S_{rev}  \Psi^{bo} (z_S, 0) = \Psi^{bo} (\mathcal S_{rev}(z_S, 0)), \qquad
W_n(x, S_{rev}q) = W_{-n}( - x, q), \quad \forall \, n \ne 0,  \  x \in \R \, , 
$$
where $q = \Psi^{bo} (z_S, 0)$. Furthermore, for any $k \ge 1$
\begin{equation}\label{symmetries 1 of expansion bo}
W^{ae, +}_k(x, \mathcal S_{rev}z_S) =  W^{ae, -}_{k}( - x, z_S)\,,  \qquad
W^{ae, -}_k(x, \mathcal S_{rev}z_S) =  W^{ae, +}_{k}( - x, z_S)\,, 
\end{equation}
and any $N \ge 1$ and $|n| > N_S$
\begin{equation}\label{symmetries 2 of expansion bo}
 {\cal R}_{N}^{W_n}(x; \mathcal S_{rev}z_S) 
 =  {\cal R}_{N}^{W_{-n}}(- x; z_S)\,.
\end{equation}
(ii)
For any $z = (z_S, z_\bot) \in \mathcal V$, $x \in \R$, and $k \ge 1$,
$$
\big(\Psi_1(\mathcal S_{rev}z_S)[\, \mathcal S_{rev}z_\bot] \big)(x) =  \big( \Psi_1(z_S)[z_\bot] \big) (-x) \, ,
$$
$$
a^+_{k}(\mathcal S_{rev}z_S; \Psi_1)(x) = a^-_{k}(z_S; \Psi_1)(-x)\, , \qquad a^-_{k}(\mathcal S_{rev}z_S; \Psi_1)(x) = a^+_{k}(z_S; \Psi_1)(-x)
$$
As a consequence,  
for any $z \in \mathcal V$ and $N \ge 1$, one has  ${\cal R}_{N}( \mathcal S_{rev}z; \Psi_1) (x) =  {\cal R}_{N}(z; \Psi_1)(-x)$ and in turn
\begin{equation}\label{symmetries of PsiL bo}
(\Psi_L( \mathcal S_{rev}z)) (x) = ( \Psi_L(z) ) (-x)\,, \qquad a^\pm_{k}(\mathcal S_{rev}z_S; \Psi_L)(x) = a^\mp_{k}(z_S; \Psi_L)(-x)\, , \quad \forall \, k \ge 1\, ,
\end{equation}
as well as ${\cal R}_{N}( \mathcal S_{rev}z; \Psi_L) (x) =  {\cal R}_{N}(z; \Psi_L)(-x)$.\\
(iii) For any $z_S \in h_S$ and $\widehat q \in L^2_0$, one has
$ \Psi_1(\mathcal S_{rev} z_S)^\top [S_{rev} \widehat q]  = \mathcal S_{rev}( \Psi_1(z_S)^\top [ \widehat q])$. Furthermore,
 for any $k \ge 1$ and $N  \ge  1$,
$$
 a_k^+( \mathcal S_{rev}z_S; \Psi_1^\top) (x)  =   a^-_{k} (z_S; \Psi_1^\top) (-x) \,,  \qquad 
  a_k^-( \mathcal S_{rev}z_S; \Psi_1^\top) (x)  =   a^+_{k} (z_S; \Psi_1^\top) (-x) \,, 
 $$
 and
 $$
 {\cal R}_N(\mathcal S_{rev}z_S;  \Psi_1^\top)[S_{rev} \widehat q]  = \mathcal S_{rev} \big( {\cal R}_N(z_S; \Psi_1^\top) [\widehat q] \big)\,.
$$
}

\noindent
{\bf Proof of Addendum to Theorem \ref{lemma asintotica bo}}
(i)
By Proposition \ref{proposizione 1 reversibilita}, it follows that for any $z \in \mathcal V$,
$$
S_{rev}  \Psi^{bo} (z) = \Psi^{bo} (\mathcal S_{rev}(z)) \, .
$$ 
Hence, for any $z_S \in \mathcal V_S$ and $q =  \Psi^{bo} (z_S, 0)$, one has for any $n \ne 0$ (cf. definition \eqref{def W_n bo}),
\begin{equation}\label{symmetries W_n}
W_{ n}(x, S_{rev}q) = W_{ -n}( - x, q) = \overline{ W_{ n}( - x, q)} \, .
\end{equation}
The identities \eqref{symmetries 1 of expansion bo} follow from  \eqref{symmetries 1} - \eqref{symmetries 2} and the definition of $W^{ae, \pm}_k$ in Theorem \ref{lemma asintotica bo}.
Combining \eqref{symmetries 1 of expansion bo} and \eqref{symmetries W_n}, one infers  \eqref{symmetries 2 of expansion bo}.\\
(ii) By item (i) one has for any $z_S \in \mathcal V_S$ and $q = \Psi^{kdv} (z_S, 0)$,
\begin{equation}\label{symmetry Psi_1}
\big(\Psi_1(\mathcal S_{rev}z_S)[\, \mathcal S_{rev}z_\bot] \big)(x) = \sum_{n \in S^\bot} z_{-n} W_n( x, S_{rev} q) = \sum_{n \in S^\bot} z_{-n} W_{-n}( - x, q) = \big( \Psi_1(z_S)[z_\bot] \big) (-x)
\end{equation}
as well as $W^{ae, +}_k(x, \mathcal S_{rev}z_S) =  W^{ae, -}_{k}( - x, z_S)$ for any $k \ge 1$. Hence by the definition \eqref{coefficients for Psi_1},
\begin{equation}\label{symmetry coefficient Psi_1}
a^+_{k}(\mathcal S_{rev}z_S; \Psi_1)(x) = a^-_{k}(z_S; \Psi_1)(-x)\, , \qquad a^-_{k}(\mathcal S_{rev}z_S; \Psi_1)(x) = a^+_{k}(z_S; \Psi_1)(-x).
\end{equation}
Combining \eqref{symmetry Psi_1} and \eqref{symmetry coefficient Psi_1} then yields
$({\cal R}_{N}(\mathcal S_{rev}z_S; \, \Psi_1)[\mathcal S_{rev}z_\bot] )(x) = \big( {\cal R}_{N}(z_S; \Psi_1)[z_\bot] \big)(-x)$.
By \eqref{definition Psi_L bo}, the claimed identities \eqref{symmetries of PsiL bo} then follow. \\
(iii) Recall that for any $z_S \in h_S$, $\widehat q \in L^2_0$, one has 
$\Psi_1(z_S)^\top [\widehat q] = (\langle W_{-n}(\cdot, q), \, \widehat q \rangle )_{n \in S^\bot}$.
It then follows from item (i) that
\begin{equation}\label{symmetry transpose Psi_1}
\Psi_1(\mathcal S_{rev}z_S)^\top  [ S_{rev} \widehat q]  = \mathcal S_{rev} \big( \Psi_1(z_S)^\top [  \widehat q] \big).
\end{equation}
Taking into account the definition \eqref{coefficients for Psi_1} of the coefficients  $a^\pm_{k}(z_S; \Psi_1)$ in the expansion for $\Psi_1(z_S)$,
the construction of the coefficients $a^\pm_{k}(z_S; \Psi_1^\top)$  in the expansion of $\Psi_1(z_S)^\top$ 
in the proof of Corollary \ref{lemma asintotiche Phi 1 q} (cf. also Lemma \ref{lemma composizione pseudo}),
and the identities $\mathcal S_{rev} \circ \mathcal F^+_{N_S} = \mathcal F^-_{N_S} \circ S_{rev}$,
$\mathcal S_{rev} \circ \mathcal F^-_{N_S} = \mathcal F^+_{N_S} \circ S_{rev}$ 
 one infers that
$$
 a^+_k (\mathcal S_{rev} z_S; \Psi_1^\top) (x) =  a^-_k (z_S; \Psi_1^\top) (-x)\,, \quad   
  a^-_k (\mathcal S_{rev} z_S; \Psi_1^\top) (x) =  a^+_k (z_S; \Psi_1^\top) (-x)\,, \qquad \forall \, k \ge 1 .
 $$
 When combined with \eqref{symmetry transpose Psi_1}, one then arrives at
 $
 {\cal R}_N(\mathcal S_{rev}z_S;  \Psi_1^\top)[ S_{rev}\widehat q] = \mathcal S_{rev} ( {\cal R}_N(z_S; \Psi_1^\top) [ \widehat q] )\,.  
$
\hfill $\square$


\medskip

In the remaining part of this section we describe the pull back $\Psi_L^* \Lambda_G$ of the symplectic form $\Lambda_G$ by the map $\Psi_L$, defined in Proposition~\ref{prop Psi L Psi bo}.
The symplectic form $\Lambda_G$ is the one defined by the Gardner Poisson structure and is given by
$$
\Lambda_G [\widehat u , \widehat v] = \langle   \widehat u , \partial_x^{-1} \widehat v \rangle = \frac{1}{2\pi} \int_0^{2\pi} \widehat u (x)  \partial_x^{-1}  \widehat v(x) d x   \,, \qquad \forall \,  \widehat u, \widehat v  \in L^2_0\,.
$$
Note that $\Lambda_G = d \lambda_G$ where the one form $\lambda_G$, defined on $L^2_0$, is given by 
\begin{equation}\label{definition lambda G}
\lambda_G(u)[ \widehat v] = \langle  u , \partial_x^{-1} \widehat v \rangle = \frac{1}{2\pi} \int_0^{2\pi}  u(x) \partial_x^{-1}\widehat v(x) d x\,, \quad \forall \, u, \widehat v \in L^2_0\,.
\end{equation}
To compute the pull back of $\Lambda_G$ by $\Psi_L$, note that for any $z = (z_S, z_\bot) \in \mathcal V = \mathcal V_S \times \mathcal V_\bot$,
the derivative $d \Psi_L(z): h_S \times h^0_\bot \to L^2_0$, when written in $1 \times 2$ matrix form, is given by (cf \eqref{forma finale Psi L def}) 
\begin{align}\label{formula differential of Psi L}
d \Psi_L(z) & = d \Psi_L(z_S, 0) \ + \ \begin{pmatrix}
d_S ( \Psi_1(z_S)[z_\bot] ) & 0
\end{pmatrix}  
 = \begin{pmatrix}
d_S\Psi^{bo}(z_S, 0) &  \Psi_1(z_S)
\end{pmatrix} \ +  \ \begin{pmatrix}
d_S ( \Psi_1(z_S)[z_\bot])  & 0
\end{pmatrix} \,. 
\end{align}
%
For any $\widehat z = (\widehat z_{S}, \widehat z_{\bot}) ,$  $\widehat w = (\widehat w_{S}, \widehat w_{\bot})  \in h_S \times h^0_\bot$ one has
\begin{align}
& (\Psi_L^* \Lambda_G)(z) [\widehat z, \widehat w]  = \Lambda_G \big[ d \Psi_L(z) [\widehat z], \, d \Psi_L(z)[\widehat w ]\big] 
 =  \big\langle  d \Psi_L(z)[\widehat z], \, \partial_x^{- 1}d \Psi_L(z)[\widehat w] \big\rangle \nonumber\\
&  =
\big\langle  d \Psi_L(z_S, 0)[\widehat z] +  d_S( \Psi_1(z_S)[z_\bot] )[\widehat z_S] , \,
\partial_x^{- 1} \big( d \Psi_L(z_S, 0)[\widehat w]\big)  +  \partial_x^{- 1} \big( d_S  \big(\Psi_1(z_S)[z_\bot] \big) [\widehat w_S] ) \big)  \big\rangle \,. \nonumber
\end{align}
Since by construction, $d \Psi_L(z_S , 0) : h^0_0  \to  L^2_0$  is symplectic, one has 
$$
(\Psi_L^* \Lambda_G)(z_S, 0) = \Lambda
$$
where $\Lambda$ is the symplectic form on $h^0_0$,
\begin{equation}\label{1 forma Lambda M}
\Lambda [\widehat z, \widehat w] :=  \langle \widehat z ,  \, J^{-1} \widehat w \rangle = \sum_{n \ne 0} \frac{1}{  \ii (-n)} \widehat z_{ n} \widehat w_{-n } \,,
\quad \forall \widehat z, \widehat w \in h^0_0 \, ,
\end{equation}
and $J^{-1}$ denotes the inverse of the diagonal operator $J$, acting on the scale of Hilbert spaces $h^s_0,$ $s \in \R,$
\begin{equation}\label{definition J}
J: h^{s+1}_0 \to h^{s}_0 :\, (z_n)_{n \ne 0} \mapsto  ( \ii n z_n)_{n \ne 0} \, .  
\end{equation}
Note that the Poisson bracket, associated to $J$, is defined for functionals $F, G: h^0_0 \to \R$ with sufficiently decaying gradients,
$$
\{ F , G \}(z) = \langle J \nabla F, \nabla G \rangle (z) = \sum_{n \ne 0} \ii n (\nabla F)_n (\nabla G)_{-n} = \sum_{n \ne 0} \ii n \,  \partial_{z_{-n}} F \cdot \partial_{z_{n}} G \, .
$$
(Here we used that for any $n \ne 0$, $(\nabla F)_n = \partial_{z_{-n}} F$.)
We remark that $\langle J \widehat z ,  \, \widehat w \rangle = \langle J  \widehat z \,  |  \,  \widehat w \rangle$.

The two form $\Lambda$ is exact, $\Lambda = d \lambda$, where $\lambda$ is the following one form on $h^0_0$, 
\begin{equation}\label{definition lambda}
\lambda(z)[\widehat w] : =  \langle  z, J^{-1} \widehat w \rangle =  \sum_{n \ne 0} \frac{1}{ -\ii n} z_n \widehat w_{-n}\,, \quad \forall z, \widehat w \in h^0_0\, . 
\end{equation}
Altogether one concludes that  
\begin{equation}\label{Pull back of LambdaG}
(\Psi_L^* \Lambda_G)(z) [\widehat z, \widehat w]  = \Lambda[\widehat z, \widehat w] +   \Lambda_L(z)[ \widehat z, \widehat w]
\end{equation}
where
$$
\begin{aligned}
 \Lambda_L(z) & [ \widehat z, \widehat w]  
 =  \big\langle  d_S( \Psi_1(z_S)[z_\bot])[\widehat z_S] , \  \partial_x^{- 1} \big( d_S  (\Psi_1(z_S)[z_\bot] ) [\widehat w_S]  \big)  \big\rangle \\
& + \big\langle  d \Psi_L(z_S, 0)[\widehat z]  , \  \partial_x^{- 1} \big( d_S  (\Psi_1(z_S)[z_\bot] ) [\widehat w_S]  \big) \big\rangle
+ \big\langle  d_S( \Psi_1(z_S)[z_\bot] )[\widehat z_S] , \ \partial_x^{- 1} \big( d \Psi_L(z_S, 0)[\widehat w] \big)  \big\rangle \, .
\end{aligned}
$$
Using the standard inner products (cf. \eqref{standard inner products}) and the definition $\Psi_L(z_S, z_\bot) = \Psi^{bo}(z_S, 0) +  \Psi_1(z_S)[z_\bot]$,  
$\Lambda_L(z)[ \widehat z, \widehat w]$ can be written in the form
\begin{equation}\label{def mathcal L}
 \Lambda_L(z)[\widehat z, \widehat w]  = \big\langle \widehat z \, | \,  \mathcal L(z)[\widehat w] \big\rangle \, , 
\end{equation}
where
$ \mathcal  L(z) : h_S \times h^0_\bot \to  h_S \times h^0_\bot$ is given by
\begin{equation}\label{definition L z}
 \mathcal L(z) = \begin{pmatrix}
 \mathcal L_S^S(z) &  \mathcal L_S^{\bot}( z) \\
 \mathcal L_{\bot}^S( z) & 0
\end{pmatrix} 
\end{equation}
with $ \mathcal L_S^S(z) : h_S \to h_S$, $ \mathcal L_S^{\bot}(z) : h^0_\bot  \to h_S$, and $ \mathcal L_\bot^S : h_S \to h^0_\bot$ given by 
\begin{equation}\label{L SS botS Sbot}
\begin{aligned}
 \mathcal L_S^S(z) & := \big(d_S \Psi_1(z_S)[z_\bot]\big)^\top \circ \partial_x^{- 1} \big( d_S \Psi_1(z_S)[ z_\bot] \big)
 + (d_S \Psi^{bo}(z_S, 0))^\top \circ \partial_x^{- 1} \big( d_S \Psi_1(z_S)[ z_\bot] \big)\\
 & + \big( d_S \Psi_1(z_S)[ z_\bot] \big)^\top \circ \partial_x^{-1} d_S \Psi^{bo}(z_S, 0) \,, \\
 \mathcal L_S^{\bot}(z) & :=  \big(d_S \Psi_1(z_S)[z_\bot] \big)^\top \circ  \partial_x^{- 1} \Psi_1(z_S) \,, \qquad
 \mathcal L_\bot^S (z)  :=\Psi_1(z_S)^\top \circ \partial_x^{- 1} \big( d_S \Psi_1(z_S)[z_\bot] \big) \, ,
\end{aligned}
\end{equation}
where we recall that $(\, \cdot \, )^\top$ denotes the transpose of an operator (such as the ones considered above) with respect to the standard inner products defined  in \eqref{standard inner products}.
For any $z = (z_S, z_\bot) \in \mathcal V$, the operators $ \mathcal L(z),$ $ \mathcal L_S^S(z),$  $ \mathcal L_S^{\bot}(z),$ and $ \mathcal L_\bot^S(z)$ are bounded.
In the sequel, we will often write the operators defined in \eqref{L SS botS Sbot} in the following way
\begin{equation}\label{L SS botS Sbot (1)}
\begin{aligned}
 \mathcal L_S^S(z) [\widehat z_S] & = 
 \big( \big\langle \partial_x^{- 1} d_S \big(\Psi_1(z_S)[z_\bot] \big)[\widehat z_S]\, | \ \partial_{z_n}\Psi_1(z_S)[z_\bot] \big\rangle \big)_{n \in S}
  + \big( \big\langle \partial_x^{- 1} d_S \big(\Psi_1(z_S)[z_\bot] \big)[\widehat z_S]\, | \ \partial_{z_n}\Psi^{bo}(z_S, 0) \big\rangle \big)_{n \in S}\\
  & +  \big( \big\langle \partial_x^{- 1} d_S (\Psi^{bo}(z_S, 0))[\widehat z_S] \, | \ \partial_{z_n}\Psi_1(z_S)[z_\bot] \big\rangle \big)_{n \in S} \,, \\
 \mathcal L_S^{\bot}(z)[\widehat z_\bot] & =  \big( \big\langle \partial_x^{- 1} \Psi_1(z_S)[\widehat z_\bot]\, | \ \partial_{z_n} \Psi_1(z_S)[z_\bot] \big\rangle \big)_{n \in S} \,, \\
 \mathcal L_\bot^S(z) [\widehat z_S] & = \Psi_1(z_S)^\top \partial_x^{- 1} d_S (\Psi_1(z_S)[z_\bot] ) [\widehat z_S] =
 \big( \big\langle   \partial_x^{- 1} \, d_S (\Psi_1(z_S)[z_\bot] ) [\widehat z_S] \, | \ W_{n}(\cdot, q) \big\rangle \big)_{n \in S^\bot} \, ,
\end{aligned}
\end{equation}
where $q = \Psi^{bo}(z_S, 0)$.
It follows from \eqref{Pull back of LambdaG} that $\mathcal L(z)$ is skew-adjoint, $\mathcal L(z)^\top = - \mathcal L(z)$. 
As a consequence, one has
\begin{equation}\label{mathcal L transpose}
  \begin{pmatrix}
 \mathcal L_S^S(z)^\top &  \mathcal L^S_{\bot}( z)^\top \\
 \mathcal L^{\bot}_S(z)^\top & 0
\end{pmatrix} 
=\begin{pmatrix}
 - \mathcal L_S^S(z) &  - \mathcal L_S^{\bot}( z) \\
-  \mathcal L_{\bot}^S( z) & 0
\end{pmatrix} \, .
\end{equation}
The operators $\mathcal L_S^S(z)$, $\mathcal L_S^{\bot}(z)$, and $\mathcal L_\bot^S(z)$ satisfy the following properties.
\begin{lemma}\label{espansione L S bot q z}
$(i)$ The maps 
$$
\mathcal V \to {\cal B}(h_S, h_S), \, z \mapsto  \mathcal L_S^S(z) \,, \qquad  
\mathcal V \to {\cal B}\big(h^0_\bot, h_S \big), \,  z \mapsto  \mathcal L_S^{\bot}(z) \, ,
$$
 are real analytic. Furthermore, the following estimates hold: for any $z= (z_S, z_\bot) \in \mathcal V$, $\widehat z_S \in h_S$, 
   and  $\widehat z_1, \ldots, \widehat z_l \in  h^0_0$, $l \geq 1$ ,
$$
\begin{aligned}
& \|  \mathcal L_S^S(z)[\widehat z_S]\| \lesssim \| \widehat z_S\| \| z_\bot \|_0 \,, \quad 
& \| d^l \big(  \mathcal L_S^S(z)[\widehat z_S] \big)[\widehat z_1, \ldots, \widehat z_l] \| 
\lesssim_{l} \| \widehat z_S\| \prod_{j = 1}^l \| \widehat z_j\|_0 \, .
\end{aligned}
$$
If in addition, $\widehat z_\bot \in h^0_\bot$,
$$
\begin{aligned}
& \|  \mathcal L_S^{\bot}(z) [\widehat z_\bot]\| \lesssim \| z_\bot \|_0 \| \widehat z_\bot \|_0\,, \qquad
 \| d^l \big(  \mathcal L_S^{\bot}(z) [\widehat z_\bot] \big)[\widehat z_1, \ldots, \widehat z_l] \| \lesssim_{l} \| \widehat z_\bot \|_0 \prod_{j = 1}^l \| \widehat z_j\|_0\,. 
\end{aligned}
$$
\noindent
$(ii)$  
For any $z = (z_S, z_\bot) \in \mathcal V$, $\widehat z_S \in h_S$, and arbitrary order $N \ge 1$,
$$ 
\mathcal L_\bot^S (z)[\widehat z_S] = \mathcal {OP}_N(z_S; \mathcal L_\bot^S)[\widehat z_S] +  {\cal R}_{N}(z; \mathcal L_\bot^S)[\widehat z_S]
$$ 
with $\mathcal {OP}_N(z_S; \mathcal L_\bot^S)[\widehat z_S] $ given by
 \begin{equation}\label{expansion L bot S}
 {\cal F}^+_{N_S} \circ \sum_{k = 1}^N  \mathcal A^+_k(z_S; \mathcal L_\bot^S)[\widehat z_S]  \cdot D^{- k} [({\cal F}^+_{N_S})^{- 1} z_\bot] 
 + {\cal F}^-_{N_S} \circ \sum_{k = 1}^N  \mathcal A^-_k(z_S; \mathcal L_\bot^S)[\widehat z_S]  \cdot (-D)^{- k} [({\cal F}^-_{N_S})^{- 1} z_\bot]  \, ,
\end{equation}
where for any $s \geq 0$, $k \ge 1,$ the maps 
$$ \mathcal V_S \mapsto {\cal B}(h_S, H^s_\C), \, z_S \mapsto \mathcal A^\pm_k(z_S; \mathcal L_\bot^S) \,, \qquad 
\mathcal V \cap h^s_0 \to {\cal B}\big(h_S, h^{s + N +1}_\bot \big), \, z \mapsto {\cal R}_{N}(z; \mathcal L_{\bot}^S) \, ,
$$ 
are real analytic and $\mathcal A^-_k(z_S; \mathcal L_\bot^S) [\widehat z_S]= \overline{\mathcal A^+_k(z_S; \mathcal L_\bot^S)[\widehat z_S]}$.
For any $\widehat z_S \in h_S$, $ \mathcal A^\pm_1(z_S; \mathcal L_\bot^S)[\widehat z_S]$ are given by
$$
\mathcal A^+_1(z_S; \mathcal L_\bot^S)[\widehat z_S] = -\ii  \, \overline{g_\infty} \cdot d_S g_\infty[\widehat z_S]  \, ,
\qquad
\mathcal A^-_1(z_S; \mathcal L_\bot^S)[\widehat z_S] =  \ii  g_\infty \cdot d_S \overline{g_\infty}[\widehat z_S] \, .
$$
In particular, the operator $\mathcal L_\bot^S(z)$ is one smoothing. More precisely, for any $s \ge 0$,
$$
 \mathcal V \cap h_0^s \to {\cal B}(h_S, h^{s+1}_\bot), \, z \mapsto  \mathcal L_\bot^S(z)
$$
is real analytic. 
The coefficients  $\mathcal A^\pm_k(z_S; \mathcal L_\bot^S)$ are independent of $z_\bot$ and satisfy for any $s \ge 0,$ $z_S \in \mathcal V_S,$
$\widehat z_S \in h_S,$ $\widehat z_1, \ldots, \widehat z_l \in h^0_0$, $l \geq 1$, the following estimates
$$
\| \mathcal A^\pm_k ( z_S; \mathcal L_{\bot}^S) [\widehat z_S]  \|_s \lesssim_{s, k}  \| \widehat z_S \|\,, \qquad
 \| d^l \big( \mathcal A^\pm_k ( z_S; \mathcal L_{\bot}^S) [\widehat z_S]  \big) [\widehat z_1, \ldots, \widehat z_l] \|_{s}  
 \lesssim_{s, k, l}  \| \widehat z_S \| \, \prod_{j = 1}^l \| \widehat z_j\|_0 \,.
$$
Furthermore, for any $s \in \Z_{\ge 0}$, $z = (z_S, z_\bot) \in \mathcal V \cap h^s_0$, $\widehat z_S \in h_S$, and
$\widehat z_1, \ldots, \widehat z_l \in h^s_0$, $l \geq 1$, ${\cal R}_{N}(z; \mathcal L_{\bot}^S) [\widehat z_S]$ satisfies
$ \| {\cal R}_{N}(z; \mathcal L_{\bot}^S) [\widehat z_S]\|_{s + N+1} \lesssim_{s, N} \| \widehat z_S\|  \| z_\bot \|_s$ and 
$$
 \| d^l \big( {\cal R}_{N}(z; \mathcal L_{\bot}^S) [\widehat z_S]  \big) [\widehat z_1, \ldots, \widehat z_l] \|_{s + N+1} 
\lesssim_{s, N, l}  \| \widehat z_S \|  \cdot \big( \sum_{j = 1}^l \| \widehat z_j\|_s \prod_{i \neq j} \| \widehat z_i\|_0 \, + \, \| z_\bot \|_s \prod_{j = 1}^l \| \widehat z_j\|_0 \big) \,. 
$$
(iii) As a consequence, for any integer $s \ge 0$, the map $\mathcal V \cap h^s_0 \to \mathcal B(h^0_0, h^{s+1}_0), z \mapsto  \mathcal L(z)$ is real analytic. Furthermore,
for any $z = (z_S, z_\bot) \in \mathcal V \cap h^s_0$ and $\widehat z \in h^0_0$, it satisfies the estimates
\begin{equation}\label{estimate mathcal L}
\|  \mathcal L(z) [\widehat z]\|_{s+1} \le C(s; \mathcal L)  \| \widehat z \|_0 \| z_\bot \|_s
\end{equation}
and if in addition $\widehat z_1, \ldots, \widehat z_l \in h^s_0$, $l \geq 1$, one has
\begin{equation}\label{estimate derivative mathcal L}
\| d^l ( \mathcal L(z) [\widehat z] )[\widehat z_1, \ldots , \widehat z_l] \|_{s+1} 
\le C(s, l; \mathcal L) \cdot  \| \widehat z\|_0 \cdot 
\big( \sum_{j = 1}^l \| \widehat z_j\|_s \prod_{i \neq j} \| \widehat z_i\|_0 \, + \, \| z_\bot \|_s \prod_{j = 1}^l \| \widehat z_j\|_0 \big) 
\end{equation}
for some constants $C(s; \mathcal L) \ge 1$, $C(s, l; \mathcal L) \ge 1$. 
\end{lemma}
\begin{remark}\label{extension of mathcal L(z)}
$(i)$ The coefficients $\mathcal A_1^\pm(z_S; \mathcal L_\bot^S)$ can be computed more explicitly. First note that 
for any $n \ge 1$, $\partial_x^{-1} e^{\ii n x} = -\ii D^{-1} e^{\ii n x} $ and $\partial_x^{-1} e^{- \ii n x} = \ii (-D)^{-1} e^{-\ii n x} $
and that by definition, $g_\infty = e^{\ii \partial_x^{-1} q}$ with $q = \Psi^{bo}(z_S, 0)$ real valued. Hence for any $n \ge 1$,
$$
\overline{g_\infty} \cdot d_S g_\infty[\widehat z_S] \cdot \partial_x^{-1} e^{\ii nx} =
\overline{g_\infty} g_\infty \ii  \partial_x^{-1} (d_S \Psi^{bo}(z_S, 0)[\widehat z_S]) \cdot (-\ii) D^{-1} e^{\ii nx}
=  \partial_x^{-1} (d_S \Psi^{bo}(z_S, 0)[\widehat z_S]) \cdot D^{-1} e^{\ii nx} \, ,
$$
yielding
$$
\mathcal A^+_1(z_S; \mathcal L_\bot^S)[\widehat z_S] = \partial_x^{-1} (d_S \Psi^{bo}(z_S, 0)[\widehat z_S]) \, , \qquad
\mathcal A^-_1(z_S; \mathcal L_\bot^S)[\widehat z_S] =  \partial_x^{-1} (d_S \Psi^{bo}(z_S, 0)[\widehat z_S]) \, .
$$
Hence $\mathcal A^+_1(z_S; \mathcal L_\bot^S)[\widehat z_S] = \mathcal A^-_1(z_S; \mathcal L_\bot^S)[\widehat z_S] $
and $\mathcal A^\pm_1(z_S; \mathcal L_\bot^S)[\widehat z_S] $ are real valued.

\smallskip
\noindent
$(ii)$ Recall that by Remark \ref{extension of remainder 1}(i),  $ \partial_x^{- 1} \Psi_1(z_S): h^{-1}_\bot \to H^0_0$ 
is a bounded linear operator for any $z \in \mathcal V$. Since $\partial_{z_n} \Psi_1(z_S)[z_\bot] \in H^0_0$ for any $n \in S$, it then follows that
$ \mathcal L_S^{\bot}(z): h^{-1}_\bot \to h_S$
and in turn $\mathcal L(z): h^{-1}_0 \to h^0_0$ are bounded linear operators. Estimates, corresponding to the ones for $\mathcal L_S^{\bot}(z)$ and $\mathcal L(z)$ 
of Lemma \ref{espansione L S bot q z}, continue to hold, when these operators are extended to $h^{-1}_\bot$ and, respectively, $h^{-1}_0$.
\end{remark}
\begin{proof}
Item (i) is verified in a straightforward way and
item (iii) is a direct consequence of item (i) and (ii). It remains to verify the statements of item (ii).
Recall that by definition, 
$$
 \mathcal L_\bot^S(z) [\widehat z_S]  = \Psi_1(z_S)^\top \partial_x^{- 1} d_S (\Psi_1(z_S)[z_\bot] ) [\widehat z_S] \,  .
 $$
By Corollary \ref{lemma asintotiche Phi 1 q}, $\Psi_1(z_S)^\top$ has the expansion
$\mathcal {OP}_N(z_S; \Psi^\top_{1}) +  {\cal R}_{N}(z_S; \Psi_1^\top) $ with
\begin{align}
\mathcal {OP}_N(z_S; \Psi^\top_{1})  :=
 {\cal F}^+_{N_S} \circ \big( - \overline{g_\infty}   +  \sum_{k = 1}^N  a^+_{k}(z_S; \Psi_1^\top) D^{- k} \, \big)
\, + \,  {\cal F}^-_{N_S} \circ  \big(  \- g_\infty  +   \sum_{k = 1}^N  a^-_{k}(z_S; \Psi_1^\top) (-D)^{- k} \, \big)\, ,  \label{espansione finale Phi 1 t bo}
\end{align}
where for any $k \ge 1$, $a_k^-( z_S; \Psi_1^\top) = \overline{a_k^+( z_S; \Psi_1^\top)}$
and for any $s \geq 0$, $k \ge 1$,
$$\mathcal V_S \to H^s_\C, \ z_S \mapsto  a_k^\pm( z_S; \Psi_1^\top) \, , \qquad
\mathcal V_S \to {\cal B}(H^s, \, h^{s + N + 1}_\bot), \ z_S \mapsto {\cal R}_N( z_S;  \Psi_1^\top),
$$ 
are real analytic. 
By Corollary \ref{pseudodiff expansion dPsi_L} , $d_S(\Psi_1(z_S)[z_\bot])[\widehat z_S]$ admits the expansion 
$$
\begin{aligned}
 & - \    d_Sg_\infty[\widehat z_S]  \cdot  ( \mathcal F^+_{N_S})^{-1} z_\bot \ \ + \ \
 \sum_{k = 1}^N  \mathcal A^+_{k}(z_S; d\Psi_1)[\widehat z_S] \cdot   D^{- k}  [( \mathcal F^+_{N_S})^{-1} z_\bot] \\
 & \,  - \   d_S \overline{g_\infty}[\widehat z_S] \cdot ( \mathcal F^-_{N_S})^{-1} z_\bot \  \ +  \ \
 \sum_{k = 1}^N  \mathcal A^-_{k}(z_S; d\Psi_1)[\widehat z_S] \cdot  (-D)^{- k} [( \mathcal F^+_{N_S})^{-1} z_\bot] \, , 
 \end{aligned}
$$
where for any $k \ge 1$, $\mathcal A^\pm_{k}(z_S; d\Psi_1)$ are the following linear operators,
$$
\mathcal A^\pm_{k}(z_S; d\Psi_1) : h_S \to L^2_\C , \, \widehat z_S \mapsto d_S a^\pm_k(z_S; \Psi_1)[\widehat z_S] \, .
$$
One has 
It then follows from the expansion of the composition $\partial_x^{-n} \circ a \, \partial_x^{-\ell}$ (cf. Lemma \ref{lemma composizione pseudo})
and the smoothing properties of Hankel operators (cf. Corollary \ref{Hankel infinitely smoothing})  
that $\mathcal L_\bot^S (z)[\widehat z_S] $ admits an expansion with the stated properties.
\end{proof}

\smallskip

Finally, we discuss the properties of the symplectic forms $\Lambda_G$, $\Lambda$, and $\Psi_L^*\Lambda_G$ with respect to
the reversible structures introduced in Section \ref{introduzione paper}. 
First note that for any $\widehat u, \widehat v \in L^2_0,$
$$
(S_{rev}^*\Lambda_G)[ \widehat u, \widehat v] = \Lambda_G[ S_{rev}\widehat u, S_{rev}\widehat v] =
\frac{1}{2\pi} \int_0^{2\pi}  \widehat u(-x) \partial_x^{-1}( \widehat v(-x) ) d x =  - \Lambda_G[ \widehat u, \widehat v]
$$
and similarly, for any $\widehat z, \widehat w \in h^0_0,$
$$
(\mathcal S_{rev}^*\Lambda)[ \widehat z, \widehat w] = \Lambda[ \mathcal S_{rev}\widehat z, \mathcal S_{rev}\widehat w] =
\sum_{n \ne 0} \frac{1}{- \ii n} \widehat z_{-n} \widehat w_{n}  = - \Lambda[ \widehat z, \widehat w]\,.
$$
Since by \eqref{symmetries of PsiL bo} 
$\mathcal S_{rev} \Psi_L = \Psi_L S_{rev}$, 
the pullback $\mathcal S_{rev}^* \Psi_L^*\Lambda_G$ can then be computed as
$$
(\mathcal S_{rev}^*\Psi_L^*) \Lambda_G  = \Psi_L^* (S_{rev}^* \Lambda_G) = - \Psi_L^*\Lambda_G \, .
$$
By \eqref{Pull back of LambdaG} it then follows that the two form $\Lambda_L$, introduced in \eqref{Pull back of LambdaG}, satisfies
$$
\mathcal S_{rev}^*\Lambda_L = - \Lambda_L \,.
$$
In view of the definitions \eqref{def mathcal L} - \eqref{definition L z}
one then concludes that the operators $ \mathcal L_S^S(z)$, $ \mathcal L_S^{\bot}(z)$, and $ \mathcal L_\bot^S(z)$ have the following symmetry properties.

\medskip

\noindent
{\bf Addendum to Lemma \ref{espansione L S bot q z}} {\em For any $z = (z_S, z_\bot) \in h_S \times h^0_\bot$ and any $\widehat z_S \in h_S$, $\widehat z_\bot \in h^0_\bot$,
$$
 \mathcal L_S^S(\mathcal S_{rev} z) [\mathcal S_{rev} \widehat z_S] = -  \mathcal S_{rev} \big( \mathcal L_S^S(z) [\widehat z_S] \big)\,, \quad 
$$
$$
 \mathcal L_S^{\bot}(\mathcal S_{rev}z)[\mathcal S_{rev} \widehat z_\bot]   = - \mathcal S_{rev} \big( \mathcal L_S^{\bot}(z)[\widehat z_\bot] \big)\,, \quad
 \mathcal L_\bot^S(\mathcal S_{rev}z) [\mathcal S_{rev} \widehat z_S]  \big)= - \mathcal S_{rev} \big( \mathcal L_\bot^S(z) [\widehat z_S] \big)\,.
$$
By \eqref{expansion L bot S} it then follows that
$$
\mathcal A^+_k(\mathcal S_{rev} z_S; \mathcal L_\bot^S) [\mathcal S_{rev} \widehat z_S] (x)  = -  \mathcal A^-_k(z_S; \mathcal L_\bot^S) [\widehat z_S] (-x) \, ,
$$
$$
\mathcal A^-_k(\mathcal S_{rev} z_S; \mathcal L_\bot^S) [\mathcal S_{rev} \widehat z_S] (x)  = -  \mathcal A^+_k(z_S; \mathcal L_\bot^S) [\widehat z_S] (-x) \, ,
$$
$$
\qquad
{\cal R}_{N}(\mathcal S_{rev}z;  \mathcal L_{\bot}^S)[\mathcal S_{rev} \widehat z_S]  = -  \mathcal S_{rev} \big({\cal R}_{N}(z; \mathcal L_{\bot}^S)[\widehat z_S] \big) \, .
$$
}

\section{The map $\Psi_C$}\label{sezione Psi C}

\noindent
In this section we construct the symplectic corrector $\Psi_C$. Our approach is based on a well known method of Moser and Weinstein,
implemented for an infinite dimensional setup in \cite{K} (cf. also \cite{Kappeler-Montalto}).
We begin by briefly outlining the construction. At the end of Section \ref{sezione mappa Psi L BO}, we introduce the symplectic forms $\Lambda$
and $\Psi_L^*\Lambda_G$ (cf. \eqref{def mathcal L}). They are defined on $\mathcal V = \mathcal V_S \times \mathcal V_\bot$ 
and are related as follows ($z \in \mathcal V,$  $\widehat z, \widehat w \in h^0_0$),
\begin{equation}\label{setup pullback Lambda_G}
\Psi^*_L \Lambda_G (z) [\widehat z, \widehat w] = \Lambda[\widehat z, \widehat w] + \Lambda_L(z) [\widehat z, \widehat w]\,, 
\qquad \Lambda [\widehat z, \widehat w] =  \langle \widehat z \, |  \, J^{-1} \widehat w \rangle \, ,
\qquad \Lambda_L(z) [\widehat z, \widehat w] = \langle \widehat z \, | \,  \mathcal L(z) [\widehat w] \rangle\,, 
\end{equation}
where $J^{-1}$ is the inverse of the diagonal operator $J$, defined by \eqref{definition J}
and $\mathcal L(z)$ the operator defined by \eqref{definition L z}. 
Our candidate for $\Psi_C$ is $\Psi_X^{0,1}$ where $X \equiv X(\tau, z)$
is a non-autonomous vector field, defined for $z \in \mathcal V$ and $0 \le \tau \le 1$,
so that $(\Psi_X^{0, 1})^* (\Psi^*_L \Lambda_G) = \Lambda$.
The  flow $\Psi^{\tau_0, \tau}_X$, corresponding to the vector field $X$, is required to be well defined on a neighborhood $\mathcal V'$ 
(cf. Lemma \ref{estimates for flow} below) for $0 \le \tau_0, \tau \le 1$, 
and to satisfy the standard normalization conditions $\Psi^{\tau_0, \tau_0}_X(z) = z$ for any $z \in \mathcal V'$ and $0 \le \tau_0 \le 1.$
To find $X$ with the desired properties, introduce the one parameter family of two forms,
 $$
 \Lambda_\tau (z) : = \Lambda + \tau \Lambda_L(z)\,, \quad 0 \le \tau \le 1\,.
 $$
 Note that $\Lambda_\tau$ is a closed two form. Indeed, 
 since $\Lambda_G$ is such a two form, so is $\Psi_L^*\Lambda_G$ and in turn $\Lambda_L = \Psi_L^*\Lambda_G - \Lambda$.
Furthermore, $\Lambda_0 = \Lambda$, $\Lambda_1 = \Psi_L^* \Lambda_G$, and $(\Psi_X^{0,0})^* \Lambda_0 = \Lambda_0.$
The desired identity $(\Psi_X^{0,1})^* \Lambda_1 = \Lambda_0$ then follows if one can show that $(\Psi_X^{0,\tau})^* \Lambda_\tau$
is independent of $\tau$, i.e., that $\partial_\tau \big( (\Psi_X^{0,\tau})^* \Lambda_\tau \big) = 0$. 
Since $(\Psi_X^{0,\tau})^* \Lambda_\tau = (\Psi_X^{0,\tau})^* \Lambda + \tau (\Psi_X^{0,\tau})^* \Lambda_L$, it follows that
$$
\partial_\tau \big( (\Psi_X^{0,\tau})^* \Lambda_\tau \big) = (\Psi_X^{0,\tau})^* (\frak L_{X} \Lambda_\tau + \Lambda_L) \, ,
$$
where $\frak L_{X}$ denotes the Lie derivative with respect to the vector  field $X$.
Using that $\Lambda_\tau$ is a closed two form, one infers from Cartan's formula that
$$
\frak L_{X} \Lambda_\tau = d(\iota_X \Lambda_\tau) + \iota_X d \Lambda_\tau = d(\iota_X \Lambda_\tau)\, ,
$$
where $\iota_X$ denotes the interior product of a form with the vector field $X$. 
Thus, the equation $\partial_\tau \big( (\Psi_X^{0,\tau})^* \Lambda_\tau \big) = 0$ can be written as
$$
(\Psi_X^{0,\tau})^* \big( \, \Lambda_L + d (\Lambda_\tau [ X(\tau, \, \cdot), \, \cdot \, ] )\, \big) = 0 \,.
$$
Hence we need to choose the vector field $X( \tau, z)$ in such a way that
\begin{equation}\label{equation for X}
\Lambda_L(z) -  d \big( \, \Lambda_\tau(z) [ \, \cdot \,  , \, X(\tau, z)] \, \big) = 0 \, .
\end{equation}
Note that 
\begin{equation}\label{operator for Lambda_tau}
\Lambda_\tau(z) [\, \cdot \, , X( \tau, z)] =  \langle \, \cdot \, | \, (J^{-1}  + \tau \mathcal L(z)) [X(\tau, z)] \rangle = \langle \,  \cdot \, | \, J^{-1} \mathcal L_\tau(z) [X(\tau, z)] \rangle 
\end{equation}
where
\begin{equation}\label{definition L tau}
\mathcal L_\tau(z) :=  {\rm Id} + \tau J \mathcal L(z) \, .
\end{equation}
We remark that by Lemma \ref{espansione L S bot q z},  the operator $ \mathcal L_\tau(z) : h^0_0  \to  h^0_0$  is bounded  for any $0 \le \tau \le 1$ and $z \in \mathcal V$.

In order to find a vector field $X$, satisfying \eqref{equation for X},
we rewrite the two form $\Lambda_L(z)$ as the differential of a suitably chosen one form. First note that
since $\Lambda_G = d \lambda_G$ (cf. \eqref{definition lambda G}) and $\Lambda = d \lambda$ (cf. \eqref{definition lambda}),
one has 
$$
\Lambda_L = \Psi^*_L \Lambda_G - \Lambda = 
d (\lambda_1 - \lambda_0)\, , \qquad  \lambda_1 := \Psi^*_L \lambda_G \, , \qquad \lambda_0 := \lambda \, .
$$
 Furthermore, 
 $\mathcal L(z_S, 0) = 0$ by \eqref{L SS botS Sbot (1)} and hence $\Lambda_L(z_S, 0) = 0$ for any $z_S \in \mathcal V_S$. 
It then follows by the Poincar\'e Lemma (cf. e.g. \cite[Appendix 1]{Kappeler-Montalto}) that  $d \lambda_L = \Lambda_L$ where
$\lambda_L$ is the one form on $\mathcal V$, obtained by the cone construction,
$$
\lambda_L(z) [\widehat z] := \int_0^1 \Lambda_L(z_S, \, tz_\bot) [(0, \, z_\bot),  (\widehat z_S, \,  t \widehat z_\bot) ] d t 
= - \int_0^1 \Lambda_L(z_S, \, tz_\bot) [(\widehat z_S, \,  t \widehat z_\bot), (0, \, z_\bot) ] d t .
$$
By the definition of the operator $\mathcal L(z)$ (cf. \eqref{setup pullback Lambda_G}), it then follows that
$$
\lambda_L(z) [\widehat z] = -  \int_0^1 \langle  (\widehat z_S, t \widehat z_\bot) \, | \, \mathcal L(z_S, tz_\bot)[(0, z_\bot)] \rangle d t  
\stackrel{\eqref{definition L z}}{=} -  \int_0^1 \langle \widehat z_S \, | \,  \mathcal L^\bot_S(z_S, t z_\bot)[z_\bot]  \rangle dt \,.
$$
Since $\mathcal L^\bot_S(z_S, t z_\bot) = t \mathcal L^\bot_S(z_S, z_\bot)$ (cf. \eqref{L SS botS Sbot (1)}),
the latter integral can be computed explicitly, 
\begin{equation}
   \lambda_{L} ( z) [ \widehat z] = - \langle \widehat z \, | \,  { \mathcal E}(z) \rangle \, , 
   \qquad 
   \mathcal E(z) :=  (\mathcal E_S(z), 0)\in h_S \times h^0_{\bot}\, ,
   \qquad z \in \mathcal V, \,\, \, \widehat z \in  h^0_0 \, ,
   \label{forma finale 1 forma Kuksin}
 \end{equation}
 where
 \begin{equation}\label{forma finale 1 forma Kuksin 2}
 \mathcal E_S: \mathcal V \to h_S, \, z \mapsto \mathcal E_S(z) :=  \frac12 \mathcal L_S^\bot(z)[z_\bot]\,.
  \end{equation}
  Combining \eqref{operator for Lambda_tau} - \eqref{forma finale 1 forma Kuksin},
 equation  \eqref{equation for X} reads
   $$
   d  \, \langle \, \cdot \, | \ { \mathcal E}(z) \,  + \, J^{-1} \mathcal L_\tau(z) [X(\tau, z)] \, \rangle  = 0 \, .
   $$
   In view of the definition \eqref{definition L tau} of $\mathcal L_\tau(z)$, we now choose $X$ to be a solution of 
 \begin{equation}\label{second equation for X}
J  \mathcal E(z) +  (  {\rm Id} + \tau J \mathcal L(z))[X(\tau, z)] = 0\,, \qquad \forall z \in \mathcal V\,, \,\, 0 \le \tau \le 1\,.
 \end{equation}
 To solve the latter equation for $X$, we need to invert $ {\rm Id} + \tau J \mathcal L(z)$.
 To this end define the standard projections
 \begin{equation}\label{Pi S Pi bot}
\Pi_S : h_S \times h^0_{ \bot } \to h_S \times h^0_{ \bot }\,,  (\widehat z_S, \widehat z_\bot) \mapsto (\widehat z_S, 0)\,, \qquad 
\Pi_\bot : h_S \times h^0_{ \bot } \to h_S \times h^0_{ \bot }\,, (\widehat z_S, \widehat z_\bot ) \mapsto (0, \widehat z_\bot) , 
\end{equation}
 the standard inclusions
\begin{equation}\label{iota S iota bot}
\iota_S : h_S  \to h_S \times h^0_{ \bot }\,,  \ \widehat z_S \mapsto (\widehat z_S, 0)\,, \qquad 
\iota_\bot :  h^0_{ \bot } \to h_S \times h^0_{ \bot }\,,  \ \widehat z_\bot  \mapsto (0, \widehat z_\bot) , 
\end{equation}
and the maps 
\begin{equation}\label{pi S}
  \pi_S : h_S \times h^0_{ \bot } \to h_S\,, \  z = (z_S, z_\bot) \mapsto z_S\,, \qquad
  \pi_\bot : h_S \times h^0_{ \bot } \to h^0_{ \bot }\,, \  z= (z_S , z_\bot) \mapsto z_\bot\,.
  \end{equation}   
\noindent
Furthermore, let 
\begin{equation}\label{def J_S, J_bot}
J_S: = \pi_S J \iota_S \, , \quad 
J_\bot: = \pi_\bot J \iota_\bot \, , \qquad
{\rm{Id}}_S: = \pi_S {\rm{Id}} \iota_S \, , \quad 
{\rm{Id}}_\bot: = \pi_\bot {\rm{Id}} \iota_\bot \, .
\end{equation}
%
 By the definition \eqref {definition L z} of $\mathcal L(z)$, one has ${\rm Id} + \tau J  \mathcal L(z) = A + \tau B$ where
 $$ 
  A:= \begin{pmatrix}
{\rm Id}_S  &  0 \\
\tau A_{21} & {\rm Id}_\bot
\end{pmatrix} \, , \qquad
A_{21}: =  J_\bot \mathcal L_{\bot}^S( z) \, ,
$$
\begin{equation}\label{formula for B}
B := \begin{pmatrix}
B_{11} &  B_{12}\\
 0 & 0
\end{pmatrix} , \qquad
B_{11}:= J_S \mathcal L_S^S(z)  \, , \qquad B_{12} := J_S \mathcal L_S^{\bot}( z) \, .
 \end{equation}
 Note that $A$ is invertible with  $A^{-1} = \begin{pmatrix}
{\rm Id}_S  &  0 \\
- \tau A_{21} & {\rm Id}_\bot
\end{pmatrix} 
$ and hence $ {\rm Id} + \tau J  \mathcal L(z) =  C A$, where
 \begin{equation}\label{formula for C}
  C:=  {\rm Id} + \tau B A^{-1} =  \begin{pmatrix}
C_{11} &  \tau B_{12} \\
 0 &  {\rm Id}_\bot
\end{pmatrix}  \, , \qquad C_{11}:=  {\rm Id}_S + \tau B_{11}  - \tau^2 B_{12} A_{21} \, .
 \end{equation}
 By Lemma \ref{espansione L S bot q z} it follows after shrinking the ball $\mathcal V_\bot$, if needed,  that $C_{11}: h_S \to h_S$ 
 is invertible. As a consequence, $C$ and in turn $CA$ are invertible. One then obtains the following formula for $(  {\rm Id} + \tau J  \mathcal L(z) )^{-1}$,
 \begin{equation}\label{formula inverse mathcal L tau}
(  {\rm Id} + \tau J  \mathcal L(z) )^{-1} = A^{-1} C^{-1} \, , \qquad 
C^{-1} = \begin{pmatrix} 
C_{11}^{-1} & - \tau C_{11}^{-1}  B_{12} \\
0 &  {\rm Id}_\bot
\end{pmatrix} 
 \end{equation}
 or, in more explicit terms,
 \begin{equation}\label{formula for inverse of mathcal L tau}
 (  {\rm Id} + \tau J  \mathcal L(z) )^{-1} = \begin{pmatrix} 
C_{11}^{-1} & - \tau C_{11}^{-1}  B_{12} \\
-\tau J_\bot \mathcal L_{\bot}^S( z) C_{11}^{-1} &  {\rm Id}_\bot + \tau^2 J_\bot \mathcal L_{\bot}^S( z) C_{11}^{-1}B_{12} 
\end{pmatrix} \, ,
 \end{equation}
 with $C_{11}: h_S \to h_S$ and $B_{12}: h_\bot^0 \to h_S$ given as above.
 In view of \eqref{second equation for X} and \eqref{formula inverse mathcal L tau},
 we then define for any $0 \le \tau \le 1$ and $z \in {\cal V}$ the vector field $X(\tau, z)$ as follows,
 \begin{equation}\label{definizione campo vettoriale ausiliario}
 X(\tau, z) := - (  {\rm Id} + \tau J  \mathcal L(z) )^{-1}[J \mathcal E(z) ]= - A^{-1} C^{-1} 
 \begin{pmatrix}
 J_S \mathcal E_S(z) \\
 0
 \end{pmatrix} 
 = -  A^{-1} \begin{pmatrix}
 C_{11}^{-1} J_S \mathcal E_S(z) \\
 0
 \end{pmatrix} 
 \end{equation}
 or, in more explicit terms, $ X(\tau, z) = \big(  X_S(\tau, z) , \,  X_\bot(\tau, z)  \big)$, where
 \begin{equation}\label{formula X}
 \begin{aligned}
 X_S (\tau, z) & = - \big(  {\rm Id}_S + \tau J_S \mathcal L_S^S(z)   - \tau^2 J_S \mathcal L_S^{\bot}( z)   J_\bot \mathcal L_{\bot}^S( z) \big)^{-1} \big[J_S\mathcal E_S(z) \big] \,  , \\
  X_\bot (\tau, z)  &  =  -  \tau  J_\bot \mathcal L_{\bot}^S( z) \big[ X_S (\tau, z)   \big] \, .
  \end{aligned}
 \end{equation}
Note that by \eqref{forma finale 1 forma Kuksin 2} and \eqref{L SS botS Sbot (1)}, $\mathcal E(z)$ and hence $\lambda_L(z)$ are quadratic expressions in $z_\bot$. 
By Lemma \ref{espansione L S bot q z}, one obtains the following estimates for $\mathcal E(z)$:
   \begin{lemma}\label{lemma E(z)}
   The map ${\cal V} \to h_S \times h^0_\bot$, $z \mapsto \mathcal E(z) = (\mathcal E_S(z), 0)$ is real analytic. Furthermore, for any 
   $z \in {\cal V}$, $ \widehat z_1, \ldots, \widehat z_l \in h^0_0$, $l \geq 1$,  one has 
   $$
   \begin{aligned}
  &  \|  \mathcal E_S(z)\| \lesssim \| z_\bot \|_0^2 \,,  \qquad \| d \mathcal E_S(z)[\widehat z_1]\| \lesssim \| z_\bot \|_0 \| \widehat z_1 \|_0 \, , \\
  & \| d^l \mathcal E_S(z)[\widehat z_1, \ldots, \widehat z_l]\|  \lesssim_{ l} \prod_{j = 1}^l \| \widehat z_j\|_0\,, \quad l \geq 2\, .  
   \end{aligned}
   $$
   \end{lemma}
In view of the formulas of the components of $X(\tau, z)$ in \eqref{formula X}, one then infers from Lemma \ref{espansione L S bot q z} and 
Lemma \ref{lemma E(z)} the following results.
 \begin{lemma}\label{estimates for X} For any $s \ge 0,$ the non-autonomous vector field
 $$
 X: [0, 1] \times (\mathcal V \cap h^s_0) \to  h^{s}_0
 $$
 is real analytic and the following estimates hold: for any $z \in \mathcal V \cap h^s_0,$ $0 \le \tau \le 1$, $\widehat z \in h^s_0$,
 $$
 \| X(\tau, z) \|_{s} \lesssim_s \|z_\bot\|_s \|z_\bot \|_0\,, \quad \|dX(\tau, z))[\widehat z]\|_{s} \lesssim_s \|z_\bot\|_s \|\widehat z\|_0 + \|z_\bot\|_0 \|\widehat z\|_s\, .
 $$
 If in addition $\widehat z_1, \,  \ldots , \widehat z_l \in h^s_0$, $l \ge 2$, then
 $$
 \| d^l X(\tau, z) [\widehat z_1, \, \ldots , \widehat z_l] \|_{s} \lesssim_{s, l}  \sum_{j=1}^l \|\widehat z_j\|_s \prod_{i \ne j} \|\widehat z_i\|_0  +  \|z_\bot\|_s \prod_{j = 1}^l \|\widehat z_j\|_0\,.
 $$
 \end{lemma}
 By a standard contraction argument, there exists an open neighborhood $\mathcal V_S' \subset \mathcal V_S$ of $\mathcal K \subset h_S$ and a ball $\mathcal V_\bot' \subset \mathcal V_\bot,$ centered at $0$,
 so that for any $\tau, \tau_0 \in [0,1]$, the flow map $\Psi_X^{\tau_0, \tau}$  of  the non-autonomous differential equation $ \partial_\tau z = X(\tau, z)$ 
 is well defined on $\mathcal V' := \mathcal V_S' \times \mathcal V_\bot'$ and
 \begin{equation}\label{definition Psi_X {tau_0, tau}}
 \Psi_X^{\tau_0, \tau} : \mathcal V' \to \mathcal V
 \end{equation}
 is real analytic. Arguing as in the proof of \cite[Lemma 4.4]{Kappeler-Montalto} one obtains the following estimates for $\Psi_X^{\tau_0, \tau} - id$. 
 \begin{lemma}\label{estimates for flow}
 Shrinking the ball ${\cal V}_\bot' \subset h^0_{\bot}$ in ${\cal V}' = {\cal V}_S' \times {\cal V}_\bot'$, if needed, it follows that 
for any $s \ge 0$, $\tau_0, \tau \in [0, 1],$ the map
$\Psi_X^{\tau_0, \tau} - id: \mathcal V' \cap h^s_0 \to h^{s}_0$ is real analytic and for any $z \in \mathcal V' \cap h^s_0$, $0 \le \tau_0, \tau \le 1$,
$\widehat z \in h^s_0$,
$$
\| \Psi_X^{\tau_0, \tau}(z)  - z\|_{s} \lesssim_s \|z_\bot\|_s \|z_\bot\|_0\,, \qquad 
\| (d \Psi_X^{\tau_0, \tau}(z)  - { \rm{Id}})[\widehat z]\|_{s}  \lesssim_s \|z_\bot\|_s \|\widehat z\|_0 + \|z_\bot\|_0 \|\widehat z\|_s \, .
$$
If in addition $\widehat z_1, \ldots , \widehat z_l \in h^s_0$, $l \ge 2$, then
$$
\| d^l \Psi_X^{\tau_0, \tau}(z)[\widehat z_1, \ldots, \widehat z_l] \|_{s}  \lesssim_s  \sum_{j = 1}^l \|\widehat z_j\|_s \prod_{i \ne j}\|\widehat z_i\|_0 + \|z_\bot\|_s \prod_{j=1}^l \|\widehat z_j\|_0 \, .
$$
 \end{lemma} 

The main goal of this section is to derive expansions for the flow maps  $\Psi_X^{\tau_0, \tau}(z)$. 
To this end we first derive such expansions for the vector field $X(\tau, z)$. In view of the formulas for the components of $X(\tau, z)$,
one infers from Lemma \ref{espansione L S bot q z}  the following 
\begin{lemma}\label{lemma campo vettoriale new}
For any $N \ge 1,$  $0 \le \tau \le 1$, and $z \in \mathcal V$,
$X(\tau, z) = - \mathcal L_\tau(z)^{-1} [J \mathcal E(z)]$  has an expansion of the form
$ \big( \, 0 , \, {\mathcal{OP}}_N(\tau, z; X) \big) + \mathcal R_N(\tau, z; X)$ where
\begin{equation}\label{espansione campo vettoriale correttore}
{\mathcal{OP}}_N(\tau, z; X)  = 
   {\cal F}^+_{N_S} \circ \sum_{k = 0}^N    a^+_k (\tau, z; X )  D^{- k} [({\cal F}^+_{N_S})^{- 1}z_\bot] 
\, +  \, {\cal F}^-_{N_S} \circ \sum_{k = 0}^N    a^-_k (\tau, z; X )  (-D)^{- k} [({\cal F}^-_{N_S})^{- 1} z_\bot] 
\end{equation}
and for any $s \ge 0$ and $k \ge 0,$ the maps
$$
 [0, 1] \times \mathcal V \to H^s_\C\,, \,\, (\tau, z) \mapsto a^\pm_k(\tau, z; X)\,, \quad
  [0, 1] \times (\mathcal V \cap h^s_0 ) \to  h^{s + N +1}_0\,, \,\, (\tau, z) \mapsto {\cal R}_N (\tau, z; X) \, ,
$$
are real analytic and $ a^-_k (\tau, z; X ) = \overline{ a^+_k (\tau, z; X )}$. 
The coefficients $a^\pm_0 (\tau, z; X )$  take values $\ii \R$.
Writing $q = \Psi^{bo}(z_S, 0)$, they are given by
\begin{equation}\label{formula a^+_0 (tau, z; X ) }
a^+_0 (\tau, z; X ) = - \ii  \tau \partial_x^{-1} \big( d_S q[X_S(\tau, z)] \big)  \, , 
\qquad a^-_0 (\tau, z; X ) = - a^+_0 (\tau, z; X ) \, .
\end{equation}
Furthermore, for any $k \ge 0$, $0 \le \tau \le 1$, $z \in \mathcal V$, $ \widehat z \in h^0_0$, 
$$
 \| a^\pm_k(\tau, z; X)\|_s \lesssim_{s, k} \| z_\bot \|_0^2\,,  \qquad \| d a^\pm_k(\tau, z; X)[\widehat z]  \|_s \lesssim_{s, k} \| z_\bot \|_0 \| \widehat z\|_0 \, .
 $$
If in addition $ \widehat z_1, \ldots, \widehat z_l \in h^0_0$, $l \ge 2$, then
$$
 \| d^l a^\pm_k(\tau, z; X)[\widehat z_1, \ldots, \widehat z_l] \|_s \lesssim_{s, k, l} \prod_{j = 1}^l \| \widehat z_j\|_0\,.
$$
For any $z \in  \mathcal V \cap h^s_0$, $0 \le \tau \le 1$,  $\widehat z \in  h^s_0$, the remainder term ${\cal R}_N(\tau, z; X)$ satisfies
$$
 \| {\cal R}_N(\tau, z; X) \|_{s + N +1} \lesssim_{s, N}  \| z_\bot \|^2_0 +  \| z_\bot \|^2_0 \| z_\bot \|_s \lesssim_{s, N}    \| z_\bot \|_0 \| z_\bot \|_s\,, 
 $$
 $$
 \| d {\cal R}_N(\tau, z; X)[\widehat z] \|_{s + N +1} \lesssim_{s, N} \| z_\bot \|^2_0 \| \widehat z\|_s + \| z_\bot \|_0 \| \widehat z\|_0(1 + \| z_\bot \|_s)  
 \lesssim_{s, N} \| z_\bot \|_0 ( \| \widehat z\|_s + \| z_\bot \|_s \| \widehat z\|_0) \, .
 $$
If in addition $\widehat z_1, \ldots, \widehat z_l \in h^s_0$, $l \geq 2$, then
 $$
 \| d^l {\cal R}_N(\tau, z; X)[\widehat z_1, \ldots, \widehat z_l] \|_{s + N + 1} \lesssim_{s, N, l} \sum_{j = 1}^l \| \widehat z_j\|_s \prod_{i \neq j} \| \widehat z_i\|_0  \, +  \, \| z_\bot \|_s \prod_{j = 1}^l \| \widehat z_j\|_0 \,. 
$$
\end{lemma}
\begin{proof}
In a first step, we compute the expansion of 
 $X_\bot (\tau, z)   =  -  \tau  J_\bot \mathcal L_{\bot}^S( z) \big[ X_S (\tau, z)   \big]$ (cf.  \eqref{formula X}).
By Lemma  \ref{espansione L S bot q z}(ii), one has
$\mathcal L_\bot^S (z)[\widehat z_S] = \mathcal {OP}_{N+1}(z_S; \mathcal L_\bot^S)[\widehat z_S] +  {\cal R}_{N+1}(z; \mathcal L_\bot^S)[\widehat z_S]$ 
where for any $\widehat z_S \in h_S$, $\mathcal {OP}_{N+1}(z_S; \mathcal L_\bot^S)[\widehat z_S] $ is the pseudo-differential operator (cf.  \eqref{expansion L bot S})
 \begin{equation}\label{expansion L bot S rep}
 {\cal F}^+_{N_S} \circ \sum_{k = 1}^{N+1}  \mathcal A^+_k(z_S; \mathcal L_\bot^S)[\widehat z_S]  \cdot D^{- k} [({\cal F}^+_{N_S})^{- 1} z_\bot] 
 + {\cal F}^-_{N_S} \circ \sum_{k = 1}^{N+1}  \mathcal A^-_k(z_S; \mathcal L_\bot^S)[\widehat z_S]  \cdot (-D)^{- k} [({\cal F}^-_{N_S})^{- 1} z_\bot]  \, .
\end{equation}
Since by the definition \eqref{definition J} of $J$, one has  $J_\bot \mathcal F_{N_S}^+ =   \mathcal F_{N_S}^+ \partial_x$
and  $\partial_x  = \ii D$, it follows that
$$
\begin{aligned}
- \tau J_\bot \mathcal {OP}_{N+1}(z_S; \mathcal L_\bot^S)[X_S(\tau, z)] = 
 & -  \ii \tau {\cal F}^+_{N_S} \circ D \sum_{k = 1}^{N+1}  \mathcal A^+_k(z_S; \mathcal L_\bot^S)[X_S(\tau, z)]  \cdot D^{- k} [({\cal F}^+_{N_S})^{- 1} z_\bot] \\
 & + \ii \tau {\cal F}^-_{N_S} \circ (-D) \sum_{k = 1}^{N+1}  \mathcal A^-_k(z_S; \mathcal L_\bot^S)[X_S(\tau, z)]  \cdot (-D)^{- k} [({\cal F}^-_{N_S})^{- 1} z_\bot]  \, 
 \end{aligned}
$$
and hence by the chain rule,
\begin{equation}\label{auxilary expansion}
\begin{aligned}
 D \sum_{k = 1}^{N+1}  \mathcal A^+_k(z_S; \mathcal L_\bot^S)[X_S(\tau, z)] & \cdot D^{- k} [({\cal F}^+_{N_S})^{- 1} z_\bot] 
 =   \sum_{k = 0}^{N}  \mathcal A^+_{k+1}(z_S; \mathcal L_\bot^S)[X_S(\tau, z)]  \cdot D^{- k} [({\cal F}^+_{N_S})^{- 1} z_\bot] \\
& - \ii   \sum_{k = 1}^{N+1} \partial_x \big( \mathcal A^+_k(z_S; \mathcal L_\bot^S)[X_S(\tau, z)] \big) \cdot D^{- k} [({\cal F}^+_{N_S})^{- 1} z_\bot] .
 \end{aligned}
\end{equation}
It then follows that  $X(\tau, z)$ has an expansion of the form
$\big( \, 0,  \, \mathcal {OP}_{N} (\tau, z; X) \big) + \mathcal R_N( \tau, z; X) $
with $ \mathcal {OP}_{N} (\tau, z; X)$ given by
$$
{\cal F}^+_{N_S} \circ \sum_{k = 0}^N    a^+_k (\tau, z; X )  D^{- k} [({\cal F}^+_{N_S})^{- 1}z_\bot] 
\, +  \, {\cal F}^-_{N_S} \circ \sum_{k = 0}^N    a^-_k (\tau, z; X )  (-D)^{- k} [({\cal F}^-_{N_S})^{- 1} z_\bot] \, ,
$$
where the coefficients $a^\pm_k (\tau, z; X)$ can be read off from \eqref{auxilary expansion}.
In view of Lemma  \ref{espansione L S bot q z}, the formula \eqref{formula X} for $X(\tau, z)$, and Lemma \ref{lemma E(z)},
the coefficients $a^\pm_k (\tau, z; X)$ and the remainder term $ \mathcal R_N( \tau, z; X) $ have the stated properties.
Finally by \eqref{auxilary expansion}, the coefficients $a^\pm_0(\tau, z; X)$ are given by
$ \mp \ii \tau \mathcal A^+_{1}(z_S; \mathcal L_\bot^S)[X_S(\tau, z)]$
and hence by Remark \ref{extension of mathcal L(z)}(i),
$$
a_0^+(\tau, z; X) = - \ii \tau \partial_x^{-1}\big( d_S q [X_S(\tau, z)] \big), \qquad
a_0^-(\tau, z; X) = \overline{a_0^+(\tau, z; X) } =   \ii \tau \partial_x^{-1}\big( d_S q [X_S(\tau, z)] \big)
$$
where for the latter identity we used that $\partial_x^{-1}\big( d_S q [X_S(\tau, z)] \big)$ is real valued.
\end{proof}

After these preliminary considerations we can now state the main result of this section, saying  that for any $\tau_0, \tau \in [0, 1]$, 
the flow map $\Psi_X^{\tau_0, \tau}$, defined on $\mathcal V'$ and taking values in $\mathcal V$, admits an expansion,
referred to as parametrix for the solution of the initial value problem of $\partial_\tau z(\tau) = X(\tau, z(\tau))$. 
\begin{theorem}\label{espansione flusso per correttore}
(i) For any $\tau_0, \tau \in [0, 1]$, $N \in \N$, and $z = (z_S, z_\bot) \in \mathcal V'$,
\begin{equation}\label{espansione asintotica Psi tau 0 tau}
\Psi_X^{\tau_0, \tau}(z) = z +  \big( 0, \, \, \mathcal{OP}_N ( z; \Psi_X^{\tau_0, \tau}) \big)  + {\cal R}_N( z; \Psi_X^{\tau_0, \tau}) ,
\end{equation}
where $ \mathcal{OP}_N ( z; \Psi_X^{\tau_0, \tau} )$ is the pseudo-differential operator
\begin{equation}\label{expansion for Psi^{tau0, tau}}
{\cal F}^+_{N_S} \circ\sum_{k = 0}^N   a^+_k(z; \Psi_X^{\tau_0, \tau}) \,\, D^{- k} [ ({\cal F}^+_{N_S})^{- 1} z_\bot] \, + \,
{\cal F}^-_{N_S} \circ\sum_{k = 0}^N   a^-_k(z; \Psi_X^{\tau_0, \tau}) \,\, (-D)^{- k}[ ({\cal F}^-_{N_S})^{- 1} z_\bot]
\end{equation}
and for any $\tau_0, \tau \in [0, 1]$, $0 \le k \le N$, and $s \geq 0$, the maps 
$$
\mathcal V' \to H^s_\C, \, z \mapsto a^\pm_k(z; \Psi_X^{\tau_0, \tau}), \qquad
\mathcal V' \cap h^s_0 \to  h^{s + N +1}_0, \, z \mapsto {\cal R}_N(z; \Psi_X^{\tau_0, \tau}) \, ,
$$ 
are real analytic and $a^-_k(z; \Psi_X^{\tau_0, \tau}) = \overline{a^+_k(z; \Psi_X^{\tau_0, \tau})}$. 
Furthermore, for any $k \ge 0$, $z \in \mathcal V'$,  $ \widehat z \in h^0_0$,
$$
 \| a^{\pm}_k(z; \Psi_X^{\tau_0, \tau} ) \|_s \lesssim_{s, k} \| z_\bot \|_0^2\,, \qquad  \| d a^{\pm}_k(z; \Psi_X^{\tau_0, \tau} )[\widehat z]\|_s 
 \lesssim_{s, k} \| z_\bot \|_0 \| \widehat z\|_0 \, .
 $$
If in addition $ \widehat z_1, \ldots, \widehat z_l \in h^0_0$, $l \geq 2$, then
$$
 \| d^l a_k(z; \Psi_X^{\tau_0, \tau} ) [\widehat z_1, \ldots, \widehat z_l]\|_s \lesssim_{s, k, l} \prod_{j = 1}^l \| \widehat z_j\|_0\,. 
$$
The remainder term satisfies the following estimates: for any $z \in \mathcal V'  \cap h^s_0$ and  $\widehat z \in h^s_0$,
$$
 \|{\cal R}_N(z; \Psi_X^{\tau_0, \tau}) \|_{s + N +1} \lesssim_{s, N} \| z_\bot \|_s \| z_\bot \|_0\,, \qquad
 \| d {\cal R}_N(z; \Psi_X^{\tau_0, \tau}) [\widehat z]\|_{s + N +1} \lesssim_{s, N} \| z_\bot \|_s \| \widehat z\|_0 + \| z_\bot \|_0 \| \widehat z\|_s \,, 
 $$
and for any  $\widehat z_1, \ldots, \widehat z_l \in h^s_0$, $l \geq 2$,
 $$
 \| d^l {\cal R}_N(z; \Psi_X^{\tau_0, \tau})[\widehat z_1, \ldots, \widehat z_l] \|_{s + N +1} \lesssim_{s, N, l} \, 
 \sum_{j = 1}^l \| \widehat z_j \|_s \prod_{i \neq j} \| \widehat z_i\|_0  \ + \ \| z_\bot \|_s \prod_{j = 1}^l \| \widehat z_j\|_0\,.
$$
(ii) In particular, the statements of item (i)  hold for $\Psi_C := \Psi_X^{0, 1}: \mathcal V' \to \Psi_X^{0, 1}(\mathcal V')$, referred to as symplectic corrector,
and $ \Psi_X^{1, 0} : \mathcal V' \to  \Psi_X^{1, 0}(\mathcal V')$,
which by a slight abuse of terminology with respect to its domain of definition, we refer to as the inverse of  $\Psi_C$ and denote by $\Psi_C^{- 1}$.
The expansion of the map $\Psi_C$ is then written as ($z \in \mathcal V'$)
$$
\Psi_C(z) = z +  \big(  0  , \,  \mathcal {OP}_N ( z; \Psi_C ) \big) + {\cal R}_N(z; \Psi_C) \, ,
$$
where $\mathcal {OP}_N ( z; \Psi_C )$ is the pseudo-differential operator
\begin{equation}\label{OP_N  Psi_C}
{\cal F}^+_{N_S}  \circ \sum_{k = 0}^N  a^+_k(z;  \Psi_C) D^{- k} [({\cal F}^+_{N_S})^{- 1} z_\bot ] \, + \,
{\cal F}^-_{N_S}  \circ \sum_{k = 0}^N  a^-_k(z;  \Psi_C) (-D)^{- k} [({\cal F}^-_{N_S})^{- 1} z_\bot ] 
\end{equation}
with
$$
a^\pm_k(z; \Psi_C) := a^\pm_k(z; \Psi_X^{0, 1})\,, \qquad {\cal R}_N( z; \Psi_C) := {\cal R}_N(z; \Psi_X^{0, 1})\,.
$$
Similarly, the expansion for the inverse $\Psi_C^{- 1}(z)$, $z \in \mathcal V'$, is of the form
$$
\Psi_C(z)^{- 1} = z +  \big( \, 0 , \ \mathcal {OP}_N ( z; \Psi_C^{-1} ) \big) + {\cal R}_N(z; \Psi_C^{- 1})
$$ 
where $\mathcal {OP}_N ( z; \Psi_C^{-1} )$ is given by
$$
{\cal F}^+_{N_S}  \circ \sum_{k = 0}^N  a^+_k(z;  \Psi_C^{-1}) D^{- k} [({\cal F}^+_{N_S})^{- 1} z_\bot ] \, + \,
{\cal F}^-_{N_S}  \circ \sum_{k = 0}^N  a^-_k(z;  \Psi_C^{-1}) (-D)^{- k} [({\cal F}^-_{N_S})^{- 1} z_\bot ] 
$$
with
$$
a_k^\pm(z;  \Psi^{- 1}_C) := a^\pm_k(z;  \Psi_X^{1, 0})\,, \qquad {\cal R}_N(z; \Psi_C^{- 1}) := {\cal R}_N(z; \Psi_X^{1, 0}).
$$ 
As a consequence, $a^-_k(z; \Psi_C) = \overline{a^+_k(z; \Psi_C)} $ 
and $a^-_k(z; \Psi_C^{-1}) = \overline{a^+_k(z; \Psi_C^{-1})} $.
\end{theorem}
\begin{proof}
Clearly, item (ii) is a direct consequence of (i). Since the proof of item (i) is quite lengthy, we divide it up into several steps.
First note that the flow map $\Psi^{\tau_0, \tau} \equiv \Psi_X^{\tau_0, \tau}$ is a bounded nonlinear operator acting on $\mathcal V' \cap h^s_0$, $s \ge 0$, 
which satisfies the differential equation
\begin{equation}\label{Duhamel ODE X}
\partial_\tau \Psi^{\tau_0, \tau}(z) =  Y_{\tau_0}(\tau, z)  \, ,
\qquad Y_{\tau_0}(\tau, z):= X(\tau, \Psi^{\tau_0, \tau}(z)) \, .
\end{equation}
Using the latter equation, the coefficients $a^\pm_k(z; \Psi^{\tau_0, \tau})$, $k \ge 0,$ 
of the parametrix  \eqref{espansione asintotica Psi tau 0 tau}
are determined inductively. 
In a first step, we want to obtain an expansion of 
$$
Y_{\tau_0}(\tau, z) =  - \mathcal L_\tau(z)^{-1} [J \mathcal E(\Psi^{\tau_0, \tau}(z))]
 \, , \qquad  0 \le  \tau_0, \, \tau \le 1, \  z \in \mathcal V' ,
$$ 
of the form $Y_{\tau_0} (\tau, z)  =  \big( \, 0 , \, {\mathcal{OP}}_N (\tau, z; Y_{\tau_0}) \big) + \mathcal R_N(\tau, z; Y_{\tau_0})$ 
where ${\mathcal{OP}}_N (\tau, z; Y_{\tau_0})$ is a pseudo-differential operator and $\mathcal R_N(\tau, z; Y_{\tau_0})$
a remainder term with the ususal properties.
By \eqref{espansione campo vettoriale correttore}, one has for any $0 \le  \tau_0$, $\tau \le 1,$ $z \in \mathcal V'$, 
\begin{equation}\label{expansion X tau Psi tau_0, t}
Y_{\tau_0}(\tau, z) = \big( \, 0 , \, {\mathcal{OP}}_N(\tau, \Psi^{\tau_0, \tau}(z); X) \big) + \mathcal R_N (\tau, \Psi^{\tau_0, \tau}(z); X ) \, ,
 \end{equation}
\begin{equation}\label{hoch schule - 1}
\begin{aligned}
{\mathcal{OP}}_N(\tau, \Psi^{\tau_0, \tau}(z); X)  & = 
  {\cal F}^+_{N_S} \circ \sum_{k = 0}^N    a^+_k (\tau, z; \tau_0 )  D^{- k} [({\cal F}^+_{N_S})^{- 1}\Psi^{\tau_0, \tau}(z)_\bot] \\
& \, +  \, {\cal F}^-_{N_S} \circ \sum_{k = 0}^N    a^-_k (\tau, z; \tau_0 )  (-D)^{- k} [({\cal F}^-_{N_S})^{- 1} \Psi^{\tau_0, \tau}(z)_\bot] \, .
\end{aligned}
\end{equation}
where,  to simplify notation, $a^\pm_k(\tau, z; \tau_0)$ is defined as
\begin{equation}\label{definition a _ k (t,z; tau)}
a^\pm_k(\tau, z; \tau_0) := a^\pm_k(\tau, \Psi^{\tau_0, \tau}(z); X) \, .
\end{equation}
In a first step, we argue formally and want to derive an expansion for $Y_{\tau_0}(\tau, z)$ 
in terms of  the (not yet determined) expansion \eqref{espansione asintotica Psi tau 0 tau}-\eqref{expansion for Psi^{tau0, tau}} of $ \Psi_X^{\tau_0, \tau}(z)$.
We begin with expanding the term $D^{- k}[ ({\cal F}^+_{N_S})^{- 1}\Psi^{\tau_0, \tau}(z)_\bot]$ in \eqref{hoch schule - 1}.

\smallskip
\noindent
{\it Expansion of $D^{- k} [({\cal F}^+_{N_S})^{- 1} \Psi^{\tau_0, \tau}(z)_\bot ]$, $0 \le k \le N$:} 
Substituting the expansion \eqref{espansione asintotica Psi tau 0 tau}-\eqref{expansion for Psi^{tau0, tau}}
into the expression $D^{- k} [({\cal F}^+_{N_S})^{- 1}\Psi^{\tau_0, \tau}(z)_\bot]$ one obtains
\begin{align}
D^{- k} [( {\cal F}^+_{N_S})^{- 1} & \Psi^{\tau_0, \tau}(z)_\bot] 
= D^{- k} \Big( [( {\cal F}^+_{N_S})^{- 1} z_\bot] +  \sum_{j = 0}^N  a^+_j(z; \Psi^{\tau_0, \tau}) \,\, D^{- j} [({\cal F}^+_{N_S})^{- 1}z_\bot] 
 + ({\cal F}^+_{N_S})^{- 1} {\cal R}_N(z; \Psi^{\tau_0, \tau})_\bot  \Big) 
\nonumber\\
& = D^{- k} \Big( [({\cal F}^+_{N_S})^{- 1} z_\bot] + \sum_{j = 0}^{N - k} a^+_j(z; \Psi^{\tau_0, \tau}) \,\, D^{- j} [({\cal F}^+_{N_S})^{- 1}z_\bot] \, \Big) 
+ {\cal R}_{N, k}^{(+, 1)}(\tau, z; \tau_0) \label{hoch schule 0}
\end{align}
where 
\begin{equation}\label{cal R n N (1)}
{\cal R}_{N, k}^{(+, 1)}(\tau, z; \tau_0) := D^{- k} \circ  \sum_{j = N - k + 1}^{N }  a^+_j(z; \Psi^{\tau_0, \tau}) \,\, D^{- j} [({\cal F}^+_{N_S})^{- 1}z_\bot] 
+ D^{- k} [ ({\cal F}^+_{N_S})^{- 1} {\cal R}_N(z; \Psi^{\tau_0, \tau})_\bot ] \,.
\end{equation}
Using Lemma \ref{lemma composizione pseudo}(i) and the notation established there, one has for any $0 \le j \le N-k$,
$$
\begin{aligned}
D^{- k} \circ  \big( a^+_j(z; \Psi^{\tau_0, \tau}) \,  D^{- j} [ ({\cal F}^+_{N_S})^{- 1} z_\bot] \big) & = 
\sum_{i = 0}^{N - k - j} C_i( k, j) \cdot D^i a^+_j(z; \Psi^{\tau_0, \tau}) \, \cdot \, D^{- k - j - i} [ ({\cal F}^+_{N_S})^{- 1}z_\bot] \\
& + {\cal R}_{N, k, j}^{(+, 2)}(\tau, z; \tau_0)
\end{aligned}
$$ 
where 
\begin{equation}\label{definition R (2)}
{\cal R}_{N,  k, j}^{(+, 2)}(\tau, z; \tau_0) :=  {\cal R}_{N,  k, j}^{\psi do}(a^+_j(z; \Psi^{\tau_0, \tau})) \,[ ( {\cal F}^+_{N_S})^{- 1} z_\bot]\,.
\end{equation}
By Lemma \ref{lemma composizione pseudo}, for any $z \in \mathcal V' \cap h^s_0$, $s \ge 0$, $0 \le \tau, \tau_0 \le 1$, $0 \le j \le N$,
\begin{equation}\label{stima cal R n k N (2)}
\| {\cal R}_{N, k, j}^{(+, 2)}(\tau, z; \tau_0)\|_{s + N + 1} \lesssim_{s, N} {\rm max}_{0 \le i  \le N}\| a^+_i (z ;  \Psi^{\tau_0, \tau})\|_{s + 2 N}  \| z_\bot \|_s\,. 
\end{equation}
Combining the above expansions, 
the formula \eqref{hoch schule 0} for $D^{- k} [( {\cal F}^+_{N_S})^{- 1} \Psi^{\tau_0, \tau}(z)_\bot ] $ then reads
$$
D^{- k} [({\cal F}^+_{N_S})^{- 1} z_\bot] + \sum_{j = 0}^{N - k} \sum_{i  =0}^{N - k - j} C_i(k, j) \cdot D^i  a^+_j(z; \Psi^{\tau_0, \tau}) \cdot   D^{- k - j - i} [( {\cal F}^+_{N_S})^{- 1}z_\bot]   +
  {\cal R}_{N, k}^{(+, 3)}( \tau, z; \tau_0)
$$
where 
\begin{equation}\label{definizione cal R n N (3)}
{\cal R}_{N, k}^{(+, 3)}(\tau, z; \tau_0) := {\cal R}_{N, k}^{(+, 1)}(\tau, z;  \tau_0) +  \sum_{j = 0}^{N - k} {\cal R}_{N, k, j}^{(+, 2)}(\tau, z;  \tau_0) \,. 
\end{equation}
Changing in the double sum $\sum_{j = 0}^{N - k} \sum_{i  =0}^{N - k - j}$  the index $i$ of summation to $n:= i+j$ and then interchanging the order of summation, one obtains
$$
\sum_{j = 0}^{N - k} \sum_{i  =0}^{N - k - j} C_i(k, j) \cdot D^i  a^+_j \cdot  D^{- k - j - i} 
=\sum_{n = 0}^{N - k} \sum_{j=0}^{n} C_{n-j} (k, j) \cdot D^{n-j}  a^+_j \cdot  D^{- k - n} \, ,
$$
implying that  $D^{- k} [ ({\cal F}^+_{N_S})^{- 1} \Psi^{\tau_0, \tau}(z)_\bot ] $ equals
\begin{equation}\label{hoch schule 2}
 D^{- k} [ ({\cal F}^+_{N_S})^{- 1}z_\bot] + \sum_{n = 0}^{N - k} \big( \sum_{ j=0}^{n} C_{n -j}(k , j) \cdot D^{n-j} a^+_j(z; \Psi^{\tau_0, \tau})  \big) \cdot D^{- k - n} [({\cal F}^+_{N_S})^{- 1}z_\bot] + 
{\cal R}_{N, k}^{(+, 3)}(\tau, z;  \tau_0) \, .
\end{equation}

%
%
%
\smallskip


\noindent
{\it Expansion of  $\sum_{k = 0}^N  a^+_k(\tau, z; \tau_0) \,  D^{- k} [({\cal F}^+_{N_S})^{- 1} \Psi^{\tau_0, \tau}(z)_\bot]$:} 
Substituting for any $0 \le k \le N$ the expansion \eqref{hoch schule 2} into 
$ \sum_{k = 0}^N  a^+_k(\tau, z; \tau_0) \,  D^{- k} [({\cal F}^+_{N_S})^{- 1} \Psi^{\tau_0, \tau}(z)_\bot]$, one gets
\begin{align}
&   \sum_{k = 0}^N  a^+_k(\tau, z; \tau_0) \,\,  D^{- k}[ ({\cal F}^+_{N_S})^{- 1} \Psi^{\tau_0, \tau}(z)_\bot]    \, 
  = \,  \sum_{k = 0}^N   a^+_k(\tau, z; \tau_0) \,\,   D^{- k}[ ( {\cal F}^+_{N_S})^{- 1} z_\bot] \,\,   \label{sumatra 0}\\
& + \,  \sum_{k = 0}^N \sum_{n = 0}^{N - k} \sum_{ j = 0}^n C_{n-j}(k, j) a^+_k(\tau, z; \tau_0) \cdot D^{n - j} a^+_j(z; \Psi^{\tau_0, \tau}) \cdot D^{-k - n } [ ({\cal F}^+_{N_S})^{- 1} z_\bot] \nonumber \\
&  + \sum_{k = 0}^N   a^+_k(\tau, z; \tau_0) {\cal R}_{N, k}^{(+, 3)}(\tau, z;  \tau_0)  \nonumber \, .
\end{align} 
Changing the index of summation $n$ to $l:= k+n$ and then interchanging the sum with respect to $k$ and $l$ and subsequently with respect to $k$ and $j$, 
the triple sum in \eqref{sumatra 0} becomes
\begin{align}\label{sumatra 1}
 & \sum_{k = 0}^N \sum_{l = k}^{N} \sum_{ j = 0}^{l - k} C_{l - k -j}(k, j) a^+_k(\tau, z; \tau_0) \cdot D^{l - k - j} a^+_j(z; \Psi^{\tau_0, \tau})  \cdot 
D^{-l }[ ( {\cal F}^+_{N_S})^{- 1} z_\bot] \nonumber \\
& = \, \sum_{l = 0}^N \Big( \sum_{k= 0}^{l } \sum_{ j = 0}^{l - k} C_{l - k -j}(k, j) a^+_k(\tau, z; \tau_0) \cdot D^{l - k - j} a^+_j(z; \Psi^{\tau_0, \tau})) \Big) \, 
D^{-l } [ ({\cal F}^+_{N_S})^{- 1}z_\bot] \nonumber \\
& = \, \sum_{l = 0}^N \Big( \sum_{j= 0}^{l } \sum_{ k = 0}^{l - j} C_{l - k -j}(k, j) a^+_k(\tau, z; \tau_0) \cdot D^{l - k - j} a^+_j(z; \Psi^{\tau_0, \tau}) \Big) \, 
D^{-l }[ ({\cal F}^+_{N_S})^{- 1} z_\bot]\,.
\end{align}
Combining \eqref{sumatra 0} - \eqref{sumatra 1} and  writing $n$ for $k$ and $k$ for $l$ in \eqref{sumatra 1}, one infers  that
\begin{align}
 \sum_{k = 0}^N  & a^+_k(\tau, z; \tau_0) \,  D^{- k} [({\cal F}^+_{N_S})^{- 1} \Psi^{\tau_0, \tau}(z)_\bot]
 = \,  \sum_{k = 0}^N   a^+_k(\tau, z; \tau_0) \,\,   D^{- k}[ ( {\cal F}^+_{N_S})^{- 1} z_\bot]   \label{sumatra 0 +}\\
& + \, \sum_{k = 0}^N \big( \sum_{j= 0}^{k} \sum_{ n = 0}^{k - j} C_{k - n -j}(n, j) a^+_n(\tau, z; \tau_0) \cdot D^{k - n - j} a^+_j(z; \Psi^{\tau_0, \tau}) \big) \cdot D^{-k }[ ({\cal F}^+_{N_S})^{- 1} z_\bot] \nonumber \\
& + \sum_{k = 0}^N   a^+_k(\tau, z; \tau_0) {\cal R}_{N, k}^{(+, 3)}(\tau, z;  \tau_0)  \nonumber  \, ,
\end{align}
with ${\cal R}_{N, k}^{(+, 3)}(\tau, z;  \tau_0) $ given by \eqref{definizione cal R n N (3)}.

\smallskip

\noindent
{\it Expansion of $ \sum_{k = 0}^N  a^-_k(\tau, z; \tau_0) \, (- D)^{- k} [({\cal F}^-_{N_S})^{- 1} \Psi^{\tau_0, \tau}(z)_\bot]$:}  
Analogous to the case '$+$', one obtains
\begin{align}
 \sum_{k = 0}^N  & a^-_k(\tau, z; \tau_0) \, (-D)^{- k} [({\cal F}^-_{N_S})^{- 1} \Psi^{\tau_0, \tau}(z)_\bot]
 = \,  \sum_{k = 0}^N   a^-_k(\tau, z; \tau_0) \,\,   (-D)^{- k}[ ( {\cal F}^-_{N_S})^{- 1} z_\bot]   \label{sumatra 0 -} \\
& + \, \sum_{k = 0}^N \big( \sum_{j= 0}^{k} \sum_{ n = 0}^{k - j} C_{k - n -j}(n, j) a^-_n(\tau, z; \tau_0) \cdot (-D)^{k - n - j} a^-_j(z; \Psi^{\tau_0, \tau}) \big) \cdot (-D)^{-k }[ ({\cal F}^-_{N_S})^{- 1} z_\bot] \nonumber \\
& + \sum_{k = 0}^N   a^-_k(\tau, z; \tau_0) {\cal R}_{N, k}^{(-, 3)}(\tau, z;  \tau_0)  \nonumber  \, ,
\end{align}
where for any $0 \le k \le N$, ${\cal R}_{N, k}^{(-, 3)}(\tau, z;  \tau_0)$ is defined analogous to ${\cal R}_{N, k}^{(+, 3)}(\tau, z;  \tau_0)$,
\begin{equation}\label{definizione cal R n N (-,3)}
{\cal R}_{N, k}^{(-, 3)}(\tau, z; \tau_0) := {\cal R}_{N, k}^{(-, 1)}(\tau, z;  \tau_0) +  \sum_{j = 0}^{N - k} {\cal R}_{N, k, j}^{(-, 2)}(\tau, z;  \tau_0) \, , 
\end{equation}
the remainder ${\cal R}_{N, k}^{(-, 1)}(\tau, z;  \tau_0)$ analogous to ${\cal R}_{N, k}^{(+, 1)}(\tau, z;  \tau_0)$, 
\begin{align}\label{cal R n N (-,1)}
{\cal R}_{N, k}^{(-, 1)}(\tau, z; \tau_0)  := & (-D)^{- k} \circ  \sum_{j = N - k + 1}^{N }  a^-_j(z; \Psi^{\tau_0, \tau}) \,\, (-D)^{- j} [({\cal F}^-_{N_S})^{- 1}z_\bot]  \nonumber\\
& + (-D)^{- k} [ ({\cal F}^-_{N_S})^{- 1} {\cal R}_N(z; \Psi^{\tau_0, \tau})_\bot ] \,.
\end{align}
and ${\cal R}_{N,  k, j}^{(-, 2)}(\tau, z; \tau_0)$ analogous to ${\cal R}_{N,  k, j}^{(+, 2)}(\tau, z; \tau_0)$,
\begin{equation}\label{definition R (-,2)}
{\cal R}_{N,  k, j}^{(-, 2)}(\tau, z; \tau_0) :=  {\cal R}_{N,  k, j}^{\psi do}(a^-_j(z; \Psi^{\tau_0, \tau})) \,[ ( {\cal F}^-_{N_S})^{- 1} z_\bot]\,.
\end{equation}
Furthermore,  for any $z \in \mathcal V' \cap h^s_0$, $s \ge 0$, $0 \le \tau, \tau_0 \le 1$, $0 \le j \le N$, one has (cf. \eqref{stima cal R n k N (2)})
\begin{equation}\label{stima cal R n k N (-,2)}
\| {\cal R}_{N, k, j}^{(-, 2)}(\tau, z; \tau_0)\|_{s + N + 1} \lesssim_{s, N} {\rm max}_{0 \le i  \le N}\| a^-_i (z ;  \Psi^{\tau_0, \tau})\|_{s + 2 N}  \| z_\bot \|_s\,. 
\end{equation}

\smallskip

\noindent
{\it Expansion of $Y_{\tau_0}(\tau, z)$:}
From  \eqref{hoch schule - 1} and \eqref{sumatra 0 +} - \eqref{sumatra 0 -} one infers that $Y_{\tau_0}(\tau, z) = X(\tau, \Psi^{\tau_0, \tau}(z))$
has an expansion of the form
\begin{equation}\label{expansion for Y_tau_0}
Y_{\tau_0} (\tau, z)  =  \big( \, 0 , \, {\mathcal{OP}}_N (\tau, z; Y_{\tau_0}) \big) + \mathcal R_N(\tau, z; Y_{\tau_0}) 
\end{equation}
with ${\mathcal{OP}}_N (\tau, z; Y_{\tau_0}) $ and $\mathcal R_N(\tau, z; Y_{\tau_0})$ given as follows:
using that $ C_0( k, j) = 1$ (cf. Lemma \ref{lemma composizione pseudo}(i)) one concludes that
the pseudo-differential operator ${\mathcal{OP}}_N (\tau, z; Y_{\tau_0}) $ reads
\begin{align}\label{psi-do}
&   \mathcal F^+_{N_S} \circ \big( a^+_0(\tau, z; \tau_0) \big(1 + a^+_0(z; \Psi^{\tau_0, \tau})  \big) \big) \cdot[ ({\cal F}^+_{N_S})^{- 1} z_\bot]  \\
 + & \, \mathcal F^+_{N_S} \circ \sum_{k = 1}^N \big( a^+_0(\tau, z; \tau_0) a^+_k(z; \Psi^{\tau_0, \tau})  
+  b^+_k(\tau, z; \tau_0) \big) \cdot D^{-k } [ ({\cal F}^+_{N_S})^{- 1} z_\bot] \nonumber\\
 +  & \,  \mathcal F^-_{N_S} \circ \big( a^-_0(\tau, z; \tau_0) \big(1 + a^-_0(z; \Psi^{\tau_0, \tau})  \big) \big) \cdot [ ({\cal F}^-_{N_S})^{- 1} z_\bot] \nonumber \\
 +  & \,  \mathcal F^-_{N_S} \circ \sum_{k = 1}^N \big( a^-_0(\tau, z; \tau_0) a^-_k(z; \Psi^{\tau_0, \tau}) 
 +  b^-_k(\tau, z; \tau_0) \big) \cdot (-D)^{-k } [ ({\cal F}^-_{N_S})^{- 1} z_\bot] \nonumber
\end{align}
where for any $k \ge 1$, 
\begin{equation}\label{definition b k}
b^\pm_k(\tau, z; \tau_0)  := a^\pm_k(\tau, z; \tau_0) + 
 \sum_{j= 0}^{k-1 } \sum_{ n = 0}^{k - j} C_{k - n -j}(n, j) a^\pm_n(\tau, z; \tau_0) \, (\pm D)^{k - n - j} a^\pm_j(z; \Psi^{\tau_0, \tau}) 
\end{equation}
and the remainder term $ \mathcal R_N(\tau, z; Y_{\tau_0})$ is defined as
\begin{align}\label{remainder of Y tau_0}
  \big( 0, \,\, {\cal F}^+_{N_S} \circ   \sum_{k = 0}^N  a^+_k(\tau, z; \tau_0) & {\cal R}_{N, k}^{(+, 3)}(\tau, z; \tau_0) 
 + \, {\cal F}^-_{N_S} \circ   \sum_{k = 0}^N  a^-_k(\tau, z; \tau_0) {\cal R}_{N, k}^{(-, 3)}(\tau, z; \tau_0) \, \big) \nonumber\\
& +  \,{\cal R}_N(\tau , \Psi^{\tau_0, \tau}(z); X) \,.
\end{align}

\noindent
{\it Definition and estimates of $a^\pm_k(z; \Psi^{\tau_0, \tau})$, $ k \ge 0$:}
The fact that the coefficients $b^+_k(\tau, z; \tau_0)$, $k \ge 1$, depend on the unknown coefficients $a^+_j(z; \Psi^{\tau_0, \tau})$ with $0 \le j \le k-1,$ 
but not on $a^+_j(z; \Psi^{\tau_0, \tau})$ with $j \ge k$,
allows to define $a^+_k(z; \Psi^{\tau_0, \tau})$ recursively by using equation \eqref{Duhamel ODE X}. 
Clearly, the coefficients $a^-_k(z; \Psi^{\tau_0, \tau})$ can be defined in the same way. 
Our candidates for $a^\pm_k(z; \Psi^{\tau_0, \tau})$ are thus obtained by solving
$\partial_\tau  \mathcal{OP}_N ( z; \Psi_X^{\tau_0, \tau}) =  {\mathcal{OP}}_N (\tau, z; Y_{\tau_0})$, or in more detail, 
\begin{equation}\label{equ for a_0}
\partial_\tau a^\pm_0(z; \Psi^{\tau_0, \tau}) = a^\pm_0(\tau, z; \tau_0) \cdot (1 + a^\pm_0(z; \Psi^{\tau_0, \tau})   ) \, ,  \quad a^\pm_0(z; \Psi^{\tau_0, \tau_0}) = 0 \, ,
\end{equation}
and for any $k \ge 1$
\begin{equation}\label{equ for a_k}
\partial_\tau a^\pm_k(z; \Psi^{\tau_0, \tau}) =  a^\pm_0(\tau, z; \tau_0) a^\pm_k(z; \Psi^{\tau_0, \tau}) +  b^\pm_k(\tau, z; \tau_0)  \, ,  \quad a^\pm_k(z; \Psi^{\tau_0, \tau_0}) = 0 \, .
\end{equation}
The solution of \eqref{equ for a_0} then leads to the definition
\begin{equation}\label{def a_0^pm}
a^\pm_0(z; \Psi^{\tau_0, \tau}) := e^{\alpha^\pm(\tau, z; \tau_0)} - 1 \, , \qquad  \alpha^\pm(\tau, z; \tau_0):= \int_{\tau_0}^\tau a^\pm_0(t, z; \tau_0) \, d t
\end{equation}
and, recursively, for any $k \ge 1 $, the one of \eqref{equ for a_k}  to 
\begin{equation}\label{def a_k^pm}
a^\pm_k(z; \Psi^{\tau_0, \tau}) := e^{\alpha^\pm(\tau, z; \tau_0)} \int_{\tau_0}^\tau  e^{-\alpha^\pm(t, z; \tau_0)} b^\pm_k(t, z; \tau_0) \, d t \, .
\end{equation}
Going through the arguments above, one verifies that for any $k \ge 0$, $a^-_k(z; \Psi^{\tau_0, \tau}) = \overline{ a^+_k(z; \Psi^{\tau_0, \tau})}$.
To prove the claimed estimates for $a^\pm_k(z; \Psi^{\tau_0, \tau})$, $k \ge 0$, we first estimate $a^\pm_k(\tau, z; \tau_0)$. 
Recall that by \eqref{definition a _ k (t,z; tau)},
$a^\pm_k(\tau, z; \tau_0) = a^\pm_k(\tau, \Psi^{\tau_0, \tau}(z); X)$. 
By Lemma \ref{lemma campo vettoriale new} and Lemma \ref{estimates for flow} one has
for any $0 \le \tau_0, \tau  \le 1$, $z \in \mathcal V'$, and $s \ge 0,$
\begin{equation}\label{estimate a k (t, z, tau 0)}
\| a^\pm_k(\tau, z; \tau_0) \|_s \lesssim_{s, k} \| \pi_\bot \Psi^{\tau_0, \tau}(z) \|_0^2 \lesssim_{s, k}  \| z_\bot \|^2_0 \,.
\end{equation}
It then follows from the definition \eqref{def a_0^pm} of $ a^\pm_0(z; \Psi^{\tau_0, \tau})$ that for any $s \ge 0$,
\begin{equation}\label{estimate for a_0^pm}
\| a^\pm_0(z; \Psi^{\tau_0, \tau})  \|_s  \lesssim_s  \, \, \| z_\bot \|_0^2 \, ,
\qquad \forall  \, 0 \le \tau_0, \tau \le 1, \, \forall \, z \in \mathcal V'\, .
\end{equation}
To prove corresponding estimates for $a^\pm_k(z; \Psi^{\tau_0, \tau})$ with $1 \le k \le N,$ we argue by induction.
Assume that 
for any $0 \le j \le k-1$ and $ s \geq 0$, 
\begin{equation}\label{stima a Psi k completa nella dim}
\| a^\pm_j(z; \Psi^{\tau_0, \tau}) \|_s \lesssim_{s, j} \| z_\bot \|_0^2, \qquad \forall  \, 0 \le \tau_0, \tau \le 1, \,  \forall  \, z \in \mathcal V' \, .
\end{equation}
By the definition \eqref{def a_k^pm} of $a^\pm_k(z; \Psi^{\tau_0, \tau})$, the definition \eqref{definition b k} of $b^\pm_k(\tau, z; \tau_0)$,
and by the induction hypothesis, Lemma \ref{estimates for flow} - Lemma \ref{lemma campo vettoriale new},  
and the interpolation Lemma \ref{lemma interpolation}, 
one then concludes that the estimate \eqref{stima a Psi k completa nella dim} is also satisfied for $j = k$. 
Using the analyticity properties established for $a^\pm_k(\tau, z; \tau_0)$ and $\Psi^{\tau_0, \tau}(z)$, 
one verifies the ones stated for the coefficients $a^\pm_k(z; \Psi^{\tau_0, \tau})$.

\smallskip

\noindent
{\em Estimates of the derivatives of $a^\pm_k(z; \Psi^{\tau_0, \tau})$:} 
By Lemma \ref{lemma campo vettoriale new}, Lemma \ref{estimates for flow}, and the chain rule one has
for any $0 \le k \le N$, $0 \le \tau_0, \tau  \le 1$, $z \in \mathcal V'$, $\widehat z \in h_0^0$, $s \ge 0,$
\begin{equation}\label{estimate d a k (t, z, tau 0)}
\| d a^\pm_k(\tau, z; \tau_0) [\widehat z] \|_s  \lesssim_{s, k}  
\| \pi_\bot \Psi^{\tau_0, \tau}(z) \|_0 \|d \Psi^{\tau_0, \tau}(z) [\widehat z] \|_0 \lesssim_s  \| z_\bot \|_0 \| \widehat z\|_0 \, .
\end{equation}
If in addition, $\widehat z_1, \ldots , \widehat z_l \in h^0_0$, $l \ge 2,$
\begin{equation}\label{estimate d ^ l a k (t, z, tau 0)}
\| d^l a^\pm_k(\tau, z; \tau_0)[\widehat z_1, \ldots , \widehat z_l]  \|_s \lesssim_s  \prod_{j=1}^l \| \widehat z_j\|_0\,.
\end{equation}
By the definition of $a^\pm_0(z; \Psi^{\tau_0, \tau})$, and  \eqref{estimate for a_0^pm}, \eqref{estimate d a k (t, z, tau 0)} - \eqref{estimate d ^ l a k (t, z, tau 0)}
it then follows that the claimed estimate for $\| d^l a^\pm_0(z; \Psi^{\tau_0, \tau})[\widehat z_1, \ldots , \widehat z_l]  \|_s$ hold for any $l \ge 1$.
To prove corresponding estimates for the derivatives of $a^\pm_k(z; \Psi^{\tau_0, \tau})$ with $1 \le k \le N,$ we again argue by induction.
Assume that for any $0 \le j \le k-1$ and $ s \geq 0$,
\begin{equation}\label{estimates for d a_k z; Psi tau_0, tau}
\| d a^\pm_j(z; \Psi^{\tau_0, \tau}) [\widehat z] \|_s \lesssim_{s, j} \| z_\bot \|_0 \| \widehat z\|_0, \quad \forall  \, 0 \le \tau_0, \tau \le 1, \,  \forall  \, z \in \mathcal V'\,, \, \widehat z \in h^0_0\,.
\end{equation}
By the definition \eqref{def a_k^pm} of $a^\pm_k(z; \Psi^{\tau_0, \tau})$ and the estimates  \eqref{estimate a k (t, z, tau 0)} - \eqref{estimate d a k (t, z, tau 0)} 
 it then follows that
\eqref{estimates for d a_k z; Psi tau_0, tau} also holds for $j = k.$
The estimates for  $\| d^l a^\pm_k(z; \Psi^{\tau_0, \tau}) [\widehat z_1, \ldots, \widehat z_l]\|_s$
with $l \ge 2$ are derived in a similar fashion.

\smallskip

\noindent
{\it Definition and estimate of ${\cal R}_N( z; \Psi^{\tau_0, \tau})$:}   
The remainder term ${\cal R}_N( z; \Psi^{\tau_0, \tau})$
 is defined so that  the identity \eqref{espansione asintotica Psi tau 0 tau} holds,
$$
{\cal R}_N( z; \Psi_X^{\tau_0, \tau}) := 
\Psi_X^{\tau_0, \tau}(z) - z  -  \big( 0, \, \, \mathcal{OP}_N ( z; \Psi_X^{\tau_0, \tau}) \big) \, .
$$
By \eqref{Duhamel ODE X} and the expansion \eqref{expansion for Y_tau_0} of $Y_{\tau_0}(\tau, z)$, 
the remainder term ${\cal R}_N( z; \Psi^{\tau_0, \tau})$ satisfies
$$
 \partial_\tau {\cal R}_N( z; \Psi_X^{\tau_0, \tau})  = \mathcal R_N(\tau, z; Y_{\tau_0})  \, , \quad
 {\cal R}_N( z; \Psi_X^{\tau_0, \tau_0})   = 0 \, ,
$$
and hence
\begin{equation}\label{remainder R N Psi^ tau0, tau}
{\cal R}_N( z; \Psi_X^{\tau_0, \tau}) = \int_{\tau_0}^\tau \mathcal R_N(t, z; Y_{\tau_0}) \, d t \, .
\end{equation}
We recall that the remainder term $\mathcal R_N(t, z; Y_{\tau_0})$ is given by \eqref{remainder of Y tau_0},
$$
\begin{aligned}
  \big( 0, \,\, {\cal F}^+_{N_S} \circ   \sum_{k = 0}^N  a^+_k(t, z; \tau_0) & {\cal R}_{N, k}^{(+, 3)}(t, z; \tau_0) 
 + \, {\cal F}^-_{N_S} \circ   \sum_{k = 0}^N  a^-_k(t, z; \tau_0) {\cal R}_{N, k}^{(-, 3)}(t, z; \tau_0) \, \big) 
 +  \,{\cal R}_N(\tau , \Psi^{\tau_0, t}(z); X) \, ,
\end{aligned}
$$
where ${\cal R}_{N, k}^{(+, 3)}(t, z; \tau_0)$ is given in \eqref{definizione cal R n N (3)}
and ${\cal R}_{N, k}^{(-, 3)}(t, z; \tau_0)$ in \eqref{definizione cal R n N (-,3)}.
We estimate the two components  $\pi_S \int_{\tau_0}^\tau {\cal R}_N(t, z; Y_{\tau_0}) \,d t$ and $\pi_\bot \int_{\tau_0}^\tau {\cal R}_N(t, z;  Y_{\tau_0}) \,d t$ 
of $\int_{\tau_0}^\tau {\cal R}_N(t, z;  Y_{\tau_0}) \,d t $ separately. 
By \eqref{remainder R N Psi^ tau0, tau} and the formula \eqref{remainder of Y tau_0} for $\mathcal R_N(\tau, z; Y_{\tau_0})$, recalled above,
$$
\pi_S \int_{\tau_0}^\tau {\cal R}_N(t, z;  Y_{\tau_0}) \,d t  = \int_{\tau_0}^\tau \pi_S {\cal R}_{N}(t, \Psi^{\tau_0, t}(z) ; X)\, d t
$$
and 
$$
\begin{aligned}
& \pi_\bot \int_{\tau_0}^\tau {\cal R}_N(t, z; Y_{\tau_0}) \,d t = \int_{\tau_0}^\tau \pi_\bot {\cal R}_{N}(t, \Psi^{\tau_0, t}(z); X )\, d t \\
&  + 
{\cal F}^+_{N_S} \circ  \sum_{k = 0}^N\int_{\tau_0}^\tau  a^+_k(t, z; \tau_0) {\cal R}_{N, k}^{(+,3)}(t, z; \tau_0)   \, d t 
+ \, {\cal F}^-_{N_S} \circ  \sum_{k = 0}^N\int_{\tau_0}^\tau  a^-_k(t, z; \tau_0) {\cal R}_{N, k}^{(-,3)}(t, z; \tau_0)   \, d t \,. 
\end{aligned}
$$
By Lemma \ref{lemma campo vettoriale new} and Lemma \ref{estimates for flow}, for any $z \in \mathcal V'$, $0 \le \tau_0, t \le 1,$
\begin{equation}\label{stima X Psi N (q)}
\|\pi_S {\cal R}_{N}(t, \Psi^{\tau_0, t}(z) ; X) \| \lesssim_{N} \| {\cal R}_{N}(t, \Psi^{\tau_0, t}(z) ; X) \|_0  \lesssim_{N} \| z_\bot \|_0^2\,,
\end{equation}
implying that 
$$
\| \int_{\tau_0}^\tau \pi_S {\cal R}_{N}(t, \Psi^{\tau_0, t}(z) ; X)\, d t \| \lesssim_{s, N} \| z_\bot \|_0^2\,.
$$
The component $\pi_\bot \int_{\tau_0}^\tau {\cal R}_N(t, z;  Y_{\tau_0}) \,d t$ is estimated by Gronwall's lemma.
First we have to determine the terms in ${\cal R}_{N, k}^{(\pm,3)}(t, z; \tau_0)$ which contain ${\cal R}_{N}(t, \Psi^{\tau_0, t}(z) ; X)$.
By the definition of ${\cal R}_{N, k}^{(\pm,3)}(t, z; \tau_0)$ (cf. \eqref{definizione cal R n N (3)}, \eqref{definizione cal R n N (-,3)}), 
$$
\int_{\tau_0}^\tau  a^\pm_k(t, z; \tau_0) {\cal R}_{N, k}^{(\pm,3)}(t,  z; \tau_0) \,d t=  \int_{\tau_0}^\tau  a^\pm_k(t, z; \tau_0){\cal R}_{N, k}^{(\pm,1)}(t, z;  \tau_0) \,d t +  
\sum_{j = 0}^{N - k} \int_{\tau_0}^\tau  a^\pm_k(t, z; \tau_0){\cal R}_{N, k, j}^{(\pm,2)}(t, z;  \tau_0) \,d t \, .
$$
Note that the terms ${\cal R}_{N, k, j}^{(\pm, 2)}(t, z;  \tau_0)$ do not involve ${\cal R}_{N}(t, \Psi^{\tau_0, t}(z) ; X)$
(cf.  \eqref{definition R (2)},  \eqref{definition R (-,2)}), whereas in contrast,  ${\cal R}_{N, k}^{(\pm,1)}(t, z;  \tau_0)$ do. Indeed,
by the definitions \eqref{cal R n N (1)} and  \eqref{cal R n N (-,1)},
$ \int_{\tau_0}^\tau  a^\pm_k(t, z; \tau_0){\cal R}_{N, k}^{(\pm,1)}(t, z;  \tau_0) \,d t $ equals 
$$
\begin{aligned}
&  \int_{\tau_0}^\tau  a^\pm_k(t, z; \tau_0) \cdot (\pm D)^{- k}  \big(  \sum_{j = N - k + 1}^{N }  a^\pm_j(z; \Psi^{\tau_0, t}) \cdot (\pm D)^{- j}[( {\cal F}^\pm_{N_S})^{- 1}z_\bot] \big)  \,d t \, \\
&  + \,  \int_{\tau_0}^\tau  a^\pm_k(t, z; \tau_0) \cdot (\pm D)^{-k} [({\cal F}^\pm_{N_S})^{- 1} \pi_\bot {\cal R}_N(z; \Psi^{\tau_0, t}) ] \,d t \,.
\end{aligned}
$$
One then infers from \eqref{remainder R N Psi^ tau0, tau} that $\pi_\bot {\cal R}_N(z; \Psi^{\tau_0, \tau})$ satisfies the integral equation
\begin{align}\label{equazione integrale resto}
\pi_\bot {\cal R}_N(z; \Psi^{\tau_0, \tau}) =  & B_N( \tau, z; \tau_0) 
+ \mathcal F^+_{N_S} \circ \int_{\tau_0}^\tau   \sum_{k = 0}^N a^+_k(t, z; \tau_0) \cdot D^{-k} [ ({\cal F}^+_{N_S})^{- 1} \pi_\bot {\cal R}_N(z; \Psi^{\tau_0, t}) ]\, d t \\
& + \mathcal F^-_{N_S} \circ \int_{\tau_0}^\tau   \sum_{k = 0}^N a^-_k(t, z; \tau_0) \cdot (-D)^{-k} [ ({\cal F}^-_{N_S})^{- 1} \pi_\bot {\cal R}_N(z; \Psi^{\tau_0, t}) ]\, d t \nonumber \, ,
\end{align}
where 
\begin{align}
B_N(\tau, z; \tau_0) :=  & \int_{\tau_0}^\tau \pi_\bot {\cal R}_{N}(t, \Psi^{\tau_0, t}(z); X )\, d t 
+ {\cal F}^+_{N_S} \circ  \sum_{k = 0}^N \sum_{j = 0}^{N - k} \int_{\tau_0}^\tau  a^+_k(t, z; \tau_0){\cal R}_{N, k, j}^{(+,2)}(t, z;  \tau_0) \,d t  \\
&  + {\cal F}^+_{N_S} \circ  \sum_{k = 0}^N \sum_{j = N - k + 1}^{N } \int_{\tau_0}^\tau  a^+_k(t, z; \tau_0) \cdot D^{- k} \big( \, a^+_j(z; \Psi^{\tau_0, t})  \cdot D^{- j} [({\cal F}^+_{N_S})^{- 1}z_\bot] \,  \big) \, d t \nonumber \\
&+ {\cal F}^-_{N_S} \circ  \sum_{k = 0}^N \sum_{j = 0}^{N - k} \int_{\tau_0}^\tau  a^-_k(t, z; \tau_0){\cal R}_{N, k, j}^{(-,2)}(t, z;  \tau_0) \, d t \nonumber \\
&  + {\cal F}^-_{N_S} \circ  \sum_{k = 0}^N \sum_{j = N - k + 1}^{N } \int_{\tau_0}^\tau  a^-_k(t, z; \tau_0) \cdot (-D)^{- k} \big( \, a^-_j(z; \Psi^{\tau_0, t})  \cdot (-D)^{- j} [({\cal F}^-_{N_S})^{- 1}z_\bot] \,  \big) \, d t \,. \nonumber
\end{align}
By the estimates of ${\cal R}_{N, k, j}^{(\pm,2)}(t, z;  \tau_0)$ in \eqref{stima cal R n k N (2)}, \eqref{stima cal R n k N (-,2)}, 
the ones of $a^\pm_k(\tau, z; X)$ and ${\cal R}_N(\tau, z; X)$ in Lemma \ref{lemma campo vettoriale new}, 
the estimates of $\Psi^{\tau_0, t}(z)$ in Lemma \ref{estimates for flow}, and the ones of $a^\pm_k(z; \Psi^{\tau_0, \tau})$ in \eqref{stima a Psi k completa nella dim},
and using the interpolation Lemma \ref{lemma interpolation}, one obtains for any $s \ge 0,$
$$
\| B_N( \tau, z; \tau_0) \|_{s + N +1 } \lesssim_{s, N} \| z_\bot \|_s \| z_\bot \|_0\,, \quad \forall z \in \mathcal V' \cap h^s_0\,, \,\,\, \forall \,\,0 \le \tau_0, \tau \le 1\,. 
$$
Note that $ \sum_{k = 0}^N a^\pm_k(t, z; \tau_0) (\pm D)^{-k}$ is a pseudo-differential operator of order $0$ where
the coefficients $a^\pm_k(t, z; \tau_0)$ satisfy $ \| a^\pm_k(t, z; \tau_0)\|_s  \lesssim_{s, k} \| z_\bot \|^2_0 $ (cf. \eqref{estimate a k (t, z, tau 0)}).
Hence for any $z \in \mathcal V' \cap h^s_0$, $0 \le \tau_0, \tau \le 1$,
$$
 \|  \sum_{k = 0}^N a^\pm_k(t, z; \tau_0) (\pm D)^{-k} [({\cal F}^\pm_{N_S})^{- 1} \pi_\bot {\cal R}_N(z; \Psi^{\tau_0, t})]  \|_{s+N+1} 
 \lesssim_{s, N} \| z_\bot \|_0^2 \,  \| ({\cal F}^\pm_{N_S})^{- 1} \pi_\bot {\cal R}_N(z; \Psi^{\tau_0, t}) \|_{s + N + 1}\,.
 $$
By Gronwall's Lemma and since $\mathcal V'_\bot$ is a ball of sufficiently small radius, the integral equation \eqref{equazione integrale resto} yields that for any $s \ge 0,$
\begin{equation}\label{stima X Psi N (z)}
\| \pi_\bot {\cal R}_{N}(z; \Psi^{\tau_0, \tau}) \|_{s + N + 1} \lesssim_{s, N} \| z_\bot \|_s \| z_\bot \|_0\,, \quad \forall z \in \mathcal V' \cap h^s_0 \,, \,\,\, \forall \,\,0 \le \tau_0, t \le 1\, ,
\end{equation}
which together with \eqref{stima X Psi N (q)}  proves the claimed estimates of ${\cal R}_{N}(z; \Psi^{\tau_0, \tau})$. 
The stated analyticity property of ${\cal R}_N( z; \Psi_X^{\tau_0, \tau})$ then follows from the already established analyticity properties of  
$\Psi_X^{\tau_0, \tau}(z)$, $a^\pm_k(\tau, z; \tau_0)$, and $a^\pm_k(z; \Psi^{\tau_0, \tau})$ (cf. e.g. \cite[Theorem A.5]{GK}).

\smallskip

\noindent
{\em Estimates of the derivatives of ${\cal R}_{N}(z ; \Psi^{\tau_0, \tau})$:} The estimates of the derivatives of ${\cal R}_{N}(z ; \Psi^{\tau_0, \tau})$  can be obtained in a similar way 
as the ones for ${\cal R}_{N}(z ; \Psi^{\tau_0, \tau})$. Indeed, for any $s \ge 0$, $z \in \mathcal V' \cap h_0^s$, $0 \le \tau_0, \tau \le 1$, $\widehat z \in h_0^s,$
$ d{\cal R}_{N}(z ; \Psi^{\tau_0, \tau})[\widehat z] = \int_{\tau_0}^\tau d  \mathcal R_N(t, z; Y_{\tau_0})[\widehat z] \, d t$ can be computed as
$$
\begin{aligned}
 \Big( 0, \ {\cal F}^+_{N_S} \circ  \int_{\tau_0}^\tau   \sum_{k = 0}^N d \big( a^+_k(t, z; \tau_0) & {\cal R}_{N, k}^{(+, 3)}(t, z; \tau_0) \big) [\widehat z]  \, d t \, 
  + {\cal F}^-_{N_S} \circ  \int_{\tau_0}^\tau \sum_{k = 0}^N d \big( a^-_k(t, z; \tau_0)  {\cal R}_{N, k}^{(-, 3)}(t, z; \tau_0) \big) [\widehat z]  \, d t \, \Big) \\
& \, +  \, \int_{\tau_0}^\tau\, d \big( {\cal R}_N(t , \Psi^{\tau_0, t}(z); X) \big) [\widehat z] \,d t\,. 
 \end{aligned}
$$
Again, we estimate 
$\pi_S \big( d{\cal R}_{N}(z ; \Psi^{\tau_0, \tau})[\widehat z] \big) = \int_{\tau_0}^\tau\,  \pi_S \big( d \big( {\cal R}_N(t , z; Y_{\tau_0}) \big) [\widehat z] \big) \,d t $
and $\pi_\bot \big( d{\cal R}_{N}(z ; \Psi^{\tau_0, \tau})[\widehat z] \big)$  
separately. By Lemma \ref{lemma campo vettoriale new}, Lemma \ref{estimates for flow}, and the chain rule, one has
\begin{equation}\label{estimate of S projection of derivative}
\| \int_{\tau_0}^\tau\,  \pi_S \big( d \big( {\cal R}_N(t , \Psi^{\tau_0, t}(z); X) \big) [\widehat z] \big) \,d t \| \lesssim_{N} \| z_\bot \|_0 \|\widehat z\|_0 \, ,
\end{equation}
whereas by \eqref{equazione integrale resto}, 
$\pi_\bot \big( d{\cal R}_{N}(z ; \Psi^{\tau_0, \tau})[\widehat z] \big)$
satisfies
\begin{align}\label{equation for the derivative of the remainder}
 \pi_\bot \big( d {\cal R}_N(z; \Psi^{\tau_0, \tau}) \big[\widehat z] \big) =  & \, B^{(1)}_N( \tau, z; \tau_0) [ \widehat z]  
 + \mathcal F^+_{N_S} \circ \int_{\tau_0}^\tau   \sum_{k = 0}^N a^+_k(t, z; \tau_0) \cdot D^{-k} \big[ ({\cal F}^+_{N_S})^{- 1} \pi_\bot d {\cal R}_N(z; \Psi^{\tau_0, t})[\widehat z]  \big] \, d t  \nonumber \\
& + \mathcal F^-_{N_S} \circ \int_{\tau_0}^\tau   \sum_{k = 0}^N a^-_k(t, z; \tau_0) \cdot (-D)^{-k} \big[ ({\cal F}^-_{N_S})^{- 1} \pi_\bot  d {\cal R}_N(z; \Psi^{\tau_0, t}) [\widehat z]  \big] \, d t  \, ,
\end{align}
with $B^{(1)}_N( \tau, z; \tau_0)[ \widehat z]$ given by
$$
\begin{aligned}
 B^{(1)}_N( \tau, z; \tau_0) [ \widehat z]  = & d B_N( \tau, z; \tau_0)  [\widehat z] 
+ \mathcal F^+_{N_S} \circ \int_{\tau_0}^\tau   \sum_{k = 0}^N d a^+_k(t, z; \tau_0)[\widehat z]  \cdot D^{-k} [ ({\cal F}^+_{N_S})^{- 1} \pi_\bot {\cal R}_N(z; \Psi^{\tau_0, t}) ]\, d t \\
& + \mathcal F^-_{N_S} \circ \int_{\tau_0}^\tau   \sum_{k = 0}^N d a^-_k(t, z; \tau_0) [\widehat z]  \cdot (-D)^{-k} [ ({\cal F}^-_{N_S})^{- 1} \pi_\bot {\cal R}_N(z; \Psi^{\tau_0, t}) ]\, d t \nonumber \, .
\end{aligned}
$$
Since
$$
\| B_N^{(1)}( \tau, z; \tau_0) [\widehat z] \|_{s + N +1 } \lesssim_{s, N} 
\| z_\bot \|_s  \| \widehat z \|_0 +  \| z_\bot \|_0 \| \widehat z \|_s \,, \quad \forall z \in \mathcal V' \cap h^s_0\,, \,\, \widehat z \in h^s_0\,,  \,\,0 \le \tau_0, \tau \le 1 \, ,
$$
we infer from \eqref{equation for the derivative of the remainder} by Gronwall's lemma that
\begin{equation}\label{estimate of bot projection of derivative}
\| \pi_\bot \big( d{\cal R}_{N}(z ; \Psi^{\tau_0, \tau})[\widehat z] \big) \|_{s+N+1} \lesssim_{s, N} 
\| z_\bot \|_s  \| \widehat z \|_0 +  \| z_\bot \|_0 \| \widehat z \|_s \,, \quad \forall z \in \mathcal V' \cap h^s_0\,,  \,\, \widehat z \in h^s_0\,,  \,\, 0 \le \tau_0, \tau \le 1\, ,
\end{equation}
which together with  \eqref{estimate of S projection of derivative} 
proves the claimed estimate for $d {\cal R}_N(z ; \Psi^{\tau_0, \tau}) [\widehat z]$.
In a similar fashion, one derives the estimates for $\| d^l {\cal R}_N(z ; \Psi^{\tau_0, \tau}) [\widehat z_1, \ldots, \widehat z_l] \|_{s + N + 1}$
with $\widehat z_1, \ldots, \widehat z_l \in h^s_0$, $l \ge 2$.
\end{proof}

\bigskip

It turns out that the flow maps $\Psi_X^{\tau_0, \tau}$ and hence the symplectic corrector $\Psi_C$ and its inverse $\Psi_C^{-1}$
preserve the reversible structures, introduced in Section \ref{introduzione paper}. To state the result in more detail, note that without loss of generality, we may assume that the neighborhood 
$\mathcal V' = \mathcal V'_S \times \mathcal V'_\bot$ (cf  Lemma \ref{estimates for flow}) is invariant  under the map $\mathcal S_{rev}$.

\noindent
{\bf Addendum to Theorem \ref{espansione flusso per correttore}}
{\em (i) For any $0 \le \tau_0, \tau \le 1$,
$\Psi_X^{\tau_0, \tau} \circ \mathcal S_{rev} = \mathcal S_{rev} \circ \Psi_X^{\tau_0, \tau}$ on $\mathcal V'$
and for any $z \in \mathcal V'$, $x \in \mathbb R$, $N \in \N,$ and  $0 \le k \le N$, 
the coefficients $ a^\pm_k(z; \Psi_X^{\tau_0, \tau})$ and 
the remainder term $ {\cal R}_N( z; \Psi_X^{\tau_0, \tau})$ 
of the expansion \eqref{expansion for Psi^{tau0, tau}} of $\Psi_X^{\tau_0, \tau}(z)$ satisfy
$$
   a^+_k( \mathcal S_{rev} z; \Psi_X^{\tau_0, \tau})(x) =  a^-_k(z; \Psi_X^{\tau_0, \tau})( - x) \,, \quad  
 a^-_k( \mathcal S_{rev} z; \Psi_X^{\tau_0, \tau})(x) =  a^+_k(z; \Psi_X^{\tau_0, \tau})( - x) \, ,
$$   
$$   
    {\cal R}_N( \mathcal S_{rev} z; \Psi_X^{\tau_0, \tau}) = \mathcal S_{rev} ( {\cal R}_N( z; \Psi_X^{\tau_0, \tau}))\,.
$$
(ii) As a consequence, $\Psi_C$ and $\Psi_C^{-1}$ are invariant under $\mathcal S_{rev}$ on $\mathcal V'$,
\begin{equation}\label{reversability and flow}
\Psi_C \circ \mathcal S_{rev} = \mathcal S_{rev} \circ \Psi_C\,, \quad  \Psi_C^{-1} \circ \mathcal S_{rev} = \mathcal S_{rev} \circ \Psi_C^{-1}
\end{equation}
and for any $z \in \mathcal V'$, $x \in \mathbb R$, $N \in \N,$ and $0 \le k \le N$,  
$$
   a^+_k( \mathcal S_{rev} z; \Psi_C)(x) =  a^-_k(z; \Psi_C)( - x) \,, \quad  
    a^-_k( \mathcal S_{rev} z; \Psi_C)(x) =  a^+_k(z; \Psi_C)( - x) \,,   
$$
$$   
    {\cal R}_N( \mathcal S_{rev} z; \Psi_C) = \mathcal S_{rev} ( {\cal R}_N( z; \Psi_C))\,.
$$
}
\noindent
{\bf Proof of Addendum to Theorem \ref{espansione flusso per correttore}} Clearly, item (ii) is a direct consequence of item (i).
By the  Addendum to Lemma \ref{espansione L S bot q z}, the operator $\mathcal L(z)$, introduced in \eqref{Pull back of LambdaG}, 
satisfies $\mathcal L( \mathcal S_{rev} z) \circ \mathcal S_{rev} = - \mathcal S_{rev} \circ \mathcal L(z)$ on $\mathcal V$.
It implies that for any $z \in \mathcal V$, 
$\mathcal E(\mathcal S_{rev} z) = - \mathcal S_{rev} \mathcal E(z)$ where $\mathcal E(z)$ has been introduced in \eqref{forma finale 1 forma Kuksin 2}.
Altogether we then conclude that the vector field $X(\tau, z),$ introduced in \eqref{definizione campo vettoriale ausiliario}, satisfies 
$$
X(\tau, \mathcal S_{rev}z) = \mathcal S_{rev}X(\tau, z)\,, \qquad \forall \, z \in \mathcal V, \,\, 0 \le \tau \le 1
$$
and hence by the uniqueness of the initial value problem of $\partial_\tau z = X(\tau, z)$,  the solution map satisfies 
$$
\Psi_X^{\tau_0, \tau}( \mathcal S_{rev}z) = \mathcal S_{rev} \Psi_X^{\tau_0, \tau}(z) \,, \quad \forall \, z \in \mathcal V' \,, \, \, 0 \le \tau_0,  \tau \le 1\,.
$$
The claimed identities for  $a^\pm_k(z; \Psi_X^{\tau_0, \tau})$ and ${\cal R}_N( z; \Psi_X^{\tau_0, \tau})$
then follow from the expansion \eqref{espansione asintotica Psi tau 0 tau}.
\hfill   $\square$

\bigskip
We finish this section with a discussion of two applications of Theorem \ref{espansione flusso per correttore}. The first one concerns the expansion 
of the transpose $(d\Psi^{0, \tau}_X(z))^\top$ of the differential $ d\Psi^{0, \tau}_X(z)$ which will be used in Section \ref{Hamiltoniana trasformata} 
in the proof of Lemma \ref{proprieta hamiltoniana cal P 3 (2b)}. Recall that for any $z \in \mathcal V',$ and $\widehat z$, $\widehat w \in h^0_0$,
$$
\Lambda_\tau(z)[\widehat z, \widehat w] = \langle \widehat z \, |  \, J^{-1}\mathcal L_\tau(z) [\widehat w]  \rangle, \qquad
\mathcal L_\tau(z) = {\rm{Id}} + \tau J \mathcal L(z), \quad 0 \le \tau \le 1 \,,
$$
and that the flow $\Psi^{0, \tau}_X$ satisfies $\partial_\tau \big( (\Psi^{0, \tau}_X)^* \Lambda_\tau \big) = 0$ and hence $ (\Psi^{0, \tau}_X)^* \Lambda_\tau = \Lambda_0$.
By the definition of the pullback this means that for any $z \in V',$ $0 \le \tau \le 1$, and $\widehat z$, $\widehat w \in h^0_0$,
$$
\langle d\Psi^{0, \tau}_X(z) [\widehat z] \, |  \, J^{-1} \mathcal L_\tau (\Psi^{0, \tau}_X(z)) \big[ d\Psi^{0, \tau}_X(z) \widehat w \big] \rangle
= \langle  \widehat z | J^{-1} \widehat w \rangle 
$$
implying that
$$
(d\Psi^{0, \tau}_X(z))^\top  J^{-1} \mathcal L_\tau (\Psi^{0, \tau}_X(z))  d\Psi^{0, \tau}_X(z) = J^{-1}.
$$
Using that $(d\Psi^{0, \tau}_X(z))^{-1} = d\Psi^{ \tau, 0}_X(\Psi^{0, \tau}_X(z))$ one obtains the following formula for $(d\Psi^{0, \tau}_X(z))^\top$,
\begin{equation}\label{formula for d Psi ^{0, tau}_X (z) ^ t }
(d\Psi^{0, \tau}_X(z))^\top = J^{-1} \circ  d\Psi^{ \tau, 0}_X(\Psi^{0, \tau}_X(z)) \circ \big(\mathcal L_\tau (\Psi^{0, \tau}_X(z))\big)^{-1} \circ J.
\end{equation}
Note that $ d\Psi^{ \tau, 0}_X(\Psi^{0, \tau}_X(z))$ and $\big( \mathcal L_\tau (\Psi^{0, \tau}_X(z)) \big)^{-1}$ are bounded linear operators on $h^0_0$,
implying that $(d\Psi^{0, \tau}_X(z))^\top$ is one on $h^1_0$, and that these operators and their derivatives depend continuously on $0 \le \tau \le 1.$
\begin{corollary}\label{expansion of  of differential of Psi}
 For any $0 \le \tau \le 1$, $z  \in {\cal V'}$, the transpose $(d\Psi^{0, \tau}_X(z))^\top$
(with respect to the standard inner product) of the differential
$d\Psi^{0, \tau}_X(z)$ is a bounded linear operator $ (d\Psi^{0, \tau}_X(z))^\top : h^1_0 \to  h^1_0$ and for any $N \in \N$, 
$(d\Psi^{0, \tau}_X(z))^\top $ admits an expansion of the form
$$
 { \rm{Id}} + \iota_\bot \circ \mathcal{OP}(\, z \, ; (d\Psi^{0, \tau}_X)^\top )  +  {\cal R}_N( \, z \, ;  (d\Psi^{0, \tau}_X)^\top ) ,
$$
where for any $\widehat z \in h^1_0$, $ \mathcal{OP}(\, z \, ; (d\Psi^{0, \tau}_X)^\top )[\widehat z]$ is given by
$$
\begin{aligned}
&  {\cal F}^+_{N_S} \circ \sum_{k = 0}^N  a^+_k(z; (d\Psi^{0, \tau}_X)^\top) \cdot  D^{- k} [(\mathcal F^+_{N_S})^{-1} \widehat z_\bot] +  
{\cal F}^+_{N_S} \circ \sum_{k = 0}^N  \mathcal A^+_k(z; (d\Psi^{0, \tau}_X)^\top)[\widehat z]  \cdot D^{- k} [( \mathcal F^+_{N_S})^{-1} z_\bot ]\\
+ \, &   {\cal F}^-_{N_S} \circ \sum_{k = 0}^N  a^-_k(z; (d\Psi^{0, \tau}_X)^\top) \cdot (-D)^{- k} [(\mathcal F^-_{N_S})^{-1} \widehat z_\bot] +  
{\cal F}^-_{N_S} \circ \sum_{k = 0}^N  \mathcal A^-_k(z; (d\Psi^{0, \tau}_X)^\top)[\widehat z] \cdot  (-D)^{- k}[ (\mathcal F^-_{N_S})^{-1} z_\bot ]
\end{aligned}
$$
and for any integer $s \ge 0$, $  k \ge 0$, and $N \ge 1$,
$$
a^\pm_k( \, \cdot \, ; (d\Psi^{0, \tau}_X)^\top) : {\cal V'} \to H^s_\C \, , \,\, z \mapsto  a^\pm_k (z; (d\Psi^{0, \tau}_X)^\top )\,, 
$$
$$
\mathcal A^\pm_k( \, \cdot \, ; (d\Psi^{0, \tau}_X)^\top) : {\cal V'} \to \mathcal B( h_0^1, H^s_\C), \,\, z \mapsto  \mathcal A^\pm_k (z ; (d\Psi^{0, \tau}_X)^\top )\,,
$$
$$
 {\cal R}_N( \, \cdot \, ;  (d\Psi^{0, \tau}_X)^\top ) : {\cal V'} \cap h^s_0 \to {\cal B}(h^{s+1}_0, h_0^{s + 1 + N +1}), \,\, z  \mapsto {\cal R}_N( z; (d\Psi^{0, \tau}_X)^\top)
$$
are real analytic maps. Furthermore, for any any $z \in {\cal V'}$, $k \ge 0$ and $\widehat z \in h^1_0$,
$$ 
a^-_k(z; (d\Psi^{0, \tau}_X)^\top) = \overline{  a^+_k(z; (d\Psi^{0, \tau}_X)^\top)} \, , \qquad
 \mathcal A^-_k(z; (d\Psi^{0, \tau}_X)^\top)[\widehat z] = \overline{ \mathcal A^+_k(z; (d\Psi^{0, \tau}_X)^\top)[\widehat z]} \, .
$$
If in addition, $\widehat z_1, \ldots, \widehat z_l \in h^0_0$, $l \ge 1$, then 
    $$
   \| a^\pm_k( z; (d\Psi^{0, \tau}_X)^\top)\|_s \lesssim_{s, k} \| z_\bot\|^2_0 \,, \qquad  \| d a^\pm_k( z; (d\Psi^{0, \tau}_X)^\top)[\widehat z_1]\|_s  \lesssim_{s, k}  \| z_\bot\|_0 \| \widehat z_1\|_0 \,,
   $$
   $$
    \| d^l a^\pm_k( z; (d\Psi^{0, \tau}_X)^\top)[\widehat z_1, \ldots, \widehat z_l]\|_s  \lesssim_{s, k, l}  \prod_{j = 1}^l \| \widehat z_j\|_0\,,
   $$
and 
 $$ 
   \| \mathcal A^\pm_k( z; (d\Psi^{0, \tau}_X)^\top) [\widehat z]\|_s \lesssim_{s, k}  \| z_\bot\|_0 \|\widehat z \|_1\,, \quad
   \| d^l (\mathcal A^\pm_k( z; (d\Psi^{0, \tau}_X)^\top)  [\widehat z] )[\widehat z_1, \ldots, \widehat z_l]\|_s  \lesssim_{s, k, l} \|\widehat z \|_1 \prod_{j = 1}^l \| \widehat z_j\|_0\,.
  $$
The remainder ${\cal R}_N( z; (d\Psi^{0, \tau}_X)^\top)$ satisfies for any $z \in {\cal V'} \cap h^s_0$, $\widehat z \in h^{s+1}_0,$ and $\widehat z_1, \ldots, \widehat z_l \in h^s_0$, $l \in \N$,
  $$
   \| {\cal R}_N( z; (d\Psi^{0, \tau}_X)^\top) [\widehat z]\|_{s + 1 + N + 1}  \lesssim_{s, N} \| z_\bot \|_0 \| \widehat z\|_{s+1} + \| z_\bot \|_s \| \widehat z\|_1\,,  \qquad \qquad \qquad \qquad \qquad
   $$
   $$
   \begin{aligned}
   \| d^l \big(  {\cal R}_N & ( z; (d\Psi^{0, \tau}_X)^\top)  [\widehat z] \big)[\widehat z_1,  \ldots,  \widehat z_l]  \|_{s + 1 + N + 1} \\
  & \lesssim_{s, N, l} \| \widehat z\|_{s+1} \prod_{j = 1}^l \| \widehat z_j\|_0 + \| \widehat z\|_1 \sum_{j = 1}^l \| \widehat z_j\|_s \prod_{i \neq j} \| \widehat z_i\|_0 + 
   \| z_\bot\|_s  \| \widehat z\|_1 \prod_{j = 1}^l \| \widehat z_j \|_0\,.
    \end{aligned}
  $$
\end{corollary}
\begin{remark}
Corollary \ref{expansion of  of differential of Psi} holds in particular for $ d\Psi_C(z)^\top = (d\Psi^{0, 1}_X(z))^\top$.
\end{remark}
\begin{proof}
The starting point is the formula \eqref{formula for d Psi ^{0, tau}_X (z) ^ t } for $d\Psi^{0, \tau}_X(z)^\top$.
Since
$$
\| J \widehat z \|_{s} \lesssim_s \| \widehat z \|_{s + 1} \, , \qquad 
\| J^{-1} \widehat z \|_{s + 1} \lesssim_s \| \widehat z \|_{s} \, ,
$$
 it suffices to derive corresponding estimates for the operator
$d\Psi^{ \tau, 0}_X(\Psi^{0, \tau}_X(z)) \circ \mathcal L_\tau (\Psi^{0, \tau}_X(z))^{-1}$.
By Theorem \ref{espansione flusso per correttore}, for any $w \in \mathcal V'$ and $\widehat w \in h^0_0$, one has
$$
d\Psi^{\tau, 0}_X(w) =   \widehat w +  \big( 0, \ d \big( \mathcal{OP}_N ( w; \Psi_X^{\tau, 0}) \big) [\widehat w] \big) 
+d {\cal R}_N( w; \Psi_X^{\tau, 0})[\widehat w] 
$$
where $d \big( \mathcal{OP}_N ( w; \Psi_X^{\tau, 0}) \big)[\widehat w]$ is given by 
$$
\begin{aligned}
& {\cal F}^+_{N_S} \circ\sum_{k = 0}^N   a^+_k(w; \Psi_X^{\tau, 0}) \cdot  D^{- k} [( {\cal F}^+_{N_S} )^{- 1}\widehat w_\bot]  
+ {\cal F}^+_{N_S} \circ\sum_{k = 0}^N   d a^+_k(w; \Psi_X^{\tau, 0})[\widehat w] \cdot D^{- k}[ ({\cal F}^+_{N_S})^{- 1}[w_\bot]  \\
& +  {\cal F}^-_{N_S} \circ\sum_{k = 0}^N   a^-_k(w; \Psi_X^{\tau, 0}) \cdot (-D)^{- k} [( {\cal F}^-_{N_S} )^{- 1}\widehat w_\bot]  
+ {\cal F}^-_{N_S} \circ\sum_{k = 0}^N   d a^-_k(w; \Psi_X^{\tau, 0})[\widehat w] \cdot (-D)^{- k}[ ({\cal F}^-_{N_S})^{- 1}[w_\bot] \, .
\end{aligned}
$$
These formulas are applied to 
$w =  \Psi^{0, \tau}_X(z) = z +  \big( 0, \  \mathcal{OP}_N ( z; \Psi_X^{0, \tau}) \big)  + {\cal R}_N( z; \Psi_X^{0, \tau})$ where
$$
\mathcal{OP}_N ( z; \Psi_X^{0, \tau}) = 
{\cal F}^+_{N_S} \circ\sum_{k = 0}^N   a^+_k(z; \Psi_X^{0, \tau}) \cdot  D^{- k} [( {\cal F}^+_{N_S} )^{- 1} z_\bot] 
+  {\cal F}^-_{N_S} \circ\sum_{k = 0}^N   a^-_k(z; \Psi_X^{0, \tau}) \cdot (-D)^{- k} [( {\cal F}^-_{N_S} )^{- 1} z_\bot] \, .
$$ 
By \eqref{formula for inverse of mathcal L tau}, $\mathcal L_\tau (w)^{-1} = \big({ \rm{Id}} +  \tau J \mathcal L (w) \big)^{-1}$ is of the form
$$
 (  {\rm Id} + \tau J  \mathcal L(w) )^{-1} = \begin{pmatrix} 
C_{11}^{-1} & - \tau C_{11}^{-1}  B_{12} \\
-\tau J_\bot \mathcal L_{\bot}^S( w) C_{11}^{-1} &  {\rm Id}_\bot + \tau^2 J_\bot \mathcal L_{\bot}^S( w) C_{11}^{-1}B_{12}
\end{pmatrix} ,
$$
with $C_{11}$ and $B_{12}$ given by \eqref{formula for C} and \eqref{formula for B}, respectively. 
We then obtain an expansion of  $ \mathcal L_\tau(w)^{-1}$ from the expansion of  $\mathcal L_{\bot}^S( w)$,
provided by Lemma \ref{espansione L S bot q z}(ii). The formulas are then again applied to 
$w =  \Psi^{0, \tau}_X(z) = z +  \big( 0, \  \mathcal{OP}_N ( z; \Psi_X^{0, \tau}) \big)  + {\cal R}_N( z; \Psi_X^{0, \tau})$.
Combining these expansions, one obtains an expansion of $d\Psi^{ \tau, 0}_X(\Psi^{0, \tau}_X(z)) \circ \mathcal L_\tau (\Psi^{0, \tau}_X(z))^{-1}$
as stated. The claimed estimates follow from Lemma \ref{espansione L S bot q z} and Theorem \ref{espansione flusso per correttore}.
\end{proof}
\medskip

As a second application of Theorem \ref{espansione flusso per correttore}, 
we compute the Taylor expansion of the symplectic corrector $\Psi_C(z_S, z_\bot)$ in $z_\bot$ around $0$.
This expansion will be needed in the subsequent section to show that the BO Hamiltonian, when expressed in the new coordinates
provided by the map $\Psi_L \circ \Psi_C$, is in Birkhoff normal
form up to order three. Note that by Theorem \ref{espansione flusso per correttore}, 
for any $z_S \in \mathcal V'_S,$ $\widehat z_\bot \in h^0_\bot$, $0 \le k \le N$,
$$
a^\pm_k((z_S, 0); \Psi_C) = 0\,, \qquad d_\bot a^\pm_k((z_S, 0); \Psi_C) [\widehat z_\bot] = 0\,,  \quad
$$
$$
{\cal R}_{N}((z_S, 0); \Psi_C) = 0 \,, \qquad d_\bot \big( {\cal R}_{N}((z_S, 0); \Psi_C)\big)[\widehat z_\bot] = 0 \,.
$$
Hence the Taylor expansion of ${\cal R}_{ N}(z; \Psi_C)$ in $z_\bot$ of order three around $ 0$ reads
\begin{equation}\label{Taylor expansion R_N}
{\cal R}_{ N}(z; \Psi_C) = {\cal R}_{ N, 2}(z; \Psi_C) + {\cal R}_{ N, 3}(z; \Psi_C)\,, \qquad 
\mathcal R_{N, 2} (z; \Psi_C) : =  \frac{1}{2} d^2_\bot {\cal R}_{N}((z_S, 0) ; \Psi_C) [z_\bot, z_\bot]
\end{equation}
with  the Taylor remainder term ${\cal R}_{N, 3}(z; \Psi_C)$ given by 
\begin{equation}\label{Taylor remainder term}
{\cal R}_{N, 3}(z; \Psi_C) = \int_0^1 d^3_\bot {\cal R}_{N}((z_S, t z_\bot); \Psi_C)[z_\bot, z_\bot, z_\bot] \, \frac{1}{2}(1 - t)^2 \, dt
\end{equation}
whereas for any $0 \le k \le N$, 
$
{\cal F}^\pm_{N_S} \big( a^\pm_k ( z; \Psi_C) (\pm D)^{- k} [({\cal F}^\pm_{N_S})^{- 1} z_\bot]  \big)
$ 
vanishes in $z_\bot$ at $0$ up to order two.

Furthermore, we need to compute the Taylor expansion of $ d\Psi_C(z)^\top = (d\Psi^{0, 1}_X(z))^\top$ in $z_\bot$ around $0$.
According to Corollary \ref{expansion of  of differential of Psi}, 
the term $ {\cal R}_{N}(z; d \Psi_C^\top)$ in the expansion of $d\Psi_C(z)^\top$ satisfies
$ {\cal R}_{N}((z_S, 0); d \Psi_C^\top)=0$ for any $(z_S, 0) \in \mathcal V'$
and hence for any $\widehat z \in h^1_0$,  the Taylor expansion of ${\cal R}_{N}(z; d \Psi_C^\top)[\widehat z]$ of order $2$ in $z_\bot$ around $0$ reads
\begin{equation}\label{Taylor expansion of of d Psi C t}
 {\cal R}_{N}(z; d \Psi_C^\top)[\widehat z] = {\cal R}_{N, 1}(z; d \Psi_C^\top)[\widehat z] + {\cal R}_{N, 2}(z; d \Psi_C^\top)[\widehat z] ,
\end{equation}
where ${\cal R}_{N, 2}(z; d \Psi_C^\top)[\widehat z]$ denotes the Taylor remainder term of order $2$.
\begin{corollary}\label{proposizione espansione taylor correttore simplettico}
 $(i)$ For any $z' \in \mathcal V'$ and any integer $N \ge 1$, the Taylor expansion of the symplectic corrector 
 $\Psi_C(z_S, z_\bot)$ in $z_\bot$ around $0$ reads   
$$
 \Psi_C(z) = (z_S, 0)  + ( 0, z_\bot) + {\cal R}_{N, 2}(z; \Psi_C) +  \Psi_{C, 3}(z) \, ,
 $$
 where 
  \begin{equation}\label{espansione Psi C ordini taglia}
  \Psi_{C, 3}(z) \equiv  \Psi_{C, N, 3}(z) :=  \big( 0, \, \, \mathcal{OP}_N ( z; \Psi_{C}) \big)  + {\cal R}_{N, 3}( z; \Psi_{C}) 
   \end{equation}
 and $ \mathcal{OP}_N ( z; \Psi_{C})$ is given by \eqref{OP_N  Psi_C}.
For any $s \geq 0$, the map $\mathcal V' \cap h^s_0 \to  h^{s + N +1}_0$, $z \mapsto {\cal R}_{N, 2}(z; \Psi_C)$,  is real analytic
and  the following estimates hold: for any $z \in \mathcal V' \cap h^s_0 $, $\widehat z \in  h^s_0$,
$$
 \| {\cal R}_{N, 2}(z; \Psi_C)\|_{s + N +1} \lesssim_{s, N} \| z_\bot \|_s \| z_\bot \|_0\,, \qquad  
 \| d {\cal R}_{N, 2}(z; \Psi_C)[\widehat z]\|_{s + N +1} \lesssim_{s, N} \| z_\bot \|_0 \| \widehat z\|_s + \| z_\bot \|_s \| \widehat z\|_0\, .
 $$
 If in addition, $ \widehat z_1, \ldots, \widehat z_l \in  h^s_0$, $l \ge 2,$
 $$
 \| d^l {\cal R}_{N, 2}(z; \Psi_C)[\widehat z_1, \ldots, \widehat z_l] \|_{s + N +1} \lesssim_{s, N, l} \sum_{j = 1}^l \| \widehat z_j \|_s \prod_{i \neq j} \| \widehat z_i \|_0 + 
\| z_\bot \|_s \prod_{j = 1}^l \| \widehat z_j\|_0\,.
$$
Similarly, for any $s \geq 0$, the map $\mathcal V' \cap h^s_0 \to  h^{s + N +1}_0$, $z \mapsto {\cal R}_{N, 3}(z; \Psi_C)$,  is real analytic
and  the following estimates hold: for any $z \in \mathcal V' \cap h^s_0 $, $ \widehat z_1, \widehat z_2 \in  h^s_0$,
$$
 \| {\cal R}_{N, 3}(z; \Psi_C)\|_{s + N + 1} \lesssim_{s, N} \| z_\bot \|_s \| z_\bot \|_0^2 \,,  
 \qquad \| d {\cal R}_{N, 3}(z; \Psi_C )[\widehat z_1]\|_{s + N +1} \lesssim_{s, N} \| z_\bot \|_s \| z_\bot \|_0 \| \widehat z_1\|_0 + \| z_\bot \|_0^2 \| \widehat z_1\|_s \,, 
 $$
$$ 
\| d^2  {\cal R}_{N, 3}(z; \Psi_C )[\widehat z_1, \widehat z_2] \|_{s + N +1} \lesssim_{s, N} 
 \| z_\bot \|_0 \big( \| \widehat z_1\|_s \| \widehat z_2\|_0 + \| \widehat z_1\|_0 \| \widehat z_2\|_s \big) + \| z_\bot \|_s \| \widehat z_1\|_0 \| \widehat z_2\|_0\, .
$$
If in addition, $\widehat z_1, \ldots, \widehat z_l \in  h^s_0$, $l \ge 3$,
$$
 \| d^l {\cal R}_{N, 3}(z; \Psi_C )[\widehat z_1, \ldots, \widehat z_l ]\|_{s + N +1} 
\lesssim_{s, N, l} \sum_{j = 1}^l \| \widehat z_j\|_s \prod_{i \neq j} \| \widehat z_i\|_0 + \| z_\bot \|_s \prod_{j = 1}^l \| \widehat z_j\|_0\,. 
$$
\noindent
$(ii)$ For any integer $N \ge 1$ and $\widehat z \in h^1_0$, the Taylor expansion of $d \Psi_C(z)^\top[\widehat z]$ with respect to the component
$z_\bot$ of $z = (z_S, z_\bot)$ around $0$ can be computed as
$$
d \Psi_C(z)^\top [\widehat z]= \widehat z + \Psi_{C, 1}^\top(z) [\widehat z] + \Psi_{C, 2}^\top(z)[\widehat z] , 
 $$
 where $\Psi_{C, 1}^\top(z) :=  {\cal R}_{N, 1}(z; d \Psi_C^\top)$ (cf. \eqref{Taylor expansion of of d Psi C t}),
 and for any $\widehat z \in h^0_0$,
 $$
 \Psi_{C, 2}^\top(z) [\widehat z] : = \big( 0 \, , \,  \mathcal{OP}(z; d \Psi_C^\top)[\widehat z] \big) + {\cal R}_{N, 2}(z; d \Psi_C^\top)[\widehat z]
 $$ 
 with ${\cal R}_{N, 2}(z; d \Psi_C^\top)$ given by \eqref{Taylor expansion of of d Psi C t} and
 $ \mathcal{OP}(z; d \Psi_C^\top)[\widehat z]$ by Corollary \ref{expansion of  of differential of Psi},
 $$
\begin{aligned}
&  {\cal F}^+_{N_S} \circ \sum_{k = 0}^N  a^+_k(z; d\Psi_C^\top) \cdot D^{- k} [(\mathcal F^+_{N_S})^{-1} \widehat z_\bot] +  
{\cal F}^+_{N_S} \circ \sum_{k = 0}^N  \mathcal A^+_k(z; d\Psi_C^\top)[\widehat z]  \cdot D^{- k} [( \mathcal F^+_{N_S})^{-1} z_\bot ]\\
+ \, &   {\cal F}^-_{N_S} \circ \sum_{k = 0}^N  a^-_k(z; d\Psi_C^\top) \cdot (-D)^{- k} [(\mathcal F^-_{N_S})^{-1} \widehat z_\bot] +  
{\cal F}^-_{N_S} \circ \sum_{k = 0}^N  \mathcal A^-_k(z; d\Psi_C^\top)[\widehat z] \cdot  (-D)^{- k}[ (\mathcal F^-_{N_S})^{-1} z_\bot ] .
\end{aligned}
$$
For any $i =1, 2$, $s \ge 0$, 
$$
 {\cal R}_{N, i}( \, \cdot \, ; \, d\Psi_{C}^\top) : {\cal V'} \cap h^s_0 \to {\cal B}(h^{s+1}_0, h_0^{s + 1 + N +1}) \, , \,\,
z  \mapsto {\cal R}_{N, i}( z; d\Psi_{C}^\top) \, ,
$$
is a real analytic map. Furthermore, for any $z \in {\cal V'} \cap h^s_0$, $\widehat z \in h^{s+1}_0,$
  $$
   \| {\cal R}_{N, 1}( z; d\Psi_{C}^\top) [\widehat z]\|_{s + 1 + N + 1}  \lesssim_{s, N}  \|  z_\bot\|_0 \| \widehat z\|_{s+1} + \| z_\bot \|_s \| \widehat z\|_1\, .
   $$
  If in addition, $\widehat z_1, \ldots, \widehat z_l \in h^s_0$, $l \in \N$,
   $$
   \begin{aligned}
   \| d^l \big(  {\cal R}_{N, 1} & ( z; d\Psi_{C}^\top)  [\widehat z] \big)[\widehat z_1,  \ldots,  \widehat z_l]  \|_{s + 1 + N + 1} \\
  & \lesssim_{s, N, l} \| \widehat z\|_{s+1} \prod_{j = 1}^l \| \widehat z_j\|_0 + \| \widehat z\|_1 \sum_{j = 1}^l \| \widehat z_j\|_s \prod_{i \neq j} \| \widehat z_i\|_0 + 
   \| z_\bot\|_s  \| \widehat z\|_1 \prod_{j = 1}^l \| \widehat z_j \|_0\,.
    \end{aligned}
  $$
  Similarly, for any $z \in {\cal V'} \cap h^s_0$, $\widehat z \in h^{s+1}_0,$ $\widehat z_1 \in h^s_0$,
  $$
   \| {\cal R}_{N, 2}( z; d\Psi_{C}^\top) [\widehat z]\|_{s + 1 + N + 1}  \lesssim_{s, N}  \|  z_\bot\|_0^2 \| \widehat z\|_{s+1} + \| z_\bot \|_s \| z_\bot\|_0 \| \widehat z\|_1\,, 
   $$
   and
   $$
    \|d \big( {\cal R}_{N, 2}  ( z; d\Psi_{C}^\top)  [\widehat z] \big)[\widehat z_1] \|_{s + 1 + N + 1} 
    \lesssim_{s, N} \|  z_\bot\|_0 \| \widehat z_1 \|_0 \| \widehat z\|_{s+1} + \|  z_\bot\|_0 \| \widehat z_1 \|_s \| \widehat z\|_1 + \| z_\bot \|_s \|  \widehat z_1 \|_0 \| \widehat z\|_1 \, .
    $$
   If in addition, $\widehat z_2, \ldots, \widehat z_l \in h^s_0$, $l \ge 2$, then
   $$
    \| d^l \big(  {\cal R}_{N, 2}  ( z; d \Psi_{C}^\top)  [\widehat z] \big)[\widehat z_1,  \ldots,  \widehat z_l]  \|_{s + 1 + N + 1}  \qquad \qquad \qquad \qquad \qquad \qquad \qquad
    $$
    $$
   \lesssim_{s, N, l} \| \widehat z\|_{s+1} \prod_{j = 1}^l \| \widehat z_j\|_0 + \| \widehat z\|_1 \sum_{j = 1}^l \| \widehat z_j\|_s \prod_{i \neq j} \| \widehat z_i\|_0 + 
   \| \widehat z\|_1  \| z_\bot\|_s  \prod_{j = 1}^l \| \widehat z_j \|_0\,.
  $$
\end{corollary}

\begin{proof} 
$(i)$ The claimed properties of ${\cal R}_{N, 2}(z; \Psi_C)$ follow directly from Theorem \ref{espansione flusso per correttore}. In view of the formula \eqref{Taylor remainder term}
the same is true for the ones of ${\cal R}_{N, 3}(z; \Psi_C)$. 
Item $(ii)$ is a direct consequence of Corollary \ref{expansion of  of differential of Psi}. 
\end{proof}


 \section{The BO Hamiltonian in new coordinates}\label{Hamiltoniana trasformata}
In this section we provide an expansion of the transformed BO Hamiltonian $\mathcal H = H^{bo} \circ \Psi$
where the map $\Psi = \Psi_L \circ \Psi_C$ is the composition of $\Psi_L$ (cf. Section \ref{sezione mappa Psi L BO}) with the symplectic corrector $\Psi_C$ (cf. Section \ref{sezione Psi C})
and $H^{bo}$ is the BO Hamiltonian, introduced in \eqref{BO Hamiltonian},
$$  
 H^{bo}(q) = \frac{1}{2 \pi}\int_0^{2\pi} \big(\,  \frac12 (|\partial_x|^{\frac 12} q)^2 -  \frac13q^3 \big) d x \, .
 $$
  First we need to make some preliminary considerations. Recall that
  for any finite subset $S_+ \subset \mathbb N$, the map $\Psi_S$, introduced in \eqref{definition Psi_S}, 
  establishes a one to one correspondance between $\mathcal M^o_S$ and the set $M^o_S$ of proper $S$-gap potentials where $S = S_+ \cup (- S_+)$.
  For any proper $S$-gap potential $q$, the corresponding BO actions 
  $$
  I = (I_S, I_\bot)\, , \qquad I_S= (I_j)_{j \in S_+} \, , \quad I_\bot = (I_j)_{j \in S^\bot_+} \, ,
  $$
  defined in terms of the Birkhoff coordinates $\Phi^{bo}(q)$, 
  satisfy $I_\bot = 0$ and $I_j > 0$ for any $j \in S_+$.
 Denote by  $\Omega_S(I_S)$ and $\Omega_\bot(I_S)$ the  diagonal linear operators defined by 
  \begin{equation}\label{Omega_S}
 \qquad \Omega_S(I_S) :=   {\rm diag}(( \Omega_n(I_S))_{n \in S}) :h_S \to  h_S\,, (z_n)_{n \in S} \mapsto ( \Omega_n(I_S) z_n)_{n \in S} \, , \qquad
  \end{equation}
   \begin{equation}\label{splitting Omega}
   \qquad \Omega_\bot(I_S) := 
 {\rm diag}(( \Omega_n(I_S))_{n \in S^\bot}) :h^1_\bot \to  h^0_\bot\,, (z_n)_{n \in S^\bot} \mapsto ( \Omega_n(I_S) z_n)_{n \in S^\bot} \, ,
  \end{equation}
   where for any $n \ge 1$, $\Omega_n(I_S)$ is defined by
   \begin{equation}\label{definition Omega n}
    \Omega_n(I_S) :=  \frac{1}{ n } \omega^{bo}_n((I_S, 0))\,, \qquad   
    \Omega_{-n}(I_S) :=  \Omega_{n}(I_S)
   \end{equation} 
  and $\omega_n^{bo}(I)$, $n \ge 1$, are the BO frequencies, viewed as a function of the actions (cf. \eqref{BO frequencies}),
  $$
   \omega_n^{bo}(I) = n^2 -2 \sum_{k=1}^\infty \min \{n, k \} I_k \, , \qquad \forall \, n \ge 1 \, .
  $$
  \begin{lemma}\label{espansione asintotica frequenza thomas}
  For any finite gap potential $q \in M_S$, $n \ne 1$, one has 
  \begin{equation}\label{asintotica frequenze kdv}
  \Omega_n(I_S) =  |n| - \frac{2}{|n|} \sum_{k \in S_+} \min \{|n|, k \} I_k \, .
  \end{equation}
  In particular, 
  \begin{equation}\label{asymptotics BO frequencies}
  \Omega_n(I_S) =  |n| - (2 \sum_{k \in S_+} k I_k) \, \frac{1}{|n|} \, , \qquad \forall \, |n| \ge N_S +1 \, .
  \end{equation}
  \end{lemma}
  For what follows, we need to consider the linearized Benjamin-Ono equation. 
   Let $\widehat u(t) = \partial_\e |_{\e = 0} u_\e(t)$ where $u_\e(t)$ is a one parameter family of solutions of \eqref{1.1},
   corresponding to a one parameter family of initial data $u_\e(0)$.
   Then $\widehat u(t)$ satisfies the linearized BO equation,
\begin{equation}\label{lin BO}
  \partial_t \widehat u(t) = \partial_x d  \nabla H^{bo}(u_0(t)) [\widehat u(t)]\,.
  \end{equation}
  Furthermore, $z_\e(t) = \Phi^{bo}(u_\e(t))$ solves 
  $$
  \partial_t z_\e(t) = J \nabla \mathcal H^{bo}(z_\e(t)) = J \Omega(I_\e) z_\e(t)\, , \qquad I_\e = (I_{\e, n})_{n \ge 1} = \big(\frac{1}{n} z_{\e,n}(0) \cdot z_{\e,-n}(0)\big)_{n \ge 1} \, .
  $$
  Hence 
  \begin{equation}\label{def widehat z}
  \widehat z(t) := \partial_\e |_{\e = 0} z_\e(t) = d\Phi^{bo}(u_0(t))[\widehat u(t)]
  \end{equation} 
  satisfies the linear equation
  \begin{equation}\label{eq widehat z}
    \partial_t \widehat z(t) = J  \Omega(I_0) [\widehat z(t)]  + J \partial_\e|_{\e = 0} \Omega(I_\e) [z_0(t)]\, , 
   \end{equation}  
   where for any $n \ge 1$,
  \begin{equation}\label{variation of I_n}
   \partial_\e |_{\e = 0} \Omega_n(I_\e) = \sum_{k \ge 1}\partial_{I_k} \Omega_n(I_0) \frac{1}{k} 
 \big(  \partial_\e |_{\e = 0}z_{\e,k}(0) \cdot z_{0,-k}(0)  + z_{0,k}(0) \cdot \partial_\e |_{\e = 0} z_{\e,-k}(0) \big)  \, .
  \end{equation}  
Now assume that  $t \mapsto q(t)$ is a solution of the BO equation \eqref{1.1} in $M^o_S$ with 
  $z(t) := \Phi^{bo}(q(t)) \in \mathcal V$,  $t \in \R$. Note that $z(t)$ is of the form $(z_S(t), 0)$,
and  the actions $I= (I_S, 0)$  of $q(t)$ satsify $I_S = ( \frac{1}{n} z_n(0) \cdot z_{-n}(0))_{n \in S_+}$
 since $I = (I_n)_{n \ge 1}$ are independent of $t$. 
  We need to investigate $\partial_x d  \nabla H^{bo} (q(t)) [\widehat q(t)]$  where 
  $\widehat q(t)$ solves   $\partial_t \widehat q(t) = \partial_x d  \nabla H^{bo}(q(t)) [\widehat q(t)]$ with
  $\widehat z(0) = d \Phi^{bo}(q(0))[\widehat q(0)]$ in $h^2_0$ and $\widehat z_S(0) = 0$.
  One has
  $$
 \widehat q(0) = \Psi_1(z_S(0))[\widehat z_\bot(0)]  \ (=d \Psi^{bo}(z_S(0), 0)[ 0, \widehat z_\bot(0)] )  \, , \qquad \widehat z_\bot(0) \in h^2_\bot .
 $$ 
 By \eqref{eq widehat z} - \eqref{variation of I_n}
 it then follows that $\widehat z_S(t) = 0$ for any $t \in \R$, hence $\widehat q(t) = \Psi_1(z_S(t)) [\widehat z_\bot(t)]$, and that
  $$\partial_t \widehat z_\bot (t) =   J_\bot \Omega_\bot(I_S) [ \widehat z_\bot (t) ]\, ,
  $$
  or more explicitly, for any  $n \in S^\bot$,
 \begin{equation}\label{linearized BO in Bikrhoff}
 \partial_t \widehat z_n (t)  =  \ii n \Omega_n (I_S) \widehat z_n(t) \,.
 \end{equation}
  By differentiating $\widehat q(t) = \Psi_1(z_S(t))[\widehat z_\bot(t)]$ with respect to $t$, one gets  
  \begin{align}
  \partial_t \widehat q (t)& = \Psi_1(z_S(t))[\partial_t \widehat z_\bot (t)] + d_S \big( \Psi_1(z_S(t))[\widehat z_\bot (t)] \big) [\partial_t z_S(t) ]\nonumber\\
  & = \Psi_1(z_S(t)) \big[ J_\bot \Omega_\bot(I_S) [ \widehat z_\bot (t)]  \big] +
  d_S \big( \Psi_1(z_S(t))[\widehat z_\bot(t)] \big)[ \partial_t z_S(t)]\,. 
  \label{maradona 1} 
  \end{align}
  Comparing \eqref{lin BO} and \eqref{maradona 1} 
  and using that $ \partial_t z_S(t) = J_S \Omega_S(I_S) [z_S(t)]$, one gets 
  \begin{align}
\partial_x d \nabla H^{bo} (q(t))  \big[ \Psi_1(z_S(t)) [\widehat z_\bot(t) ] \big] 
& =  \Psi_1(z_S(t)) \big[ J_\bot \Omega_\bot(I_S) [\widehat z_\bot(t) ] \big] \,  \nonumber\\
  & \quad + \, d_S \big( \Psi_1(z_S(t))[\widehat z_\bot(t) ] \big) [J_S \Omega_S(I_S) [z_S(t)]]\,.  \label{maradona 3}
  \end{align}
  Applying $\Psi_1(z_S(t))^{- 1}$ to both sides of \eqref{maradona 3}, then yields
  \begin{align}
\Psi_1(z_S(t))^{- 1}\partial_x & d \nabla H^{bo} (q(t))  \big[ \Psi_1(z_S(t)) [\widehat z_\bot(t)]  \big] 
=  J_\bot \Omega_\bot(I_S)  [\widehat z_\bot(t)]\,  \nonumber\\
  & \quad + \, \Psi_1(z_S(t))^{- 1} d_S \big( \Psi_1(z_S(t))[\widehat z(t) ] \big)[J_S \Omega_S(I_S) [z_S(t)]]\,.  \label{maradona 3-1}
  \end{align}
  Since $\Psi_1(z_S)$ is symplectic one has 
  $ \Psi_1(z_S)^\top \partial_x^{- 1} \Psi_1(z_S) = J_\bot^{- 1}$ or
   $ \Psi_1(z_S)^{- 1} \partial_x  = J_\bot \Psi_1(z_S)^\top,$
  implying that 
   \begin{align}
J_\bot \Psi_1(z_S(t))^\top d \nabla H^{bo} (q(t)) & \big[ \Psi_1(z_S(t))[\widehat z_\bot(t)]  \big] 
=  J_\bot \Omega_\bot(I_S) [\widehat z_\bot(t)]\,  \nonumber\\
  & \quad + \, \Psi_1(z_S(t))^{- 1} d_S \big( \Psi_1(z_S(t))[\widehat z_\bot(t) ] \big)[J_S \Omega_S(I_S) [z_S(t)]]   \label{maradona 3-2}
  \end{align}
and hence for any $z_S \in \mathcal V_S$, $I_S= (\frac{1}{ n} z_n z_{-n})_{n \in S_+}$, $q = \Psi^{bo}(z_S, 0)$, and $\widehat z_\bot \in h^2_\bot$,
   \begin{align}
 \Psi_1(z_S)^\top d \nabla H^{bo} (q) & \big[ \Psi_1(z_S) [\widehat z_\bot]  \big] 
=   \Omega_\bot(I_S) [\widehat z_\bot] + \,{\cal G}(z_S)[\widehat z_\bot]   \label{maradona 4}
  \end{align}
  where $ {\cal G}(z_S): h^0_\bot \to h^0_\bot$ is given by
  \begin{equation}\label{definizione cal M (wS)}
  \begin{aligned}
  {\cal G}(z_S)[\widehat z_\bot] & := J_\bot^{- 1} \Psi_1(z_S)^{- 1} d_S \big( \Psi_1(z_S)[\widehat z_\bot] \big)[ J_S \Omega_S(I_S) [z_S]]\,. 
  \end{aligned}
  \end{equation}
  In the next lemma we record an expansion for the operator ${\cal G}(z_S)$. 
      \begin{lemma}\label{stime cal M(wS)}
  For any integer $N \ge 1$, the operator ${\cal G}(z_S): h^0_\bot \to h^0_\bot$ admits an expansion of the form 
 $ \mathcal{OP}(z_S; \mathcal G) +  {\cal R}_N(z_S;  {\cal G})$, 
  $$
    \mathcal{OP}(z_S; \mathcal G) =  {\cal F}^+_{N_S} \circ \sum_{k =  0}^{N}  a^+_k( z_S;  {\cal G}) D^{- k} \circ ({\cal F}^+_{N_S})^{- 1}
     +  {\cal F}^-_{N_S} \circ \sum_{k =  0}^{N}  a^-_k( z_S;  {\cal G}) (-D)^{- k} \circ ({\cal F}^-_{N_S})^{- 1} \, ,
  $$  
  where for any $0 \le k \le N,$ $s \ge 0,$ the maps 
  $$
  \mathcal V_S \to H^s_\C, \, z_S  \mapsto a^\pm_k(z_S; {\cal G})\,, \qquad 
  \mathcal V_S \mapsto {\cal B}(h^s_\bot, h^{s + N + 1}_\bot), \, z_S  \mapsto {\cal R}_N(z_S; \mathcal G) \, ,
  $$
   are real analytic and $a^-_k( z_S;  {\cal G}) = \overline{a^+_k( z_S;  {\cal G})}$. 
  \end{lemma}
  \begin{proof}
  In view of the definition \eqref{definizione cal M (wS)} of ${\cal G}$,
  the lemma follows from  Corollary \ref{pseudodiff expansion Psi_L} 
  and  Lemma \ref{lemma composizione pseudo}.
    \end{proof}

  \smallskip
  
    After this preliminary discussion, we can now study the transformed Hamiltonian ${H}^{bo} \circ \Psi$ where $\Psi = \Psi_L \circ \Psi_C$. 
    We split the analysis into two parts. First we expand ${\cal H}^{(1)} := { H}^{bo} \circ \Psi_L$ and then we analyze ${\cal H}^{(2)} = {\cal H}^{(1)} \circ \Psi_C$. 
   
   \medskip
   
  \noindent 
 {\bf Expansion of ${\cal H}^{(1)} := H^{bo} \circ \Psi_L$}
 
 \noindent
To expand $H^{bo} \circ \Psi_L$, it is useful to write $H^{bo}(u) $ as $H^{bo}(u) =H_2^{bo} (u)+ H_3^{bo}(u)$ where
  \begin{equation}\label{forma compatta hamiltoniana d-NLS}
  H_2^{bo}(u) := \frac12 \langle  |\partial_x| u \, , \, u \rangle\,, 
  \qquad H_3^{bo}(u) := \frac{1}{2\pi}\int_0^{2\pi} - \frac 13 u^3 \, d x\,.
  \end{equation}
   The $L^2$-gradient $\nabla H^{bo}$ of $H^{bo}$ and its derivative are then given by 
 \begin{equation}\label{espressione gradiente Hamiltoniana originaria nls}
  \nabla H^{bo}(u) =  |\partial_x|  u - u^2 \,, \qquad
  d \nabla H^{bo}(u) = |\partial_x| - 2u \, .
  \end{equation} 
Let $z_S \in \mathcal V_S$ and $q = \Psi^{bo}(z_S, 0)$ be given. 
The Taylor expansion of $H^{bo}(q+v)$ around $q$ in direction 
$v = \Psi_1(z_S)[z_\bot]$ 
with $z_\bot \in \mathcal V_\bot \cap h_\bot^1$ reads
\begin{align}
H^{bo}(q+v) =  H^{bo}(q )  + \langle \nabla H^{bo}(q ) ,  v \rangle  +
\frac12 \langle d \nabla H^{bo}(q )[v]\, , \, v \rangle +  \frac{1}{2\pi} \int_0^{2\pi} - \frac 13 v^3 \, dx \, . \label{espansione taylor H4 nls}
\end{align}
Since $v = d\Psi^{bo}(z_S,0)[0, z_\bot]$ one has
$\langle \nabla H^{bo}(q ) ,  v \rangle  = \partial_\e |_{\e=0} H^{bo} (\Psi^{bo}(z_S , \e z_\bot))$.
Recall that $\mathcal H^{bo} = H^{bo} \circ \Psi^{bo}$ is a function of the actions $I=(I_n)_{n \ge 1}$ alone and that $I_n = \frac{1}{n}z_n \cdot z_{-n}$, $n \ge 1$.
It implies that  
$$
\partial_\e |_{\e=0} H^{bo} (\Psi^{bo}(z_S, \e z_\bot)) = \sum_{n \in S^\bot_+} \omega_n(I_S, 0) \partial_\e |_{\e=0} \, \e^2 I_n =0 
$$
and hence $\langle \nabla H^{bo}(q ) , v \rangle  =0$.
Since $\Psi_L(z) = q  + \Psi_1(z_S)[z_\bot]$, the Hamiltonian $ {\cal H}^{(1)}(z)  = H^{bo}(\Psi_L(z))$ can be computed as
  \begin{align}
 {\cal H}^{(1)}(z)  & = H^{bo}(q)  + \frac12 \big\langle d \nabla H^{bo} (q) \big[ \Psi_1(z_S) z_\bot  \big]\, , \, \Psi_1(z_S) [z_\bot] \big\rangle + 
\frac{1}{2\pi} \int_0^{2\pi} - \frac 13 (\Psi_1(z_S) [z_\bot] )^3 \, d x \,.  \nonumber
  \end{align}
  By formula \eqref{maradona 4}, 
  $$
   \big\langle d \nabla {H}^{bo}(q )  \big[ \Psi_1(z_S)[z_\bot] \big]\, , \, \Psi_1(z_S)[z_\bot ] \big\rangle
   = \big\langle \Omega_\bot(I_S) [z_\bot]\, , \, z_\bot \big\rangle  +
  \big\langle{\cal G}(z_S) [z_\bot]\, , \, z_\bot \big\rangle\,.
  $$
   Since $\Psi_1(z_S)^\top \circ d \nabla H^{bo} (q) \circ \Psi_1(z_S)$ and  $\Omega_\bot(I_S)$ are symmetric,
  so is the operator ${\cal G}(z_S)$. 
In summary,  
     \begin{align}
{\cal H}^{(1)}(z) = \mathcal H^{bo}_S(z) + 
\frac12 \big\langle { \Omega}_\bot(I_S)[ z_\bot] \, ,  \,  z_\bot \big\rangle \,
+ {\cal P}_2^{(1)}(z) + {\cal P}_3^{(1)}(z)    \label{forma semifinale H nls circ Psi}
  \end{align}
where for any $z = (z_S, z_\bot) \in \mathcal V \cap h^1_0,$ 
  \begin{align}\label{perturbazione composizione con Psi L}
  \mathcal H^{bo}_S(z) & :=  H^{bo}(\Psi^{bo}(z_S, 0)) \, , \qquad
   {\cal P}_2^{(1)}(z) :=  \frac12 \langle {\cal G}(z_S)[z_\bot]\, , \, z_\bot \rangle\,,  \\
 & {\cal P}_3^{(1)}(z) := \frac{1}{2\pi}\int_0^{2\pi} - \frac 13 (\Psi_1(z_S)[z_\bot] )^3\, d x\,.  \nonumber
  \label{perturbazione composizione con Psi L}
  \end{align}
  Here, the superscript $(1)$ in ${\cal P}_2^{(1)}(z)$ and ${\cal P}_3^{(1)}(z)$ refers to the Hamiltonian $\mathcal H^{(1)}$
  whereas the subscripts in these functionals refer to the fact that  ${\cal P}_2^{(1)}(z)$ is quadratic in $z_\bot$ and ${\cal P}_3^{(1)}(z)$ is at least of order three  in $z_\bot$.
  Furthermore, note that $ \mathcal H^{bo}_S(z) = \mathcal H^{bo}_S(\Pi_S z)$ where we recall that
$\Pi_S : h_S \times h^0_{ \bot } \to h_S \times h^0_{ \bot }$ denotes the projection, given by  
$(\widehat z_S, \widehat z_\bot) \mapsto (\widehat z_S, 0)$ (cf. \eqref{Pi S Pi bot}).

  Recall from \eqref{definition prarproduct} that for any $a \in H^1$, the paraproduct $T_a u$ of the function $a$  with $v \in L^2$ 
  with respect to the cut-off function $\chi$ is defined as $(T_a v) (x) = \sum_{k, n \in \Z} \chi(k, n) a_k v_n e^{\ii 2 \pi (k + n) x}$
  with $v_n$, $n \in \Z$, denoting the Fourier coefficients of $v$ and $a_k$, $k \in \Z$, the ones of $a$.
  \begin{lemma}\label{lemma stima cal P2 P3}
 For any integer $N \ge 1$, there exists an integer $\sigma_N \ge N$ (loss of regularity) so that on $\mathcal V \cap h^{\sigma_N}_0$,
  the $L^2$-gradient $\nabla{\cal P}_3^{(1)}$ of  ${\cal P}_3^{(1)}$ admits the asymptotic expansion of the form $\big( 0, \mathcal{OP}( z;  \nabla {\cal P}_3^{(1)} ) \big)  + {\cal R}_N(z; \nabla {\cal P}_3^{(1)})$,
  where $ \mathcal{OP}( z;  \nabla {\cal P}_3^{(1)} )$ is the para-differential operator
  \begin{equation}\label{epansione d nabla P3}
  \mathcal{OP}( z;  \nabla {\cal P}_3^{(1)} )=  {\cal F}^+_{N_S} \circ  \sum_{k = 0}^N   T_{a^+_k(z; \nabla {\cal P}_3^{(1)})} D^{- k}[({\cal F}^+_{N_S})^{- 1} z_\bot] 
  + {\cal F}^-_{N_S} \circ  \sum_{k = 0}^N   T_{a^-_k(z; \nabla {\cal P}_3^{(1)})} (-D)^{- k}[({\cal F}^-_{N_S})^{- 1} z_\bot] 
  \end{equation}
  and where for any $s \geq 0$, $0 \le k \le N$, the maps
   $$
   \mathcal V \cap h^{s + \sigma_N}_0 \to H^s_\C, \, z  \mapsto a^\pm_k(z; \nabla {\cal P}_3^{(1)})\,, \qquad
   \mathcal V \cap h^{s \lor \sigma_N}_0 \to h^{s + N + 1}_0, \, z \mapsto {\cal R}_N(z; \nabla {\cal P}_3^{(1)})
   $$
 are real analytic and $a^+_k(z; \nabla {\cal P}_3^{(1)}) = \overline{a^+_k(z; \nabla {\cal P}_3^{(1)})}$. 
 Furthermore, for any $z \in \mathcal V \cap h^{s + \sigma_N}_0$ with  $\| z \|_{\sigma_N} \leq 1$, 
 $ \| a^\pm_k(z; \nabla {\cal P}_3^{(1)}) \|_s \lesssim_{s, N} \| z_\bot\|_{s + \sigma_N}$.
 If in addition $\widehat z_1, \ldots , \widehat z_l \in h^{s+ \sigma_N}_0$, $l \ge 1$, then
  \begin{equation}\label{stima P3 dopo Psi L}
   \| d^l a^\pm_k(z; \nabla {\cal P}_3^{(1)}) [\widehat z_1 , \ldots, \widehat z_l] \|_s \, \lesssim_{s, k, l} \,
   \sum_{j = 1}^l \| \widehat z_j\|_{s + \sigma_N} \prod_{i \neq j} \| \widehat z_i \|_{\sigma_N} + \| z_\bot \|_{s + \sigma_N} \prod_{j = 1}^l \| \widehat z_j\|_{\sigma_N}\,.
   \end{equation}
  Similarly, for any $ z \in \mathcal V \cap h^{s \lor \sigma_N}_0$ with  $\| z \|_{\sigma_N} \leq 1$, $\widehat z \in h^{s \lor \sigma_N}_0$,
  $ \| {\cal R}_N(z; \nabla {\cal P}_3^{(1)}) \|_{s + N + 1} \lesssim_{s, N} \| z_\bot \|_{s \lor \sigma_N} \| z_\bot \|_{\sigma_N}$ and
   \begin{equation}\label{stima P3 dopo Psi L (part 2)}
  \| d {\cal R}_N(z; \nabla {\cal P}_3^{(1)}) [\widehat z] \|_{s + N + 1} \lesssim_{s, N} 
  \| z_\bot\|_{s \lor \sigma_N } \| \widehat z \|_{\sigma_N} + \| z_\bot\|_{\sigma_N} \| \widehat z \|_{s \lor \sigma_N } \, .
  \end{equation}
  If in addition $\widehat z_1, \ldots , \widehat z_l \in h^{s \lor \sigma_N}_0$, $l \ge 2$, then
  \begin{equation}\label{stima P3 dopo Psi L (part 3)}
   \| d^l {\cal R}_N(z; \nabla {\cal P}_3^{(1)})[\widehat z_1, \ldots, \widehat z_l] \|_{s + N + 1} \lesssim_{s, N, l} 
  \sum_{j = 1}^l \| \widehat z_j\|_{s \lor \sigma_N} \prod_{i \neq j} \| \widehat z_i\|_{\sigma_N} + \| z_\bot \|_{s \lor\sigma_N} \prod_{j = 1}^l \| \widehat z_j\|_{\sigma_N}\,.
  \end{equation}
   \end{lemma}
  \begin{proof}
  By a straightforward calculation, one has
  $\nabla_\bot {\cal P}_3^{(1)}(z) = - \Psi_1(z_S)^\top ( \Psi_1(z_S)[z_\bot] )^2$.
 By the Bony decomposition given in Lemma \ref{primo lemma paraprodotti}$(ii)$,
  $$
( \Psi_1(z_S)[z_\bot])^2 = 2  \, T_{\Psi_1(z_S)[z_\bot]} \Psi_1(z_S)[z_\bot] + \, {\cal R}^{(B)} (\Psi_1(z_S)[z_\bot]\,,\, \Psi_1(z_S)[z_\bot] ) \,. 
  $$
 The expansion \eqref{epansione d nabla P3} and the stated estimates follow from Corollary \ref{pseudodiff expansion Psi_L},
 Corollary \ref{lemma asintotiche Phi 1 q}, and Lemmata  \ref{primo lemma paraprodotti}, \ref{lemma composizione pseudo}, and \ref{lemma interpolation}.    
  \end{proof}
  
 \smallskip

\noindent  
{\bf Expansion of ${\cal H}^{(2)} := {\cal H}^{(1)} \circ \Psi_C$}

\noindent
To compute the expansion of ${\cal H}^{(2)}(z)= {\cal H}^{(1)} ( \Psi_C(z))$ on $\mathcal V' \cap h^1_0$, 
we study the composition of each of the terms in \eqref{forma semifinale H nls circ Psi} with the symplectic corrector $\Psi_C$ separately. 
Recall that $\Psi_C$ is defined on $\mathcal V' $ and takes values in $\mathcal V$.

\smallskip

\noindent
{\em Term $\mathcal H_S^{bo} $.} \ 
By Corollary \ref{proposizione espansione taylor correttore simplettico}, $\Psi_C(z)$ has a Taylor expansion in $z_\bot$ around $0$ of the form
$$
\Psi_C(z) = (z_S, 0) + (0, z_\bot) + \tilde \Psi_C(z), \qquad
\tilde \Psi_C(z):= {\cal R}_{N, 2}(z; \Psi_C) +  \Psi_{C, 3}(z)\,, \qquad  \Psi_{C, 3}(z) \equiv  \Psi_{C, N, 3}(z) \, ,
$$ 
where $R_{N, 2} (z; \Psi_C)$ is the term of order two, given by $R_{N, 2} (z; \Psi_C) =  \frac{1}{2} d^2_\bot {\cal R}_{N}((z_S, 0) ; \Psi_C) [z_\bot, z_\bot]$ (cf. \eqref{Taylor expansion R_N}),
and $\Psi_{C, 3}(z)$ is given by \eqref{espansione Psi C ordini taglia}
$$
  \Psi_{C, 3}(z)  =  \big( 0, \, \, \mathcal{OP}_N ( z; \Psi_{C}) \big)  + {\cal R}_{N, 3}( z; \Psi_{C}) 
$$
with ${\cal R}_{ N, 3} (z; \Psi_{C})$ denoting the Taylor remainder term \eqref{Taylor remainder term}. 
Since $ \mathcal H^{bo}_S(z) =  \mathcal H^{bo}_S(\Pi_S z)$ (cf. \eqref{perturbazione composizione con Psi L}),
the Taylor expansion of $\mathcal H^{bo}_S( \Psi_C(z)) = \mathcal H^{bo}_S(z  + \tilde \Psi_C(z))$ reads
\begin{align}\label{h nls circ PsiC}
\mathcal H_S^{bo}(\Psi_C(z)) & = \mathcal H_S^{bo}(z) + \langle \nabla_S \mathcal H_S^{bo}(z) \,  , \,  \pi_S {\cal R}_{N, 2}(z; \Psi_C) \rangle + {\cal P}^{(2a)}_3(z)\,, 
\end{align}
where ${\cal P}^{(2a)}_3(z)$ is the Taylor remainder term of order three, given by
$$
 \langle \nabla_S \mathcal H_S^{bo}(z) \, , \pi_S  \Psi_{C, 3}(z) \rangle + 
\int_0^1 (1 - \tau) \big\langle \, d_S  \nabla_S \mathcal H_S^{bo} (z + \tau \tilde \Psi_C(z) )
[ \pi_S  \tilde \Psi_C(z)] \, , \, \pi_S  \tilde \Psi_C(z) \, \big\rangle \, d \tau
$$
and $\pi_S : h_S \times h^0_{ \bot } \to h_S$ denotes the map given by $z = (z_S, z_\bot) \mapsto z_S$ (cf. \eqref{pi S}).
Since $\pi_S  \Psi_{C, 3}(z) = \pi_S {\cal R}_{ N, 3} (z; \Psi_{C})$ and $\pi_S \tilde \Psi_C(z) = \pi_S {\cal R}_{N}(z; \Psi_C) = \pi_S \big(  {\cal R}_{N, 2}(z; \Psi_C) + {\cal R}_{N, 3}(z; \Psi_C) \big)$ (cf. \eqref{Taylor expansion R_N}),
one has
\begin{align}\label{definizione h3 nls}
{\cal P}^{(2a)}_3(z)  & =   \langle \nabla_S \mathcal H_S^{bo}(z) \, , \, \pi_S  R_{N, 3}(z; \Psi_C) \rangle \nonumber\\
 & +  \int_0^1 (1 - \tau) \big\langle \, d_S  \nabla_S \mathcal H_S^{bo} (z + \tau \tilde \Psi_C(z) ) 
[\pi_S {\cal R}_{N}(z; \Psi_C)] \, ,  \, \pi_S {\cal R}_{N}(z; \Psi_C)  \, \big\rangle \, d \tau .
\end{align}

In the next lemma we show that $ \nabla {\cal P}^{(2a)}_3(z)$ is in $h^{s+N + 1}_0$ for any $z \in \mathcal V' \cap h^s_0.$ 
\begin{lemma}\label{stime tame h3 nls}
The Hamiltonian ${\cal P}^{(2a)}_3 : \mathcal V'  \to \R$ is real analytic and for any integers $s \geq 0$, $N \ge 1$,
the map $\mathcal V' \cap h^s_0 \to h^{s + N + 1}_0$, $z \mapsto \nabla {\cal P}_3^{(2a)}(z)$ is real analytic. 
Furthermore, for any $z \in \mathcal V' \cap h^s_0 $, and
$ \widehat z \in h^s_0$,
  $$
  \| \nabla {\cal P}^{(2a)}_3(z)\|_{s + N + 1} \lesssim_{s, N}  \, \| z_\bot\|_s \| z_\bot \|_0\,, \quad 
  \| d  \nabla {\cal P}^{(2a)}_3 (z) [\widehat z]\|_{s + N + 1} \lesssim_{s, N} \, \| z_\bot \|_s \| \widehat z\|_0 +  \| z_\bot \|_0 \| \widehat z\|_s  \, .
  $$
  If in addition $ \widehat z_1, \ldots, \widehat z_l \in h^s_0$, $l \geq 2$, 
  $$
  \| d^l \nabla {\cal P}^{(2a)}_3(z) [\widehat z_1, \ldots, \widehat z_l]\|_{s + N + 1} \lesssim_{s, N, l} \,
  \sum_{j = 1}^l \| \widehat z_j\|_s\prod_{i \neq j} \|\widehat z_i \|_0 + 
  \| z_\bot \|_s \prod_{j = 1}^l \|\widehat z_j \|_0\,.
  $$
\end{lemma}
\begin{proof}
We begin by analyzing the first term
$\langle \nabla_S \mathcal H_S^{bo}(z) \, | \,  \pi_S  R_{N, 3} (z; \Psi_C) \rangle$ on the right hand side of \eqref{definizione h3 nls}.
It is given by the finite sum $\sum_{n \in S} h_n(z)$  where 
$$
h_{n}(z) :=  (\nabla \mathcal H_S^{bo}(z))_{n} \cdot (R_{N, 3}(z))_{-n}  
=\partial_{z_{-n}} \mathcal H_S^{bo}(z) \cdot \langle {\cal R}_{ N, 3} (z; \Psi_{C}) \, , \,  e_n \rangle\,, 
\quad \forall n \in S \,,
$$
and $(e_n)_{n \in S}$ denotes the standard basis of $h_S$.
The derivative of $h_n$ in direction $\widehat z \in h^0_0$ then reads
$$
\begin{aligned}
\langle  \nabla h_{n}(z) \, , \, \widehat z \rangle 
& = \langle  \nabla \partial_{z_{-n}} \mathcal H_S^{bo}(z) \, , \,  \widehat z  \rangle \cdot \langle {\cal R}_{ N, 3} (z; \Psi_{C}) \, , \,  e_n \rangle
+ \partial_{z_{-n}} \mathcal H_S^{bo}(z)  \cdot \langle d  {\cal R}_{ N, 3} (z; \Psi_{C}) [\widehat z] \, , \,  e_{n} \rangle\\
 & =  \langle  \nabla \partial_{z_{-n}} \mathcal H_S^{bo}(z) \, , \,  \widehat z \rangle \cdot \langle {\cal R}_{ N, 3} (z; \Psi_{C}) \, , \,  e_n \rangle
+  \partial_{z_{-n}} \mathcal H_S^{bo}(z) \cdot \langle   (d {\cal R}_{ N, 3} (z; \Psi_{C}))^\top [e_{n}] \, , \,  \widehat z \rangle \, ,
  \end{aligned}
$$
 implying that 
 $$
 \nabla h_{n}(z) = \langle {\cal R}_{ N, 3} (z; \Psi_{C}) \, , \,  e_n \rangle  \cdot \nabla \partial_{z_{-n}} \mathcal H_S^{bo}(z) 
 + \partial_{z_{-n}} \mathcal H_S^{bo}(z) \cdot  (d {\cal R}_{ N, 3} (z; \Psi_{C}))^\top [e_{n}] \,.  
 $$
By 
Corollary \ref{proposizione espansione taylor correttore simplettico},  for any $s \ge 0$,
$\mathcal V' \cap h^s_0 \to h^{s + N + 1}_0$, $z \mapsto \nabla  h_{n}(z)$ is real analytic and satisfies the estimates 
$\| \nabla h_{n}(z)\|_{s + N + 1} \lesssim_{s, N}  \| z_\bot \|_s \| z_\bot \|_0\,.$
The estimates for the higher order derivatives of $h_n$, $n \in S,$ are obtained by differentiating the expression for $ \nabla h_{n}(z)$ 
and using the estimates of 
Corollary \ref{proposizione espansione taylor correttore simplettico}.  

In order to analyze the second term on the right hand side of \eqref{definizione h3 nls}
 it suffices to study the functions $h_{n, k}(z; \tau)$, $n, k \in S$, given by
$$
h_{n, k}(z; \tau) := \partial_{z_{-n}} \partial_{z_{-k}} \mathcal H_S^{bo}(z + \tau \tilde \Psi_C(z) ) \cdot
\langle {\cal R}_{ N} (z; \Psi_{C}), e_{n} \rangle \cdot \langle {\cal R}_{ N} (z; \Psi_{C}), e_{k} \rangle 
$$
where $0 \le \tau \le 1$. Clearly, $h_{n, k}(z; \tau)$ depends continuously on $\tau$ as do all the derivatives with respect to the variable $z$.
Since $\mathcal H_S^{bo}( z + \tau \tilde \Psi_C(z) )$ only depends on $\pi_S(z + \tau \tilde \Psi_C(z))$ 
one sees that
$$
\begin{aligned}
\langle \nabla h_{n, k}(z; \tau), \,\widehat z \rangle & = 
\langle \nabla_S \big( \, \partial_{z_{-n}} \partial_ {z_{-k}}  \mathcal H_S^{bo}(z + \tau  \tilde \Psi_C(z)) \big), 
\, \pi_S \big({\rm{ Id}} +  \tau \, d \tilde \Psi_C(z) \big) [\widehat z]  \rangle
	\cdot \langle {\mathcal R}_N (z; \Psi_C), e_n \rangle \cdot 
	\langle {\cal R}_{ N} (z; \Psi_{C}) , e_{k} \rangle \nonumber\\
& \quad + \partial_{z_{-n}} \partial_ {z_{-k}} \mathcal H_S^{bo}(z + \tau  \tilde \Psi_C(z)) 
	\cdot \langle  (d {\cal R}_{ N} (z; \Psi_{C}))^\top [e_{n}], \widehat z \rangle \cdot \langle \mathcal R_N (z; \Psi_C), e_k \rangle \nonumber\\
& \quad + \partial_{z_{-n}} \partial_ {z_{-k}}  \mathcal H_S^{bo}(z + \tau  \tilde \Psi_C(z)) 
	\cdot \langle {\cal R}_{ N} (z; \Psi_{C}), e_{n } \rangle \cdot \langle  (d {\cal R}_{ N} (z; \Psi_{C}))^\top [e_{k}], \widehat z \rangle  \, , \nonumber\\
\end{aligned}
$$
implying that 
$$
\begin{aligned}
\nabla h_{n, k}(z; \tau) & = \big({\rm{ Id}} +  \tau \ d \tilde \Psi_C(z) \big)^\top
\big[\Pi_S \nabla_S \big( \, \partial_{z_{-n}} \partial_ {z_{-k}}  \mathcal H_S^{bo}(z + \tau  \tilde \Psi_C(z)) \, \big) \big]   
	\cdot \langle {\cal R}_{ N} (z; \Psi_{C}), e_n \rangle \cdot \langle {\cal R}_{ N} (z; \Psi_{C}), e_k \rangle \nonumber\\
& \quad + \partial_{z_{-n}} \partial_ {z_{-k}} \mathcal H_S^{bo}(z + \tau  \tilde \Psi_C(z)) 
	\cdot \langle {\cal R}_{ N} (z; \Psi_{C}), e_{k} \rangle \cdot  (d {\cal R}_{ N} (z; \Psi_{C}))^\top [e_{n}]   \nonumber\\
& \quad + \partial_{z_{-n}} \partial_ {z_{-k}}  \mathcal H_S^{bo}(z + \tau  \tilde \Psi_C(z)) \cdot
	\langle {\cal R}_{ N} (z; \Psi_{C}), e_{n} \rangle \cdot (d {\cal R}_{ N} (z; \Psi_{C}))^\top [e_{k}] \, . \nonumber\\
\end{aligned}
$$
By Corollary \ref{proposizione espansione taylor correttore simplettico}, for any $s \geq 0$, 
the map $\mathcal V'  \cap h^s_0 \to h^{s + N + 1}_0$, $z \mapsto \nabla h_{n, k}(z; y)$ is real analytic and satisfies the estimate 
$\| \nabla h_{n, k}(z; y)\|_{s + N +1} \lesssim_{s, N} \| z_\bot \|_s \| z _\bot\|_0^2$. 
The estimates for the higher order derivatives are obtained by differentiating $\nabla h_{n, k}$ 
and  applying again Corollary \ref{proposizione espansione taylor correttore simplettico}.
\end{proof}

\smallskip

In a next step we analyze ${\cal H}_\Omega(\Psi_C(z))$ where for any $z \in h^1_0$,
\begin{equation}\label{def H_Omega}
{\cal H}_\Omega(z) :=  \frac{1}{2} \langle \Omega_\bot(I_S) [z_\bot] \, , \, z_\bot \rangle \, ,
\end{equation}
and according to \eqref{splitting Omega}, 
  \begin{equation}\label{splitting Omega bot (IS)}
  \Omega_\bot(I_S) =  |D_\bot|  + \Omega_\bot^{(0)}(I_S)\,,  \qquad |D_\bot| :=   {\rm diag}_{n \in S^\bot} (|n|) \, ,
  \end{equation}
  where by \eqref{splitting Omega} - \eqref{asintotica frequenze kdv},
 \begin{equation}\label{definizione Omega bot (0) (IS)}
  \Omega_\bot^{(0)}(I_S) := {\rm diag}_{n \in S^\bot} (\Omega_n(I_S) - |n|)
  = {\rm diag}_{n \in S^\bot}(- \frac{2}{|n|} \sum_{k \in S_+} \min \{|n|, k \} I_k) \,.
    \end{equation}
  
  \noindent
  {\em Term ${\cal H}_\Omega(z)$.}  \
First, for further reference, we rewrite $\Omega_\bot^{(0)}(I_S)$ in the form of an expansion as follows.
  \begin{lemma}\label{lemma Omega bot q}
  The operator $\Omega_\bot^{(0)}(I_S)$ can be written in the form
  \begin{equation}\label{identity Omega bot}
  {\cal F}^+_{N_S}  \circ \big( a^+_1(I_S; \Omega_\bot^{(0)})  \,  D^{- 1} \big) ({\cal F}^+_{N_S})^{- 1}
 + {\cal F}^-_{N_S} \circ \big(  a^-_1(I_S; \Omega_\bot^{(0)}) \,  (-D)^{- 1} \big) ({\cal F}^-_{N_S})^{- 1} 
 +  {\cal R}_N( I_S; \Omega_\bot^{(0)}) \, ,
 \end{equation}
 where 
 $$
 a^+_1(I_S; \Omega_\bot^{(0)}) := - 2 \sum_{k \in S_+} k I_k \, , \qquad
 a^-_1(I_S; \Omega_\bot^{(0)}) := a^+_1(I_S; \Omega_\bot^{(0)}) \, ,
 $$
and  ${\cal R}_N( I_S; \Omega_\bot^{(0)})$ is defined by the identity \eqref{identity Omega bot}. For any $s \ge 0$, 
 $$
 \mathbb R_{> 0}^{S_+} \to \R, \, I_S \mapsto a^+_1(I_S; \Omega_\bot^{(0)})\,, \qquad  
 \mathbb R^{S_+}_{> 0} \to {\cal B}(h^s_\bot, h^{s + N +1}_\bot), \, I_S \mapsto  {\cal R}_N( I_S; \Omega_\bot^{(0)}) \, ,
 $$
 are real analytic. 
  \end{lemma}
  To analyze ${\cal H}_\Omega(\Psi_C(z))$, we write the quadratic form 
  $2 {\cal H}_\Omega(\Psi_C(z)) =  \langle \Omega_\bot(I_S) [z_\bot], \, z_\bot \rangle$, $z_\bot \in h^1_\bot$, as a sum
  $$
   \langle \Omega_\bot(I_S) [z_\bot], \, z_\bot \rangle  =  
    \ \langle |D_\bot| [z_\bot], \, z_\bot  \rangle 
   +  \ \langle  \Omega_\bot^{(0)}(I_S) [z_\bot] \,, \, z_\bot  \rangle
  $$
  and consider $\langle |D_\bot | [z_\bot] \, , \, z_\bot  \rangle$ and $ \langle  \Omega_\bot^{(0)}(I_S) [z_\bot] \, , \, z_\bot  \rangle$ separately.
Substituting $\pi_\bot \Psi_C(z) = z_\bot + \pi_\bot \tilde \Psi_C(z)$ for $z_\bot$ 
in $ \big\langle |D_\bot| [z_\bot] \, , \, z_\bot \big\rangle$, one gets
\begin{align}
  \big\langle |D_\bot| [z_\bot +& \pi_\bot \tilde \Psi_C(z)] \, ,  \, z_\bot  + \pi_\bot \tilde \Psi_C(z) \big\rangle
    =   \big\langle |D_\bot|  [z_\bot]  \, , \, z_\bot  \big\rangle +  
 \big\langle |D_\bot|  [z_\bot] \, , \,  \pi_\bot \tilde \Psi_C(z) \big\rangle \\
& +   \big\langle |D_\bot| [ \pi_\bot \tilde \Psi_C(z)] \, , \, z_\bot  \big\rangle +  
 \big\langle |D_\bot| [ \pi_\bot \tilde \Psi_C(z)] \, ,  \, \pi_\bot \tilde \Psi_C(z) \big\rangle \, ,   \nonumber
\end{align}
where the map $\pi_\bot$ is defined in \eqref{pi S}. With a view towards the expansion of $H^{bo} \circ \Psi$, stated in Theorem \ref{modified Birkhoff map},
we treat the difference
$$
  \frac12  \langle |D_\bot| [z_\bot + \pi_\bot \tilde \Psi_C(z)] \, , \, z_\bot + \pi_\bot \tilde \Psi_C(z) \rangle  
\, - \, \frac12   \langle | D_\bot | [z_\bot] , \, z_\bot   \rangle
$$
as part of the error term ${\cal P}_3(z)$.  It needs special attention since the Hamiltonian vector fields, associated to the functionals
$$
\langle |D_\bot| [z_\bot] \, ,  \, \pi_\bot \tilde \Psi_C(z) \rangle \, ,  \qquad 
 \langle |D_\bot| [ \pi_\bot \tilde \Psi_C(z)] \, , \, z_\bot  \rangle \, ,
$$ 
could be unbounded. We write
  \begin{equation}\label{parte pericolosa Psi C}
  {\cal H}_\Omega(\Psi_C(z)) = {\cal H}_\Omega(z) + {\cal P}_3^{(2b)}(z)\,, 
  \qquad {\cal P}_3^{(2b)}(z) := {\cal H}_\Omega(\Psi_C(z)) - {\cal H}_\Omega(z)\, ,
  \end{equation}
  with ${\cal H}_\Omega(z)$ given by \eqref{def H_Omega}.
 Note that by the mean value theorem, 
  \begin{align}
  & {\cal P}_3^{(2b)}(z) = \int_0^1 {\cal P}_{\Omega}(\tau, \Psi_{X}^{0, \tau}(z)) \, d \tau \, ,   \label{caniggia 0}
  \end{align}
where for $\tau \in [0, 1]$, $z \in \mathcal V$,
  \begin{equation}\label{definition cal H Omega tau}
{\cal P}_{\Omega} (\tau, z) := \langle \nabla {\cal H}_\Omega(z) \, , \, X(\tau, z) \rangle \, .
  \end{equation}
In a first step we analyze ${\cal P}_{\Omega} (\tau, z)$.   One has 
  \begin{align}
   \langle \nabla {\cal H}_\Omega(z) \, , \, X(\tau, z) \rangle &  = 
     \langle \nabla_S {\cal H}_\Omega(z) \, ,\, \pi_S X(\tau, z) \rangle + 
   \frac{1}{2} \langle \Omega_\bot(I_S) z_\bot \, , \,  \pi_\bot {X}(\tau, z)  \rangle\,. \label{termine delicato sergey}
  \end{align}
Since ${\cal H}_\Omega = \frac{1}{2} \langle \Omega_\bot(I_S) [z_\bot] \, , \, z_\bot \rangle$
and $\Omega_\bot(I_S) =  |D_\bot| +  \Omega_\bot^{(0)}(I_S)$ one has
\begin{align}\label{formula for nabla_S H_ Omega}
 \langle \nabla_S {\cal H}_\Omega(z)\,,\, \pi_S X(\tau, z) \rangle  
 & =  \sum_{j \in S} \, \partial_{z_{-j}} \mathcal H_\Omega(z) \cdot \langle X(\tau, z), e_j \rangle \nonumber \\ 
 & =  \frac{1}{2} \sum_{j \in S} \, \langle \partial_{z_{-j}} \Omega^{(0)}_\bot(I_S) [z_\bot] \, , \, z_\bot \rangle \, \langle X(\tau, z) \, , \, e_j \rangle \,.
 \end{align}
 Concerning the term $\frac{1}{2} \langle \Omega_\bot(I_S) z_\bot, \,  \pi_\bot {X}(\tau, z)  \rangle$ in \eqref{termine delicato sergey}, recall that 
 (cf.  \eqref{definizione campo vettoriale ausiliario}, \eqref{definition L tau})
  $$
  X(\tau, z) = - {\cal L}_\tau(z)^{- 1} [ J \mathcal E(z)] \, ,  \qquad
  {\cal L}_\tau(z) = {\rm{Id}} + \tau J \mathcal L(z) \, ,
  $$
  and hence 
  \begin{equation}\label{proprieta Neumann F tau w}
  X(\tau, z)  = - J \mathcal E(z) +  \tau J \mathcal L(z) [X(\tau, z)]\,.
  \end{equation}
 Since $\mathcal E(z) = (\mathcal E_S(z), 0)$ and $J^\top = - J$,
 one gets
  \begin{align}
  \langle \Omega_\bot(I_S) z_\bot \, , \,  \pi_\bot {X}(\tau, z)  \rangle  & = 
  \big\langle \Omega_\bot(I_S) z_\bot , \, \pi_\bot \tau J \mathcal L(z) X(\tau, z) \big\rangle \nonumber \\
  &  =  -   \tau  \big\langle J_\bot \Omega_\bot(I_S) z_\bot , \, \pi_\bot  \mathcal L(z) X(\tau, z) \big\rangle \, . \label{torres 0}
  \end{align}
  By \eqref{definition L z} the component $\mathcal L_\bot^\bot(z)$ of $\mathcal L(z)$ vanishes, implying that
  $ \pi_\bot \mathcal L(z) X(\tau, z) = \mathcal L_\bot^S(z) \pi_S X(\tau, z)$.  Substituting the latter expression into \eqref{torres 0}  
  and using that by \eqref{mathcal L transpose}, $ \mathcal L_\bot^S(z)^\top  = - \mathcal L_S^\bot(z)$, one concludes that
  \begin{align}\label{cappello 10}
  \langle \Omega_\bot(I_S) z_\bot \, , \,  \pi_\bot {X}(\tau, z)  \rangle
   & =  -   \tau \langle J_\bot \Omega_\bot(I_S) z_\bot \, , \, \mathcal L_\bot^S(z) \pi_S X(\tau, z) \rangle  \nonumber \\
  & =   \tau \langle \mathcal L_S^\bot(z) J_\bot \Omega_\bot(I_S) z_\bot  \, , \, \pi_S X(\tau, z) \rangle.
   \end{align}
 Furthermore, by \eqref{L SS botS Sbot (1)},
  \begin{align}
  &  \mathcal L_S^\bot(z) [J_\bot \Omega_\bot(I_S) z_\bot ]= 
  \big( \big\langle \partial_x^{- 1} \Psi_1(z_S)[ J_\bot \Omega_\bot(I_S ) z_\bot ] \, | \, \partial_{z_n}  \Psi_1(z_S)[ z_\bot]   \big\rangle \big)_{n \in S} 
\,. \label{cappello 0}
  \end{align}
  Since $ \Omega_\bot(I_S ) =  |D_\bot|  +  \Omega_\bot^{(0)}(I_S)$ and 
  $$
   J_\bot |D_\bot| = \ii D_\bot |D_\bot | \, , \qquad  D_\bot :=  {\rm diag}_{n \in S^\bot} (n)\, ,
  $$
  we need to analyze $\Psi_1(z_S) \ii D_\bot |D_\bot|$. 
%
%
%
%
  By Remark \ref{extension of remainder 1}(i),
 $\Psi_1(z_S) \ii D_\bot |D_\bot|$ is a bounded linear operator $h^s_\bot \to H^{s-2}_0$ for any $s \ge 0$.
  \begin{lemma}\label{lemma cal T(q)}
  For any integer $N \ge 0$ and any $z_S \in \mathcal V_S$, the operator
  $$
  {\cal T}(z_S) := \Psi_1(z_S) \ii D_\bot |D_\bot|  - \partial_x  |\partial_x|  \Psi_1(z_S) 
  $$ 
  admits an expansion of the form $\mathcal{OP}(z_S;  {\cal T}) + {\cal R}_N(z_S;  {\cal T})$ where
  
  $$
  \begin{aligned}
   \mathcal{OP}(z_S;  {\cal T})  = \sum_{k = - 1}^N a^+_k( z_S; {\cal T}) D^{- k} ({\cal F}^+_{N_S})^{- 1} + 
   \sum_{k = - 1}^N a^-_k( z_S; {\cal T}) (-D)^{- k} ({\cal F}^-_{N_S})^{- 1}
  \end{aligned}
  $$
  and for any $s \geq 0$, $-1 \le k \le N$, the maps 
  $$
  \mathcal V_S \to H^s_\C, \, z_S \mapsto a^{\pm}_k( z_S; {\cal T})\,, \qquad 
  \mathcal V_S \to {\cal B}(h^s_\bot, H^{s + N + 1}), \, z_S \mapsto {\cal R}_N( z_S;  {\cal T}) \, ,
  $$ 
  are real analytic and $a^-_k( z_S; {\cal T}) = \overline{a^+_k( z_S; {\cal T})}$.
  A similar statement holds for the transpose ${\cal T}(z_S)^\top$ of ${\cal T}(z_S)$. 
  \end{lemma}
  \begin{proof} 
  Since by Corollary \ref{pseudodiff expansion Psi_L}), $\Psi_1(z_S)$ has an expansion of order zero and $D_\bot |D_\bot| $ is a diagonal
  operator of order two, the operator $ {\cal T}(z_S)$, being of commutator type,
  has an expansion of order one.
  The claimed statements then follow from Corollary \ref{pseudodiff expansion Psi_L} (expansion of $\Psi_1(z_S)$)
  and Corollary \ref{lemma asintotiche Phi 1 q} (expansion of $\Psi_1(z_S)^\top$).
  \end{proof}
  Using that  $J_\bot \Omega_\bot(I_S) =  \ii D_\bot |D_\bot| + J_\bot \Omega_\bot^{(0)}(I_S)$,
  the operator $\partial_x^{- 1}  \Psi_1(z_S)  J_\bot   \Omega_\bot(I_S)$, 
  appearing in formula \eqref{cappello 0}, reads
  \begin{equation}\label{auxilary identity}
  \partial_x^{- 1}  \Psi_1(z_S) J_\bot \Omega_\bot(I_S) =
    \partial_x^{-1}\Psi_1(z_S)  \ii D_\bot |D_\bot| +  \partial_x^{- 1}  \Psi_1(z_S) J_\bot \Omega^{(0)}_\bot(I_S).
  \end{equation}
  By the definition of ${\cal T}(z_S)$ 
  one then gets
  $$
 \partial_x^{- 1}  \Psi_1(z_S) \ii D_\bot |D_\bot| 
 =   \ii  |\partial_x| \Psi_1(z_S) +  \partial_x^{- 1} \mathcal T(z_S) \, .
  $$
  By  \eqref{cappello 0} it then follows that for any $n \in S$,  
  \begin{align}
  &  \big( \mathcal L_S^\bot(z) [J_\bot \Omega_\bot(I_S) z_\bot ] \big)_n
  =  \ii \langle  | \partial_x |  \Psi_1(z_S)[ z_\bot ] \, | \, \partial_{z_n}  \Psi_1(z_S)[ z_\bot]   \rangle + \langle {\cal T}_{1, n}(z_S)[ z_\bot ] \, | \,  z_\bot  \rangle \, , \label{cappello 20}
  \end{align}
  where for any $z_S \in \mathcal V_S$ and $n \in S$, the operator ${\cal T}_{1, n}(z_S)$ is given by (cf. \eqref{auxilary identity})
  \begin{equation}\label{definizione cal T1 (q)}
  {\cal T}_{1, n}(z_S) := (\partial_{z_n} \Psi_1(z_S))^\top \partial_x^{- 1} {\cal T}(z_S) +  (\partial_{z_n} \Psi_1(z_S))^\top \partial_x^{- 1} \Psi_1(z_S) J_\bot \Omega^{(0)}_\bot(I_S )\,. 
  \end{equation}
  Since $(\partial_{z_n} \Psi_1(z_S))^\top$ (Corollary  \ref{lemma asintotiche Phi 1 q}) and 
  $\partial_x^{- 1} {\cal T}(z_S)$ (Lemma \ref{lemma cal T(q)}) have both an expansion of order zero, one infers that
  $  {\cal T}_{1, n}(z_S)$ has also such an expansion. More precisely, the following result holds. 
  \begin{lemma}\label{proprieta cal T 1 j (q)}
  For any $n \in S$ and $N \in \N$, the operator ${\cal T}_{1, n}(z_S)$, defined by \eqref{definizione cal T1 (q)} for $z_S \in \mathcal V_S$, admits an expansion 
  of the form $\mathcal{OP}(z_S; {\cal T}_{1, n})  + {\cal R}_N( z_S;   {\cal T}_{1, n})$ where
  $$
  \mathcal{OP}(z_S; {\cal T}_{1, n})   =  
  {\cal F}^+_{N_S} \circ \sum_{k = 0}^N  a^+_k(z_S;  {\cal T}_{1, n}) D^{- k} ({\cal F}^+_{N_S})^{- 1} 
  +  {\cal F}^-_{N_S} \circ \sum_{k = 0}^N  a^-_k(z_S;  {\cal T}_{1, n}) (-D)^{- k} ({\cal F}^-_{N_S})^{- 1} 
  $$
  and for any $s \geq 0$, $0 \le k \le N$, the maps 
  $$
  \mathcal V_S \to H^s_\C, \, z_S \mapsto a^\pm_k( z_S;  {\cal T}_{1, n})\,, \qquad 
  \mathcal V_S \to {\cal B}(h^s_\bot , h^{s + N +1}_\bot), \, z_S \mapsto {\cal R}_N( z_S;  {\cal T}_{1, n}) \, ,
  $$ 
  are real analytic and $a^-_k( z_S; {\cal T}_{1, n}) = \overline{a^+_k( z_S;  {\cal T}_{1, n})}$.
  \end{lemma}
  \begin{proof}
  The claimed statements follow from Corollary  \ref{lemma asintotiche Phi 1 q}  (expansion of $\Psi_1(z_S)^\top$), 
  Lemma \ref{lemma cal T(q)} (expansion of ${\cal T}(z_S)$),  
  Corollary \ref{pseudodiff expansion Psi_L}  (expansion of $\Psi_1(z_S)$)
  Lemma \ref{lemma Omega bot q} (expansion of  $\Omega_\bot^{(0)}(I_S)$),  
  and Lemma \ref{lemma composizione pseudo}. 
  \end{proof}
We now turn our attention to the term $ \ii \langle  | \partial_x | \Psi_1(z_S)[ z_\bot ] \, | \, \partial_{z_n}  \Psi_1(z_S)[ z_\bot]   \rangle$
 in \eqref{cappello 20}. 
By \eqref{forma compatta hamiltoniana d-NLS}
$$
d \nabla H^{bo}(q) =  |\partial_x| + d \nabla H_3^{bo}(q) =  |\partial_x| -  2 q \, , 
\qquad H_3^{bo}(q) := \frac{1}{2\pi} \int_0^{2\pi} - \frac 13 q^3 dx \, ,
$$ 
and hence $|\partial_x| = d \nabla H^{bo}(q) + 2q$.
Since $\Psi_1(z_S)[ z_\bot ]$ is real valued, one then infers that for any $n \in S$
\begin{align}
  \langle & |\partial_x|  \Psi_1(z_S)[ z_\bot ] \, | \, \partial_{z_n}  \Psi_1(z_S)[ z_\bot]   \rangle 
 =   \partial_{z_{-n}} \frac 12 \langle  |\partial_x|  \Psi_1(z_S)[ z_\bot ] \, | \,  \Psi_1(z_S)[ z_\bot]   \rangle \nonumber\\
& =  \frac{1}{2} \partial_{z_{-n}} \langle d \nabla H^{bo}(q) \big[ \Psi_1(z_S)[ z_\bot ]\big] \, | \,  \Psi_1(z_S)[ z_\bot]   \rangle 
+ \frac 12 \partial_{z_{-n}} \langle 2 q  \Psi_1(z_S)[ z_\bot ] \, | \,  \Psi_1(z_S)[ z_\bot]   \rangle \, . \nonumber\\
\end{align}
Since by \eqref{maradona 4}, 
$$  
\Psi_1(z_S)^\top d \nabla H^{bo}(q)  \Psi_1(z_S) =   \Omega_\bot(I_S) + \,{\cal G}(z_S)
$$ 
we conclude that $\ii \langle  |\partial_x|  \Psi_1(z_S)[ z_\bot ] \, | \, \partial_{z_n}  \Psi_1(z_S)[ z_\bot]   \rangle$ equals
\begin{align}
  \frac{\ii}{2}  \langle \partial_{z_{-n}} \big( \Omega_\bot(I_S) + \partial_{z_{-n}}{\cal G}(z_S) \big) [z_\bot]  \, ,\,  z_\bot  \rangle  
+    \frac{\ii}{2}  \langle  \partial_{z_{-n}} \big( 2 \Psi_1(z_S)^\top \, q \, \Psi_1(z_S) \big) [ z_\bot ] \, , \,   z_\bot \rangle\,.
\end{align}
Using again that for any $n \in S$,
$\partial_{z_{-n}} \Omega_\bot(I_S) =  \partial_{z_{-n}} \Omega_\bot^{(0)}(I_S)$ (cf. \eqref{splitting Omega}),
one thus obtains 
$$
 \ii \langle  |\partial_x|  \Psi_1(z_S)[ z_\bot ] \, | \, \partial_{z_n}  \Psi_1(z_S)[ z_\bot]   \rangle =   \langle {\cal T}_{2, n}(z_S)[z_\bot]\,, z_\bot \rangle
$$
where
\begin{equation}\label{definizione cal T2}
 {\cal T}_{2, n}(z_S) :=   \frac{\ii}{2} \partial_{z_{-n}} \big( \,  \Omega_\bot^{(0)}(I_S) + {\cal G}(z_S) + 2 \Psi_1(z_S)^\top q  \Psi_1(z_S) \,  \big)\,. 
\end{equation}
\begin{lemma}\label{lemma cal T2 j}
For any $n \in S$ and any integer $N \ge 0$, the operator ${\cal T}_{2, n}(z_S) : h^0_\bot \to h^0_\bot$, defined by \eqref{definizione cal T2} for $z_S \in \mathcal V_S$, 
admits an expansion of the form $\mathcal{OP}( z_S; {\cal T}_{2, n}) + {\cal R}_N( z_S; {\cal T}_{2, n})$ where
$$
\mathcal{OP}( z_S; {\cal T}_{2, n})
 = {\cal F}^+_{N_S} \circ  \sum_{k = 0}^{N} a^+_k( z_S; {\cal T}_{2, n}) D^{- k}  ({\cal F}^+_{N_S})^{- 1} + 
  {\cal F}^-_{N_S} \circ  \sum_{k = 0}^{N} a^-_k( z_S; {\cal T}_{2, n}) (-D)^{- k}  ({\cal F}^-_{N_S})^{- 1} 
$$
and where for any $s \geq 0$, $0 \le k \le N,$  the maps 
$$
\mathcal V_S \to H^s_\C, \, z_S \mapsto a^\pm_k( z_S;  {\cal T}_{2, n})\,,  \qquad 
\mathcal V_S \to {\cal B}(h^s_\bot , h^{s + N + 1}_\bot), \, z_S \mapsto  {\cal R}_N ( z_S; {\cal T}_{2, n}) \,,
$$ 
are real analytic and $a^-_k( z_S; {\cal T}_{2, n}) = \overline{a^+_k( z_S;  {\cal T}_{2, n})}$.
A similar statement holds for the transpose ${\cal T}_{2, n}(z_S)^\top$ of the operator ${\cal T}_{2, n}(z_S)$. 
\end{lemma}
\begin{proof}
The claimed results follow from Lemmata \ref{lemma asintotiche Phi 1 q}, \ref{stime cal M(wS)}, \ref{lemma Omega bot q}, and Lemma \ref{lemma composizione pseudo}. 
\end{proof}

\smallskip

By \eqref{definition cal H Omega tau} - \eqref{termine delicato sergey}, 
 \eqref{cappello 10} -- 
 \eqref{cappello 20}, 
 and \eqref{definizione cal T2} 
 the Hamiltonian ${\cal P}_\Omega(\tau, z)$, defined in \eqref{definition cal H Omega tau}, is given by
  $$
{\cal P}_\Omega(\tau, z)  =   \langle \nabla_S {\cal H}_\Omega(z)\, , \, \pi_S X(\tau, z) \rangle 
+ \frac{\tau}{2} \sum_{j \in S}  \big\langle \big(  {\cal T}_{1 , j}(z_S) +   {\cal T}_{2, j}(z_S) \big) [z_\bot]\, , \, z_\bot \big\rangle \cdot
\langle X(\tau, z)\,,\, e_j \rangle  \, .
$$
Hence by \eqref{formula for nabla_S H_ Omega},
  \begin{equation}\label{forma finale cal P Omega tau}
{\cal P}_\Omega(\tau, z)
 = \frac12 \sum_{j \in S}  \langle {\cal T}_{3, j}(\tau, z_S)[z_\bot]\, , \, z_\bot \rangle  \cdot  \langle X(\tau, z)\, , \, e_j \rangle
  \end{equation}
where for any $j \in S$, $z_S \in \mathcal V_S$, and $0 \le \tau \le 1,$ the operator ${\cal T}_{3, j}(\tau, z_S): h^0_\bot \to h^0_\bot$ is defined by
\begin{equation}\label{definition cal T 3 j}
{\cal T}_{3, j}(\tau, z_S) :=  \partial_{z_{-j}} \Omega_\bot^{(0)}(I_S)  +  \tau {\cal T}_{1 , j}(z_S) + \tau {\cal T}_{2, j}(z_S)\,. 
\end{equation}
The Hamiltonian ${\cal P}_\Omega(\tau, z)$ has the following properties. 
\begin{lemma}\label{proprieta hamiltoniana cal P Omega tau psi}
For any $0 \le \tau \le 1$ and any integer $N \ge 0$, the Hamiltonian ${\cal P}_\Omega(\tau, \cdot ) : \mathcal V \to \R$ is real analytic and $\nabla {\cal P}_\Omega (\tau, z)$ 
admits the expansion of the form $ \big( 0, \,  \mathcal{OP}(\tau, z;  \nabla {\cal P}_\Omega) \big) + {\cal R}_N(\tau, z;  \nabla {\cal P}_\Omega)$ where
$$
 \mathcal{OP}(\tau, z;  \nabla {\cal P}_\Omega)  =  {\cal F}^+_{N_S} \circ \sum_{k = 0}^N a^+_k( \tau, z; \nabla {\cal P}_\Omega) D^{- k} [ ({\cal F}^+_{N_S})^{- 1} z_\bot]
 + {\cal F}^-_{N_S} \circ \sum_{k = 0}^N a^-_k( \tau, z; \nabla {\cal P}_\Omega) (-D)^{- k} [ ({\cal F}^-_{N_S})^{- 1} z_\bot]
$$
and where for any $s \geq 0$, $0 \le k \le N,$ the maps 
$$
\mathcal V \to H^s_\C, \, z \mapsto a^\pm_k(\tau, z; \nabla {\cal P}_\Omega)\,, \qquad 
\mathcal V \cap h^s_0  \to  h^{s + N + 1}_0, \, z \mapsto {\cal R}_N(\tau, z; \nabla {\cal P}_\Omega) \, ,
$$ 
are real analytic and $a^-_k( \tau, z; \nabla {\cal P}_\Omega) = \overline{a^+_k( \tau, z;  \nabla {\cal P}_\Omega)}$.
Furthermore, for any $0 \le \tau \le 1$, $z \in \mathcal V$, $\widehat z \in h^0_0$,
$$
 \|  a^\pm_k(\tau, z; \nabla {\cal P}_\Omega) \|_s \lesssim_{s} \| z_\bot \|_0^2 \, ,
 \qquad  \| d a^\pm_k (\tau, z;  \nabla {\cal P}_\Omega)[\widehat z] \|_s \lesssim_s  \| z_\bot \|_0 \| \widehat z\|_0 \, .
 $$
If in addition, $\widehat z_1, \ldots, \widehat z_l \in h^0_0$, $l \ge 2$, then
$$
 \| d^l  a^\pm_k(\tau, z;  \nabla {\cal P}_\Omega)[\widehat z_1, \ldots, \widehat z_l] \|_s  \lesssim_{s, l} \prod_{j = 1}^l \| \widehat z_j\|_0 \,.
 $$
 Similarly, for any $0 \le \tau \le 1$, $z \in \mathcal V \cap h^s_0$, $\widehat z_1, \widehat z_2 \in h^s_0,$
$$
 \| {\cal R}_N (\tau, z; \nabla {\cal P}_\Omega) \|_{s + N + 1} \lesssim_{s, N} \| z_\bot \|_s \| z_\bot \|_0^2 \,,  \qquad \qquad \qquad \qquad \qquad \qquad  \quad \ \
 $$
 $$
 \| d {\cal R}_N(\tau, z;  \nabla {\cal P}_\Omega)[\widehat z_1] \|_{s + N + 1} \lesssim_{s, N}  \| z_\bot \|_0^2 \| \widehat z_1\|_s
 +  \| z_\bot \|_s \| z_\bot \|_0 \| \widehat z_1\|_0\,,   \qquad \qquad \qquad \ \ \
 $$
 $$
  \| d^2 {\cal R}_N(\tau, z;  \nabla {\cal P}_\Omega)[\widehat z_1, \widehat z_2] \|_{s + N + 1} \lesssim_{s, N}  
 \| z_\bot \|_0 \big( \| \widehat z_1\|_s \| \widehat z_2\|_0 +  \| \widehat z_1\|_0 \| \widehat z_2\|_s\big) + \|z_\bot  \|_s \| \widehat z_1\|_0 \| \widehat z_2 \|_0\,,  
 $$
 and if in addition $\widehat z_1, \ldots, \widehat z_l \in h^s_0$, $l \ge 3$, then
 $$
  \| d^l {\cal R}_N(\tau, z;  \nabla {\cal P}_\Omega) [\widehat z_1, \ldots, \widehat z_l]\|_{s+N+1}  \lesssim_{s, N, l} 
 \sum_{j = 1}^l \| \widehat z_j\|_s \prod_{i \neq j} \| \widehat z_i\|_0 + \| z_\bot \|_s \prod_{j = 1}^l \| \widehat z_j\|_0\,. 
$$
\end{lemma}
\begin{proof}
One has $\nabla_S {\cal P}_\Omega(\tau, z) = (\partial_{z_{-n}} {\cal P}_\Omega(\tau, z))_{n \in S}$ with
$$
\begin{aligned}
 \partial_{z_{-n}} {\cal P}_\Omega (\tau, z) & = \frac12 \sum_{j \in S} \,  \langle \partial_{z_{-n}} {\cal T}_{3, j}(\tau, z_S)[z_\bot]\,,\,  z_\bot \rangle  \cdot  \langle X(\tau, z)\,,\, e_j \rangle  \\
& + \frac12 \sum_{j \in S} \,  \langle {\cal T}_{3, j}(\tau, z_S)[z_\bot]\,,\, z_\bot \rangle  \cdot \langle \partial_{z_{-n}}X(\tau, z)\,,\, e_j \rangle,
\end{aligned}
$$
whereas $\nabla_\bot {\cal P}_\Omega(\tau, z)$ can be computed to be
$$
\nabla_\bot {\cal P}_\Omega(\tau, z) =  \sum_{j \in S} \,  \langle X(\tau, z)\,,\, e_j \rangle \,   {\cal T}_{3, j}(\tau, z_S)[z_\bot]    
+ \frac12 \sum_{j \in S}  \, \langle {\cal T}_{3, j}(\tau, z_S)[z_\bot]\,,\, z_\bot \rangle  \, (d_\bot X(\tau, z))^\top [e_j]\,. 
$$
The claimed statements then follow by Lemmata \ref{lemma campo vettoriale new}, \ref{lemma Omega bot q}, \ref{proprieta cal T 1 j (q)}, \ref{lemma cal T2 j}.
\end{proof}
We are now ready to analyze the gradient of the Hamiltonian ${\cal P}_3^{(2b)}(z) := \int_0^1 {\cal P}_{\Omega}(\tau, \Psi_{X}^{0, \tau}(z)) \, d \tau$ (cf.  \eqref{caniggia 0}). 
\begin{lemma}\label{proprieta hamiltoniana cal P 3 (2b)}
The Hamiltonian ${\cal P}_3^{(2b)} : \mathcal V' \to \R$ is real analytic and for any integer $N \ge 0$, its gradient $\nabla {\cal P}_3^{(2b)} (z)$ admits the expansion
of the form $ \big( 0, \,  \mathcal{OP}( z; \nabla {\cal P}_3^{(2b)}) \big)+  {\cal R}_N( z; \nabla {\cal P}_3^{(2b)})$ where
$$
 \mathcal{OP}( z; \nabla {\cal P}_3^{(2b)}) = {\cal F}^+_{N_S} \circ \sum_{k = 0}^N  a^+_k(z; \nabla {\cal P}_3^{(2b)}) \, D^{- k} [ ({\cal F}^+_{N_S})^{- 1} z_\bot]  
 + {\cal F}^-_{N_S} \circ \sum_{k = 0}^N  a^-_k(z; \nabla {\cal P}_3^{(2b)}) \, (-D)^{- k} [ ({\cal F}^-_{N_S})^{- 1} z_\bot] 
$$
and where for any $s \geq 0$, $0 \le k \le N$, the maps 
$$
\mathcal V' \to H^s_\C, \, z \mapsto a^\pm_k( z; \nabla {\cal P}_3^{(2b)})\,, \qquad  
\mathcal V' \cap h^s_0 \to  h^{s + N + 1}_0, \, z \mapsto {\cal R}_N( z;  \nabla {\cal P}_3^{(2b)}) \, ,
$$ 
are real analytic and $a^-_k( z; \nabla {\cal P}_3^{(2b)}) = \overline{a^+_k( z;  \nabla {\cal P}_3^{(2b)})}$.
Furthermore, the following estimates hold: for any $z \in \mathcal V'$, $\widehat z \in h^0_0$,
$$
 \|  a^\pm_k( z;  \nabla {\cal P}_3^{(2b)}) \|_s \lesssim_{s} \| z_\bot \|_0^2 \,, \qquad \| d a^\pm_k( z; \nabla {\cal P}_3^{(2b)})[\widehat z] \|_s 
 \lesssim_s  \| z_\bot \|_0 \| \widehat z\|_0\,, 
 $$
 and if in addition $\widehat z_1, \ldots, \widehat z_l \in h^0_0,$ $l \ge 2$, then
 $\| d^l  a^\pm_k(z; \nabla {\cal P}_3^{(2b)})[\widehat z_1, \ldots, \widehat z_l] \|_s  \lesssim_{s, l} \prod_{j = 1}^l \| \widehat z_j\|_0$.
 
 \noindent
 Similarly, for any $z \in \mathcal V' \cap h^s_0$, $\widehat z_1, \widehat z_2 \in h^s_0$, one has
 $$ 
 \| {\cal R}_N( z;  \nabla {\cal P}_3^{(2b)}) \|_{s + N + 1} \lesssim_{s, N} \| z_\bot \|_s \| z_\bot \|_0^2 \, ,  \qquad \qquad \qquad  \qquad \qquad \qquad \qquad
 $$
 $$
 \| d {\cal R}_N( z;  \nabla {\cal P}_3^{(2b)})[\widehat z_1] \|_{s + N + 1} \lesssim_{s, N}  \| z_\bot \|_0^2 \| \widehat z_1\|_s  +  \| z_\bot \|_s \| z_\bot \|_0 \| \widehat z_1\|_0\,, 
 \qquad \qquad \qquad  \quad
 $$
 $$
  \| d^2 {\cal R}_N( z; \nabla {\cal P}_3^{(2b)})[\widehat z_1, \widehat z_2] \|_{s + N + 1} \lesssim_{s, N}  
 \| z_\bot \|_0 \big( \| \widehat z_1\|_s \| \widehat z_2\|_0 +  \| \widehat z_1\|_0 \| \widehat z_2\|_s\big) + \|z_\bot  \|_s \| \widehat z_1\|_0 \| \widehat z_2 \|_0\,, 
 $$
 and if in addition $\widehat z_1, \ldots, \widehat z_l \in h^s_0,$ $l \ge 2$, then
 $$
  \| d^l   {\cal R}_N( z; \nabla {\cal P}_3^{(2b)}) [\widehat z_1, \ldots, \widehat z_l]\|_{s + N + 1}  \lesssim_{s, N, l} 
 \sum_{j = 1}^l \| \widehat z_j\|_s \prod_{i \neq j} \| \widehat z_i\|_0 + \| z_\bot \|_s \prod_{j = 1}^l \| \widehat z_j\|_0 \,. 
$$
\end{lemma}
\begin{proof}
By a straightforward computation, one has for any $z \in \mathcal V'$,
$$
\nabla {\cal P}_3^{(2b)}(z) = \int_0^1 (d \Psi_X^{0, \tau}(z))^\top \nabla {\cal P}_{\Omega}(\tau, \Psi_{X}^{0, \tau}(z)) \, d \tau\,.
$$
The claimed statements then follow by applying Corollary \ref{expansion of  of differential of Psi} (expansion of $d \Psi_X^{0, \tau}(z)^\top$), 
Lemma \ref{proprieta hamiltoniana cal P Omega tau psi} (expansion of $\nabla {\cal P}_\Omega(\tau, z)$), 
Theorem \ref{espansione flusso per correttore} (expansion of $\Psi_{X}^{0, \tau}(z)$), and Lemma \ref{lemma composizione pseudo}.  
\end{proof}
 
 \smallskip
  
  \noindent
  {\em Terms ${\cal P}_2^{(1)}$ and ${\cal P}_3^{(1)}$.} \ Recall that the Hamiltonians ${\cal P}_2^{(1)}$ and ${\cal P}_3^{(1)}$ 
  were introduced in \eqref{perturbazione composizione con Psi L}. We write
  \begin{align}\label{cal P 12 Psi C}
&  {\cal P}_2^{(1)}(\Psi_C(z)) + {\cal P}_3^{(1)}(\Psi_C(z)) = {\cal P}_2^{(1)}(z) + {\cal P}_3^{(2c)}(z) \,,  \nonumber \\
&  {\cal P}_3^{(2c)}(z):=   {\cal P}_2^{(1)}(\Psi_C(z)) - {\cal P}_2^{(1)}(z) + {\cal P}_3^{(1)}(\Psi_C(z)) \, ,
  \end{align}
  where by the mean value theorem
  $$
 {\cal P}_2^{(1)}(\Psi_C(z)) - {\cal P}_2^{(1)}(z) = \int_0^1\big\langle \nabla {\cal P}_2^{(1)}\big( z + y (\Psi_C(z) - z)   \big)\,,\, \Psi_C(z) - z \big\rangle d y \,. 
  $$
  The Hamiltonian ${\cal P}_3^{(2c)}(z)$ has the following properties.
  \begin{lemma}\label{composizione cal P 12 Psi C}
  The Hamiltonian ${\cal P}_3^{(2c)} : \mathcal V' \cap h^1_0 \to \R$ is real analytic and for any integer $N \ge 0$ its gradient $\nabla {\cal P}_3^{(2c)} (z)$ admits the expansion
  of the form $ \big( 0, \, \mathcal{OP}(z;  \nabla {\cal P}_3^{(2c)}) \big) + {\cal R}_N (z;  \nabla {\cal P}_3^{(2c)})$ where
   $$
  \mathcal{OP}(z;  \nabla {\cal P}_3^{(2c)})  =  {\cal F}^+_{N_S} \circ  \sum_{k = 0}^N T_{a^+_k(z; \nabla {\cal P}_3^{(2c)})}  D^{- k} [ ( {\cal F}^+_{N_S})^{- 1} z_\bot] 
  +  {\cal F}^-_{N_S} \circ  \sum_{k = 0}^N T_{a^-_k(z; \nabla {\cal P}_3^{(2c)})} (- D)^{- k} [ ( {\cal F}^-_{N_S})^{- 1} z_\bot] 
  $$ 
  with the property that there exists an integer $\sigma_N \ge N$ (loss of regularity) such that for any $s \geq 0$, $0 \le k \le N$, the maps 
  $$
  \mathcal V' \cap h^{s + \sigma_N}  \to H^s_\C, \, z \mapsto a^\pm_k(z; \nabla {\cal P}_3^{(2c)})\,,
  \qquad 
  \mathcal V'  \cap h_0^{s \lor \sigma_N} \to  h_0^{s + N + 1}, \, z \mapsto {\cal R}_N(z; \nabla {\cal P}_3^{(2c)})
  $$
   are real analytic and $a^-_k( z; \nabla {\cal P}_3^{(2c)}) = \overline{a^+_k( z;  \nabla {\cal P}_3^{(2c)})}$.
   Furthermore, for any $s \ge 0,$ $z \in \mathcal V' \cap h^{s + \sigma_N}_0$ with  $\| z \|_{\sigma_N} \leq 1$, $\widehat z_1, \ldots, \widehat z_l \in h^{s + \sigma_N}_0$, $l \ge 1$,
  $$
  \begin{aligned}
   & \|   a^\pm_k( z; \nabla {\cal P}_3^{(2c)})  \|_s \lesssim_{s, N} \| z_\bot \|_{s + \sigma_N}\,, \\
   & \| d^l a^\pm_k( z; \nabla {\cal P}_3^{(2c)}) [\widehat z_1, \ldots, \widehat z_l] \|_s \lesssim_{s, N, l} 
  \sum_{j = 1}^l \| \widehat z_j\|_{s + \sigma_N} \prod_{i \neq j} \| \widehat z_j\|_{\sigma_N} + \| z_\bot \|_{s + \sigma_N} \prod_{j = 1}^l \| \widehat z_j\|_{\sigma_N}\,.
    \end{aligned}
  $$
  Similarly, for any $s \ge 0,$ $z \in \mathcal V' \cap h^{s \lor \sigma_N}_0$ with  $\| z \|_{\sigma_N} \leq 1$, $\widehat z\in h^{s \lor \sigma_N}_0$, 
  $$
  \begin{aligned}
   &\| {\cal R}_N(z ; \nabla {\cal P}_3^{(2c)})\|_{s + N + 1} \lesssim_{s, N} \| z_\bot \|_{s \lor \sigma_N} \| z_\bot \|_{\sigma_N} \,, \\
   &\| d {\cal R}_N( z; \nabla {\cal P}_3^{(2c)}) [\widehat z]\|_{s + N + 1} \lesssim_{s, N} \| z_\bot \|_{\sigma_N} \| \widehat z\|_{s \lor \sigma_N} + 
   \| z_\bot \|_{s \lor \sigma_N} \| \widehat z\|_{\sigma_N}\,, 
   \end{aligned}
   $$
   and if in addition $\widehat z_1, \ldots, \widehat z_l \in h^{s \lor \sigma_N}_0$, $ l \geq 2$, then
   $$
   \| d^l {\cal R}_N( z; \nabla {\cal P}_3^{(2c)}) [\widehat z_1, \ldots, \widehat z_l] \|_{s + N + 1} \lesssim_{s, N, l} 
  \sum_{j = 1}^l \|\widehat z_j \|_{s \lor \sigma_N} \prod_{i \neq j} \| \widehat z_i\|_{\sigma_N} 
  + \| z_\bot \|_{s \lor \sigma_N} \prod_{j = 1}^l \| \widehat z_j\|_{\sigma_N}\,. 
  $$
   \end{lemma}
  \begin{proof}
   The lemma follows by differentiating  the Hamiltonian ${\cal P}_3^{(2c)}$, defined  in \eqref{cal P 12 Psi C} and then  applying Corollary \ref{proposizione espansione taylor correttore simplettico}, 
   Lemmata \ref{stime cal M(wS)}, \ref{lemma stima cal P2 P3} and using Lemmata \ref{primo lemma paraprodotti}, \ref{lemma composizione pseudo}.   
  \end{proof}
  
  \bigskip
  


\smallskip
Altogether we have found the following expansion of ${\cal H}^{(2)} (z) = {\cal H}^{(1)} (\Psi_C(z)) $: 
By \eqref{forma semifinale H nls circ Psi} (expansion of ${\cal H}^{(1)} $) and by  the definition
${\cal H}_\Omega(z) =  \frac{1}{2} \langle \Omega_\bot(I_S) [z_\bot], [z_\bot] \rangle $ (cf.  \eqref{def H_Omega}), one has for any $z \in \mathcal V'$
$$
\begin{aligned}
{\cal H}^{(2)} (z)  & =  {\cal H}_S^{bo}( \Psi_C(z)) + \mathcal H_\Omega(\Psi_C(z)) + {\cal P}_2^{(1)}(\Psi_C(z)) +  {\cal P}_3^{(1)}(\Psi_C(z)) \\
& = {\cal H}_S^{bo}(z) +  \mathcal H_\Omega(z) + {\cal P}_2^{(1)}(z) + 
\big({\cal H}_S^{bo}( \Psi_C(z)) - {\cal H}_S^{bo}(z) \big) + \big( \mathcal H_\Omega(\Psi_C(z)) -  \mathcal H_\Omega(z) \big)+  {\cal P}_3^{(2c)}(z) \, ,
\end{aligned}
$$
where ${\cal P}_3^{(2c)}(z) = {\cal P}_2^{(1)}(\Psi_C(z)) - {\cal P}_2^{(1)}(z) +  {\cal P}_3^{(1)}(\Psi_C(z))$ (cf. \eqref{cal P 12 Psi C}).
Using $H^{bo}(\Psi^{bo}(z_S, 0)) =  {\cal H}^{bo}(I_S )$ (cf. \eqref{perturbazione composizione con Psi L}) and the identity (cf. \eqref{h nls circ PsiC}),
$$ 
\mathcal H_S^{bo}(\Psi_C(z)) - \mathcal H_S^{bo}(z) =  \langle \nabla_S \mathcal H_S^{bo}(z) \,  , \,  \pi_S {\cal R}_{N, 2}(z; \Psi_C) \rangle + {\cal P}^{(2a)}_3(z) \, ,
$$
as well as the definition ${\cal P}_3^{(2b)}(z) = {\cal H}_\Omega(\Psi_C(z)) - {\cal H}_\Omega(z)$ (cf. \eqref{parte pericolosa Psi C}) 
it follows that for $z= (z_S, z_\bot) \in \mathcal V'$, the Hamiltonian 
${\cal H}^{(2)}(z)$ is given by 
\begin{equation}\label{cal H2 nls}
{\cal H}^{(2)} (z) = {\cal H}^{bo}(I_S ) + \frac12 \langle \Omega_\bot(I_S)[z_\bot] \, , \, z_\bot \rangle
+ {\cal P}_2^{(2)}(z) + {\cal P}_3^{(2)}(z) \, ,
\end{equation}
where 
\begin{equation} \label{definizione cap P2 (3)}
{\cal P}_2^{(2)} (z)  :=  \langle \nabla_S \mathcal H_S^{bo}(z) \, , \,  \pi_S {\cal R}_{N, 2}(z; \Psi_C) \rangle + {\cal P}_2^{(1)}(z) \,, \quad 
{\cal P}_3^{(2)} (z)  := {\cal P}^{(2a)}_3(z) +{\cal P}_3^{(2b)}(z) + {\cal P}_3^{(2c)}(z).
\end{equation}
We recall that by \eqref{Taylor expansion R_N} and \eqref{perturbazione composizione con Psi L},
\begin{align}\label{definizione cap P2 (2)} 
\mathcal R_{N, 2} (z; \Psi_C) =  \frac{1}{2} d^2_\bot {\cal R}_{N}((z_S, 0) ; \Psi_C) [z_\bot, z_\bot]\,, \qquad
\quad {\cal P}_2^{(1)}(z) =  \frac12 \langle {\cal G}(z_S)[z_\bot] \, , \, z_\bot \rangle\,. \quad
\end{align}
Note that ${\cal P}_2^{(2)}(z)$ is quadratic with respect to $z_\bot$, whereas ${\cal P}^{(2)}_3$
is a remainder term of order three in  $z_\bot$. 
Being quadratic with respect to $z_\bot$, ${\cal P}^{(2)}_2$ can be written as 
\begin{equation}\label{bla bla cal P2}
{\cal P}_2^{(2)} (z) = \frac12 \langle d_\bot \nabla_\bot {\cal P}_2^{(2)}(z_S, 0)   [z_\bot] \, , \, z_\bot \rangle \, .
\end{equation}
 The following lemma is the analogue of a corresponding result for the KdV equation, due to Kuksin \cite{K}.  
      \begin{lemma}\label{cancellazione finale termini quadratici}
  The Hamiltonian ${\cal P}^{(2)}_2$ vanishes on $ \mathcal V'$. 
    \end{lemma}
  \begin{proof}
  In view of \eqref{bla bla cal P2}, it suffices to prove that for any $z_S \in \mathcal V_S'$, the operator $d_\bot \nabla_\bot {\cal P}_2^{(2)}(z_S, 0)$ vanishes.
  We establish that $d_\bot \nabla_\bot {\cal P}_2^{(2)}(z_S, 0) = 0$ by studying the linearization of $\partial_t w = J \nabla \mathcal H^{(2)}(w)$ along an arbitrary
  solution $w(t)$ of the form $w(t) = (w_S(t), 0)$. 
  First we need to make some preliminary considerations.
Let $t \mapsto q(t) \in M^o_S$ be a solution of the BO equation $\partial_t q = \partial_x \nabla H^{bo}(q)$
and denote by $t \mapsto z(t) := (z_S(t), 0)$ the corresponding solution in Birkhoff coordinates, defined by $q(t) = \Psi^{bo}(z(t))$.
It satisfies $\partial_t z(t) = J \Omega(I_S) [z(t)]$ (cf. \eqref{BO in Birkhoff}).
Furthermore, let $\widehat q(t)$ be the solution of the equation, obtained by  linearizing the BO equation along $q(t)$, 
$$
\partial_t \widehat q(t) = \partial_x d \nabla H^{bo}(q(t)) [\widehat q(t)] \,,
$$ 
with initial data $\widehat q^0 := d\Psi^{bo}(z_S(0), 0)[ 0, \widehat z_{\bot}^0]$ and  $\widehat z_{\bot}^0 \in h^2_\bot$. 
Similarly, denote by $\widehat z(t)$ the solution of the equation, obtained by linearizing $\partial_t z = J \Omega(I_S) [z]$
along the solution $z(t)$ with initial data $\widehat z^0 = (0, \widehat z_{\bot}^0)$,
$\partial_t \widehat z (t) =  J d \nabla  \mathcal H^{bo}(z(t)) [\widehat z (t)]$.
Since  $\partial_t z(t) = J \Omega(I_S) [z(t)]$ one concludes that (cf. \eqref{linearized BO in Bikrhoff})
\begin{equation}\label{equation widehat z bot}
 \widehat z(t) = (0, \widehat z_\bot(t))\,, \qquad \partial_t \widehat z_\bot (t) =   J_\bot \Omega_\bot(I_S) [\widehat z_\bot (t)] \,.
\end{equation}
Since $\Psi^{bo}$ is symplectic and $\mathcal H^{bo} = H^{bo} \circ \Psi^{bo}$, one has  
$\widehat q(t) = d\Psi^{bo}(z_S(t), 0)[\widehat z(t)]$. 
Recall that for any $z_S \in \mathcal V'_S$,  $\Psi_L(z_S, 0) = \Psi^{bo} (z_S, 0)$ (cf. definition \eqref{definition Psi_L bo} of $\Psi_L$) 
and $\Psi_C(z_S, 0) = (z_S, 0)$ 
(cf. Corollary \ref{proposizione espansione taylor correttore simplettico}),
implying that $ \Psi(z_S, 0) = \Psi^{bo}(z_S, 0)$ and hence $q(t) = \Psi (z_S(t), 0)$ for any $t$.
Since  $\Psi : \mathcal V' \to H^0_0$ is symplectic and $\mathcal H^{(2)} = H^{bo}\circ \Psi$, 
one sees that $z(t) = (z_S(t), 0)$ is also a solution of the equation 
$\partial_t w = J \nabla {\cal H}^{(2)}(w)$.
 With these preliminary considerations made, we are ready to prove that $d_\bot \nabla_\bot {\cal P}_2^{(2)}(z_S, 0)$ vanishes.
 To this end consider the solution $\widehat w(t)$ of the equation obtained by linearizing  $\partial_t w = J \nabla {\cal H}^{(2)}(w)$ along the solution $z(t) = (z_S(t), 0)$
 with initial data $\widehat w^0 = (0, \widehat z_{\bot}^0)$. Again using that the map $\Psi$ is symplectic and $\mathcal H^{(2)} = H^{bo}\circ \Psi$,
 it follows that $d\Psi(z(t))[\widehat w(t)]$ solves the linearized BO equation. Since $d\Psi(z(0)) = d\Psi^{bo}(z(0))$ and $\widehat w^0 = \widehat z^0$,
 one then concludes from the uniqueness of the initial value problem that  $d\Psi(z(t))[\widehat w(t)] = d\Psi^{bo}(z(t))[\widehat z(t)]$ and hence 
 $\widehat w(t) = \widehat z(t)$ for any $t$. It means that $ \widehat z(t)$ satisfies also the linear equation
 $$
 \partial_t \widehat z (t) =  J d \nabla  \mathcal H^{(2)}(z(t)) [\widehat z (t)]\,.
 $$
 In view of the expansion \eqref{cal H2 nls} of $\mathcal H^{(2)}$ one then infers that
$$
\partial_t \widehat z_\bot(t) = J_\bot \Omega_\bot (I_S)[\widehat z_\bot(t)] + J_\bot d_\bot \nabla_\bot {\cal P}^{(2)}_2(z_S(t), 0)[\widehat z_\bot(t)]\,. 
$$
Comparing the latter identity with \eqref{equation widehat z bot} one concludes that in particular, 
$d_\bot \nabla_\bot {\cal P}^{(2)}_2 (z_S(0), 0) = 0$. Since the initial data  $z_S(0) \in \mathcal V_S'$ can be chosen arbitrarily,
we thus have $d_\bot \nabla_\bot {\cal P}^{(2)}_2 (z_S, 0) = 0$ for any $z_S \in \mathcal V_S'$ as claimed. 
 \end{proof}
In summary, we have proved the following results on the Hamiltonian $\mathcal {\cal H}^{(2)}= H^{bo} \circ \Psi$.  
\begin{theorem} \label{stime finali grado 3 perturbazione}
The Hamiltonian $\mathcal {\cal H}^{(2)}: \mathcal V' \cap h^1_0 \to \R $ has an expansion of the form
\begin{equation}\label{forma finalissima hamiltoniana trasformata}
{\cal H}^{(2)}(z) = H^{bo}(q) + \frac12 \langle {\Omega}_\bot(I_S)[z_\bot], \, z_\bot \rangle + {\cal P}_3^{(2)}(z)
\end{equation}
where $\Omega_\bot(I_S)$ is given by \eqref{splitting Omega} and the remainder term  ${\cal P}_3^{(2)}$, defined by  \eqref{definizione cap P2 (3)},
 satisfies the following:
 ${\cal P}_3^{(2)} : \mathcal V' \cap h^1_0  \to \R$ is real analytic and for any integer $N \ge 1$,  its gradient $\nabla {\cal P}_3^{(2)}(z)$ 
 admits the asymptotic expansion $\big( 0, \, \mathcal{OP}(z; \nabla {\cal P}_3^{(2)}) \big)+ {\cal R}_N( z; \nabla {\cal P}_3^{(2)})$ where
  $$
  \mathcal{OP}(z; \nabla {\cal P}_3^{(2)}) = 
  {\cal F}^+_{N_S} \circ \sum_{k = 0}^N  T_{a^+_k(z; \nabla {\cal P}_3^{(2)})}  D^{- k} [({\cal F}^+_{N_S})^{- 1} z_\bot] 
  +  {\cal F}^-_{N_S} \circ \sum_{k = 0}^N  T_{a^-_k(z; \nabla {\cal P}_3^{(2)})} (- D)^{- k} [({\cal F}^-_{N_S})^{- 1}z_\bot]
  $$ 
  with the property that there there exists an integer  $\sigma_N \ge N$ (loss of regularity) so that for any $s \ge 0$, $0 \le k \le N$, the maps 
  $$
  \mathcal V'  \cap h^{s + \sigma_N}_0 \to H^s_\C, \, z \mapsto a^\pm_k( z; \nabla{\cal P}_3^{(2)})\,, \qquad
  \mathcal V' \cap h^{s \lor \sigma_N}_0 \to h^{s + N + 1}_0, \, z \mapsto {\cal R}_N( z; \nabla {\cal P}_3^{(2)})
  $$
   are real analytic and $a^-_k( z; \nabla {\cal P}_3^{(2)}) = \overline{a^+_k( z;  \nabla {\cal P}_3^{(2)})}$.
    Furthermore, they satisfy the following estimates: 
    for any $s \ge 0,$ $z \in \mathcal V' \cap h^{s + \sigma_N}_0$ with  $\| z \|_{\sigma_N} \leq 1$,
  and $\widehat z_1, \ldots, \widehat z_l \in h^{s + \sigma_N}_0$, $l \ge 1$,
  $$
  \begin{aligned}
  & \|   a^\pm_k(z; \nabla {\cal P}_3^{(2)})  \|_s \lesssim_{s, N} \| z_\bot \|_{s + \sigma_N}\,, \\
  & \| d^l a^\pm_k( z; \nabla {\cal P}_3^{(2)}) [\widehat z_1, \ldots, \widehat z_l] \|_s \lesssim_{s, N, l} 
  \sum_{j = 1}^l \| \widehat z_j\|_{s + \sigma_N} \prod_{i \neq j} \| \widehat z_j\|_{\sigma_N} + \| z_\bot \|_{s + \sigma_N} \prod_{j = 1}^l \| \widehat z_j\|_{\sigma_N}\,.
  \end{aligned}
  $$
  Similarly, for any $s \ge 1$ and $z \in \mathcal V' \cap h^{s \lor \sigma_N}_0$ with $\| z_\bot \|_{\sigma_N} \leq 1$,  $\widehat z\in h^{s \lor \sigma_N}_0$,   
  $$
    \begin{aligned}
  & \| {\cal R}_N( z; \nabla {\cal P}_3^{(2)})\|_{s + N + 1} \lesssim_{s, N} \| z_\bot \|_{s \lor \sigma_N} \| z_\bot \|_{\sigma_N} \,, \\
  &\| d {\cal R}_N( z; \nabla {\cal P}_3^{(2)}) [\widehat z]\|_{s + N + 1} \lesssim_{s, N} \| z_\bot \|_{\sigma_N} \| \widehat z\|_{s \lor \sigma_N} + \| z_\bot \|_{s \lor \sigma_N} \| \widehat z\|_{\sigma_N}\, . 
  \end{aligned}
  $$
 If in addition  $\widehat z_1, \ldots, \widehat z_l \in h^{s \lor \sigma_N}_0$, $ l \geq 2$, then
 $$
 \| d^l {\cal R}_N( z; \nabla {\cal P}_3^{(2)}) [\widehat z_1, \ldots, \widehat z_l] \|_{s + N + 1} \lesssim_{s, N, l} 
  \sum_{j = 1}^l \|\widehat z_j \|_{s \lor \sigma_N} \prod_{i \neq j} \| \widehat z_i\|_{\sigma_N} + \| z_\bot \|_{s \lor \sigma_N} \prod_{j = 1}^l \| \widehat z_j\|_{\sigma_N}\,. 
  $$
\end{theorem} 
\begin{proof}
The identity \eqref{forma finalissima hamiltoniana trasformata} folllows from formula \eqref{cal H2 nls} and Lemma \ref{cancellazione finale termini quadratici}.
The claimed asymptotic expansion of  $\nabla {\cal P}_3^{(2)}$ and its properties follow from
Lemmata \ref{stime tame h3 nls}, \ref{proprieta hamiltoniana cal P 3 (2b)}, \ref{composizione cal P 12 Psi C} and Lemma \ref{primo lemma paraprodotti}. 
\end{proof}


\section{Summary of the proofs of Theorem \ref{modified Birkhoff map} and Theorem \ref{modified Birkhoff map 2}}\label{synopsis of proof}

In this section we summarize the proofs of Theorem \ref{modified Birkhoff map}, of its addendum, and of Theorem \ref{modified Birkhoff map 2}. 
First recall that in view of the envisioned applications, these theorems are formulated in terms of action angle coordinates on the submanifold $M_S^o$ 
of proper $S$-gap potentials.
 Denote by $\Xi $ the map relating action angle variables and complex Birkhoff coordinates,
$$
\Xi: \T^{S_+} \times \R_{>0}^{S_+} \times h^0_\bot \to h_S \times h^0_\bot, \,  (\theta_S, I_S, z_\bot) \mapsto (z_S(\theta_S, I_S), z_\bot)\,, 
$$
where 
$$
z_S(\theta_S, I_S) = (z_n(\theta_S, I_S))_{n \in S} \, , \qquad 
z_{\pm n} = \sqrt{2\pi n I_n} e^{\mp\ii \theta_n}\,, \quad \forall  n \in S_+ \, .
$$
Clearly, $\Xi$ is symplectic and, for any $s \geq 0$, the map  $\Xi : \T^{S_+} \times \R_{>0}^{S_+} \times h^s_\bot \to h_S \times h^s_\bot, $ is real analytic.
Furthermore, in view of the definition  \eqref{definition reversible structure for actions angles}, the map $\Xi$ preserves the reversible structure.
Hence the claimed results for the map $\Psi_L \circ \Psi_C \circ \Xi$ follow from the corresponding ones for the map $\Psi_L \circ \Psi_C$.
In what follows we summarize the proofs of the results for $\Psi_L \circ \Psi_C$ corresponding to the ones claimed for  $\Psi_L \circ \Psi_C \circ \Xi$.

\smallskip

\noindent 
{\em Proof of Theorem \ref{modified Birkhoff map}.}  For notational convenience, we denote  the composition $\Psi_L \circ \Psi_C$ by $\Psi$
(cf. the discussion in the paragraph above with regard to the map $\Psi$ of Theorem \ref{modified Birkhoff map}).
By \eqref{definition Psi_X {tau_0, tau}}, $\Psi$ is defined on the neighborhood $\mathcal V' = \mathcal V'_S \times \mathcal V'_\bot$
where $\mathcal V'_S$ is a bounded neighborhood of any given compact subset $\mathcal K \subset h_S$ and $\mathcal V'_\bot$ is a ball in $h^0_\bot$ 
of radius smaller than $1$, centered at $0$. The expansion of $\Psi$, corresponding to the one of {\bf (AE1)}, 
follows from the expansion for the map $\Psi_L$, provided by Corollary  \ref{pseudodiff expansion Psi_L},
and the one for the map $\Psi_C$, provided by Theorem \ref{espansione flusso per correttore}(ii).

\noindent
The expansion of the transpose $d\Psi(z)^\top$ of the derivative $d\Psi(z)$, corresponding to the one of {\bf (AE2)}, 
follows from the fact that $\Psi: \mathcal V'  \to L^2_0$ is symplectic, meaning that 
for any $z \in \mathcal V' $, the operator $d\Psi(z)^\top: H^1_0 \to  h^1_0$ satisfies $d\Psi(z)^\top= J^{-1} (d\Psi(z))^{-1} \partial_x$. 
The expansion of $\Psi(z)$ in {\bf (AE1)} then leads to an expansion 
of $d\Psi(z)$ and in turn of $(d\Psi(z))^{-1}$ and hence of $d\Psi(z)^\top$. The reality conditions of the coefficients in the various expansions
follow from the way they are constructed.



\noindent
The expansion of the Hamiltonian $\mathcal H^{(2)} = H^{bo} \circ \Psi$ and of the remainder term ${\cal P}_3^{(2)}$
in the Taylor expansion \eqref{forma finalissima hamiltoniana trasformata}, 
corresponding to the one in {\bf (AE3)}, are provided in Theorem \ref{stime finali grado 3 perturbazione}.
By arguing as in Section \ref{Hamiltoniana trasformata}, one sees that $\mathcal H^{(mo, 2)}(z) := H^{mo} \circ \Psi (z)$ has a Taylor expansion of the form
$$
\mathcal H^{(mo, 2)}(z) = \frac 12 \langle z_S , z_S \rangle + \frac 12 \langle z_\bot , z_\bot \rangle + \mathcal P^{(mo, 2)}_3(z) 
$$
As in the case of the Benjamin-Ono Hamiltonian, one obtains  expansions of the Hamiltonian $\mathcal H^{(mo, 2)}$
and of the remainder term $ \mathcal P^{(mo, 2)}_3$, corresponding to the ones stated in {\bf (AE3)}. Actually, in this case
the computations considerably simplify.
\hfill $\square$

\smallskip

\noindent
{\em Proof of Addendum to Theorem \ref{modified Birkhoff map}.} Clearly, the Fourier transform $\mathcal F$ and its inverse preserve 
the reversible structure and by Proposition \ref{proposizione 1 reversibilita}, so do the Birkhoff map $\Phi^{bo}$  and its inverse $\Psi^{bo}$.
Furthermore,  by item (ii) of the Addendum to Theorem 2.1,
and by the Addendum to Theorem 3.1,
also the maps $\Psi_L$  and $\Psi_C$ and hence $\Psi_L \circ \Psi_C$ preserve the reversible structure, as do the coefficients and the remainder terms
in their expansions as well as the transpose of their derivatives. 

\noindent
Clearly, the BO Hamiltonian $H^{bo}$ and $H^{mo}$ are reversible and therefore so are $\mathcal H^{(2)} = H^{bo} \circ \Psi$ and $\mathcal H^{(mo,2)} = H^{mo} \circ \Psi$.
By \eqref{forma finalissima hamiltoniana trasformata} one then concludes that also the remainders ${\cal P}_3^{(2)}$ and ${\cal P}_3^{(mo,2)}$ are reversible.
\hfill $\square$

\smallskip

\noindent
{\em Proof of Theorem \ref{modified Birkhoff map 2}.} The estimates of the coefficients and the remainder in the expansion of $\Psi = \Psi_L \circ \Psi_C$, corresponding to the ones of {\bf (Est1)}, 
follow from the estimates of the coefficients and the remainder in the expansion of the map $\Psi_L$, provided by Corollary \ref{pseudodiff expansion Psi_L} ,
and the ones of the coefficients and the remainder in the expansion of the map $\Psi_C$, provided by Theorem \ref{espansione flusso per correttore}.

The estimates of the coefficients and the remainder in the expansion of $d\Psi(z)^\top$, corresponding to the one of {\bf (Est2)},
follow from the fact that $\Psi: \mathcal V'  \to L^2_0$ is symplectic, meaning that for any $z \in \mathcal V' $, 
$d\Psi(z)^\top: H^1_0 \to h^1_0$ satisfies  $d\Psi(z)^\top = J^{-1} (d\Psi(z))^{-1} \partial_x$ 
and the estimates  {\bf (Est1)} of the coefficients and the remainder in the expansion of $\Psi(z)$ which lead to corresponding estimates
of the coefficients and the remainder in the expansion of $d\Psi(z)$ and in turn of $(d \Psi(z))^{-1}.$

The estimates of the remainder term $ {\cal P}_3^{(2)}$ in the expansion of the Hamiltonian $\mathcal H^{(2)} = H^{bo} \circ \Psi$, corresponding to {\bf (Est3)}, 
are provided by Theorem \ref{stime finali grado 3 perturbazione}. The ones for the remainder term $ {\cal P}_3^{(mo,2)}$ are derived in the same way.
\hfill $\square$

\appendix

\section{Spectral theory of the Lax operator $L_u$}\label{spectral theory}

In this appendix we review the spectral theory of the Lax operator of the Benjamin-Ono equation \eqref{1.1},
\begin{equation}\label{Lax operator}
L_u f:= - \ii \partial_x f - \Pi[uf]\, , \qquad f \in \mbox{dom}(L_u):= H^1_+ \, , \qquad H^1_+:= H^1 \cap H^0_+ \, ,
\end{equation} 
with potential $u \in L^2 \equiv H^0$, where
$L^2_+ \equiv  H^0_+ := \{ f \in L^2 : \, \langle f | e^{\ii nx} \rangle = 0 \ \forall \, n < 0\}$ denotes the Hardy space of $\T$
and $\Pi: L^2 \to L^2_+,  f \mapsto \sum_{n \ge 0} \widehat f(n) e^{\ii nx}$ the Szeg\H{o} projector.
The operator $L_u$ appears in the Lax pair formulation of \eqref{1.1}, $\partial_t L_u = B_u L_u - L_u B_u$, where $B_u$ is a certain skew adjoint
pseudo-differential operator -- see \cite[Remark 2.1, Appendix A]{GK1}. For the convenience of the reader, we recall some of the notations,
introduced in the main body of the paper.

For any $u \in L^2$, the Lax operator $L_u$ is the pseudo-differential operator of order one, acting on $L^2_+$, 
which is self-adjoint, bounded from below, and has compact resolvent (cf. \cite{GK1}, \cite{GKT2}). 
Hence its spectrum $\rm{spec}(L_u)$ consists of an unbounded sequence of real eigenvalues, each of finite multiplicity, which
can be listed in increasing order,
$$
\lambda_0(u) \le \lambda_1(u) \le \lambda_2(u) \le \cdots
$$
In \cite{GK1}, \cite{GKT1}, the following results are proved.
\begin{theorem}\label{spec Lax operator}
$(i)$ For any $u \in L^2$, the eigenvalues $\lambda_n(u)$, $n \ge 0$, of $L_u$ are separated by at least one, 
$$
\gamma_n(u) := \lambda_n(u) - \lambda_{n-1}(u) -1 \ge 0 \, , \qquad \forall \, n \ge 1 \, .
$$
In particular, $\rm{spec}(L_u)$ is simple. Furthermore, for any $n \ge 0$,
$$ 
\lambda_n(u) = n - \sum_{k \ge n+1} \gamma_k(u) \, , 
$$
and the following trace formulas hold, 
$$
\langle u | 1 \rangle = - \lambda_0(u) - \sum_{n \ge 1} \gamma_n(u) \, , \qquad \quad
\| u - \langle u | 1 \rangle \|^2 = 2 \sum_{n \ge 1} n \gamma_n(u) \, .
$$
$(ii)$ Conversely, for any sequence $(r_n)_{n \ge 1}$ of nonnegative numbers, satisfying $\sum_{n \ge 1} n r_n < \infty$,
there exists $u \in L^2$ with $\langle u | 1 \rangle = 0$ so that $\gamma_n(u) = r_n$ for any $n \ge 1$. It means that such sequences are a moduli
space for the spectra of the Lax operators \eqref{Lax operator}.\\
$(iii)$ For any potential $u$ in $L^2$ and any $s \ge 0$, $\sum_{n \ge 1} n^{1 + 2s} \gamma_n(u) < \infty$ if and only if
$u$ is in $H^s$. \\
$(iv)$ For any $n \ge 0$, $\lambda_n : L^2 \to \R$ is real analytic. Hence for any $n \ge 1$, $\gamma_n : L^2 \to \R$ is real analytic as well.\\
$(v)$ For any $u \in L^2$, $u$ is constant if and only if $\gamma_n(u) = 0$ for any $n \ge 1$.
\end{theorem}
\begin{remark} 
$(i)$ It has been shown in \cite[Appendix C]{GK1} that the spectrum of $L_u$, when considered as an operator on 
the Hardy space of the real line, consists of the union of bands of length one, $\cup_{n \ge 0}[\lambda_{n}, \lambda_{n} + 1]$.
These bands are separated by the gaps $(\lambda_{n-1} +1, \lambda_{n})$, $n \ge 1$. For any $n \ge 1$, $\gamma_n(u)$
is the length of the gap $(\lambda_{n-1} +1, \lambda_{n})$ and is referred to as the nth gap length of the spectrum of $L_u$.
$(ii)$ In \cite{GKT1}, Theorem \ref{spec Lax operator} has been extended to $L_u$ with potentials $u \in H^s$, $-1/2 < s < 0$.
For any $u\in H^s$, $\sum_{n \ge 1} n^{1+2s}\gamma_n(u) < \infty$.
\end{remark}

For any $u \in L^2$, let  $h_n(\cdot, u)$, $n \ge 0$, be an orthonormal basis of eigenfunctions of $L_u$. Note that for any $n \ge 0$, 
$h_n(\cdot, u)$ is in $H^1_+$ and $\|h_n(\cdot, u)\| =1$. Since $\lambda_n(u)$ is simple, 
$h_n(\cdot, u)$ is then uniquely determined up to a phase factor.
By \cite[Lemma 2.5, Lemma 2.7]{GK1}, 
$$
\langle 1 | h_0(\cdot, u) \rangle \ne 0 \, , \qquad \langle h_n(\cdot, u) | e^{\ii x}h_{n-1}(\cdot, u) \rangle \ne 0 \, , \quad \forall \, n \ge 1.
$$
\begin{definition}
For any $u \in L^2$, we denote by $f_n(\cdot, u)$, $n \ge 0$, the orthonormal basis of $L^2_+$ of eigenfunctions of $L_u$,
corresponding to the eigenvalues $(\lambda_n(u))_{n \ge 0}$, uniquely determined by the normalization conditions
\begin{equation}\label{normalized ef}
\langle 1 | f_0(\cdot, u) \rangle > 0 \, , \qquad \langle f_n(\cdot, u) | e^{\ii x}f_{n-1}(\cdot, u) \rangle > 0 \, , \quad \forall \, n \ge 1.
\end{equation}
\end{definition}
By \cite[Corollary 3.4]{GK1} one concludes that for any $u \in L^2$ and $n \ge 1$, $\gamma_n(u) = 0$ if and only if $\langle 1 | f_n(\cdot, u) \rangle = 0$.
More generally, the following holds.
\begin{lemma}\label{def kappa}
$(i)$ For any $u \in L^2$ and any $n \ge 1$,
\begin{equation}\label{product formula kappa}
| \langle 1 | f_n(\cdot, u) \rangle |^2 = \gamma_n(u) \kappa_n(u)\, , \qquad
\kappa_n(u):= \frac{1}{\lambda_n(u) - \lambda_0(u)} \prod_{p \ne n} \big(1 - \frac{\gamma_p(u)}{\lambda_p(u) - \lambda_n(u)}  \big) \, ,
\end{equation}
where the infinite product in \eqref{product formula kappa} is absolutely convergent.
Furthermore, for any $n \ge 1$ and $u \in L^2$, $\gamma_n(u) = 0$ if and only if $f_n(\cdot, u) = e^{ix} f_{n-1}(\cdot, u)$. \\
$(ii)$ For any $u \in L^2$ and $n \ge 1$, $\kappa_n(u)$, defined in \eqref{product formula kappa}, satisfies $\kappa_n(u) > 0$ 
and there exists a constant $C \ge 1$ so that
$$
\sup_{n \ge 1}  \, n \kappa_n(u) \le C \, ,  \qquad \sup_{n \ge 1} \,  \frac{1}{n \kappa_n(u)} \le C \, .
$$ 
In particular, $\sqrt{\kappa_n(u)}$, $n \ge 1$, is well-defined and $\sqrt{\kappa_n(u)} > 0$, where $\sqrt{\cdot} \equiv \sqrt[+]{\cdot}$
denotes the principal branch of the square root.\\
$(iii)$ For any $n \ge 1$, $f_n : L^2 \to H^1_+$ and $\kappa_n: L^2 \to \R$ are real analytic.
\end{lemma}
Next, we compare the spectrum of $L_{u_*}$ with the one of $L_u$, where for any $u \in L^2$,
\begin{equation}\label{def u_*}
u_*(x) := u(-x) \, .
\end{equation}
By \cite[Proposition 4.5]{GK1}, eigenvalues and eigenfunctions of $L_{u_*}$ and $L_u$ are related as follows.
\begin{lemma}\label{even symmetry}
$(i)$ For any $u \in L^2$ and $n \ge 0$,
\begin{equation}\label{Sym 1}
\lambda_n(u_*) = \lambda_n(u)\, , \qquad f_n(x, u_*) = \overline{f_n(-x, u)} \, .
\end{equation}
$(ii)$ As a consequence,
\begin{equation}\label{Sym 2}
\gamma_n(u_*) = \gamma_n(u)\, , \quad \kappa_n(u_*) = \kappa_n(u)  \, , \qquad \forall \, n \ge 1 \, , \ u \in L^2 \, ,
\end{equation}
and (when combined with Lemma \ref{def kappa}(iii)) $f_n : L^2 \to H^1_+, \, u \mapsto \overline{f_n(- \, \cdot , u)} = f_n( \cdot, u_*)$ as well as
\begin{equation}\label{complex conjugate}
 L^2 \to \C , \, u \mapsto \langle 1 | f_n(\cdot, u) \rangle = \frac{1}{2\pi} \int_0^{2\pi} f_n(x, u_*) d x \, ,
\quad L^2 \to \C , \, u \mapsto \overline{\langle 1 | f_n(\cdot, u) \rangle} = \frac{1}{2\pi} \int_0^{2\pi} f_n(x, u) d x \, ,
\end{equation}
are real analytic.
\end{lemma}

We finish this appendix with a discussion of finite gap potentials. Let
$$
L^2_0 := \{ u \in L^2 : \, \frac{1}{2\pi} \int_0^{2\pi} u \, d x = 0 \} .
$$
\begin{definition}\label{(proper) S gap}
An element $u$  in $L^2_0$ is said to be a finite gap potential if the set $S_u:= \{ n \in \N : \, \gamma_n(u) > 0 \}$ is finite.
For any given finite subset $S_+ \subset \N$, $u$ is said to be a $S$-gap potential if $S_u = S_+$
where $S$ is defined as
$S:= S_+ \cup (-S_+)$.
The set of $S$-gap potentials in $L^2_0$
is denoted by $M_S$ and the subset $M^o_S:= \{ q \in M_S : \, \gamma_n(u) > 0 \ \forall n \in S_+ \}$ is referred
to as the set of proper $S$-gap potentials.
\end{definition}
For any integer $N \ge 1$, let
$$
\mathcal U_N := \{ u \in L^2_0 : \, \gamma_N(u) > 0, \, \gamma_n(u) = 0 \  \forall \, n > N  \} \, .
$$
Note that the set of finite gap potentials in $L^2_0$ is given by the disjoint union $ \cup_{N \ge 1} \mathcal U_N \cup \{ 0 \}$
and that by Theorem \ref{spec Lax operator}, $\mathcal U_N \subset \cap_{s \ge 0} H^s_0$ for any $N \ge 1$.
Furthermore, by \cite[Section 7]{GK1}, any $u \in \mathcal U_N$ is of the form 
\begin{equation}\label{formula finite gap}
u(x) = \sum_{j=1}^{N} \big(  \frac{1 - r_j^2}{1 - 2r_j \cos(x +\alpha_j)+ r_j^2} -1  \big)\, , 
\qquad 0 < r_j < 1, \  0 \le \alpha_j < 2\pi \, , \quad \forall \, 1 \le j \le N \, .
\end{equation}
In \cite[Section 7]{GK1} (cf. also \cite[Appendix A]{GKT3}) the following is proved.
\begin{lemma}\label{results finite gap potentials}
$(i)$ For any $s > -1/2$, the disjoint union $ \cup_{N \ge 1} \mathcal U_N \cup \{ 0 \}$ is dense in $H^s_0$.\\
$(ii)$ For any $u \in \mathcal U_N$, $N \ge 1$, one has for any $n \ge N$,
$$
\lambda_n(u) = n \, , \qquad f_n(x, u) = e^{\ii nx} g_\infty(x, u)\, , \qquad g_\infty(x, u):= e^{\ii \partial_x^{-1} u(x)} \, .
$$
$(iii)$ For $u = 0$, one has $\lambda_0(0) = 0$, $f_0(x, 0) = 1$, $g_\infty(x, 0) = 1$, and for any $n \ge 1$,
$$
\lambda_n(0) = n \, , \qquad f_n(x, 0) = e^{\ii nx}\, ,  \qquad n\kappa_n(0) = 1 \, .
$$
Furthermore, $\nabla \langle 1 | f_n \rangle  (x, 0) = - \frac{1}{n} e^{-\ii nx}$ 
for any $n \ge 1$ (cf. \cite[Remark 5.5]{GK1}).
\end{lemma}


\section{Birkhoff map}\label{Birkhoff map for BO}
In this appendix we review the definition and properties of the Birkhoff coordinates, needed in this paper. 
We refer to \cite{GK1} - \cite{GKT5} for proofs of the results stated and for more details in these matters.
In \cite{GKT6}, most of the results obtained on the Birkhoff map $\Phi$ are summarized.
We point out that in the main body of the paper, we exclusively use a rescaled version $\Phi^{bo}$ of the Birkhoff map $\Phi$,
which we also refer to as Birkhoff map.
For the convenience of the reader, we recall some of the notations, introduced in the main body of the paper.

The Benjamin-Ono equation \eqref{1.1} is known to be globally in time wellposed on $H^s$ for any $s > -1/2$. See \cite{GKT1} and references therein.
For our purposes, it suffices to consider \eqref{1.1} on $H^s$ for $s \ge 0$.
Furthermore, since $\langle u | 1 \rangle = \frac{1}{2\pi} \int_0^{2\pi} u \, dx$ is a prime integral of \eqref{1.1}, we may assume that $\langle u | 1 \rangle = 0$. 
Indeed, for any solution $u(t, x)$ of \eqref{1.1}
in $H^s$, $s > - 1/2$, and any $c \in \R$, $c + u(t, x - 2ct)$ is again a solution of \eqref{1.1} in $H^s$.
Denote by $S(t)$, $t \in \R$, the flow map of \eqref{1.1} on $L^2_0$, meaning that 
$S(t)u$ denotes the solution of \eqref{1.1} with initial value $S(0) u = u \in L^2_0$, where 
$$
L^2_0 :=  \{ u \in L^2 : \, \langle u | 1 \rangle = 0 \} \, , \qquad
H^s_0 := H^s \cap L^2_0 \, , \quad \forall \, s \ge 0 \, .
$$

By Lemma \ref{even symmetry}, for any $n \ge 1$, the map
$$
L^2_0 \to \C , \, u \mapsto \zeta_n(u) := \frac{\langle 1 | f_n(\cdot, u) \rangle}{\sqrt{\kappa_n(u)}}
$$
is well-defined and satisfies $|\zeta_n(u) |^2 = \gamma_n(u) $. Hence by Theorem \ref{spec Lax operator}(iii)
one has for any $u \in L^2_0$ and $s \ge 0$,
$$
u \in H^s_0 \quad \mbox{if and only if} \quad \zeta(u):= (\zeta_n(u))_{n \ge 1} \in h^{s+\frac 12}_+ \, .
$$
For any $C^1$-functionals $F$, $G : L^2_0 \to \C$ with sufficiently regular $L^2_0$-gradient $\nabla F$, $\nabla G$, we denote
by $\{ F, G\}$ the Poisson bracket, due to Gardner and Faddeev$\&$Zakharov,
\begin{equation}\label{Gardner bracket B}
\{ F, G\}(u) := \langle \partial_x \nabla F , \,  \nabla G\rangle (u)= \frac{1}{2\pi} \int_0^{2\pi} \partial_x( \nabla F)(u) \cdot \nabla G(u) d x \, .
\end{equation}
For any $s \in \R$, let
$$
h^s_+ := \{ w = (w_n)_{n \ge 1} \subset \C : \, \| w \|_s < \infty \} \, , \qquad
\|w\|_s := (\sum_{n \ge 1} n^{2s} |w_n|^2)^{\frac 12} \, .
$$
%
\begin{theorem}\label{Theorem Birkhoff coordinates BO}
The map
$$
\Phi : \bigsqcup_{s \ge 0} H^s_{0} \to  \bigsqcup_{s \ge 0} h_+^{s+\frac 12}, \ u \mapsto \zeta(u) = (\zeta_n(u))_{n \ge 1} \, ,
$$
has the following properties:\\
(NF1) For any $s \ge 0$, $\Phi : H^s_{r,0} \to h_+^{s+1/2}$ is a real analytic diffeomorphism and $\Phi$ and its inverse $\Phi^{-1}$ map bounded
subsets to bounded ones.\\
(NF2) The Poisson brackets between the functionals $\zeta_n(u)$, $\overline{ \zeta_n(u)}$, $n \ge 1$,
are well-defined and one has (cf. \cite[Corolllary 7.3]{GK1}),
$$
\{\zeta_n, \overline{\zeta_k} \}(u) = -\ii \delta_{n,k}\, , \quad
\{\zeta_n, \zeta_k \}(u) = 0\, , \qquad  \forall \, n, k \ge 1 \, .
$$
In particular,
$$
\{\gamma_n, \, \gamma_k \}(u)  = \{|\zeta_n|^2, \, |\zeta_k|^2 \}(u) = 0\, , \qquad  \forall \, n, k \ge 1 \, .
$$
(NF3) For any $u \in L^2_{0}$, $n \ge 1$, and $t \in \R$,
\begin{equation}\label{formula frequencies intro}
\zeta_n(\mathcal S(t)u) = e^{\ii t \omega^{bo}_n } \zeta_n(u) \, , \qquad
\omega^{bo}_n \equiv \omega^{bo}_n(u) := n^2 - 2\sum_{k=1}^n k \gamma_k(u) - 2n \sum_{k > n}  \gamma_k(u) \, .
\end{equation}
For any $n \ge 1$, $|\zeta_n(\mathcal S(t)u)|^2$ is independent of $t$ and $\zeta(\mathcal S(t)u) = (\zeta_n(\mathcal S(t)u))_{n \ge 1}$
evolves on the torus
\begin{equation}\label{def Tor}
\mbox{Tor}(\Phi(u)) := \{w = (w_n)_{n \ge 1} \in h^{1/2}_+ \, : \, |w_n|^2 = \gamma_n(u) \ \forall \, n \ge 1  \} \, .
\end{equation}
(NF4) The differential of $\Phi$ at $0$ is given by (cf. Lemma \ref{results finite gap potentials}(iii))
$$
d_0\Phi: L^2_0 \to h^{\frac 12}_+ \, , \ u \mapsto (- \frac{u_n}{\sqrt n})_{n \ge 1} \, .
$$
The map $\Phi$ is referred to as the Birkhoff map for the Benjamin-Ono equation and $\zeta_n(u)$, $n \ge 1$,
as the Birkhoff coordinates of $u$. Furthermore, $\gamma_n(u)$, $n \ge 1$, are referred to as action variables.
\end{theorem}
\begin{remark}\label{extension + sym}
$(i)$ For any $-1/2 < s < 0$, $\Phi$ can be extended as a real analytic map, 
$H^s_0 \to h^{s+\frac 12}_+$, but it does not extend to $H^{-\frac 12}_0$  -- see \cite{GKT1}, \cite{GKT3}, \cite{GKT4}.\\
$(ii)$ For any $u \in L^2_0$, $u_*(x) = u(-x)$ is also in $L^2_0$. It then follows from \eqref{Sym 1}- \eqref{Sym 2} that
\begin{equation}\label{Sym zeta_n}
\zeta_n(u_*) = \overline{\zeta_n(u)} \, , \qquad  \forall \, n \ge 1 \, , \ u \in L^2_0 .
\end{equation}
\end{remark}
\begin{corollary}\label{restriction to finite gap}
For any finite subset $S_+ \subset \N$ and $S:= S_+ \cup (-S_+)$, 
the set of $S$-gap potentials and the one of proper $S$-gap potentials, introduced in Definition \ref{(proper) S gap},
are given by  
$$
M_S = \Phi^{-1} \{ (w_n)_{n \ge 1} : \, w_n = 0 \ \  \forall n \in \N \setminus S_+\} \, , \qquad 
M^0_S = \Phi^{-1} \{ (w_n)_{n \ge 1} : \, w_n \ne 0 \ {\mbox{ if and only if }} \   n \in S_+ \} .
$$
$M_S$ and $M^0_S$ are real analytic, symplectic submanifolds of $L^2_0$ of dimension $S_+$, where
$L^2_0$ is endowed with the symplectic form $\Lambda_G$, induced by the Poisson bracket \eqref{Gardner bracket B},
$$
\Lambda_G [u , v] = \langle   u , \,  \partial_x^{-1} v \rangle =\frac{1}{2\pi} \int_0^{2\pi} u(x)   \partial_x^{-1} v(x) d x   \,, 
\qquad \forall  u,  v  \in L^2_0\,.
$$
\end{corollary} 
We finish this appendix with a discussion of the isospectral set $\mbox{Iso}(u)$, containing $u \in L^2_0$, defined as
$$
\mbox{Iso}(u) := \{ v \in L^2_0 : \, \lambda_n(v) = \lambda_n(u) \ \forall \, n \ge 0 \} .
$$
\begin{corollary} For any $u \in L^2_0$, 
$$
\mbox{Iso}(u) = \Phi^{-1} (\mbox{Tor}(\Phi(u)))
$$
where $\mbox{Tor}(\Phi(u))$ is given by \eqref{def Tor}. In particular, 
 $\mbox{Iso}(u) $ can be viewed as a torus of dimension $| \{ n \ge 1 : \, \gamma_n(u) > 0\} |$. 
\end{corollary}


\section{Reversibility structure}\label{AppendixReversability}

In this appendix we prove that the Birkhoff map $\Phi^{bo}$, defined in \eqref{def Phi^bo}, and hence also its inverse $\Psi^{bo}$, preserve the reversible structure,
defined by the maps 
\begin{equation}\label{reversible structure appendix}
 S_{rev}: L^2_0 \to L^2_0\,, \ ( S_{rev} q) := q_* \, , \quad  q_*(x) = q(-x)\, ,
 \qquad \mathcal S_{rev}: h^0_0 \to h^0_0\,,  \ (\mathcal S_{rev} w)_n := w_{-n}\,, \quad \forall n \ne 0 \, .
\end{equation}
\begin{proposition}\label{proposizione 1 reversibilita}
For any $q \in L^2_0$,
\begin{equation}\label{reversability BO}
\Phi^{bo} ( { S}_{rev}(q)) = \mathcal S_{rev} ( \Phi^{bo}(q))\,. 
\end{equation}
As a consequence, ${ S}_{rev} \circ \Psi^{bo} = \Psi^{bo} \circ \mathcal S_{rev}$ and by the chain rule, for any $q \in L^2_0(\T)$ and $w \in h^0_0$,
$$
(d_{{ S}_{rev}(q)} \Phi^{bo}) \circ { S}_{rev} = \mathcal S_{rev} \circ d_q \Phi^{bo}\,, \qquad  (d_{{\mathcal S}_{rev}(w)} \Psi^{bo}) \circ { \mathcal S}_{rev} =  S_{rev} \circ d_w \Psi^{bo}\,.
$$ 
\end{proposition}
\begin{proof} By Remark \ref{extension + sym}(ii) and the definition  \eqref{definition z_n} of $(z_n(q))_{n \ne 0}$, one infers that $z_n(q_*) = z_{-n}(q)$ 
for any $n \ne 0$ and $q \in L^2_0$.
Identity \eqref{reversability BO} then follows from the definition \eqref{def Phi^bo} of $\Phi^{bo}$.
\end{proof}


\section{Properties of pseudodifferential and paradifferential calculus}\label{appendice B}
In this appendix we record some well known facts about pseudodifferential and paradifferential calculus on the torus $\T$. We refer to \cite{Met} for further details.
For the convenience of the reader, we recall some of the notations, introduced in the main body of the paper.

Let $\chi \in C^\infty(\R^2, \R)$ be an admissible cut-off function. It means that $\chi$ is an even function
and that there exist
$0 < \e_1 < \e_2 < 1$ so that for any $(\vartheta, \eta) \in \R^2$ and $\alpha, \beta \in \Z_{\ge 0}$,
\begin{equation}\label{cut-off paraprodotto}
\chi(\vartheta, \eta) = 1, \quad \forall |\vartheta| \leq \e_1 +  \e_1|\eta|\,, \qquad \chi(\vartheta, \eta) = 0, \quad \forall |\vartheta| \geq \e_2 +  \e_2 |\eta|\,,
\end{equation}
\begin{equation}\label{order 0 cut-off}
|\partial^\alpha_\vartheta \partial^\beta_\eta \chi(\vartheta, \eta)) | \le C_{\alpha, \beta} (1 +  |\eta|)^{-\alpha -\beta}\,.
\end{equation}
For any $a \in H^1_\C$, the paraproduct $T_a u$ of the function $a$  with $u \in L^2_\C$ (with respect to the cut-off function $\chi$) is defined as
\begin{equation}\label{definition prarproduct}
(T_a u) (x) := \sum_{k, n \in \Z} \chi(k, n) a_k u_n e^{\ii (k + n) x} \, ,
\end{equation}
where $u_n$, $n \in \Z$, denote the Fourier coefficients of $u$, $u_n = \frac{1}{2\pi}\int_0^{2\pi} u(x) e^{-  \ii n x}\, d x$. 
Note that if $a$, $u$ are real valued, $T_a u$ is real valued as well since $\chi$ is real valued and even.
Given any $s, s' \in \Z$, we denote by ${\cal B}(H^s_\C, H^{s'}_\C)$ the Banach space of all bounded linear operators
$H^s_\C \to H^{s'}_\C$, endowed with the operator norm $\| \cdot \|_{{\cal B}(H^s_\C, H^{s'}_\C)}$. 
In case $s=s',$ we also write ${\cal B}(H^s_\C)$ instead of ${\cal B}(H^s_\C, H^{s}_\C)$. Given any linear operator $A \in {\cal B}(H^s_\C, H^{s'}_\C)$,
we denote by $A^\top$ the transpose of $A$ with respect to the standard $L^2$-inner product. It is an element in ${\cal B}((H^{s'}_\C)^*, (H^{s}_\C)^*)$
where $(H^{s}_\C)^*$ denotes the dual of $H^{s}_\C$.

\begin{lemma}\label{primo lemma paraprodotti}
$(i)$ For any $s \in \Z_{\ge 0}$ and $a \in H^1_\C,$ the linear operator $T_a: u \mapsto  T_a u$ is in ${\cal B}(H^s_\C)$. Furthermore 
the linear map $H^1_\C \to  {\cal B}(H^s_\C),$ $a \mapsto T_a$, is bounded,  $\| T_a\|_{{\cal B}(H^s_\C)} \lesssim_s \| a \|_{1}$. 

\noindent
$(ii)$ Let $a \in H^{s_1}_\C, b \in H^{s_2}_\C$ and $s_1, s_2 \in \Z_{ \geq 1}$. Then 
$$
a b = T_a b + T_b a + {\cal R}^{(B)}(a, b) \, ,
$$
where the bilinear map ${\cal R}^{(B)}: H^{s_1}_\C \times H^{s_2}_\C \to H^{s_1 + s_2 - 1}_\C$, $(a, b) \mapsto {\cal R}^{(B)}(a, b)$, is continuous 
and satisfies the estimate 
$$
\| {\cal R}^{(B)}(a, b)\|_{s_1 + s_2 -  1} \lesssim_{s_1, s_2} \| a \|_{s_1} \| b \|_{s_2} \, . 
$$

\noindent
$(iii)$ Let $a \in H^\rho_\C$ with $\rho \in \Z_{ \ge 2}$. Then for any $s \geq 0$, $T_a^\top - T_{ a} \in {\cal B}(H^s_\C, H^{s + \rho - 1}_\C)$ and
$$
\| T_a^\top - T_{ a}\|_{{\cal B}(H^s_\C, H^{s + \rho - 1}_\C)}  \lesssim_{s, \rho}  \| a \|_{\rho} \, .
$$

\noindent
$(iv)$ Let $a, b \in H^\rho_\C$ with $\rho \in \Z_{ \ge 1}$. Then for any $s \geq 0$, $T_a \circ T_b - T_{ab} \in {\cal B}\big( H^s_\C , H^{s + \rho - 1}_\C \big)$ and 
$$
\| T_a \circ T_b - T_{ab} \|_{{\cal B}(H^s_\C, H^{s + \rho - 1}_\C)} \lesssim_{s, \rho} \| a \|_\rho \| b \|_\rho \, . 
$$
\end{lemma}

Recall that for any $k \ge 1$, $\partial_x^{-k} : H^s_{ \C} \to H^{s+k}_{0,\C}$ is  the bounded linear operator, defined by
$$
\partial_x^{-k}[e^{ \ii nx}] = \frac{1}{( \ii n)^k} e^{\ii nx}\, , \quad \forall n \ne 0 \, ,  
\qquad \quad   \partial_x^{-k}[1] = 0 \, .
$$ 
\begin{lemma}\label{lemma composizione pseudo}
$(i)$ Let $k, \ell \in \Z_{\geq 0}$ and $a \in C^\infty(\T, \C)$. Then for any $s \in \Z_{\ge 0}$ and $N \in \N$ with $N \geq  k+ \ell$, the composition
$\partial_x^{- k} \circ a \partial_x^{- \ell}$ is a bounded linear operator  $H^s_\C \to H^{s+k+\ell}_{0,\C}$, admitting an expansion of the form
$$
\partial_x^{- k} \circ a  \partial_x^{- \ell} =  \sum_{j = 0}^{N - k - \ell} C_j(k, \ell) \,  (\partial_x^j a ) \, \partial_x^{- k -\ell - j} + {\cal R}^{\psi do}_{N, k, \ell}(a) \, ,
$$
where $C_j(k, \ell)$, $0 \le j \le N-k-\ell$, are real constants with $C_0(k, \ell) = 1$, and the remainder ${\cal R}^{\psi do}_{N, k, \ell}(a)$ is a bounded linear operator 
$H^s_\C \to  H^{s + N + 1}_{0,\C}$,
 satisfying the estimate 
\begin{equation}\label{stima resti lemma astratto OPS}
\| {\cal R}^{\psi do}_{N, k, \ell}(a) \|_{{\cal B}(H^s_\C, H^{s + N + 1}_{0,\C})} \lesssim_{s, N}  \| a \|_{s + 2N }\,.
\end{equation} 
In particular for any $N \ge 1$,
$
\partial_x^{-1} \circ a =  \sum_{j= 0}^{N -1} (-1)^j\,  
(\partial_x^j a ) \, \partial_x^{- 1 - j} + 
{\cal R}^{\psi do}_{N, 1 , 0}(a)$ where
for any $h \in H^s$,
$$
{\cal R}^{\psi do}_{N, 1 , 0}(a)[h] = 
(-1)^N \partial_x^{-1} [(\partial_x^N a) \partial_x^{-N} h] + 
(\partial_x^{-1}  a) \langle h | 1 \rangle \, .
$$

\noindent
$(ii)$ Let $k, \ell \in \Z_{\geq 0}$ and $N \geq k + \ell$. There exists a constant $\sigma_N > N - k - \ell +1$ so that for any $a \in H^{\sigma_N}_\C$
and any $s \in \Z_{\ge 0}$, the composition $\partial_x^{- k} \circ T_a  \circ \partial_x^{- \ell}$ is a bounded linear operator 
$H^s_\C \to H^{s+k+\ell}_\C$ which admits an expansion of the form
$$
\partial_x^{- k} \circ T_a  \circ \partial_x^{- \ell} =  
\sum_{j= 0}^{N - k - \ell} C_j(k, \ell) T_{\partial_x^j a} \partial_x^{- k - \ell - j} + {\cal R}_{N, k, \ell}^{(B)}(a) \, ,
$$
where $C_j(k, \ell)$, $1 \le j \le N-k-\ell$, are the same constants as in item (i) and for any $s \geq 0$, the remainder ${\cal R}_{N, k, \ell}^{(B)}(a)$
is a bounded linear operator $H^s_\C \to  H^{s + N+1}_\C$, satisfying the estimate 
\begin{equation}\label{stima resti lemma astratto OPS}
\| {\cal R}_{N, n, k}^{(B)}(a) \|_{{\cal B}(H^s_\C, H^{s + N + 1}_\C)} \lesssim_{s, N}  \| a \|_{ \sigma_N } \, . 
\end{equation}
\end{lemma}

The next results concern Hankel operators which are defined as follows.
For any $s \ge 0$, let $H^s_\pm := H^s_\C \cap L^2_\pm$ where $L^2_\pm$ denote the Hardy spaces of $\T$, given by
$$
L^2_\pm \equiv  H^0_\pm := \{ f \in L^2_\C : \, \langle f \, | \, e^{\ii nx} \rangle = 0 \ \forall \, \pm n < 0\} .
$$
Furthermore, denote by 
$$
\Pi \equiv \Pi^+: L^2_\C \to L^2_+,  \, f \mapsto \sum_{n \ge 0} \widehat f(n) e^{\ii nx}\, , \qquad
\Pi^-: L^2_\C \to L^2_-, \,  f \mapsto \sum_{n \le 0} \widehat f(n) e^{\ii nx}\, ,
$$ the Szeg\H{o} projectors.
For any $u \in H^s_\C$ with $s > 1/2$, the Hankel operators $H_u \equiv H^+_u$ and $H^-_u$ are given by
$$
H^+_u : L^2_- \to L^2_+, \, f \mapsto \Pi^+(uf) \, , \qquad  H^-_u : L^2_+ \to L^2_-, \, f \mapsto \Pi^-(uf) \, .
$$
The following results follow from smoothing properties of Hankel operators, obtained in \cite{GKT5}.
For the convenience of the reader we include the proof, given in  \cite{GKT5}.
\begin{lemma}\label{smoothing Hankel}
For any $u \in H^{s+\alpha}_\C$ with $\alpha \ge 0$ and $f \in H^s_-$, the following holds:
\begin{itemize}
\item[(i)] If $s > 1/2$, then $\| H^+_u f \|_{s + \alpha} \lesssim_{s, \alpha} \|u\|_{s+\alpha} \, \| f \|_s$ .
\item[(ii)] If $s=1/2$ and $\e > 0$, then $\| H^+_u f \|_{\frac 12 + \alpha - \varepsilon} \lesssim_{ \alpha} \|u\|_{\frac 12 +\alpha} \, \| f \|_{\frac 12}$ .
\item[(iii)] If $0 \le  s < 1/2$ and $\alpha \ge 1/2 - s$, then  
$$
\| H^+_u f \|_{s + \beta} \lesssim_{s, \alpha} \|u\|_{s+\alpha} \, \| f \|_s \, , \qquad  \beta:=  \alpha - (1/2 - s) \, .
$$
\end{itemize}
Corresponding results hold for the operator $H_u^-$.
\end{lemma}
\begin{proof}
Since the estimates for $H_u^-$ can be proved in the same way as the ones for $H_u^+$, we only consider the latter ones.
Let $u = \sum_{k \in \Z} \widehat u(k) e^{ikx} \in H^{s+\alpha}_\C$ and $f = \sum_{p \ge 0} \widehat f(-p) e^{-ipx} \in H^s_-$. 
Then with $n := - (k+p)$,
$$
g(x) := \Pi \big( \sum_{k \in \Z, p \ge 0}  \widehat u(-k) \widehat f(-p) e^{-(k+p)x}  \big) = \sum_{n \ge 0} \widehat g(n) e^{inx} 
$$
where
$$
 \widehat g(n) := \sum_{p \ge 0} \widehat u(n+p) \widehat f(-p) \, , \qquad \forall \, n \ge 0 \, .
$$
By Cauchy Schwarz, one obtains
$$
 | \widehat g(n) |^2 \le \|f \|_s^2 \sum_{p \ge 0} \frac{1}{\langle p \rangle^{2s} }|\widehat u(n+p) |^2 
$$
and thus
$$
\| g \|^2_{s+\alpha} = \sum_{n \ge 0} \langle n \rangle^{2(s+\alpha)}  | \widehat g(n) |^2 
\le  \|f \|_s^2 \sum_{\ell \ge 0} |\widehat u(\ell) |^2   \sum_{p + n = \ell} \frac{\langle n \rangle^{2(s+\alpha)}}{\langle p \rangle^{2s} }
$$
(i) In the case $s > 1/2$, one has $2s > 1$ and hence
$$
\sum_{p + n = \ell} \frac{\langle n \rangle^{2(s+\alpha)}}{\langle p \rangle^{2s} } 
\le \langle \ell \rangle^{2(s+\alpha)} \sum_{0 \le p \le \ell} \frac{1}{\langle p \rangle^{2s}}
\lesssim_s \langle \ell \rangle^{2(s+\alpha)} \big(1 + \frac{1}{\langle \ell \rangle^{2s -1}} \big) ,
$$
yielding
$$
\| g \|_{s+\alpha} \lesssim_s \|f \|_s  \, \| u \|_{s+\alpha} .
$$
(ii) In the case $s = 1/2$, one has $2s = 1$ and hence
$$
\sum_{p + n = \ell} \frac{\langle n \rangle^{2(s+ \alpha - \varepsilon)}}{\langle p \rangle^{2s} } 
\le \langle \ell \rangle^{2(\frac 12+\alpha - \varepsilon)} \sum_{0 \le p \le \ell} \frac{1}{\langle p \rangle}
\lesssim  \langle \ell \rangle^{2(\frac 12+\alpha)}  \langle \ell \rangle^{ - 2\e} \log \langle \ell \rangle ,
$$
yielding
$$
\| g \|_{\frac 12+\alpha - \e} \lesssim \|f \|_\frac 12  \, \| u \|_{\frac 12+\alpha} .
$$
(iii) In the case $0 \le s < 1/2$, one has $2s < 1$ and hence $1-2s > 0$. Therefore
$$
\sum_{p + n = \ell} \frac{\langle n \rangle^{2(s+\beta)}}{\langle p \rangle^{2s} } 
\le \langle \ell \rangle^{2(s+\beta)} \sum_{0 \le p \le \ell} \frac{1}{\langle p \rangle^{2s}}
\lesssim_s  \langle \ell \rangle^{2(s+\beta)} \langle \ell \rangle^{2(\frac 12 - s)} =  \langle \ell \rangle^{2(s + \alpha)} \, ,
$$
where we used that by definition $\beta + 1/2 = s + \alpha$.  It thus follows that
$\| g \|_{s+ \beta} \lesssim_s \|f \|_s  \, \| u \|_{s+\alpha}$.
\end{proof}

\begin{corollary}\label{Hankel infinitely smoothing}
Assume that $u \in \cap_{k \ge 0} H^k_\C$. Then $H^+_u$ and $H^-_u$ are infinitely smoothing, i.e.,  
$$
H^+_u \in \mathcal B(H^s_- , \,  H^{s+ N}_+) \, , \qquad  H^-_u \in \mathcal B (H^s_+ , \,  H^{s+N}_- ) \, , 
\qquad \forall \, s \ge 0, \ N \ge 1 \, .
$$
\end{corollary}

\smallskip

Finally, we record the following well known tame estimates of products of functions in Sobolev spaces. 
\begin{lemma}\label{lemma interpolation}
For any $s \in \Z_{\ge 1}$,
$$
\| u v \|_s \lesssim_s \| u \|_s \| v \|_1 + \| u \|_1 \| v \|_s\,, \quad \forall u, v \in H^s_\C\,.
$$
\end{lemma}

\vspace{1.0cm}

\noindent
T. Kappeler, 
Institut f\"ur Mathematik, 
Universit\"at Z\"urich, Winterthurerstrasse 190, CH-8057 Z\"urich;\\
${}\qquad$  email: thomas.kappeler@math.uzh.ch \\

\noindent
R. Montalto, 
Department of Mathematics,  
University of Milan, Via Saldini 50, 20133, Milano, Italy;\\
${}\qquad$ email: riccardo.montalto@unimi.it


\begin{thebibliography}{99}

\bibitem{Ben} {\sc T. Benjamin}, {\em Internal waves of permanent form in fluids of great depth},
J. Fluid Mech. 29(1967), 559-592.

\bibitem{DA} {\sc R. Davis, A. Acrivos}, {\em Solitary internal waves in deep water}, 
J. Fluid Mech.  29(1967), 593-607.

\bibitem{GK1} {\sc P. G\'erard, T. Kappeler}, 
{\em On the integrability of the Benjamin-Ono equation on the torus},
Comm. Pure Appl. Math. 74(2021), no. 8, 1685-1747.

\bibitem{GKT1} {\sc P. G\'erard, T. Kappeler, P. Topalov}, 
{\em Sharp well-posedness results for the Benjamin-Ono equation in $H^{s}(\mathbb T, \mathbb R)$ and qualitative
properties of its solutions}, to appear in Acta Math., arXiv:2004.04857.

\bibitem{GKT2} {\sc P. G\'erard, T. Kappeler, P. Topalov},
{\em On the spectrum of the Lax operator of the Benjamin-Ono equation on the torus}, J. Funct. Anal. 279(2020), no. 12,
108762, 75 pp. 

\bibitem{GKT6} {\sc P. G\'erard, T. Kappeler, P. Topalov}, 
{\em On the Benjamin--Ono equation on $\mathbb T$ and its periodic and quasiperiodic solutions},
to appear in J. of Spectr. Theory, arXiv:2103.0929.

\bibitem{GKT3} {\sc P. G\'erard, T. Kappeler, P. Topalov}, 
{\em On the analytic Birkhoff normal form of the Benjamin-Ono equation and applications},
arXiv:2103.0798.

\bibitem{GKT4} {\sc P. G\'erard, T. Kappeler, P. Topalov}, 
{\em Analytic extension of the Birkhoff map of the Benjamin-Ono equation in the large},
preprint.

\bibitem{GKT5} {\sc P. G\'erard, T. Kappeler, P. Topalov}, 
{\em On smoothing properties and Tao's gauge transform of the Benjamin-Ono equation on the torus}, arXiv:2109.00610.

\bibitem{GK} {\sc B. Gr\'ebert, T. Kappeler}, {\em The defocusing NLS equation and its normal form}, 
 EMS Series of Lectures in Mathematics, European Mathematical Society (EMS), Z\"urich, 2014.

\bibitem{Kappeler-Montalto} {\sc T. Kappeler, R. Montalto}, {\em Canonical coordinates with tame estimates for the defocusing NLS equation on the circle},
 Int. Math. Res. Not. IMNR, 2018, issue 5, 6 March 2018, 1473-1531.
 
\bibitem{KM1} {\sc T. Kappeler, R. Montalto}, {\em Normal form coordinates for the KdV equation having expansions in terms of pseudodifferential operators}, 
Comm. in Math. Phys. 375(2020), no. 1, 833-913.

\bibitem{KM2} {\sc T. Kappeler, R. Montalto}, {\em On the stability of periodic multi-solitons of the KdV equation}, 
Comm. Math. Phys. 385(2021), no. 3, 1871-1956.


\bibitem{KP} {\sc T. Kappeler, J. P{\"o}schel}, {\em KdV {\&} KAM},  Ergebnisse der Mathematik und ihrer Grenzgebiete. 3. Folge. 
A Series of Modern Surveys in Mathematics, 45. Springer-Verlag, Berlin, 2003.

\bibitem{KST2} {\sc T. Kappeler, B. Schaad, P. Topalov}, {\em Qualitative features of periodic solutions of KdV}, 
Comm. in PDEs 38(2013), no. 9, 1626-1673.

\bibitem{KST3} {\sc T. Kappeler, B. Schaad, P. Topalov}, {\em Semi-linearity  of the nonlinear Fourier transform of the defocusing NLS equation},
Int. Math. Res. Notices, IMRN vol. 2016, no. 23, 7212-7229.

\bibitem{Krichever} {\sc I. Krichever}, {\em Perturbation theory in periodic problems for two-dimensional integrable systems}, 
Soviet Scientific Reviews C. Math. Phys. 9(1991), 1-103.


\bibitem{K} {\sc S. Kuksin}, {\em Analysis of Hamiltonian PDEs}, Oxford University Press, 2000.

\bibitem{Kuksin-Perelman} {\sc S. Kuksin, G. Perelman}, {\em Vey theorem in infinite dimensions and its applications to KdV}, 
Disc. Cont. Dyn. Syst. Serie A, 27(2010), 1-24. 

\bibitem{Met} {\sc G. M\'etivier,} {\em Para-differential calculus and applications to the Cauchy problem for nonlinear systems},
Publications of the Scuola Normale Superiore, Book 5, Edizioni della Normale, 2008.

\bibitem{S} {\sc J.-C. Saut}, {\em Benjamin-Ono and intermediate long wave equations: modeling, IST, and PDE},
Fields Institute Communications, 83, Springer, 2019.

\end{thebibliography}
\end{document}